\documentclass[leqno,12pt]{article}

\usepackage{amssymb, amsthm, amscd, amsmath}
\usepackage{stmaryrd}
\usepackage{mathrsfs, mathtools}
\usepackage{epic,eepic}
\usepackage[all]{xy}
\usepackage{color}

\theoremstyle{plain}
\newtheorem{theorem}{Theorem}[section]
\newtheorem{lemma}[theorem]{Lemma}
\newtheorem{proposition}[theorem]{Proposition}

\newtheorem{problem}[theorem]{Problem}

\newtheorem{weakisomGC}{Weak Isom-version Conjecture}

\newtheorem{solweakisomGC}{Solvable Weak Isom-version Conjecture}

\newtheorem{homeconj}{Homeomorphism Conjecture}

\newtheorem{solhomeconj}{Solvable Homeomorphism Conjecture}

\renewcommand{\theprob}

\renewcommand{\thequest}

\newtheorem{corollary}[theorem]{Corollary}
\theoremstyle{definition}
\newtheorem{definition}[theorem]{Definition}

\newtheorem{conditiona}{Condition A}

\newtheorem{conditionb}{Condition B}

\newtheorem{conditionc}{Condition C}

\newtheorem{remarkA}{Remark}

\newtheorem{remarkB}{Remark}

\newtheorem{remarkC}{Remark}

\newcommand{\be}{\begin{enumerate}}
\newcommand{\ee}{\end{enumerate}}
\newcommand{\bcd}{\[\begin{CD}}
\newcommand{\ecd}{\end{CD}\]}
\newcommand{\bit}{\begin{itemize}}
\newcommand{\eit}{\end{itemize}}
\newcommand{\bq}{\begin{quote}}
\newcommand{\eq}{\end{quote}}
\newcommand{\bpf}{\begin{proof}}
\newcommand{\epf}{\end{proof}}

\newcommand{\mcA}{\mathcal{A}}

\newcommand{\mcC}{\mathcal{C}}

\newcommand{\mcF}{\mathcal{F}}

\newcommand{\mcM}{\mathcal{M}}
\newcommand{\mcN}{\mathcal{N}}
\newcommand{\mcO}{\mathcal{O}}
\newcommand{\mcP}{\mathcal{P}}

\newcommand{\mcS}{\mathcal{S}}

\newcommand{\mbA}{\mathbb{A}}

\newcommand{\mbF}{\mathbb{F}}
\newcommand{\mbG}{\mathbb{G}}

\newcommand{\mbN}{\mathbb{N}}

\newcommand{\mbP}{\mathbb{P}}
\newcommand{\mbQ}{\mathbb{Q}}
\newcommand{\mbR}{\mathbb{R}}

\newcommand{\mbZ}{\mathbb{Z}}

\newcommand{\mfM}{\mathfrak{M}}

\newcommand{\mfP}{\mathfrak{P}}

\newcommand{\mfS}{\mathfrak{S}}
\newcommand{\mfT}{\mathfrak{T}}

\newcommand{\msC}{\mathscr{C}}

\newcommand{\msG}{\mathscr{G}}

\newcommand{\SR}{\stackrel}
\newcommand{\mr}{\mathrm}

\newcommand{\defeq}{\SR{\mr{def}}{=}}

\newcommand{\spec}{\mathrm{Spec} \, }

\newcommand{\migi}{\rightarrow}

\newcommand{\isom}{\stackrel{\sim}{\migi}}

\newcommand{\migiinje}{\hookrightarrow}

\newcommand{\migisurj}{\twoheadrightarrow}

\newcommand{\invlim}{\varprojlim}
\newcommand{\chr}{ {\text{\rm char}} }

\newcommand{\et}{\rm \text{\'et}}

\newcommand{\bp}{\bullet}
\newcommand{\muge}{\infty}

\begin{document}


\title{Moduli Spaces of Fundamental Groups of Curves in Positive Characteristic I}
\author{\sc Yu Yang}
\date{}
\maketitle


\begin{abstract}
In this series of papers, we investigate a new anabelian phenomenon of curves over algebraically closed fields of positive characteristic which shows that the topological structures of moduli spaces of curves can be understood from open continuous homomorphisms of fundamental groups of curves. Let $p$ be a prime number, and let $\overline M_{g, n}$ be the coarse moduli space of the moduli stack over an algebraic closure of the finite field $\mbF_{p}$ classifying pointed stable curves of type $(g, n)$. We introduce a topological space $\overline \Pi_{g, n}$ which we call {\it the moduli space of admissible fundamental groups of pointed stable curves of type $(g, n)$ in characteristic $p$}, whose underlying set consists of the set of isomorphism classes of the admissible fundamental groups of pointed stable curves of type $(g, n)$, and whose topology is determined by the sets of finite quotients of the admissible fundamental groups. By introducing a certain equivalence relation $\sim_{fe}$ on the underlying topological space $|\overline M_{g, n}|$ of $\overline M_{g, n}$ determined by Frobenius actions, we obtain a topological space $\overline \mfM_{g, n}\defeq |\overline M_{g, n}|/\sim_{fe}$ whose topology is induced by the Zariski topology of $\overline M_{g, n}$. Moreover, there is a natural continuous map $$\pi_{g,n}^{\rm adm}: \overline \mfM_{g, n} \migi \overline \Pi_{g, n}.$$  The topological space $\overline \Pi_{g, n}$ gives us a new insight into the theory of the anabelian geometry of curves over algebraically closed fields of characteristic $p$ based on the following anabelian philosophy: Every topological property concerning $\overline \Pi_{g, n}$ is equivalent to an anabelian property concerning pointed stable curves over algebraically closed fields of characteristic $p$. Furthermore, {\it the Homeomorphism Conjecture} says that $\pi_{g,n}^{\rm adm}$ is a homeomorphism, which is the main conjecture of the theory developed in the present series of papers. The Homeomorphism Conjecture generalizes all the conjectures in the theory of anableian geometry of curves over algebraically closed fields of characteristic $p$, and means that moduli spaces of curves {\it can be reconstructed group-theoretically as topological spaces} from the admissible fundamental groups of curves. One of main results of the present series of papers says that the Homeomorphism Conjecture holds when $\text{dim}(\overline M_{g, n})=1$ (i.e., $(g, n)=(0,4)$ or $(g, n)=(1,1)$). In the present paper, we establish two fundamental tools to analyze the geometric behavior of curves from open continuous homomorphisms of admissible fundamental groups, which play central roles in the theory developed in the series of papers. Moreover, we prove that $\pi_{0,n}^{\rm adm}([q])$ is a closed point of $\overline \Pi_{0,n}$ when $[q]$ is a closed point of $\overline \mfM_{0, n}$. In particular, we obtain that the Homeomorphism Conjecture holds when $(g, n)=(0, 4)$. 


\end{abstract}


\tableofcontents


\section*{Introduction}\label{sec-1} 

In the present paper, we study the anabelian geometry of curves over algebraically closed fields of positive characteristic. Let $p$ be a prime number, and let $$X^{\bp}=(X, D_{X})$$ be a pointed stable curve of  type $(g_{X}, n_{X})$ over a field $k$ of characteristic $\chr(k)$, where $X$ denotes the underlying curve, $D_{X}$ denotes the set of marked points, $g_{X}$ denotes the genus of $X$, and $n_{X}$ denotes the cardinality $\#D_{X}$ of $D_{X}$. First, we explain some background of the anabelian geometry of curves. Suppose that $X^{\bp}$ is smooth over $k$. When $k$ is an {\it arithmetic field} (e.g. number field, $p$-adic field, finite field, etc.), A. Grothendieck suggested a theory of arithmetic geometry called {\it anabelian geometry} (cf. \cite{G1}, \cite{G2}), roughly speaking, in the case of curves, whose ultimate goal is the following question.
\begin{quote}
Can we recover the geometric information of $X^{\bp}$ group-theoretically from various versions of its fundamental group?
\end{quote}
\noindent
The various formulations of this question are called {\it Grothendieck's anabelian conjecture for curves} or the Grothendieck conjecture, for short. The Grothendieck conjecture has been proven in many cases. For example, the conjecture was proved by H. Nakamura, A. Tamagawa, and S. Mochizuki in the case of number fields  (cf. \cite{Nakm1}, \cite{Nakm2}, \cite{T0}, \cite{M0}), and was proved by Tamagawa, J. Stix, and M. Sa\"idi-Tamagawa in the case of finitely generated fields over the finite field $\mbF_{p}$ (cf. \cite{T0}, \cite{Sti1}, \cite{Sti2}, \cite{ST1}, \cite{ST2}). All the proofs of the Grothendieck conjecture for curves over arithmetic fields mentioned above require the use of {\it the highly non-trivial outer Galois representations} induced by the fundamental exact sequences of fundamental groups.

Next, let us return to the case where $X^{\bp}$ is an arbitrary pointed stable curve, and suppose that $k$ is {\it an algebraically closed field}. By choosing a suitable base point of $X^{\bp}$, we have the admissible fundamental group $$\pi_{1}^{\rm adm}(X^{\bp})$$ of $X^{\bp}$ (cf. Definition \ref{def-2}). In particular, if $X^{\bp}$ is smooth over $k$, then $\pi_{1}^{\rm adm}(X^{\bp})$ is naturally isomorphic to the tame fundamental group $\pi^{\rm t}_{1}(X^{\bp})$. Write $\pi_{1}^{\rm adm}(X^{\bp})^{p'}$ for the maximal prime-to-$p$ quotient of $\pi_{1}^{\rm adm}(X^{\bp})$ if $\chr(k)=p$. The profinite group $\pi_{1}^{\rm adm}(X^{\bp})$ (resp. $\pi_{1}^{\rm adm}(X^{\bp})^{p'}$) is isomorphic to the profinite completion (resp. pro-prime-to-$p$ completion) of the following group (cf. \cite[Th\'eor\`eme 2.2 (c)]{V})  $$\langle a_{1}, \dots, a_{g_{X}}, b_{1}, \dots, b_{g_{X}}, c_{1}, \dots, c_{n_{X}} \ | \ \prod_{i=1}^{g_{X}}[a_{i}, b_{i}]\prod_{j=1}^{n_{X}}c_{j}=1\rangle$$ when $\chr(k)=0$ (resp. $\chr(k)=p$). In the case of algebraically closed fields of characteristic $0$, since the admissible fundamental groups of curves depend only on the types of curves, the anabelian geometry of curves does not exist in this situation. On the other hand, if $\chr(k)=p$, the situation is quite different from that in characteristic $0$. The admissible fundamental group $\pi_{1}^{\rm adm}(X^{\bp})$ is very mysterious and its structure is no longer known. In the remainder of the introduction, we assume that $k$ is an algebraically closed field of characteristic $p$. 

Since the late 1990s, some developments of F. Pop, M. Raynaud, Sa\"idi, Tamagawa, J. Tong, and the author (cf. \cite{PS}, \cite{R2},  \cite{T1}, \cite{T2}, \cite{T3}, \cite{T4}, \cite{To}, \cite{Y1}, \cite{Y2}, \cite{Y3}) showed evidence for very strong {\it anabelian phenomena} for curves over {\it algebraically closed fields of positive characteristic.} In this situation, the Galois group of the base field is trivial, and the {\it arithmetic} fundamental group coincides with the {\it geometric} fundamental group, thus there is a total absence of a Galois action of the base field. This kinds of anabelian phenomenon go beyond Grothendieck's anabelian geometry (because no Galois actions exist), and this is the reason that we do not have an explicit description of the geometric fundamental group of any pointed stable curve in positive characteristic. Moreover, we may think that the anabelian geometry of curves over algebraically closed fields of characteristic $p$ is a theory based on the following rough consideration: The admissible fundamental group of a pointed stable curve over an algebraically closed field of characteristic $p$ must encode {\it ``moduli"} of the curve.

Let us explain the anabelian geometry of curves over algebraically closed fields of positive characteristic from the point of view of moduli spaces. Let $\overline \mbF_{p}$ be an algebraically closed field of $\mbF_{p}$, and let $\overline \mcM_{g, n}$ be the moduli stack over $\overline \mbF_{p}$ classifying pointed
stable curves of type $(g, n)$, $\mcM_{g, n} \subseteq \overline \mcM_{g, n}$ the open substack classifying smooth pointed stable curves, $\overline M_{g, n}$ the coarse moduli space of $\overline \mcM_{g,n}$, and $M_{g, n}$ the coarse moduli space of  $\mcM_{g, n}$. Let $q \in \overline M_{g, n}$ be an arbitrary point, $k(q)$ the residue field of $\overline M_{g, n}$, and $k_{q}$ an algebraically closed field which contains $k(q)$. Then the composition of natural morphisms $$\spec k_{q} \migi \spec k(q) \migi \overline M_{g, n}$$ determines a pointed stable curve $X^{\bp}_{k_{q}}$ of type $(g, n)$ over $k_{q}$. In particular, if $k_{q}$ is an algebraic closure of $k(q)$, we shall write $X_{q}^{\bp}$ for $X^{\bp}_{k_{q}}$. Write $\pi_{1}^{\rm adm}(X^{\bp}_{k_{q}})$ for the admissible fundamental group $X^{\bp}_{k_{q}}$ and $\Gamma_{X_{k_{q}}^{\bp}}$ for the dual semi-graph of $X_{k_{q}}^{\bp}$. Since the isomorphism classes of $\pi_{1}^{\rm adm}(X^{\bp}_{k_{q}})$ and $\Gamma_{X_{k_{q}}^{\bp}}$ do not depend on the choice of $k_{q}$, we shall denote by $$\pi_{1}^{\rm adm}(q), \ \Gamma_{q}$$ the admissible fundamental group $\pi_{1}^{\rm adm}(X^{\bp}_{k_{q}})$ and the dual semi-graph $\Gamma_{X_{k_{q}}^{\bp}}$, respectively. Moreover, we write $v(\Gamma_{q})$, $e^{\rm op}(\Gamma_{q})$, and $e^{\rm cl}(\Gamma_{q})$ for the set of vertices of $\Gamma_{q}$, the set of open edges of $\Gamma_{q}$, and the set of closed edges of $\Gamma_{q}$, respectively.

Let $\overline \Pi_{g, n}$ be the set of isomorphism classes (as profinite groups) of admissible fundamental groups of pointed stable curves of type $(g, n)$ over algebraically closed fields of characteristic $p$. Then the fundamental group functor $\pi_{1}^{\rm adm}$ induces a natural sujective map from the underlying topological space $|\overline M_{g,n}|$ of $\overline M_{g,n}$ to $\overline \Pi_{g, n}$ as follows: $$[\pi_{1}^{\rm adm}]: |\overline M_{g, n}| \migisurj \overline \Pi_{g, n}, \ q \mapsto [\pi_{1}^{\rm adm}(q)],$$ where $[\pi_{1}^{\rm adm}(q)]$ denotes the isomorphism class of $\pi_{1}^{\rm adm}(q)$. Note that the map $[\pi_{1}^{\rm adm}]$ is not a bijection in general. For example, let $q$, $q'\in \overline M_{g, n}$ be arbitrary points such that $X_{q}\setminus D_{X_{q}}$ is isomorphic to $X_{q'}\setminus D_{X_{q'}}$ as schemes (e.g. $X_{q}^{\bp}$ is a Frobenius twist of $X_{q'}^{\bp}$). Then we have that $[\pi_{1}^{\rm adm}(q)]=[\pi_{1}^{\rm adm}(q')]$. On the other hand, we introduces an equivalence relation $\sim_{fe}$ on $|\overline M_{g,n}|$ which we call {\it Frobenius equivalence} (cf. \cite[Definition 3.4]{Y7} or Definition \ref{fe} of the present paper). Roughly speaking, $q_{1} \sim_{fe} q_{2}$ for any points $q_{1}, q_{2}\in \overline M_{g, n}$ if there exists an isomorphism $\rho: \Gamma_{q_{1}} \isom \Gamma_{q_{2}}$ of dual semi-graphs of $X_{q_{1}}^{\bp}$ of $X_{q_{2}}^{\bp}$ such that the pointed stable curves $\widetilde X^{\bp}_{q_{1}, v_{1}}$ and $\widetilde X^{\bp}_{q_{2}, v_{2}}$ associated to $v_{1}$ and $v_{2} \defeq \rho(v_{1})$ (cf. Section \ref{sec-1}), respectively, are isomorphic as schemes for every $v_{1} \in v(\Gamma_{q_{1}})$.  In particular, when $q_{1}\in M_{g, n}$ (i.e., $X_{q_{1}}^{\bp}$ is a non-singular curve), then $q_{1} \sim_{fe} q_{2}$ if and only if $X_{q_{1}}\setminus D_{X_{q_{1}}}$ is isomorphic to $X_{q_{2}}\setminus D_{X_{q_{2}}}$ as schemes. Moreover, \cite[Proposition 3.7]{Y7} shows that $[\pi_{1}^{\rm adm}]$ factors through the following quotient set $$\overline \mfM_{g, n}\defeq |\overline M_{g, n}|/\sim_{fe}.$$ Then we obtain a natural surjective map $$\pi_{g, n}^{\rm adm}: \overline \mfM_{g, n} \migisurj \overline \Pi_{g, n}$$ induced by $[\pi_{1}^{\rm adm}]$.

One of the main conjectures in the theory of anabelian geometry of curves is the following weak Isom-version of the Grothendieck conjecture of curves over algebraically closed fields of characteristic $p$ (=the Weak Isom-version Conjecture):
\begin{weakisomGC}
We maintain the notation introduced above. Then the surjective map $$\pi_{g, n}^{\rm adm}: \overline \mfM_{g, n} \migisurj \overline \Pi_{g, n}, \ [q] \mapsto [\pi_{1}^{\rm adm}(q)],$$ is a bijection, where $[q]$ denotes the image of $q$ of the natural quotient map $|\overline M_{g, n}| \migi \overline \mfM_{g, n}$.
\end{weakisomGC}
\noindent
The Weak Isom-version Conjecture was essentially formulated by Tamagawa in the case of smooth pointed stable curves, and by the author in the case of arbitrary pointed stable curves (cf. \cite{T2}, \cite{Y7}), which means that the moduli spaces of curves in positive characteristic {\it can be reconstructed group-theoretically as sets} from admissible fundamental groups of pointed stable curves in positive characteristic. The Weak Isom-version Conjecture is very difficult, which was proved completely only in the case where $(g, n)=(0,4)$. More precisely, we have the following result obtained by Tamagawa and the author (cf. \cite[Theorem 0.2]{T2}, \cite[Theorem 3.8]{Y7}):
\begin{theorem}\label{them-0-1}
We maintain the notation introduced above. Write $\overline \mfM_{g, n}^{\rm cl}$ for the images of the set of closed points of $|\overline M_{g, n}|$. Then we have that $\pi_{0, n}^{\rm adm}(\overline \mfM_{0, n}^{\rm cl})\cap \pi_{0, n}^{\rm adm}(\overline \mfM_{0, n} \setminus\overline \mfM_{0, n}^{\rm cl})=\emptyset$, and that $$\pi_{0, n}^{\rm adm}|_{\overline \mfM^{\rm cl}_{0, n}}: \overline \mfM^{\rm cl}_{0, n} \migi \overline \Pi_{0, n}$$ is an injection. In particular, the Weak Isom-version Conjecture holds when $(g, n)=(0, 4)$. 
\end{theorem}

\begin{remarkA}
In other words, Theorem \ref{them-0-1} is equivalent to the following anabelian result:
\begin{quote}
 {\it Let $q_{1}$ $q_{2} \in \overline M_{0, n}$ be arbitrary points. Suppose that $q_{1}$ is {\it closed}, and that $\pi_{1}^{\rm adm}(q_{1})$ is isomorphic to $\pi_{1}^{\rm adm}(q_{2})$ as profinite groups. Then we have $q_{1} \sim_{fe} q_{2}$.}
\end{quote}
Suppose that $g$ is an arbitrary non-negative integer number. We also want to mention the following {\it finiteness theorem} (cf.  \cite{PS}, \cite{R2}, \cite{T4}, \cite{To}, \cite{Y1}): 
\begin{quote}
{\it Let $[q] \in \overline \mfM_{g, n}^{\rm cl}$. Then we have $\#((\pi_{g, n}^{\rm adm})^{-1}([\pi_{1}^{\rm adm}(q)]) \cap \overline \mfM_{g, n}^{\rm cl})<\muge.$}
\end{quote}
\end{remarkA}
\noindent
At the time of writing, almost all of the researches concerning the anabelian geometry of curves over algebraically closed fields of characteristic $p$ focus on the Weak Isom-version Conjecture, and the conjecture cannot give us any new insight into the anabelian phenomena of curves over algebraically closed fields of characteristic $p$. On the other hand, the results proved by the author in \cite{Y2} show that 
\begin{quote}
it is possible that the {\it topological structures} of moduli spaces of curves in positive characteristic can be reconstructed group-theoretically from the geometric fundamental groups of curves in positive characteristic.
\end{quote}
This is the main observation that motivated the theory developed in the present series of papers. 

From now on, we shall regard $\overline \mfM_{g, n}$ as a topological space whose topology is induced naturally by the Zariski topology of $|\overline M_{g, n}|$. Let $\msG$ be the category of finite groups and $G \in \msG$ a finite group. We put $$U_{\overline \Pi_{g, n}, G} \defeq \{ [\pi_{1}^{\rm adm}(q)] \in \overline \Pi_{g, n}\ | \  \text{Hom}_{\rm surj}(\pi_{1}^{\rm adm}(q), G) \neq \emptyset\},$$ where $\text{Hom}_{\rm surj}(-, -)$ denotes the set of surjective homomorphisms of profinite groups. We define a topological space  $$(\overline \Pi_{g, n}, O_{\overline \Pi_{g,n}})$$ group-theoretically from the set of isomorphism classes of admissible fundamental groups of pointed stable curves $\overline \Pi_{g, n}$, whose underlying set is $\overline \Pi_{g, n}$, and whose topology $O_{\overline \Pi_{g,n}}$ is generated by $\{U_{\overline \Pi_{g, n}, G}\}_{G \in \msG}$ as open subsets. For simplicity, we still use the notation $\overline \Pi_{g, n}$ to denote the topological space $(\overline \Pi_{g, n}, O_{\overline \Pi_{g, n}})$, and shall say $$\overline \Pi_{g, n}$$ {\it the moduli space of admissible fundamental groups of pointed stable curves of type $(g, n)$ over algebraically closed fields of characteristic $p$} or {\it the moduli space of admissible fundamental groups of type $(g, n)$ in  characteristic $p$}, for short.  Theorem \ref{continuous} (or Theorem \ref{continuousmap}) of the present paper implies that the surjective map $$\pi_{g, n}^{\rm adm}: \overline \mfM_{g, n} \migisurj \overline \Pi_{g, n}$$ is also a continuous map. Moreover, we pose the following conjecture, which is the main conjecture of the theory developed in the present series of papers:

\begin{homeconj}
We maintain the notation introduced above. Then we have that $$\pi_{g, n}^{\rm adm}: \overline \mfM_{g, n} \migisurj \overline \Pi_{g, n}$$ is a homeomorphism.
\end{homeconj}
\noindent
The Homeomorphism Conjecture means that the moduli spaces of curves in positive characteristic {\it can be reconstructed group-theoretically as topological spaces} from admissible fundamental groups of pointed stable curves in positive characteristic. This conjecture gives us a new insight into the theory of the anabelian geometry of curves over algebraically closed fields of characteristic $p$ based on the following philosophy:
\begin{quote}
The anabelian properties of pointed stable curves over algebraically closed fields of characteristic $p$ are equivalent to the topological properties of the topological space $\overline \Pi_{g, n}$.
\end{quote}
This new anabelian philosophy has raised a host of questions which cannot be seen if we only consider
the Weak Isom-version Conjecture (e.g. Problem \ref{prob-5-7} of the present paper). 

Next, let us explain the difference between the Weak Isom-version Conjecture and the Homeomorphism Conjecture from the aspect of group theory. The mean of {\it anabelian geometry} around the Weak Isom-version Conjecture (i.e., the theory developed in \cite{PS}, \cite{R2},  \cite{T1}, \cite{T2}, \cite{T3}, \cite{T4}, \cite{To}, \cite{Y1}, \cite{Y2}, \cite{Y3}) is the following:
\begin{quote}
Let $\mcF_{i}$, $i\in \{1, 2\}$, be a geometric object in a certain category and $\Pi_{\mcF_{i}}$ the fundamental group associated to $\mcF_{i}$. Then the set of isomorphisms of geometric objects $\text{Isom}(\mcF_{1}, \mcF_{2})$ can be understood from the set of isomorphisms of group-theoretical objects $\text{Isom}(\Pi_{\mcF_{1}}, \Pi_{\mcF_{2}})$. The term ``{\it anabelian}" means that the geometric properties can be determined by the isomorphism classes of the fundamental groups. On the other hand, we do not know the relation of $\mcF_{1}$ and $\mcF_{2}$ if $\Pi_{\mcF_{1}}$ is not isomorphic to $\Pi_{\mcF_{2}}$.
\end{quote}
In the theory developed in the present series of papers, we consider anabelian geometry in a completely different way. The mean of {\it anabelian geometry} around the Homeomorphism Conjecture is the following:
\begin{quote}
The relation of $\mcF_{1}$ and $\mcF_{2}$ in a certain moduli space can be understood from the set of open continuous homomorphisms $\text{Hom}(\Pi_{\mcF_{1}}, \Pi_{\mcF_{2}})$. Moreover, $\text{Hom}(\Pi_{\mcF_{1}}, \Pi_{\mcF_{2}})$ contains the geometric information of the moduli space. The term ``{\it anabelian}" means the geometric properties of a certain {\it moduli space} (i.e., not only a single geometric object but also the moduli space of geometric objects), where the moduli space can be reconstructed group-theoretically from fundamental groups.
\end{quote}
Note that we cannot consider the set of morphisms of geometric objects $\text{Hom}(\mcF_{1}, \mcF_{2})$, since $\text{Hom}(\mcF_{1}, \mcF_{2})=\emptyset$ and $\text{Hom}(\Pi_{\mcF_{1}}, \Pi_{\mcF_{2}})\neq \emptyset$ in general (e.g. the specialization homomorphisms). In fact, the existence of specialization homomorphisms is the reason that Tamagawa cannot formulate a ``Hom-type" conjecture for tame fundamental groups of smooth pointed stable curves in general (cf. \cite[Remark 1.34]{T2}). Thus, roughly speaking, the Weak Isom-version Conjecture is an ``{\it Isom-type}" problem, and the Homeomorphism Conjecture is a ``{\it Hom-type}" problem. In fact, in \cite{Y2}, the author formulated the so-called {\it Weak Hom-version Conjecture} for smooth pointed stable curves which is equivalent to Homeomorphism Conjecture when $q\in M_{g, n}$. Similar to other theory in anabelian geometry, Hom-type problems are so much harder than the Isom-type problems.

Now, our main  result of the present paper is as follows (see also Theorem \ref{mainthem-form-2}):
\begin{theorem}\label{them-0-2}
We maintain the notation introduced above. Let $[q] \in \overline \mfM^{\rm cl}_{0, n}$ be an arbitrary closed point. Then $\pi_{0, n}^{\rm adm}([q])$ is a closed point of $\overline \Pi_{0, n}$. In particular, the Homeomorphism Conjecture holds when $(g, n)=(0, 4)$.  
\end{theorem}

\begin{remarkA}
In \cite{Y8}, we will prove that the Homeomorphism Conjecture also holds when $(g, n)=(1,1)$. Then the Homeomorphism Conjecture holds when the dimension of $\overline M_{g, n}$ is $1$. In \cite{Y8-2}, we will prove a generalized version of Tamagawa's essential dimension conjecture for closed points of $\overline M_{1, n}$.  In \cite{Y9}, by equipping the sets of inertia subgroups with certain orders, we  define clutching morphisms and forgetting morphisms for moduli spaces of admissible fundamental groups, and prove the clutching morphisms and the forgetting morphisms are continuous maps.
\end{remarkA}
\noindent
Denote by ${\rm Hom}^{\rm open}_{\rm pro\text{-}gps}(-, -)$ and ${\rm Isom}_{\rm pro\text{-}gps}(-, -)$ the set of open continuous homomorphisms of profinit groups and the set of isomorphisms of profinite groups, respectively. Then Theorem \ref{them-0-2} follows immediately from the following strong (Hom-type) anabelian result, which is a ultimate generalization of \cite[Theorem 0.2]{T2} when $g=0$ and $q_{1}$ is closed (see also Theorem \ref{mainthem-form-1}).
\begin{theorem}\label{them-0-3}
Let $q_{1}$, $q_{2} \in \overline M_{0, n}$ be arbitrary points. Suppose that $q_{1}$ is closed. Then we have that $${\rm Hom}^{\rm open}_{\rm pro\text{-}gps}(\pi_{1}^{\rm adm}(q_{1}), \pi_{1}^{\rm adm}(q_{2})) \neq \emptyset$$ if and only if $q_{1} \sim_{fe} q_{2}$. In particular, if this is the case, we have that $q_{2}$ is a closed point,  and that $${\rm Hom}^{\rm open}_{\rm pro\text{-}gps}(\pi_{1}^{\rm adm}(q_{1}), \pi_{1}^{\rm adm}(q_{2})) = {\rm Isom}_{\rm pro\text{-}gps}(\pi_{1}^{\rm adm}(q_{1}), \pi_{1}^{\rm adm}(q_{2})).$$ 
\end{theorem}

\begin{remarkA}
In fact, in the present paper, we will prove a slightly stronger version of Theorem \ref{them-0-3} by replacing $\pi_{1}^{\rm adm}(q_{1})$ and $\pi_{1}^{\rm adm}(q_{2})$ by the maximal pro-solvable quotients $\pi_{1}^{\rm adm}(q_{1})^{\rm sol}$ and $\pi_{1}^{\rm adm}(q_{2})^{\rm sol}$ of $\pi_{1}^{\rm adm}(q_{1})$ and $\pi_{1}^{\rm adm}(q_{2})$, respectively. Then we obtain a solvable version of Theorem \ref{them-0-2} which is slightly stronger than Theorem \ref{them-0-2}. In particular, we obtain that {\it the Solvable Homeomorphism Conjecture} (cf. Section \ref{sec-5-2}) holds when $(g, n)=(0, 4)$. 
\end{remarkA}

\begin{remarkA}
Note that Theorem \ref{them-0-3} is essentially different from Theorem \ref{them-0-1}. The reason is as follows: Let $q_{1}$, $q_{2}\in | \overline M_{g, n}|$ be arbitrary points such that $q_{1}$ is not closed, and that $q_{2}$ is a closed point contained in the topological closure of $q_{1}$ in $|\overline M_{g, n}|$. Then every open continuous homomorphism $\pi_{1}^{\rm adm}(q_{1}) \migi \pi_{1}^{\rm adm}(q_{2})$ is not an isomorphism (cf. \cite[Theorem 0.3]{T4}).
\end{remarkA}

Next, we explain the method of proving Theorem \ref{them-0-3} (or Theorem \ref{them-0-2}). We establish two fundamental tools to analyze the geometric behavior of curves from open continuous homomorphisms of admissible fundamental groups, which play central roles in the theory of moduli spaces of admissible fundamental groups in positive characteristic. The first tool is the following result, which says that the inertia subgroups and field structures associated to inertia subgroups of marked points can be reconstructed group-theoretically  from arbitrary surjective open continuous homomorphisms of admissible fundamental groups (cf. Theorem \ref{mainstep-1} and Theorem \ref{prop-3-12} for more precise statements):
\begin{theorem}\label{them-0-4}
Let $X_{i}^{\bp}$, $i \in \{1, 2\}$, be a pointed stable curve of type $(g_{X_{i}}, n_{X_{i}})$ over an algebraically closed field $k_{i}$ of characteristic $p>0$, and $\Gamma_{X_{i}^{\bp}}$ the dual semi-graph of $X_{i}^{\bp}$. Let $\Pi_{X_{i}^{\bp}}$ be either the admissible fundamental group $\pi_{1}^{\rm adm}(X^{\bp}_{i})$ of $X_{i}^{\bp}$ or the maximal pro-solvable quotient $\pi_{1}^{\rm adm}(X^{\bp}_{i})^{\rm sol}$ of $\pi_{1}^{\rm adm}(X^{\bp}_{i})$, and $I_{i}\subseteq \Pi_{X_{i}^{\bp}}$ an closed subgroup associated to an open edge of $\Gamma_{X^{\bp}_{i}}$ (i.e., a closed subgroup which is (outer) isomorphic to the inertia subgroup of the marked point corresponding to an open edge of $\Gamma_{X^{\bp}_{i}}$). Suppose that $(g_{X_{1}}, n_{X_{1}})=(g_{X_{2}}, n_{X_{2}})$. Let $$\phi: \Pi_{X_{1}^{\bp}} \migisurj \Pi_{X_{2}^{\bp}}$$ be an arbitrary surjective open continuous homomorphism of profinite groups. Then the following statements hold:

(i)  $\phi(I_{1}) \subseteq \Pi_{X_{2}^{\bp}}$ is a closed subgroup associated to an open edge of $\Gamma_{X^{\bp}_{2}}$, and that there exists a closed subgroup $I' \subseteq \Pi_{X_{1}^{\bp}}$ associated to an open edge of $\Gamma_{X^{\bp}_{1}}$ such that $\phi(I')=I_{2}$. 

(ii) The field structures associated to inertia subgroups of marked points can be reconstructed group-theoretically from $\Pi_{X_{i}^{\bp}}$, and that $\phi$ induces a field isomorphism between the fields associated to $I_{1}$ and $\phi(I_{1})$ group-theoretically.
\end{theorem}
\noindent
The main ingredient in the proof of Theorem \ref{them-0-4} is a formula for the maximum generalized Hasse-Witt invariant $\gamma^{\rm max}(\Pi_{X^{\bp}_{i}})$ of prime-to-$p$ cyclic admissible coverings of $X^{\bp}_{i}$, which was proved by the author in \cite{Y6}. 

The second tool is the following result, which we call combinatorial Grothendieck conjecture for surjections, and which says that the geometry (i.e., topological and combinatorial structures) of pointed stable curves can be completely reconstructed group-theoretically from open continuous homomorphisms of admissible fundamental groups (cf. Theorem \ref{mainstep-2} for a more precise statement, and Theorem \ref{comgc} for a more general form under certain assumptions):
\begin{theorem}\label{them-0-5}
Let $X_{i}^{\bp}$, $i \in \{1, 2\}$, be a pointed stable curve of type $(0,n)$ over an algebraically closed field $k_{i}$ of characteristic $p>0$, and $\Gamma_{X_{i}^{\bp}}$ the dual semi-graph of $X_{i}^{\bp}$. Let $\Pi_{X_{i}^{\bp}}$ be the maximal pro-solvable quotient of the admissible fundamental group of $X^{\bp}_{i}$ and $\Pi_{i} \subseteq \Pi_{X_{i}^{\bp}}$ a closed subgroup associated to a vertex (i.e., a closed subgroup which is (outer) isomorphic to the admissible fundamental group of the smooth pointed stable curve associated to a vertex of $\Gamma_{X^{\bp}_{i}}$), and $I_{i}\subseteq \Pi_{X_{i}^{\bp}}$ an closed subgroup associated to a closed edge (i.e., a closed subgroup which is (outer) isomorphic to the inertia subgroup of the node corresponding to a closed edge of   $\Gamma_{X^{\bp}_{i}}$). Suppose that $\#v(\Gamma_{X^{\bp}_{1}})=\#v(\Gamma_{X^{\bp}_{2}})$ and $\#e^{\rm cl}(\Gamma_{X^{\bp}_{1}})=\#e^{\rm cl}(\Gamma_{X^{\bp}_{2}})$, where $\#(-)$ denotes the cardinality of $(-)$. Let $$\phi: \Pi_{X_{1}^{\bp}} \migisurj \Pi_{X_{2}^{\bp}}$$ be an arbitrary surjective open continuous homomorphism of profinite groups. Then the following statements hold: 

(i)  $\phi(\Pi_{1}) \subseteq \Pi_{X_{2}^{\bp}}$ is a closed subgroup associated to a vertex of $\Gamma_{X_{2}^{\bp}}$, and that there exists a closed subgroup $\Pi' \subseteq \Pi_{X_{1}^{\bp}}$ associated to a vertex of $\Gamma_{X_{1}^{\bp}}$ such that $\phi(\Pi')=\Pi_{2}$.

(ii) $\phi(I_{1}) \subseteq \Pi_{X_{2}^{\bp}}$ is a closed subgroup associated to a closed edge of $\Gamma_{X^{\bp}_{2}}$, and that there exists a closed subgroup $I' \subseteq \Pi_{X_{1}^{\bp}}$ associated to a closed edge of $\Gamma_{X^{\bp}_{1}}$ such that $\phi(I')=I_{2}$. 

(iii)  $\phi$ induces an isomorphism $$\phi^{\rm sg}: \Gamma_{X_{1}^{\bp}} \isom \Gamma_{X_{2}^{\bp}}$$ of dual semi-graphs group-theoretically.
\end{theorem}
\noindent
The main ingredient in the proof of Theorem \ref{them-0-5} is a formula for the limit of $p$-averages $\text{Avr}_{p}(\Pi_{X_{i}^{\bp}})$ of the admissible fundamental group of $X_{i}^{\bp}$, which was proved by Tamagawa and the author  (cf. \cite{T2}, \cite{Y5}). The key observations of the proofs of Theorem \ref{them-0-4} and Theorem \ref{them-0-5} are the following:
\begin{quote}
The inequalities of $\gamma^{\rm max}(\Pi_{X_{i}^{\bp}})$ and $\text{Avr}_{p}(\Pi_{X^{\bp}_{i}})$ induced by $\phi$ play roles of the comparability of (outer) Galois representations in the theory of the anabelian geometry of curves over algebraically closed fields of characteristic $p>0$.
\end{quote}
In fact, under certain assumptions, Theorem \ref{them-0-5} also holds for arbitrary types (cf. Theorem \ref{comgc} and Remark \ref{rem-comgc}). Moreover, the author believes that Theorem \ref{them-0-4}, Theorem \ref{them-0-5}, and Theorem \ref{comgc} will play important roles in the proof of the Homeomorphism Conjecture for arbitrary types. 

We observe that the scheme structure of a smooth pointed stable curve of type $(0,n)$ over $\overline \mbF_{p}$ is determined (via generalized Hasse-Witt invariants of  admissible coverings) completely by its inertia subgroups of marked points and the field structures associated to the inertia subgroups. Then Theorem \ref{them-0-4} implies that, if $X^{\bp}_{q_{1}}$ is non-singular, the scheme structure of $X^{\bp}_{q_{2}}$ can be determined by the scheme structure of $X^{\bp}_{q_{1}}$ via an open continuous homomorphism between their admissible fundamental groups. Moreover, by applying Theorem \ref{them-0-4}, the geometric operation (=removing a subset of marked points of a pointed stable curve and contracting the ($-1$)-curves and the ($-2$)-curves of a pointed semi-stable curve) can be translated to the group-theoretical operation (=quotient of a closed subgroup of the admissible fundamental group of a pointed stable curve, where the closed subgroup is generated by the inertia subgroups corresponding to a subset of marked points of the pointed stable curve). Then we can reduce Theorem \ref{them-0-3} to the case where $\#v(\Gamma_{q_{1}})=\#v(\Gamma_{q_{2}})$ and $\#e^{\rm cl}(\Gamma_{q_{1}})=\#e^{\rm cl}(\Gamma_{q_{2}})$. Moreover, by applying Theorem \ref{them-0-5}, we can reduce Theorem \ref{them-0-3} further to the case where $q_{1}$ and $q_{2}$ are contained in $M_{0, n}$ (i.e., $X_{q_{1}}^{\bp}$ and $X_{q_{2}}^{\bp}$ are non-singular). Then Theorem \ref{them-0-3} follows from \cite[Theorem 1.2]{Y2} proved by the author. This completes the proof of our main theorem.

The present paper is organized as follows. In Section \ref{sec-1}, we fix some notation concerning admissible coverings and admissible fundamental groups. In Section \ref{sec-2}, we recall the definition of generalized Hasse-Witt invariants, a formula for maximum generalized Hasse-Witt invariants of prime-to-$p$ admissible coverings, and a formula for limits of $p$-averages of admissible fundamental groups. In Section \ref{sec-5}, we introduce the moduli spaces of admissible fundamental groups and formulate the Homeomorphism Conjecture. In Section \ref{sec-3}, we prove Theorem \ref{them-0-4}. In Section \ref{sec-4}, we prove Theorem \ref{them-0-5}. In Section \ref{sec-6}, we prove our main theorem. In Section \ref{sec-7}, we prove the continuity of $\pi_{g, n}^{\rm adm}$.

\vspace{5mm}
\begin{center}
\textsc{Acknowledgements}
\end{center}
\vspace{1mm}

The main result of the present paper was obtained in July 2019, the author would like to thank Zhi Hu, Yuji Odaka, Akio Tamagawa, and Kazuhiko Yamaki for comments. This work was supported by JSPS KAKENHI Grant Number 20K14283, and by the Research Institute for Mathematical Sciences (RIMS), an International Joint Usage/Research Center located in Kyoto University.

\section{Admissible coverings and admissible fundamental groups}\label{sec-1}

In this section, we recall some notation and definitions concerning admissible coverings and admissible fundamental groups. 

\begin{definition}\label{def-1}
Let $\mbG$ be a semi-graph (cf. \cite[Definition 2.1]{Y6}). 

(a) We shall denote by $v(\mbG)$, $e^{\rm op}(\mbG)$, and $e^{\rm cl}(\mbG)$ the set of vertices of $\mbG$, the set of open edges of $\mbG$, and the set of closed edges of $\mbG$, respectively. 

(b) The semi-graph $\mbG$ can be regarded as a topological space with natural topology induced by $\mbR^{2}$. We define an {\it one-point compactification} $\mbG^{\rm cpt}$ of $\mbG$ as follows: if $e^{\rm op}(\mbG) =\emptyset$, we put $\mbG^{\rm cpt}=\mbG$; otherwise, the set of vertices of $\mbG^{\rm cpt}$ is the disjoint union $v(\mbG^{\rm cpt})\defeq v(\mbG) \sqcup \{v_{\infty}\}$, the set of closed edges of $\mbG^{\rm cpt}$ is $e^{\rm cl}(\mbG^{\rm cpt})\defeq e^{\rm cl}(\mbG) \cup e^{\rm op}(\mbG)$, the set of open edges of $\mbG$ is empty, and every edge $e \in e^{\rm op}(\mbG) \subseteq e^{\rm cl}(\mbG^{\rm cpt})$ connects $v_{\infty}$ with the vertex that is abutted by $e$. 

(c) Let $v \in v(\mbG)$. We shall say that $\mbG$ is {\it $2$-connected} at $v$ if $\mbG \setminus \{v\}$ is either empty or connected. Moreover, we shall say that $\mbG$ is $2$-connected if $\mbG$ is $2$-connected at each $v\in v(\mbG)$. Note that, if $\mbG$ is connected, then $\mbG^{\rm cpt}$ is $2$-connected at each $v\in v(\mbG) \subseteq v(\mbG^{\rm cpt})$ if and only if $\mbG^{\rm cpt}$ is $2$-connected. We put $$b(v)\defeq \sum_{e\in e^{\rm op}(\mbG) \cup e^{\rm cl}(\mbG)} b_{e}(v),$$ where $b_{e}(v) \in \{0, 1, 2\}$ denotes the number of times that $e$ meets $v$. We put $$v(\mbG)^{b\leq 1}\defeq \{v\in  v(\mbG)\ | \ b(v)\leq 1 \},$$ and  denote by $e^{\rm cl}(\mbG)^{b\leq 1}$ the set of closed edges of $\mbG$ which meet a vertex of $v(\mbG)^{b\leq 1}$.


\end{definition}

Let $p$ be a prime number, and let $$X^{\bp}=(X, D_{X})$$ be a pointed {\it semi-stable} curve of  type $(g_{X}, n_{X})$ over an algebraically closed field $k$ of characteristic $p$, where $X$ denotes the underlying curve, $D_{X}$ denotes the set of marked points, $g_{X}$ denotes the genus of $X$, and $n_{X}$ denotes the cardinality $\#D_{X}$ of $D_{X}$.  Write $\Gamma_{X^{\bp}}$ for the dual semi-graph of $X^{\bp}$ (cf. \cite[Definition 3.1]{Y0} for the definition of the dual semi-graph of a pointed semi-stable curve) and $r_{X}\defeq\text{dim}_{\mbQ}(H^{1}(\Gamma_{X^{\bp}}, \mbQ))$ for the Betti number of the semi-graph $\Gamma_{X^{\bp}}$. 

Let $v \in v(\Gamma_{X^{\bp}})$ and $e \in e^{\rm op}(\Gamma_{X^{\bp}}) \cup e^{\rm cl}(\Gamma_{X^{\bp}})$. We write $X_{v}$ for the irreducible component of $X$ corresponding to $v$, write $x_{e}$ for the node of $X$ corresponding to $e$ if $e \in e^{\rm cl}(\Gamma_{X^{\bp}})$, and write $x_{e}$ for the marked point of $X$ corresponding to $e$ if $e \in e^{\rm op}(\Gamma_{X^{\bp}})$. Moreover, write $\widetilde X_{v}$ for the {\it smooth} compactification of $U_{X_{v}}\defeq X_{v} \setminus X_{v}^{\rm sing}$, where $(-)^{\rm sing}$ denotes the singular locus of $(-)$. We define a smooth pointed semi-stable curve of type $(g_{v}, n_{v})$ over $k$ to be $$\widetilde X_{v}^{\bp}=(\widetilde X_{v}, D_{\widetilde X_{v}}\defeq(\widetilde X_{v} \setminus U_{X_{v}})\cup (D_{X}\cap X_{v})).$$ We shall say that {\it $\widetilde X^{\bp}_{v}$ is the smooth pointed semi-stable curve of type $(g_{v}, n_{v})$ associated to $v$}, or {\it the smooth pointed semi-stable curve associated to $v$} for short. In particular, we shall say that $\widetilde X^{\bp}_{v}$ is the smooth pointed {\it stable} curve associated to $v$ if $\widetilde X^{\bp}_{v}$ is a pointed stable curve over $k$.

\begin{definition}\label{def-2}
Let $Y^{\bp}=(Y, D_{Y})$ be a pointed semi-stable curve over $k$, $f^{\bp}: Y^{\bp} \migi X^{\bp}$ a {\it finite} morphism of pointed semi-stable curves over $k$, and $f: Y \migi X$ the morphism of underlying curves induced by $f^{\bp}$. 

We shall say $f^{\bp}$ a {\it Galois admissible covering} over $k$ (or Galois admissible covering for short) if the following conditions are satisfied: 

(i) There exists a finite group $G\subseteq \text{Aut}_{k}(Y^{\bp})$ such that $Y^{\bp}/G=X^{\bp}$, and $f^{\bp}$ is equal to the quotient morphism $Y^{\bp} \migi Y^{\bp}/G$.

(ii) For each $y\in Y^{\rm sm}\setminus D_{Y}$, $f$ is \'etale at $y$, where $(-)^{\rm sm}$ denotes the smooth locus of $(-)$. 

(iii) For any $y \in Y^{\rm sing}$, the image $f(y)$ is contained in $X^{\rm sing}$.

(iv) For each $y\in Y^{\rm sing}$, we write $D_{y}\subseteq G$ for the decomposition group of $y$ and $\#D_{y}$ for the cardinality of $D_{y}$. Then we have that  $(\#D_{y}, p)=1$, and that the local morphism between two nodes induced by $f$ may be described as follows: 
\[
\begin{array}{ccc}
 \widehat \mcO_{X, f(y)} \cong k[[u,v]]/uv & \migi & \widehat \mcO_{Y,y}\cong k[[s,t]]/st
\\
u & \mapsto & s^{\#D_{y}}
\\
v & \mapsto & t^{\#D_{y}},
\end{array}
\]
where $\#(-)$ denotes the cardinality of $(-)$. 
Moreover, we have that $\tau(s)=\zeta_{\#D_{y}}s$ and $\tau(t)=\zeta^{-1}_{\#D_{y}}t$ for each $\tau\in D_{y}$, where $\zeta_{\#D_{y}}$ is a primitive $\#D_{y}$th root of unit.

(v) The local morphism between two marked points induced by $f$ may be described as follows:
\[
\begin{array}{ccc}
 \widehat \mcO_{X, f(y)} \cong k[[a]] & \migi & \widehat \mcO_{Y,y}\cong k[[b]]
\\
a & \mapsto & b^{m},
\end{array}
\]
where $(m, p)=1$ (i.e., a tamely ramified extension). 

Moreover, we shall say $f^{\bp}$ an {\it admissible covering} if there exists a morphism of pointed semi-stable curves $h^{\bp}: W^{\bp} \migi Y^{\bp}$ over $k$ such that the composite morphism $f^{\bp}\circ h^{\bp}: W^{\bp} \migi X^{\bp}$ is a Galois admissible covering over $k$. We shall say an admissible covering $f^{\bp}$ {\it \'etale} if $f$ is an \'etale morphism.

Let $Z^{\bp}$ be a disjoint union of finitely many pointed semi-stable curves over $k$. We shall say that a morphism $f_{Z}^{\bp}: Z^{\bp} \migi X^{\bp}$ over $k$ is a {\it multi-admissible covering} if the restriction of $f_{Z}^{\bp}$ to each connected component of $Z^{\bp}$ is admissible.
\end{definition}

\begin{definition}\label{def-1-3}
Let $f^{\bp}: Y^{\bp} \migi X^{\bp}$ be an admissible covering over $k$ of degree $m$. Let $e\in e^{\rm op}(\Gamma_{X^{\bp}})\cup e^{\rm cl}(\Gamma_{X^{\bp}})$ and $x_{e}$ the closed point of $X$ corresponding to $e$. We put
$$e_{f}^{\rm cl, ra}\defeq \{e\in e^{\rm cl}(\Gamma_{X^{\bp}}) \ | \ \#f^{-1}(x_{e})=1\},$$
$$e_{f}^{\rm cl, \text{\'et}}\defeq \{e\in e^{\rm cl}(\Gamma_{X^{\bp}}) \ | \ \#f^{-1}(x_{e})=m\},$$
$$e_{f}^{\rm op, ra}\defeq \{e\in e^{\rm op}(\Gamma_{X^{\bp}}) \ | \ \#f^{-1}(x_{e})=1\},$$
$$e_{f}^{\rm op, \et}\defeq \{e\in e^{\rm op}(\Gamma_{X^{\bp}}) \ | \ \#f^{-1}(x_{e})=m\},$$
$$v_{f}^{\rm ra}\defeq \{v\in v(\Gamma_{X^{\bp}}) \ | \ \#\text{Irr}(f^{-1}(X_{v}))=1 \},$$
$$ v_{f}^{\rm sp}\defeq \{v\in v(\Gamma_{X^{\bp}}) \ | \ \#\text{Irr}(f^{-1}(X_{v}))=m \},$$
where $\text{Irr}(-)$ denotes the set of irreducible components of $(-)$, ``ra" means ``ramification", and ``sp" means ``split". Note that if the Galois closure of $f^{\bp}$ is a Galois admissible covering whose Galois group is a $p$-group, then the definition of admissible coverings implies that $\#e_{f}^{\rm cl, ra}=\#e_{f}^{\rm op, ra}=0.$
\end{definition}

Let $\msC$ be a category. We shall write $\text{Ob}(\msC)$ for the class of objects of $\msC$, and write $\text{Hom}(\msC)$ for the class of morphisms of $\msC$. We  denote by $$\text{Cov}^{\text{adm}}(X^{\bp})\defeq(\text{Ob}(\text{Cov}^{\text{adm}}(X^{\bp})), \text{Hom}(\text{Cov}^{\text{adm}}(X^{\bp})))$$ the category which consists of the following data: (i) $\text{Ob}(\text{Cov}^{\text{adm}}(X^{\bp}))$ consists of an empty object and all the pairs $(Z^{\bp}, f^{\bp}_{Z}: Z^{\bp} \migi X^{\bp})$, where $Z^{\bp}$ is a disjoint union of finitely many pointed semi-stable curves over $k$, and $f_{Z}^{\bp}$ is a multi-admissible covering over $k$; (ii) for any $(Z^{\bp}, f^{\bp}_{Z})$, $(Y^{\bp}, f^{\bp}_{Y}) \in \text{Ob}(\text{Cov}^{\text{adm}}(X^{\bp}))$, we define $$\text{Hom}((Z^{\bp}, f^{\bp}_{Z}), (Y^{\bp}, f^{\bp}_{Y}))\defeq\{g^{\bp} \in \text{Hom}_{k}(Z^{\bp}, Y^{\bp}) \ | \ f^{\bp}_{Y}\circ g^{\bp}=f^{\bp}_{Z}\},$$ where $\text{Hom}_{k}(Z^{\bp}, Y^{\bp})$ denotes the set of $k$-morphisms of pointed semi-stable curves. It is well known that $\text{Cov}^{\text{adm}}(X^{\bp})$ is a Galois category. Thus, by choosing a base point $x \in X^{\rm sm}\setminus D_{X}$, we obtain a fundamental group $\pi_{1}^{\text{adm}}(X^{\bp}, x)$ which is called the {\it admissible fundamental group} of $X^{\bp}$. For simplicity of notation, we omit the base point and denote the admissible fundamental group by $$\pi_{1}^{\rm adm}(X^{\bp}).$$ Note that, by the definition of admissible coverings, the admissible fundamental group of $X^{\bp}$ is naturally isomorphic to the tame fundamental group of $X^{\bp}$ when $X^{\bp}$ is smooth over $k$. Let $v \in v(\Gamma_{X^{\bp}})$. Write $\pi_{1}^{\rm adm}(\widetilde X_{v}^{\bp})$ for the admissible fundamental group of the smooth pointed semi-stable curve $\widetilde X_{v}^{\bp}$ associated to $v$. Then we have a natural (outer) injection
$$\pi_{1}^{\rm adm}(\widetilde X_{v}^{\bp}) \migiinje \pi_{1}^{\rm adm}(X^{\bp}).$$ On the other hand, the structure of maximal pro-prime-to-$p$ quotient $\pi_{1}^{\rm adm}(X^{\bp})^{p'}$ of $\pi_{1}^{\rm adm}(X^{\bp})$ is well-known, which is isomorphic to the pro-prime-to-$p$ completion of the following group (cf. [V, Th\'eor\`eme 2.2 (c)])  $$\langle a_{1}, \dots, a_{g_{X}}, b_{1}, \dots, b_{g_{X}}, c_{1}, \dots, c_{n_{X}} \ | \ \prod_{i=1}^{g_{X}}[a_{i}, b_{i}]\prod_{j=1}^{n_{X}}c_{j}=1\rangle.$$  We shall denote by $\pi_{1}^{\rm adm}(X), \ \pi_{1}^{\et}(X), \ \pi_{1}^{\rm top}(\Gamma_{X^{\bp}})$ the admissible fundamental group of the pointed semi-stable curve $(X, \emptyset)$, the \'etale fundamental group of the underlying curve $X$ of $X^{\bp}$, and the profinite completion of the topological fundamental group of $\Gamma_{X^{\bp}}$, respectively. Then we have the following natural surjective open continuous homomorphisms (for suitable choices of base points): $$\pi_{1}^{\rm adm}(X^{\bp})\migisurj \pi_{1}^{\rm adm}(X)\migisurj \pi_{1}^{\et}(X) \migisurj \pi_{1}^{\rm top}(\Gamma_{X^{\bp}}).$$ Note that the isomorphism classes of $\pi_{1}^{\rm adm}(X^{\bp})$, $\pi_{1}^{\rm adm}(X)$,  $\pi_{1}^{\et}(X)$,  and $\pi_{1}^{\rm top}(\Gamma_{X^{\bp}})$ depend only on the pointed {\it stable} curve associated to $X^{\bp}$ (i.e., the pointed stable curve obtained by contracting the $(-1)$-curves and $(-2)$-curves of $X^{\bp}$). 

The admissible fundamental groups of pointed stable curves can be also described by using logarithmic geometry. Let $\overline \mcM_{g_{X}, n_{X}, \mbZ}$ be the moduli stack over $\spec \mbZ$  parameterizing pointed stable curves of type $(g_{X}, n_{X})$ and $\mcM_{g_{X}, n_{X}, \mbZ}$ the open substack of $\overline \mcM_{g_{X}, n_{X}, \mbZ}$ parameterizing smooth pointed stable curves. Write $\overline \mcM_{g_{X}, n_{X}, \mbZ} ^{\log}$ for the log stack obtained by equipping $\overline \mcM_{g_{X}, n_{X}, \mbZ}$ with the natural log structure associated to the divisor with normal crossings $\overline \mcM_{g_{X}, n_{X}, \mbZ}\setminus\mcM_{g_{X}, n_{X}, \mbZ} \subset \overline \mcM_{g_{X}, n_{X}, \mbZ}$ relative to $\spec \mbZ$. The pointed stable curve $X^{\bullet}$ over $k$ induces a morphism $\spec k \migi \overline \mcM_{g_{X}, n_{X}, \mbZ}$. Write $s_{X}^{\log}$ for the log scheme whose underlying scheme is $\spec k$, and whose log structure is the pulling-back log structure induced by the morphism $\spec k \migi \overline \mcM_{g_{X}, n_{X}, \mbZ}$. We obtain a natural morphism $s_{X}^{\log} \migi \overline \mcM_{g_{X}, n_{X}, \mbZ}^{\log}$ induced by the morphism $\spec k \migi \overline \mcM_{g_{X}, n_{X}, \mbZ}$ and a stable log curve $$X^{\log}\defeq s_{X}^{\log}\times_{\overline \mcM_{g_{X}, n_{X}, \mbZ}^{\log}}\overline \mcM_{g_{X}, n_{X}+1, \mbZ}^{\log}$$ over $s_{X}^{\log}$ whose underlying scheme is $X$. Let $Y^{\log} \migi X^{\log}$ be an arbitrary Kummer log \'etale covering. One can prove that there exists a Kummer log \'etale covering $t_{X}^{\log} \migi s_{X}^{\log}$ such that $Y^{\log} \times_{s_{X}^{\log}} t^{\log}_{X} \migi X^{\log} \times_{s_{X}^{\log}} t^{\log}_{X}$ is a log admissible covering (cf. \cite[\S 3.5 Definition]{M-1}) over $t^{\log}_{X}$. Then the admissible fundamental group of $X^{\bp}$ does not depend on the log structure of $X^{\log}$, and \cite[\S 3.11 Proposition]{M-1} implies that the admissible fundamental group $\pi_{1}^{\rm adm}(X^{\bp})$ of $X^{\bp}$ is naturally isomorphic to the geometric log \'etale fundamental group of $X^{\log}$ (i.e., $\text{ker}(\pi_{1}(X^{\log}) \migi \pi_{1}(s_{X}^{\log}))$).


Let $$\pi_{1}^{\rm adm}(X^{\bp})^{\rm sol},\ \pi_{1}^{\rm adm}(X)^{\rm sol}, \ \pi_{1}^{\et}(X)^{\rm sol}, \ \pi_{1}^{\rm top}(\Gamma_{X^{\bp}})^{\rm sol}$$ be the maximal pro-solvable quotients of $\pi_{1}^{\rm adm}(X^{\bp}), \ \pi_{1}^{\rm adm}(X), \ \pi_{1}^{\et}(X), \ \pi_{1}^{\rm top}(\Gamma_{X^{\bp}})$, respectively. Then we obtain the following natural surjective open continuous homomorphisms $$\pi_{1}^{\rm adm}(X^{\bp})^{\rm sol} \migisurj \pi_{1}^{\rm adm}(X)^{\rm sol} \migisurj \pi_{1}^{\et}(X)^{\rm sol} \migisurj \pi_{1}^{\rm top}(\Gamma_{X^{\bp}})^{\rm sol}.$$ We shall say $$\pi_{1}^{\rm adm}(X^{\bp})^{\rm sol}$$ the {\it solvable admissible fundamental group} of $X^{\bp}$. Let $v \in v(\Gamma_{X^{\bp}})$. Write $\pi_{1}^{\rm adm}(\widetilde X_{v}^{\bp})^{\rm sol}$ for the solvable admissible fundamental group of the smooth pointed semi-stable curve $\widetilde X_{v}^{\bp}$ associated to $v$. Then the natural (outer) injection $\pi_{1}^{\rm adm}(\widetilde X_{v}^{\bp}) \migiinje \pi_{1}^{\rm adm}(X^{\bp})$ induces a homomorphism $$\pi_{1}^{\rm adm}(\widetilde X_{v}^{\bp})^{\rm sol} \migi \pi_{1}^{\rm adm}( X^{\bp})^{\rm sol}.$$ We see that this homomorphism is an {\it injection}. Indeed, let $\widetilde f_{v}^{\bp}: \widetilde Y^{\bp}_{v} \migi \widetilde X_{v}^{\bp}$ be a Galois admissible covering over $k$ whose Galois group is an abelian group. Then we see immediately that there exists a Galois admissible covering $g^{\bp}: Z^{\bp} \migi X^{\bp}$ over $k$ whose Galois group is a solvable group such that the following is satisfied: let $Z_{v}$ be an irreducible component of $Z^{\bp}$ such that $g(Z_{v})=X_{v}$; then the Galois admissible covering $\widetilde Z^{\bp}_{v} \migi \widetilde X_{v}^{\bp}$ over $k$ induced by $g^{\bp}$ factors through $\widetilde f_{v}^{\bp}$. This means that the homomorphism $\pi_{1}^{\rm adm}(\widetilde X_{v}^{\bp})^{\rm sol} \migi \pi_{1}^{\rm adm}( X^{\bp})^{\rm sol}.$ mentioned above is an injection.

In the remainder of the present paper, {\it we shall denote by $$\Pi_{X^{\bp}}$$ either $\pi_{1}^{\rm adm}(X^{\bp})$ or $\pi_{1}^{\rm adm}(X^{\bp})^{\rm sol}$  unless indicated otherwise.}  If $\Pi_{X^{\bp}}=\pi_{1}^{\rm adm}(X^{\bp})$, we denote by $$\Pi_{X^{\bp}}^{\rm cpt}\defeq \pi_{1}^{\rm adm}(X),\ \Pi^{\text{\'et}}_{X^{\bp}} \defeq \pi_{1}^{\et}(X), \ \Pi^{\rm top}_{X^{\bp}}\defeq \pi_{1}^{\rm top}(\Gamma_{X^{\bp}}).$$ If $\Pi_{X^{\bp}}=\pi_{1}^{\rm adm}(X^{\bp})^{\rm sol}$, we denote by $$\Pi_{X^{\bp}}^{\rm cpt}\defeq \pi_{1}^{\rm adm}(X)^{\rm sol},\ \Pi^{\text{\'et}}_{X^{\bp}} \defeq \pi_{1}^{\et}(X)^{\rm sol}, \ \Pi^{\rm top}_{X^{\bp}}\defeq \pi_{1}^{\rm top}(\Gamma_{X^{\bp}})^{\rm sol}.$$

Let  $H \subseteq \Pi_{X^{\bp}}$ be an arbitrary open subgroup. We write $X_{H}^{\bp}$ for the pointed semi-stable curve of type $(g_{X_{H}}, n_{X_{H}})$ over $k$ corresponding to $H$, $\Gamma_{X_{H}^{\bp}}$ for the dual semi-graph of $X^{\bp}_{H}$, and $r_{X_{H}}$ for the Betti number of $\Gamma_{X_{H}^{\bp}}$. Then we obtain an admissible covering $$f_{H}^{\bp}: X^{\bp}_{H} \migi X^{\bp}$$ over $k$ induced by the natural injection $H \migiinje \Pi_{X^{\bp}}$, and obtain a natural map of dual semi-graphs $$f_{H}^{\rm sg}: \Gamma_{X^{\bp}_{H}} \migi \Gamma_{X^{\bp}}$$ induced by $f_{H}^{\bp}$, where ``sg" means ``semi-graph". Moreover, if $H$ is an open {\it normal} subgroup, then $\Gamma_{X^{\bp}_{H}}$ admits an action of $\Pi_{X^{\bp}}/H$ induced by the natural action of $\Pi_{X^{\bp}}/H$ on $X^{\bp}_{H}$. Note that the quotient of $\Gamma_{X^{\bp}_{H}}$ by $\Pi_{X^{\bp}}/H$ coincides with $\Gamma_{X^{\bp}}$, and that $H$ is isomorphic to the admissible fundamental group (resp. solvable admissible fundamental group) $\Pi_{X^{\bp}_{H}}$ of $X_{H}^{\bp}$ if $\Pi_{X^{\bp}}=\pi_{1}^{\rm adm}(X^{\bp})$ (resp. $\Pi_{X^{\bp}}=\pi_{1}^{\rm adm}(X^{\bp})^{\rm sol}$). We also use the notation $$H^{\rm cpt}, \ H^{\text{\rm \'et}}, \ H^{\rm top}$$ to denote $\Pi_{X^{\bp}_{H}}^{\rm cpt}$, $\Pi_{X^{\bp}_{H}}^{\text{\rm \'et}}$, and $\Pi_{X^{\bp}_{H}}^{\rm top}$, respectively.

We put $$\widehat X \defeq \invlim_{H \subseteq \Pi_{X^{\bp}}\ \text{open}} X_{H}, \ D_{\widehat X}\defeq \invlim_{H \subseteq \Pi_{X^{\bp}} \ \text{open}} D_{X_{H}}, \ \Gamma_{\widehat X^{\bp}} \defeq \invlim_{H \subseteq \Pi_{X^{\bp}}\ \text{open}} \Gamma_{X_{H}^{\bp}}.$$ We shall say that $$\widehat X^{\bp} =(\widehat X, D_{\widehat X})$$ is the universal admissible covering (resp. universal solvable admissible covering) of $X^{\bp}$ corresponding to $\Pi_{X^{\bp}}$ if $\Pi_{X^{\bp}}=\pi_{1}^{\rm adm}(X^{\bp})$ (resp.  $\Pi_{X^{\bp}}=\pi_{1}^{\rm adm}(X^{\bp})^{\rm sol}$), and that $\Gamma_{\widehat X^{\bp}}$ is the dual semi-graph of $\widehat X^{\bp}$. Note that we have that $\text{Aut}(\widehat X^{\bp} /X^{\bp})=\Pi_{X^{\bp}}$, and that $\Gamma_{\widehat X^{\bp}}$ admits a natural action of $\Pi_{X^{\bp}}$.

Let $v \in v(\Gamma_{X^{\bp}})$, $e \in e^{\rm op}(\Gamma_{X^{\bp}})\cup e^{\rm cl}(\Gamma_{X^{\bp}})$, $\widehat v \in v(\Gamma_{\widehat X^{\bp}})$ a vertex over $v$, and $\widehat e \in e^{\rm op}(\Gamma_{\widehat X^{\bp}}) \cup e^{\rm cl}(\Gamma_{\widehat X^{\bp}})$ an edge over $e$. We denote by $$\Pi_{\widehat v}\subseteq \Pi_{X^{\bp}}, \ I_{\widehat e} \subseteq \Pi_{X^{\bp}}$$ the stabilizer subgroups of $\widehat v$ and $\widehat e$, respectively. We see immediately that $\Pi_{\widehat v}$ is (outer) isomorphic to $\Pi_{\widetilde X_{v}^{\bp}}$ of $\widetilde X^{\bp}_{v}$, and that $I_{\widehat e}$ is (outer) isomorphic to an inertia subgroup associated to the closed point of $X$ corresponding to $e$. Then we have that $I_{\widehat e} \cong \widehat \mbZ(1)^{p'}$, where $(-)^{p'}$ denotes the maximal pro-prime-to-$p$ quotient of $(-)$. We put $$\text{Ver}(\Pi_{X^{\bp}}) \defeq \{\Pi_{\widehat v}\}_{\widehat v \in v(\Gamma_{\widehat X^{\bp}})},$$ $$\text{Edg}^{\rm op}(\Pi_{X^{\bp}}) \defeq \{I_{\widehat e}\}_{\widehat e \in e^{\rm op}(\Gamma_{\widehat X^{\bp}})},$$ $$\text{Edg}^{\rm cl}(\Pi_{X^{\bp}})\defeq \{I_{\widehat e}\}_{\widehat e \in e^{\rm cl}(\Gamma_{\widehat X^{\bp}})}.$$  Moreover, if $\widehat e$ abuts on $\widehat v$, then we have the following injections $$I_{\widehat e} \migiinje \Pi_{\widehat v} \migiinje \Pi_{X^{\bp}}.$$ Note that $\text{Ver}(\Pi_{X^{\bp}})$, $\text{Edg}^{\rm op}(\Pi_{X^{\bp}})$, and $\text{Edg}^{\rm cl}(\Pi_{X^{\bp}})$ admit natural actions of $\Pi_{X^{\bp}}$ (i.e., the conjugacy actions), and that we have the following natural bijections $$\text{Ver}(\Pi_{X^{\bp}})/\Pi_{X^{\bp}}\isom v(\Gamma_{X^{\bp}}),$$ $$\text{Edg}^{\rm op}(\Pi_{X^{\bp}})/\Pi_{X^{\bp}}\isom e^{\rm op}(\Gamma_{X^{\bp}}),$$ $$\text{Edg}^{\rm cl}(\Pi_{X^{\bp}})/\Pi_{X^{\bp}}\isom e^{\rm cl}(\Gamma_{X^{\bp}}).$$

\section{Maximum and averages of generalized Hasse-Witt invariants}\label{sec-2}

In this section, we recall some results concerning Hasse-Witt invariants (or $p$-rank) and generalized Hasse-Witt invariants. 

\begin{definition}\label{def-3}
Let $Z^{\bp}$ be a disjoint union of finitely many pointed semi-stable curves over $k$. We define the {\it $p$-rank} (or {\it Hasse-Witt invariant}) of $Z^{\bp}$ to be $$\sigma_{Z}\defeq\text{dim}_{\mbF_{p}}(H^{1}_{\text{\'et}}(Z, \mbF_{p})).$$ In particular, if $Z^{\bp}$ is a pointed semi-stable curve, then $$\sigma_{Z}=\text{dim}_{\mbF_{p}}(\Pi_{Z^{\bp}}^{\rm ab} \otimes \mbF_{p}),$$ where $\Pi_{Z^{\bp}}$ is either the admissible fundamental group or the solvable admissible fundamental group of $Z^{\bp}$, and $(-)^{\rm ab}$ denotes the abelianization of $(-)$.
\end{definition}

Let $X^{\bp}$ be a pointed stable curve over of type $(g_{X}, n_{X})$ over an algebraically closed field $k$ of characteristic $p>0$, $\Gamma_{X^{\bp}}$ the dual semi-graph of $X^{\bp}$, and $\Pi_{X^{\bp}}$ either the admissible fundamental group or the solvable admissible fundamental group of $X^{\bp}$. Let $n$ be an arbitrary positive natural number prime to $p$ and $\mu_{n} \subseteq k^{\times}$ the group of $n$th roots of unity. Fix a primitive $n$th root $\zeta_{n}$, we may identify $\mu_{n}$ with $\mbZ/n\mbZ$ via the map $\zeta_{n}^{i} \mapsto i$. Let $\alpha \in \text{Hom}(\Pi_{X^{\bp}}^{\rm ab}, \mbZ/n\mbZ)$. We denote by $X^{\bp}_{\alpha}=(X_{\alpha}, D_{X_\alpha})$ the Galois multi-admissible covering with Galois group $\mbZ/n\mbZ$ corresponding to $\alpha$.  Write $F_{X_{\alpha}}$ for the absolute Frobenius morphism on $X_{\alpha}$. Then there exists a decomposition (cf. \cite[Section 9]
{Se}) $$H^{1}(X_{\alpha}, \mcO_{X_\alpha})=H^{1}(X_{\alpha}, \mcO_{X_\alpha})^{\rm st} \oplus H^{1}(X_{\alpha}, \mcO_{X_\alpha})^{\rm ni},$$ where $F_{X_{\alpha}}$ is a bijection on $H^{1}(X_{\alpha}, \mcO_{X_\alpha})^{\rm st}$ and is nilpotent on $H^{1}(X_{\alpha}, \mcO_{X_\alpha})^{\rm ni}$. Moreover, we have $$H^{1}(X_{\alpha}, \mcO_{X_\alpha})^{\rm st}=H^{1}(X_{\alpha}, \mcO_{X_\alpha})^{F_{X_{\alpha}}}\otimes_{\mbF_{p}}k,$$ where $(-)^{F_{X_\alpha}}$ denotes the subspace of $(-)$ on which  $F_{X_{\alpha}}$ acts trivially. Then Artin-Schreier theory implies that we may identify $$H_{\alpha}\defeq H^{1}_{\text{\'et}}(X_{\alpha}, \mbF_{p}) \otimes_{\mbF_{p}}k$$ with the largest subspace of $H^{1}(X_{\alpha}, \mcO_{X_\alpha})$ on which $F_{X_{\alpha}}$ is a bijection.

The finite dimensional $k$-vector spaces $H_{\alpha}$ is a finitely generated $k[\mu_{n}]$-module induced by the natural action of $\mu_{n}$ on $X_{\alpha}$. We have the following canonical decomposition $$H_{\alpha}=\bigoplus_{i\in \mbZ/n\mbZ} H_{\alpha, i},$$ where $\zeta_{n} \in \mu_{n}$ acts on $H_{\alpha, i}$ as the $\zeta_{n}^{i}$-multiplication. We define $$\gamma_{\alpha, i}\defeq\text{dim}_{k}(H_{\alpha, i}), \ i \in \mbZ/n\mbZ.$$ We shall say that $\gamma_{\alpha, i},$ $i \in \mbZ/n\mbZ,$ is a  {\it generalized Hasse-Witt invariant} (cf. \cite{N}) of the cyclic multi-admissible covering $X^{\bp}_{\alpha} \migi X^{\bp}$.  Note that the decomposition above implies that $$\sigma_{X_{\alpha}}=\text{dim}_{k}(H_{\alpha})=\sum_{i \in \mbZ/n\mbZ}\gamma_{\alpha, i}.$$

Let $t \in \mbN$ be an arbitrary positive natural number, $K_{p^{t}-1}$ the kernel of the natural surjection $$\Pi_{X^{\bp}} \migisurj \Pi_{X^{\bp}}^{\rm ab}\otimes\mbZ/(p^{t}-1)\mbZ,$$ and $X_{K_{p^{t}-1}}^{\bp}$ the pointed stable curve over $k$ determined by $K_{p^{t}-1}$. Next, we define two important invariants associated to $X^{\bp}$. We put $$\gamma^{\rm max}(X^{\bp})\defeq $$$$\text{max}_{n \in \mbN \ \text{s.t.} \ (n, p)=1}\{\gamma_{\alpha, i} \ | \ \alpha \in \text{Hom}(\Pi_{X^{\bp}}^{\rm ab}, \mbZ/n\mbZ), \ \alpha\neq 0, \ i \in (\mbZ/n\mbZ) \setminus \{0\} \},$$ and shall say that $\gamma^{\rm max}(X^{\bp})$ is {\it the maximum generalized Hasse-Witt invariant of prime-to-$p$ cyclic admissible coverings of $X^{\bp}$}. We put $$\text{Avr}_{p}(X^{\bp}) \defeq \lim_{t \migi \infty} \frac{\sigma_{X_{K_{p^{t}-1}}}}{\#(\Pi_{X^{\bp}}^{\rm ab}\otimes \mbZ/(p^{t}-1)\mbZ)},$$ and shall say that $\text{Avr}_{p}(X^{\bp})$ is {\it the limit of $p$-averages of $X^{\bp}$}.

On the other hand, let $\overline \mbF_{p}$ be an arbitrary algebraic closure of the finite field $\mbF_{p}$, $\chi \in \text{Hom}(\Pi_{X^{\bp}}, \overline \mbF_{p}^{\times})$ such that $\chi \neq 1$, and $\Pi_{\chi} \subseteq \Pi_{X^{\bp}}$ the kernel of $\chi$. The profinite group $\Pi_{\chi}$ admits a natural action of $\Pi_{X^{\bp}}$ via conjugation. We put
$$\text{Hom}(\Pi_{\chi}, \mbZ/p\mbZ)[\chi]\defeq\{a \in \text{Hom}(\Pi_{\chi}, \mbZ/p\mbZ)\otimes_{\mbF_{p}} \overline \mbF_{p} \ | \ \tau(a)=\chi(\tau)a$$ $$ \text{for all} \ \tau \in \Pi_{X^{\bp}}\},$$  $$\gamma_{\chi}(\text{Hom}(\Pi_{\chi}, \mbZ/p\mbZ)) \defeq \text{dim}_{\overline \mbF_{p}}(\text{Hom}(\Pi_{\chi}, \mbZ/p\mbZ)[\chi]).$$ We define the following two {\it group-theoretical} invariants associated to $\Pi_{X^{\bp}}$: $$\gamma^{\rm max}(\Pi_{X^{\bp}}) \defeq \text{max}\{ \gamma_{\chi}(\text{Hom}(\Pi_{\chi}, \mbZ/p\mbZ)) \ | \ \chi \in \text{Hom}(\Pi_{X^{\bp}}, \overline \mbF_{p}^{\times}) \ \text{such that} \ \chi \neq 1\},$$  $$\text{Avr}_{p}(\Pi_{X^{\bp}}) \defeq \lim_{t \migi \infty} \frac{\text{dim}_{\mbF_{p}}(K^{\rm ab}_{p^{t}-1}\otimes \mbF_{p})}{\#(\Pi_{X^{\bp}}^{\rm ab}\otimes \mbZ/(p^{t}-1)\mbZ)}.$$ We see immediately that $$\gamma^{\rm max}(\Pi_{X^{\bp}})=\gamma^{\rm max}(X^{\bp}),$$ $$\text{Avr}_{p}(\Pi_{X^{\bp}})=\text{Avr}_{p}(X^{\bp}).$$

Moreover, we have the following important formulas for $\gamma^{\rm max}(\Pi_{X^{\bp}})$ and $\text{Avr}_{p}(\Pi_{X^{\bp}})$ which were proved by Tamagawa and the author by using the theory of Raynaud-Tamagawa theta divisors (cf. \cite[Theorem 1.4]{Y6} for $\gamma^{\rm max}(\Pi_{X^{\bp}})$, \cite[Theorem 0.5]{T3} and \cite[Theorem 5.2, Remark 5.2.1, and Remark 5.2.2]{Y5} for $\text{Avr}_{p}(\Pi_{X^{\bp}})$).

\begin{theorem}\label{max and average}
We maintain the notation introduced above. 

(a) We have 
\begin{eqnarray*}
\gamma^{\rm max}(\Pi_{X^{\bp}})=\left\{ \begin{array}{ll}
g_{X}-1, & \text{if} \ n_{X} =0,
\\
g_{X}+n_{X}-2, & \text{if} \ n_{X} \neq 0.
\end{array} \right.
\end{eqnarray*} 

(b) Suppose that $\Gamma_{X^{\bp}}^{\rm cpt}$ is $2$-connected. Then we have $${\rm Avr}_{p}(\Pi_{X^{\bp}})=g_{X}-r_{X}-\#v(\Gamma_{X^{\bp}})^{b\leq 1}+\#e^{\rm cl}(\Gamma_{X^{\bp}})^{b \leq 1}.$$
\end{theorem}

\begin{remarkA}\label{rem-them-1-1-2}
Suppose that $\Gamma_{X^{\bp}}^{\rm cpt}$ is $2$-connected. Note that $\#v(\Gamma_{X^{\bp}})^{b\leq 1}\neq 0$ if one of the following conditions holds: (i) $X^{\bp}$ is smooth over $k$ and $\#e^{\rm op}(\Gamma_{X^{\bp}}) \leq 1$; (ii) $\#v(\Gamma_{X^{\bp}})=2$, $\#e^{\rm op}(\Gamma_{X^{\bp}})=0$,  $\#e^{\rm cl}(\Gamma_{X^{\bp}})=1$, and $r_{X}=0$. In the case (i), we have $${\rm Avr}_{p}(\Pi_{X^{\bp}})=g_{X}-1.$$
In the case (ii), we have $${\rm Avr}_{p}(\Pi_{X^{\bp}})=g_{X}-2+1=g_{X}-1.$$
\end{remarkA}

\begin{lemma}\label{lem-0}
Let $X_{i}^{\bp}$, $i \in \{1, 2\}$, be a pointed stable curve of type $(g_{X_{i}}, n_{X_{i}})$ over an algebraically closed field $k_{i}$ of characteristic $p>0$ and $\Pi_{X^{\bp}_{i}}$ either the admissible fundamental group of $X_{i}^{\bp}$ or the solvable admissible fundamental group of $X_{i}^{\bp}$. Let $$\phi: \Pi_{X^{\bp}_{1}} \migisurj \Pi_{X^{\bp}_{2}}$$ be an arbitrary surjective open continuous homomorphism of profinit groups, $H_{2} \subseteq \Pi_{X^{\bp}_{2}}$ an arbitrary open normal subgroup, and $H_{1} \defeq \phi^{-1}(H_{2})$. Then the following statements hold:

(a) We have $$\gamma^{\rm max}(H_{1}) \geq \gamma^{\rm max}(H_{2}).$$

(b) Suppose that $(g_{X}, n_{X})=(g_{X_{1}}, n_{X_{1}})=(g_{X_{2}}, n_{X_{2}})$. Moreover, suppose either that $G\defeq \Pi_{X^{\bp}_{2}}/H_{2}$ is a $p$-group, that $(\#G, p)=1$, or that $G$ is a solvable group. Then we have $${\rm Avr}_{p}(H_{1}) \geq {\rm Avr}_{p}(H_{2}).$$ 
\end{lemma}

\begin{proof}
(a) Let $n \in \mbZ_{>0}$ be a positive natural number prime to $p$, and $\alpha_{2} \in \text{Hom}(H_{2}^{\rm ab}, \mbZ/n\mbZ)$ such that $\alpha_{2} \neq 0$. Let $j \in \mbZ/n\mbZ$ such that $\gamma_{\alpha_{2, j}}=\gamma^{\rm max}(H_{2})$. Write $Q_{2}$ for the kernel of the composition of the following homomorphisms $$H_{2} \migisurj H_{2}^{\rm ab} \overset{\alpha_{2}}\migisurj \mbZ/n\mbZ,$$ $Q_{1} \defeq \phi^{-1}(Q_{2})$, and $\alpha_{1}\in  \text{Hom}(H_{1}^{\rm ab}, \mbZ/n\mbZ)$ for the homomorphism induced by $\phi|_{H_{1}}$ and $\alpha_{2}$. Let $\overline \mbF_{p}$ be an algebraic closure of $\mbF_{p}$. Then $Q_{i}^{p, \rm ab} \otimes_{\mbF_{p}} \overline \mbF_{p}$ admits a natural $\overline \mbF_{p}[\mbZ/n\mbZ]$-module structure. Moreover, we see immediately that $\phi|_{H_{1}}$ induces a surjective homomorphism of $\overline \mbF_{p}[\mbZ/n\mbZ]$-modules $$Q_{1}^{p, \rm ab} \otimes_{\mbF_{p}} \overline \mbF_{p} \migisurj Q_{2}^{p, \rm ab} \otimes_{\mbF_{p}} \overline \mbF_{p}.$$ Then we obtain that $\gamma_{\alpha_{1}, j} \geq \gamma_{\alpha_{2}, j}.$ Thus, we have  $$\gamma^{\rm max}(H_{1}) \geq \gamma^{\rm max}(H_{2}).$$

(b) Let $t \in \mbN$ be an arbitrary positive natural number, $K_{H_{i}, p^{t}-1}$ the kernel of the natural surjection $$H_{i} \migisurj H_{i}^{\rm ab}\otimes\mbZ/(p^{t}-1)\mbZ.$$ Suppose that $G$ is a $p$-group. We have that Galois admissible covering $X^{\bp}_{H_{i}} \migi X_{i}^{\bp}$ corresponding to $H_{i}$ is \'etale. This implies that  $X^{\bp}_{H_{1}}$ and $X^{\bp}_{H_{2}}$ are equal types. We obtain that $$\#(H_{1}^{\rm ab}\otimes\mbZ/(p^{t}-1)\mbZ)=\#(H_{2}^{\rm ab}\otimes\mbZ/(p^{t}-1)\mbZ).$$ Suppose that $(\#G, p)=1$.  Since  $X_{1}^{\bp}$ and $X_{2}^{\bp}$ are equal types, we have that $$\#(H_{1}^{\rm ab}\otimes\mbZ/(p^{t}-1)\mbZ)=\#(H_{2}^{\rm ab}\otimes\mbZ/(p^{t}-1)\mbZ).$$

Then $\phi|_{H_{1}}$ implies that $${\rm Avr}_{p}(H_{1})\defeq \lim_{t \migi \infty} \frac{\text{dim}_{\mbF_{p}}(K^{\rm ab}_{H_{1}, p^{t}-1}\otimes \mbF_{p})}{\#(H_{1}^{\rm ab}\otimes \mbZ/(p^{t}-1)\mbZ)}$$ $$\geq {\rm Avr}_{p}(H_{2})\defeq \lim_{t \migi \infty} \frac{\text{dim}_{\mbF_{p}}(K^{\rm ab}_{H_{2}, p^{t}-1}\otimes \mbF_{p})}{\#(H_{2}^{\rm ab}\otimes \mbZ/(p^{t}-1)\mbZ)}.$$

Suppose that $G$ is solvable. Then the lemma follows immediately from the lemma if either $G$ is a $p$-group, or $(\#G, p)=1$. This completes the proof of the lemma.
\end{proof}

\section{Moduli spaces of admissible fundamental groups and the Homeomorphism Conjecture}\label{sec-5}

In this section, we define the moduli spaces of fundamental groups and formulate the Homeomorphism Conjecture, which are main research objects of the present paper.

\subsection{The Weak Isom-version Conjecture}\label{sec-5-1}

Let $p$ be a prime number, $\mbF_{p}$ the prime field of characteristic $p$, and $\overline \mbF_{p}$ an algebraically closed field of $\mbF_{p}$. Let $\overline \mcM_{g, n}$ be the moduli stack over $\overline \mbF_{p}$ classifying pointed
stable curves of type $(g, n)$ and $\mcM_{g, n} \subseteq \overline \mcM_{g, n}$ the open substack classifying smooth pointed stable curves. Let $\overline M_{g, n}$ and $M_{g, n}$ be the coarse moduli spaces of $\overline \mcM_{g,n}$ and $\mcM_{g, n}$, respectively.

Let $q \in \overline M_{g, n}$ be an arbitrary point, $k(q)$ the residue field of $\overline M_{g, n}$, and $k_{q}$ an algebraically closed field which contains $k(q)$. Then the composition of natural morphisms $$\spec k_{q} \migi \spec k(q) \migi \overline M_{g, n}$$ determines a pointed stable curve $X^{\bp}_{k_{q}}$  of type $(g, n)$ over $k_{q}$. Write $\pi_{1}^{\rm adm}(X^{\bp}_{k_{q}})$ for the admissible fundamental group $X^{\bp}_{k_{q}}$ and $\pi_{1}^{\rm adm}(X^{\bp}_{k_{q}})^{\rm sol}$ for the solvable admissible fundamental group of $X^{\bp}_{k_{q}}$. Since the isomorphism classes of $\pi_{1}^{\rm adm}(X^{\bp}_{k_{q}})$ and $\pi_{1}^{\rm adm}(X^{\bp}_{k_{q}})^{\rm sol}$ do not depend on the choice of $k_{q}$, we shall denote by $$\pi_{1}^{\rm adm}(q), \ \pi_{1}^{\rm sol}(q)$$ the admissible fundamental groups $\pi_{1}^{\rm adm}(X^{\bp}_{k_{q}})$ and the solvable admissible fundamental group $\pi_{1}^{\rm adm}(X^{\bp}_{k_{q}})^{\rm sol}$, respectively. Moreover, we shall denote by $$X^{\bp}_{q}$$ the pointed stable curve $X^{\bp}_{\overline {k(q)}}$ and $\Gamma_{q}$ the dual semi-graph of $X^{\bp}_{q}$,  where $\overline {k(q)}$ is an algebraic closure of $k(q)$. Let $v \in v(\Gamma_{q})$. Then the smooth pointed stable curve $\widetilde X^{\bp}_{q, v}$ of type $(g_{v}, n_{v})$ associated to $v$ determines a morphism $$\spec \overline {k(q)} \migi M_{g_{v}, n_{v}}.$$ We shall write $q_{v} \in M_{g_{v}, n_{v}}$ for the image of the morphism and say {\it $q_{v}$ the point of type $(g_{v}, n_{v})$ associated to $v$}.

\begin{definition}\label{fe}

(a) Let $q_{i} \in M_{g, n}$, $i \in \{1, 2\}$,  be an arbitrary point. We shall say that $q_{1}$ is {\it Frobenius equivalent} to $q_{2}$ if $X_{q_{1}} \setminus D_{X_{q_{1}}}$ is isomorphic to $X_{q_{2}} \setminus D_{X_{q_{2}}}$ as schemes.

(b) Let $q_{i} \in \overline M_{g, n}$, $i \in \{1, 2\}$,  be an arbitrary point. We shall say that $q_{1}$ is {\it Frobenius equivalent} to $q_{2}$ if the following conditions are satisfied: 

(i) There exists an isomorphism $\rho: \Gamma_{q_{1}} \isom \Gamma_{q_{2}}$ of dual semi-graphs. 

(ii) Let $v_{1} \in v(\Gamma_{q_{1}})$, $v_{2} \defeq \rho(v_{1}) \in v(\Gamma_{q_{2}})$, $q_{1, v_{1}}$ the point of type $(g_{v_{1}}, n_{v_{1}})$ associated to $v_{1}$, and $q_{2, v_{2}}$ the point of type $(g_{v_{2}}, n_{v_{2}})$ associated to $v_{2}$. We have that $q_{1,v_{1}}$ is Frobenius equivalent to $q_{2, v_{2}}$. 

(iii) Let $\rho_{v_{1}, v_{2}}: \Gamma_{q_{1, v_1}} \isom \Gamma_{q_{2, v_2}}$ be the isomorphism of dual semi-graphs induced by $\rho$. There exists a morphism $f_{v_{1}, v_{2}}^{\bp}: X_{q_{1,v_{1}}}^{\bp} \isom X^{\bp}_{q_{2,v_{2}}}$ such that the morphism $X_{q_{1,v_{1}}} \setminus D_{X_{q_{1,v_{1}}}} \migi X_{q_{2,v_{2}}} \setminus D_{X_{q_{2,v_{2}}}}$ induced by  $f_{v_{1}, v_{2}}^{\bp}$ is an isomorphism as schemes, and that the isomorphism of dual semi-graphs $f_{v_{1}, v_{2}}^{\rm sg}: \Gamma_{q_{1, v_1}} \isom \Gamma_{q_{2, v_2}}$ induced by $f_{v_{1}, v_{2}}^{\bp}$ coincides with $\rho_{v_{1}, v_{2}}$.


We shall denote by $$q_{1} \sim_{fe} q_{2}$$ if $q_{1}$ is Frobenius equivalent to $q_{2}$. Note that $\sim_{fe}$ is an equivalence relation on the underlying topological space $|\overline M_{g, n}|$ of $\overline M_{g, n}$

(c) Let $q_{i} \in \overline M_{g, n}$, $i \in \{1, 2\}$,  be an arbitrary point, $k_{q_{i}}$ an algebraically closed field which contains $k(q_{i})$, and $X^{\bp}_{k_{q_{i}}}$ the pointed stable curve of type $(g, n)$ over $k_{q_{i}}$. We shall say that $X^{\bp}_{k_{q_{1}}}$ {\it is Frobenius equivalent} to $X^{\bp}_{k_{q_{2}}}$ if $q_{1}$ is Frobenius equivalent $q_{2}$.
\end{definition}

The following result was proved by the author.

\begin{proposition}\label{prop-5-2}
Let $q_{i} \in \overline M_{g, n}$, $i \in \{1, 2\}$, be an arbitrary point. Suppose that $q_{1}\sim_{fe}q_{2}$. Then we have that $\pi_{1}^{\rm adm}(q_{1})$ is isomorphic to $\pi_{1}^{\rm adm}(q_{2})$ as profinite groups. In particular, we have that $\pi_{1}^{\rm sol}(q_{1})$ is isomorphic to $\pi_{1}^{\rm sol}(q_{2})$ as profinite groups. 
\end{proposition}

\begin{proof}
See \cite[Proposition 3.7]{Y7}.
\end{proof}

We put $$\mfM_{g, n} \defeq | M_{g, n}|/\sim_{fe} \subseteq \overline \mfM_{g, n} \defeq |\overline M_{g, n}|/\sim_{fe},$$ $$\Pi_{g, n} \defeq \{[\pi_{1}^{\rm adm}(q)] \ | \ q\in  M_{g, n}\}\subseteq \overline \Pi_{g, n} \defeq \{[\pi_{1}^{\rm adm}(q)] \ | \ q\in \overline M_{g, n}\},$$  $$ \Pi^{\rm sol}_{g, n} \defeq \{[\pi_{1}^{\rm sol}(q)] \ | \ q\in  M_{g, n}\}\subseteq \overline \Pi_{g, n}^{\rm sol} \defeq \{[\pi_{1}^{\rm sol}(q)] \ | \ q\in \overline M_{g, n}\},$$  where $[\pi_{1}^{\rm adm}(q)]$ and $[\pi_{1}^{\rm sol}(q)]$ denote the isomorphism classes (as profinite groups) of $\pi_{1}^{\rm adm}(q)$ and $\pi_{1}^{\rm sol}(q)$, respectively. Let $q \in \overline M_{g, n}$. We shall write $[q]$ for the image of $q$ in $\overline \mfM_{g, n}$. Then there are natural surjective maps of {\it sets} as follows: $$sol:  \overline \Pi_{g, n} \migisurj \overline \Pi_{g, n}^{\rm sol}, \ [\pi_{1}^{\rm adm}(q)] \mapsto [\pi_{1}^{\rm sol}(q)],$$ $$\pi_{g, n}^{\rm adm}: \overline \mfM_{g, n} \migisurj \overline \Pi_{g, n}, \ [q] \mapsto [\pi_{1}^{\rm adm}(q)],$$ $$\pi_{g, n}^{\rm sol}\defeq sol\circ \pi_{g, n}^{\rm adm}: \overline \mfM_{g, n} \migisurj \overline \Pi_{g, n}^{\rm sol},$$ $$\pi_{g, n}^{\rm t}\defeq \pi_{g, n}^{\rm adm}|_{\mfM_{g, n}}:  \mfM_{g, n} \migisurj  \Pi_{g, n},$$ $$\pi_{g, n}^{\rm t, sol}\defeq \pi_{g, n}^{\rm sol}|_{\mfM_{g, n}}:  \mfM_{g, n} \migisurj  \Pi_{g, n}^{\rm sol},$$ where ``t" means ``tame". Moreover, we have the following commutative diagrams:
\[
\begin{CD}
\mfM_{g, n} @>\pi^{\rm t}_{g, n}>> \Pi_{g, n}
\\
@VVV@VVV
\\
\overline \mfM_{g, n} @>\pi_{g, n}^{\rm adm}>> \overline \Pi_{g, n}
\end{CD}
\]
and
\[
\begin{CD}
\mfM_{g, n} @>\pi^{\rm t, sol}_{g, n}>> \Pi^{\rm sol}_{g, n}
\\
@VVV@VVV
\\
\overline \mfM_{g, n} @>\pi_{g, n}^{\rm sol}>> \overline \Pi_{g, n}^{\rm sol},
\end{CD}
\]
where all vertical arrows are natural injections.

\begin{proposition}\label{prop-5-3}
We maintain the notation introduced above. Then we have $$\pi_{g, n}^{\rm adm}(\overline \mfM_{g, n}\setminus \mfM_{g, n}) = \overline \Pi_{g, n} \setminus \Pi_{g, n},$$ $$\pi_{g, n}^{\rm sol}(\overline \mfM_{g, n}\setminus \mfM_{g, n}) = \overline \Pi_{g, n}^{\rm sol} \setminus \Pi_{g, n}^{\rm sol}.$$
\end{proposition}

\begin{proof}
The proposition follows immediately from \cite[Theorem 1.2, Remark 1.2.1, Remark 1.2.2, and Proposition 6.1]{Y1} (see also Theorem \ref{types} of the present paper).
\end{proof}

We may formulate the weak Isom-version of the Grothendieck conjecture for pointed stable curves over algebraically closed fields of characteristic $p>0$ (=the Weak Isom-version Conjecture) as follows:

\begin{weakisomGC}
We maintain the notation introduced above. Then we have that $$\pi_{g, n}^{\rm adm}: \overline \mfM_{g, n} \migisurj \overline \Pi_{g, n}$$ is a bijection as sets.
\end{weakisomGC}
\noindent
The Weak Isom-version Conjecture was formulated by Tamagawa in the case of $q\in M_{g, n}$, and by the author in the general case (cf. \cite{T2}, \cite{Y7}), which is the ultimate goal of \cite{PS}, \cite{R2}, \cite{T1}, \cite{T2}, \cite{T3}, \cite{T4}, \cite{To}, and \cite{Y1}. The Weak Isom-version Conjecture shows that moduli spaces of curves over algebraically closed fields of characteristic $p>0$ can be reconstructed group-theoretically as ``{\it sets}" from admissible fundamental groups of curves.

Moreover, we have the following solvable version of the Weak Isom-version Conjecture which is slightly stronger than the original version.

\begin{solweakisomGC}
We maintain the notation introduced above. Then we have that $$\pi_{g, n}^{\rm sol}: \overline \mfM_{g, n} \migisurj \overline \Pi^{\rm sol}_{g, n}$$ is a bijection as sets.
\end{solweakisomGC}

Write $\overline \mfM^{\rm cl}_{g, n}$ for the image of the set of closed points $\overline M_{g, n}^{\rm cl}$ of the natural map $|\overline M_{g, n}| \migisurj \overline \mfM_{g, n}$. Then we have the following result.

\begin{theorem}\label{them-5-4}
We maintain the notation introduced above. Then the following statements hold:

(a) We have that $$\pi_{g, n}^{\rm sol}|_{\overline \mfM^{\rm cl}_{g, n}}: \overline \mfM^{\rm cl}_{g, n} \migi \overline  \Pi_{g, n}^{\rm sol}$$ is quasi-finite (i.e., $\#(\pi_{g, n}^{\rm sol}|_{\overline \mfM^{\rm cl}_{g, n}})^{-1}([\pi_{1}^{\rm sol}(q)])<\muge$ for every $[\pi_{1}^{\rm sol}(q)] \in \overline \Pi_{g, n}^{\rm sol}$). 

(b) Suppose that $g=0$. Then we have that $$\pi_{g, n}^{\rm sol}|_{\overline \mfM^{\rm cl}_{g, n}}: \overline \mfM^{\rm cl}_{g, n} \migi \overline  \Pi_{g, n}^{\rm sol}$$ is an injection, and that $$\pi_{g, n}^{\rm sol}(\overline \mfM_{g, n} \setminus \overline \mfM_{g, n}^{\rm cl}) \subseteq  \overline \Pi_{g, n}^{\rm sol} \setminus \pi_{g, n}^{\rm sol}(\overline \mfM_{g, n}^{\rm cl}).$$ In particular, the Solvable Weak Isom-version Conjecture holds if $(g, n)=(0, 4)$.
\end{theorem}

\begin{proof}
Since \cite[Theorem 0.2]{T3} and \cite[Theorem 0.1]{T4} also hold for the maximal pro-solvable quotients of tame fundamental groups,  the theorem follows immediately from \cite[Theorem 0.2]{T3}, \cite[Theorem 0.1]{T4}, \cite[Theorem 1.2, Remark 1.2.1, Remark 1.2.2, and Proposition 6.1]{Y1}, and Proposition \ref{prop-5-3}.
\end{proof}

\subsection{Moduli spaces of admissible fundamental groups and the Homeomorphism Conjecture}\label{sec-5-2}

We maintain the notation introduced in Section \ref{sec-5-1}. Moreover, in the reminder of the present section, we regard $\overline \mfM_{g, n}$ and $\mfM_{g, n}$ as  topological spaces whose topologies are induced by the Zariski topologies of $|\overline M_{g, n}|$ and $|M_{g, n}|$, respectively. 

Let $\msG$ be the category of finite groups, $G \in \msG$ an arbitrary finite group, $\Pi$ an arbitrary profinite group, and $\text{Hom}_{\rm surj}(-,-)$ the set of surjective homomorphisms. We put $$U_{\overline \Pi_{g, n}, G} \defeq \{ [\pi_{1}^{\rm adm}(q)] \in \overline \Pi_{g, n}\ | \  \text{Hom}_{\rm surj}(\pi_{1}^{\rm adm}(q), G) \neq \emptyset\},$$ $$U_{\Pi_{g, n}, G} \defeq \{ [\pi_{1}^{\rm adm}(q)] \in  \Pi_{g, n}\ | \  \text{Hom}_{\rm surj}(\pi_{1}^{\rm adm}(q), G) \neq \emptyset\},$$
$$U_{\overline \Pi^{\rm sol}_{g, n}, G} \defeq \{ [\pi_{1}^{\rm sol}(q)] \in \overline \Pi_{g, n}^{\rm sol}\ | \  \text{Hom}_{\rm surj}(\pi_{1}^{\rm sol}(q), G) \neq \emptyset\},$$
$$U_{\Pi^{\rm sol}_{g, n}, G} \defeq \{ [\pi_{1}^{\rm sol}(q)] \in \Pi_{g, n}^{\rm sol}\ | \  \text{Hom}_{\rm surj}(\pi_{1}^{\rm sol}(q), G) \neq \emptyset\}.$$

Then we obtain the following topological spaces $$(\overline \Pi_{g, n}, O_{\overline \Pi_{g, n}}), \ (\Pi_{g, n}, O_{\Pi_{g, n}}),$$ $$(\overline \Pi^{\rm sol}_{g, n}, O_{\overline \Pi^{\rm sol}_{g, n}}), \  (\Pi^{\rm sol}_{g, n}, O_{\Pi^{\rm sol}_{g, n}})$$ whose topologies $O_{\overline \Pi_{g, n}}$, $O_{ \Pi_{g, n}}$, $O_{\overline \Pi_{g, n}^{\rm sol}}$, and $O_{\Pi_{g, n}^{\rm sol}}$ are generated by $\{U_{\overline \Pi_{g, n}, G} \}_{G \in \msG}$, $\{U_{\Pi_{g, n}, G} \}_{G \in \msG}$, $\{U_{\overline \Pi^{\rm sol}_{g, n}, G} \}_{G \in \msG}$, and $\{U_{\Pi^{\rm sol}_{g, n}, G} \}_{G \in \msG}$ as open subsets, respectively. For simplicity, we still use the notation $$\overline \Pi_{g, n}, \ \Pi_{g, n}, \ \overline \Pi_{g, n}^{\rm sol}, \ \Pi_{g, n}^{\rm sol}$$ to denote the topological spaces $(\overline \Pi_{g, n}, O_{\overline \Pi_{g, n}})$, $(\Pi_{g, n}, O_{\Pi_{g, n}})$, $(\overline \Pi^{\rm sol}_{g, n}, O_{\overline \Pi^{\rm sol}_{g, n}})$, and $(\Pi^{\rm sol}_{g, n}, O_{\Pi^{\rm sol}_{g, n}})$, respectively.

\begin{theorem}\label{continuous}
We maintain the notation introduced above. Then we have that $$\pi_{g, n}^{\rm adm}: \overline \mfM_{g, n} \migi \overline \Pi_{g, n},$$ $$\pi_{g, n}^{\rm sol}: \overline \mfM_{g, n} \migi \overline \Pi^{\rm sol}_{g, n}$$ are continuous maps.
\end{theorem}

\begin{proof}
The theorem will be proved in Section \ref{sec-7}, see Theorem \ref{continuousmap}.
\end{proof}

\begin{proposition}\label{prop-5-5}
We maintain the notation introduced above. Then the following statements hold.

(a) Let $[\pi_{1}^{\rm adm}(q)] \in \overline \Pi_{g, n}$ and $[\pi_{1}^{\rm sol}(q)] \in \overline \Pi_{g, n}^{\rm sol}$ be arbitrary points. Then we have $$V([\pi_{1}^{\rm adm}(q)])=\{[\pi_{1}^{\rm adm}(q')] \in  \overline \Pi_{g, n}\ | \ {\rm Hom}_{\rm surj}(\pi_{1}^{\rm adm}(q), \pi_{1}^{\rm adm}(q'))\neq \emptyset\},$$  $$V([\pi_{1}^{\rm sol}(q)])=\{[\pi_{1}^{\rm sol}(q')] \in  \overline \Pi_{g, n}^{\rm sol}\ | \ {\rm Hom}_{\rm surj}(\pi_{1}^{\rm sol}(q), \pi_{1}^{\rm sol}(q'))\neq \emptyset\},$$ where $V([\pi_{1}^{\rm adm}(q)])$ and $V([\pi_{1}^{\rm sol}(q)])$ denote the topological closures of $[\pi_{1}^{\rm adm}(q)]$ and $[\pi_{1}^{\rm sol}(q)]$ in $\overline \Pi_{g, n}$ and $\overline \Pi_{g, n}^{\rm sol}$, respectively.

(b) We have that 
$$\Pi_{g, n} \subseteq \overline \Pi_{g, n},\ \Pi_{g, n}^{\rm sol} \subseteq \overline \Pi_{g, n}^{\rm sol}$$ are open subsets. 

(c) Let $Z$ be an arbitrary irreducible closed subset of $\overline \mfM_{g, n}$. Then $V(\pi_{g, n}^{\rm adm}(Z))$ and $V(\pi_{g, n}^{\rm sol}(Z))$ are irreducible closed subsets of $\overline \Pi_{g, n}$ and $\overline \Pi_{g, n}^{\rm sol}$, respectively, where $V(\pi_{g, n}^{\rm adm}(Z))$ and $V(\pi_{g, n}^{\rm sol}(Z))$ denote the topological closures of $\pi_{g, n}^{\rm adm}(Z)$ and $\pi_{g, n}^{\rm sol}(Z)$ in $\overline \Pi_{g, n}$ and $\overline \Pi_{g, n}^{\rm sol}$, respectively. In particular, the topological spaces $\overline \Pi_{g, n}$ and $\overline \Pi_{g, n}^{\rm sol}$ are irreducible.

(d) Let $V$ be either an irreducible closed subset of $\overline \Pi_{g, n}$ or an irreducible closed subset of $\overline \Pi^{\rm sol}_{g, n}$. Then $V$ has a unique generic point.
\end{proposition}

\begin{proof}
(a) follows immediately from the definitions of $O_{\overline \Pi_{g, n}}$ and $O_{\overline \Pi_{g, n}^{\rm sol}}$, respectively.

(b) Let $[\pi^{\rm adm}_{1}(q)] \in \Pi_{g, n}$ be an arbitrary point and $\pi_{A}^{\rm adm}(q)$ the set of finite quotients of $\pi^{\rm adm}_{1}(q)$. Moreover, since $\pi_{1}^{\rm adm}(q)$ is topologically finitely generated, we have a subset of open normal subgroups $\{H_{j}\}_{j \in \mbN}$ of $\pi_{1}^{\rm adm}(q)$ such that $H_{j_{1}} \subseteq H_{j_{2}}$ for any $j_{1}\geq j_{2}$, and that $$\pi_{1}^{\rm adm}(q) \cong \invlim_{j \in \mbN}\pi_{1}^{\rm adm}(q)/H_{j}.$$ We put $$S(q)\defeq \{\pi_{1}^{\rm adm}(q)/H_{j}, \ j \in \mbN\} \subseteq \pi_{A}^{\rm adm}(q).$$ We see that, in order to prove that $\Pi_{g, n}$ is an open subset of $\overline \Pi_{g, n}$, it is sufficient to prove that, for every point $[q_{2}] \in \mfM_{g, n}$, there exists a finite group $G \in S(q_{2})$ such that $U_{\overline \Pi_{g, n}, G}$ is contained in $\Pi_{g, n}$. 

Suppose that $U_{\overline \Pi_{g, n}, G} \cap (\overline \Pi_{g, n} \setminus \Pi_{g, n}) \neq\emptyset$ for every $G \in S(q_{2})$. Since $\pi_{g, n}^{\rm adm}$ is continuous (cf. Theorem \ref{continuous}) and the set of generic points of $\overline \mfM_{g, n} \setminus \mfM_{g, n}$ is finite, there exists a generic point $[q_{1}]$ of $\overline \mfM_{g, n} \setminus \mfM_{g, n}$ such that $$[\pi_{1}^{\rm adm}(q_{1})] \in \bigcap_{G \in S(q_{2})}U_{\overline \Pi_{g, n}, G}.$$ Then  the set $$\text{Hom}_{\rm surj}(\pi_{1}^{\rm adm}(q_{1}), \pi_{1}^{\rm adm}(q_{2}))=\invlim_{G \in S(q_{2})}\text{Hom}_{\rm surj}(\pi_{1}^{\rm adm}(q_{1}), G)$$ is not empty. Thus, there is a surjective open continuous homomorphism $\phi: \pi_{1}^{\rm adm}(q_{1}) \migisurj \pi_{1}^{\rm adm}(q_{2}).$ Note that $\phi$ induces an isomorphism of maximal prime-to-$p$ quotients $$\phi^{p'}: \pi_{1}^{\rm adm}(q_{1})^{p'} \isom \pi_{1}^{\rm adm}(q_{2})^{p'}.$$ 

By applying \cite[Lemma 6.3]{Y1}, there exists an open characteristic subgroup $H_{1} \subseteq \pi_{1}^{\rm adm}(q_{1})^{p'}$ such that the pointed stable curve $X_{H_{1}}^{\bp}$ of type $(g_{X_{H_{1}}}, n_{X_{H_{1}}})$ over $k_{q_{1}}$ corresponding to $H_{1}$ satisfying the following conditions: (1) $\Gamma_{X_{H_{1}}^{\bp}}^{\rm cpt}$ is $2$-connected; (2) $\#(v(\Gamma_{X_{H_{1}}^{\bp}})^{b\leq 1})=0$; (3) the Betti number $r_{X_{H_{1}}}$ of the dual semi-graph of $X_{H_{1}}^{\bp}$ is positive. Let $H_{2} \defeq \phi^{p'}(H_{1}) \subseteq \pi_{1}^{\rm adm}(q_{2})^{p'}$. Then we obtain a smooth pointed stable curve $X_{H_{2}}^{\bp}$ of type $(g_{X_{H_{2}}}, n_{X_{H_{2}}})$ over $k_{q_{2}}$ corresponding to $H_{2}$. Since $H_{i}$ is an open characteristic subgroup, we obtain that $(g_{X_{H_{1}}}, n_{X_{H_{1}}})=(g_{X_{H_{2}}}, n_{X_{H_{2}}})$. Then Theorem \ref{max and average} (b) and Lemma \ref{lem-0} (b) imply that $$r_{X_{H_{1}}} \leq 0.$$ This contradicts $r_{X_{H_{1}}}>0$. 

Similar arguments to the arguments given in the proof above imply that $\Pi_{g, n}^{\rm sol}$ is an open subset of $\overline \Pi_{g, n}^{\rm sol}$. This completes the proof of (b).

(c) is trivial.

(d) We only treat the case where $V$ is an irreducible closed subset of $\overline \Pi_{g, n}$. Let ${\rm Gen}(V)$ be the set of generic points of $V$. Since every closed subset of $\overline \mfM_{g, n}$ has a non-empty set of generic points, we have that ${\rm Gen}(V) \neq \emptyset$. Let $[\pi_{1}^{\rm adm}(q_{1})]$, $[\pi_{1}^{\rm adm}(q_{2})] \in {\rm Gen}(V)$ be arbitrary generic points. Let $G \in \pi_{A}^{\rm adm}(q_{1})$ be an arbitrary finite group. Then $U_{\overline \Pi_{g, n}, G} \cap V \neq \emptyset$. Thus, $[\pi_{1}^{\rm adm}(q_{2})] \in U_{\overline \Pi_{g, n}, G} \cap V$. This means that $\pi_{A}^{\rm adm}(q_{1}) \subseteq \pi_{A}^{\rm adm}(q_{2}).$ Similar arguments to the arguments given in the proof above imply $\pi_{A}^{\rm adm}(q_{1}) \supseteq \pi_{A}^{\rm adm}(q_{2}).$ Then we have $$\pi_{A}^{\rm adm}(q_{1}) = \pi_{A}^{\rm adm}(q_{2}).$$ Since admissible fundamental groups of pointed stable curves are topologically finitely generated, \cite[Proposition 16.10.6]{FJ} implies that $[\pi_{1}^{\rm adm}(q_{1})]=[\pi_{1}^{\rm adm}(q_{2})]$. This completes the proof of the proposition.
\end{proof}

\begin{definition}
We shall say that $$\overline \Pi_{g, n}$$ is {\it the moduli space of admissible fundamental groups of pointed stable curves of type $(g, n)$ over algebraically closed fields of characteristic $p$} (or {\it the moduli space of admissible fundamental groups of type $(g, n)$ in characteristic $p$} for short), and that $$\overline \Pi_{g, n}^{\rm sol}$$ is {\it the moduli space of solvable admissible fundamental groups of pointed stable curves of type $(g, n)$ over algebraically closed fields of characteristic $p>0$} (or {\it the moduli space of solvable admissible fundamental groups of type $(g, n)$ in characteristic $p$} for short). Moreover, we shall say that $O_{\overline \Pi_{g, n}}$ and $O_{\overline \Pi_{g, n}^{\rm sol}}$ are the {\it anabelian topologies} of $\overline \Pi_{g, n}$ and $\overline \Pi_{g, n}^{\rm sol}$, respectively.
\end{definition}

Next, we formulate the main conjectures of the theory developing in the present paper.

\begin{homeconj}
We maintain the notation introduced above. Then we have that $$\pi_{g, n}^{\rm adm}: \overline \mfM_{g, n} \migisurj \overline \Pi_{g, n}$$ is a homeomorphism.
\end{homeconj}
\noindent
Moreover, we have a solvable version of the Homeomorphism Conjecture as follows, which is slightly stronger than the original version.

\begin{solhomeconj}
We maintain the notation introduced above. Then we have that $$\pi_{g, n}^{\rm sol}: \overline \mfM_{g, n} \migisurj \overline \Pi^{\rm sol}_{g, n}$$ is a homeomorphism. 
\end{solhomeconj}
\noindent
The Homeomorphism Conjecture (or the Solvable Homeomorphism Conjecture) shows that moduli spaces of curves over algebraically closed fields of characteristic $p>0$ can be reconstructed group-theoretically as ``{\it topological spaces} " from admissible fundamental groups (or solvable admissible fundamental groups) of curves. Moreover, the conjectures give a new insight into the
theory of anabelian geometry of curves over algebraically closed fields of characteristic $p>0$ based on the following anabelian philosophy:
\begin{quote}
The topological space $\overline \Pi_{g, n}$ (or $\overline \Pi^{\rm sol}_{g, n}$) contains all anabelian informations of pointed stable curves of type $(g, n)$ over algebraically closed fields of characteristic $p>0$, and every topological property concerning the topological space $\overline \Pi_{g, n}$ (or $\overline \Pi^{\rm sol}_{g, n}$) can be regarded as an anabelian property of pointed stable curves of type $(g, n)$ over algebraically closed fields of characteristic $p > 0$.
\end{quote}
The main theorem of the present paper is the following, which will be proved in Section \ref{sec-6} (cf. Theorem \ref{mainthem-form-2}).

\begin{theorem}\label{main-them}
We maintain the notation introduced above. Let $[q] \in \overline \mfM^{\rm cl}_{0, n}$ be an arbitrary closed point. Then $\pi_{0, n}^{\rm adm}([q])$ and $\pi_{0, n}^{\rm sol}([q])$ are closed points of $\overline \Pi_{0, n}$ and $\overline \Pi_{0, n}^{\rm sol}$, respectively. In particular, the Homeomorphism Conjecture and the Solvable Homeomorphism Conjecture hold when $(g, n)=(0, 4)$.  

\end{theorem}

On the other hand, some major problems concerning $\overline \Pi_{g, n}$ are as follows:

\begin{problem}\label{prob-5-7} 
We maintain the notation introduced above.

\begin{enumerate}

\item Is $\overline \Pi_{g, n}$ a noetherian topological space? 

\item Let $[q] \in \overline \mfM_{g, n}^{\rm cl}$. Is $V(\pi_{g, n}^{\rm adm}([q]))$ a finite set?

\item Let $Z$ be an irreducible closed subset of $\overline \mfM_{g, n}$. Is $\pi_{g, n}^{\rm adm}(Z)$ an irreducible closed subset of $\overline \Pi_{g, n}$?

\item Let $i \in \{1, 2\}$, and let $V_{i, m_{i}} \subseteq \dots \subseteq V_{i, 1} \subseteq V_{i, 0}\defeq \overline \Pi_{g, n}$ be an arbitrary maximal chain of irreducible closed subsets of $\overline \Pi_{g, n}$. Does $m_{1}=m_{2}$ hold?

\item Let $Z$ be an irreducible closed subset of $\overline \mfM_{g, n}$. Does ${\rm dim}(Z)={\rm dim}(V(\pi_{g, n}^{\rm adm}(Z)))$ hold? Here,  ${\rm dim}(V(\pi_{g, n}^{\rm adm}(Z)))$ denotes the Krull dimension of $V(\pi_{g, n}^{\rm adm}(Z))$. In particular, does ${\rm dim}(\overline \mfM_{g, n})={\rm dim}(\overline \Pi_{g, n})$ and ${\rm dim}(V(\pi_{g, n}^{\rm adm}([q])))=0$ for every $[q] \in \overline \mfM_{g, n}^{\rm cl}$ hold? Moreover, Is $\pi_{g, n}^{\rm adm}([q])$ a closed point of $\overline \Pi_{g, n}$ for every $[q] \in \overline \mfM_{g, n}^{\rm cl}$?

\item Prove the Homeomorphism Conjecture. In particular, prove the Homeomorphism Conjecture for $(0, n)$.

\end{enumerate}
\end{problem}

\begin{remarkA}
We may also ask the problems mentioned above for $\overline \Pi_{g, n}^{\rm sol}$.
\end{remarkA}

\begin{remarkB}
Problem \ref{prob-5-7} (2) is equivalent to the following anabelian property of pointed stable curves:
\begin{quote}
Let $[q] \in \overline \mfM_{g, n}^{\rm cl}$. Then we have that the set $$\{q' \in \overline \mfM_{g, n}^{\rm cl} \ | \ \text{Hom}_{\rm surj}(\pi_{1}^{\rm adm}(q), \pi_{1}^{\rm adm}(q')) \neq \emptyset\}$$ is a finite set. 
\end{quote}
This property is a generalized version of Theorem \ref{them-5-4} (a).
\end{remarkB}

\begin{remarkC}
We maintain the notation introduced above. Tamagawa posed a conjecture as follows (cf. \cite[Conjecture 5.3 (ii)]{T2}), which is called {\it Essential Dimension Conjecture}: 
\begin{quote}
Let $i\in \{1, 2\}$, and let $q_{i} \in \mfM_{g, n}$ and $V(q_{i})$ the topological closure of $q_{i}$ in $\overline \mfM_{g, n}$. Then we have that $\text{dim}(V(q_{1}))=\text{dim}(V(q_{2}))$ if $[\pi_{1}^{\rm adm}(q_{1})]=[\pi_{1}^{\rm adm}(q_{2})]$.
\end{quote}
We see immediately that Problem \ref{prob-5-7} (5) is a generalized version of Tamagawa's Essential Dimension Conjecture.
\end{remarkC}

\begin{proposition}\label{prop-5-8}
Let $[q] \in \overline \mfM_{g, n}^{\rm cl}$. Then we have that ${\rm dim}(V(\pi_{g, n}^{\rm adm}([q])))=0$ if and only if $\pi_{g, n}^{\rm adm}([q])$ is a closed point of $\overline \Pi_{g, n}$. 
\end{proposition}

\begin{proof}
The ``if" part of the proposition is trivial. We only need to prove the ``only if" part of the proposition.

Let $[\pi_{1}^{\rm adm}(q')] \in V(\pi_{g, n}^{\rm adm}([q]))$ be an arbitrary point and $V([\pi_{1}^{\rm adm}(q')])$ the topological closure of $[\pi_{1}^{\rm adm}(q')]$ in $\overline \Pi_{g, n}$. Then we have that $V([\pi_{1}^{\rm adm}(q')])$ is an irreducible closed subset which is contained in $V(\pi_{g, n}^{\rm adm}([q]))$. Since $V(\pi_{g, n}^{\rm adm}([q]))$ is an irreducible closed subset of dimension $0$, we obtain that $$V(\pi_{g, n}^{\rm adm}([q]))=V([\pi_{1}^{\rm adm}(q')]).$$ This means that there exist surjective open continuous homomorphisms $$\pi_{1}^{\rm adm}(q) \migisurj \pi_{1}^{\rm adm}(q'),$$ $$\pi_{1}^{\rm adm}(q') \migisurj \pi_{1}^{\rm adm}(q).$$ Then we obtain $\pi_{A}^{\rm adm}(q)=\pi_{A}^{\rm adm}(q')$. Since admissible fundamental groups of pointed stable curves are topologically finitely generated, \cite[Proposition 16.10.6]{FJ} implies that $[\pi_{1}^{\rm adm}(q)]=[\pi_{1}^{\rm adm}(q')]$. Thus, we obtain $V(\pi_{g, n}^{\rm adm}([q]))=[\pi_{1}^{\rm adm}(q)]$. This completes the proof of the proposition.
\end{proof}



\section{Reconstruction of inertia subgroups and field structures from surjections}\label{sec-3}

In this section, we will prove that the inertia subgroups assoicated to marked points can be reconstructed group-theoretically from surjective homomorphisms of admissible fundamental groups (or solvable admissible fundamental groups). 

Let $\mcP$ be a category of profinite groups whose class of objects $\text{Ob}(\mcP)$ consists of profinite groups, and whose class of morphisms $\text{Hom}_{\mcP}(\Pi, \Pi')$ is the class of open continuous homomorphisms of $\Pi$ and $\Pi'$. Let $\Pi \in \mcP$, and let $\mfS$ be a category whose class of objects $\text{Ob}(\mfS)$ is a set of subgroups of $\Pi$, and whose class of morphisms $\text{Hom}_{\mfS}(H, H')$ for any $H, H' \in \mfS$ is defined as follows: the unique element of $\text{Hom}_{\mfS}(H, H')$ is the natural inclusion when $H \subseteq H'$; otherwise, $\text{Hom}_{\mfS}(H, H')$ is empty. We shall say that $\mfS$ is a category associated to $\Pi$.

Let $\mfS$ and $\mfS'$ be categories associated to profinite groups $\Pi$ and $\Pi'$, respectively. We define a class of functors $\text{Hom}_{\mcS}(\mfS, \mfS')$ as follows: $\theta_{\mcS}\in \text{Hom}_{\mcS}(\mfS, \mfS')$ if there exists an open continuous homomorphism $\theta: \Pi \migi \Pi'$ such that $\mfS=\{H\defeq \theta^{-1}(H')\}_{H' \in \mfS'}$, and that $\theta_{\mcS}: \mfS\migi \mfS'$, $H \mapsto H' $; otherwise, $\text{Hom}_{\mcS}(\mfS, \mfS')$ is empty. We denote by $$\mcS$$ the category whose class of objects $\text{Ob}(\mcS)$ is the class of categories associated to profinite groups, and whose class of morphisms $\text{Hom}_{\mcS}(\mfS, \mfS')$ consists of the classes of functors defined above. Then we have a natural functor $\pi: \mcS \migi \mcP$ defined as follows: Let $\mfS$, $\mfS' \in \mcS$ be categories associated to profinite groups $\Pi$, $\Pi'$, respectively; we have $\pi(\mfS)=\Pi$, $\pi(\mfS')=\Pi'$, and $\pi(\theta_{\mcS})=\theta$.  We see immediately that $$\pi: \mcS \migi \mcP$$ is a fibered category over $\mcP$.

\begin{definition}\label{def-3-1}
Let $i \in \{1, 2\}$. Let $\mcF_{i}$ be a geometric object (in a certain category), $\Pi_{\mcF_{i}}$ a profinite group associated to the geometric object $\mcF_{i}$, and $\mfS_{i}$ a category associated to $\Pi_{\mcF_{i}}$. Let $\text{Inv}_{\mcF_{i}}$ be an invariant depending on the isomorphism class of $\mcF_{i}$ (in a certain category) and $\text{Add}_{\mcF_{i}}(\mfS_{i})$ an additional structure associated to $\mfS_{i}$  (e.g., $\text{Add}_{\mcF_{i}}(\mfS_{i})=\mfS_{i}$) on the profinite group $\Pi_{\mcF_{i}}$ depending functorially on $\mcF_{i}$ and $\mfS_{i}$. 

(a) We shall say that $\text{Inv}_{\mcF_{i}}$ can be {\it reconstructed group-theoretically} from $\Pi_{\mcF_{i}}$ (or $\Pi_{\mcF_{i}}$ induces $\text{Inv}_{\mcF_{i}}$ group-theoretically) if $\Pi_{\mcF_{1}} \cong \Pi_{\mcF_{2}}$ implies $\text{Inv}_{\mcF_{1}}=\text{Inv}_{\mcF_{2}}$.

(b) We shall say that $\text{Add}_{\mcF_{2}}(\mfS_{2})$ can be {\it reconstructed group-theoretically} from $\Pi_{\mcF_{2}}$ (or $\Pi_{\mcF_{2}}$ induces $\text{Add}_{\mcF_{2}}(\mfS_{2})$ group-theoretically) if every isomorphism $\theta:\Pi_{\mcF_{1}} \isom \Pi_{\mcF_{2}}$ induces a bijection $\theta_{\rm ad}: \text{Add}_{\mcF_{1}}(\mfS_{1}) \isom \text{Add}_{\mcF_{2}}(\mfS_{2})$ which preserves the structures $\text{Add}_{\mcF_{1}}(\mfS_{1})$ and $\text{Add}_{\mcF_{2}}(\mfS_{2})$, where $\mfS_{1}\defeq \Pi_{\mcF_{1}}\times_{\theta, \Pi_{\mcF_{2}}}\mfS_{2}$ (i.e., the fiber product in the fibered category $\mcS$ over $\mcP$). 

(c) Let $j_{1}$, $ j_{2} \in \{1, 2\}$ distinct from each other, and let $\theta: \Pi_{\mcF_{1}} \migi \Pi_{\mcF_{2}}$ be an open continuous homomorphism of profinite groups and $\mfS_{1}=\Pi_{\mcF_{1}}\times_{\theta, \Pi_{\mcF_{2}}}\mfS_{2}$. We shall say that a map  $\theta_{\rm ad}: \text{Add}_{\mcF_{j_1}}(\mfS_{j_1}) \migi \text{Add}_{\mcF_{j_2}}(\mfS_{j_2})$ can be {\it reconstructed group-theoretically} from $\theta: \Pi_{\mcF_{1}} \migi \Pi_{\mcF_{2}}$ (or $\theta: \Pi_{\mcF_{1}} \migi \Pi_{\mcF_{2}}$ induces $\theta_{\rm ad}: \text{Add}_{\mcF_{j_1}}(\mfS_{j_1}) \migi \text{Add}_{\mcF_{j_2}}(\mfS_{j_2})$ group-theoretically) if the following condition is satisfied: Let $\mcF'_{i}$ be a geometric object, $\Pi_{\mcF'_{i}}$ a profinite group associated to the geometric object $\mcF'_{i}$, $\theta_{i}:\Pi_{\mcF'_{i}} \isom \Pi_{\mcF_{i}}$ an isomorphism of profinite groups, $\theta':\Pi_{\mcF_{1}'} \migi \Pi_{\mcF_{2}'}$, $\mfS_{i}'\defeq \Pi_{\mcF'_{i}}\times_{\theta_{i}, \Pi_{\mcF_{i}}}\mfS_{i}$, $\text{Add}_{\mcF'_{i}}(\mfS'_{i})$ an additional structure on the profinite group $\Pi_{\mcF'_{i}}$. Suppose that the following commutative diagram of profinite groups holds 
\[
\begin{CD}
\Pi_{\mcF'_{1}} @>\theta'>> \Pi_{\mcF'_{2}}
\\
@V\theta_{1}VV@V\theta_{2}VV
\\
\Pi_{\mcF_{1}} @>\theta>> \Pi_{\mcF_{2}}.
\end{CD}
\]
Then the commutative diagram of profinite groups above induces the following commutative diagram of additional structures
\[
\begin{CD}
\text{Add}_{\mcF'_{j_1}}(\mfS_{j_1}') @>\theta_{\rm ad}'>> \text{Add}_{\mcF'_{j_2}}(\mfS_{j_2}')
\\
@V\theta_{j_1, \rm ad}VV@V\theta_{j_2, \rm ad}VV
\\
\text{Add}_{\mcF_{j_1}}(\mfS_{j_1}) @>\theta_{\rm ad}>> \text{Add}_{\mcF_{j_2}}(\mfS_{j_2})
\end{CD}
\] 
which preserves the structures of additional structures.
\end{definition}

\begin{remarkA}
Let us explain the philosophy of {\it mono-anabelian geometry} introduced
by Mochizuki. The classical point of view of anabelian geometry (i.e., the
anabelian geometry considered in \cite{G1}, \cite{G2}) focuses on a comparison between two geometric objects via their fundamental groups. Moreover, the term ``group-theoretical", in the classical point of view, means that ``preserved by an arbitrary isomorphism between
the fundamental groups under consideration''. We shall refer to the classical point of view as ``{\it bi-anabelian geometry}''. Then Definition \ref{def-3-1} is a definition from the point of view of bi-anabelian geometry.

On the other hand, mono-anabelian geometry focuses on the establishing a group-theoretic algorithm whose input datum is an abstract topological group
which is isomorphic to the fundamental group of a given geometric object of interest
(resp. a continuous homomorphism of abstract topological groups which are isomorphic to a continuous homomorphism of the fundamental groups of given geometric objects of interest), and whose output datum is a geometric object which is isomorphic to the given geometric object of interest (resp. a morphism of geometric objects which is isomorphic to a morphism of given geometric objects of interest). In the point of view of mono-anabelian geometry, the term ``group-theoretic algorithm" is used to mean that ``the algorithm in a discussion is phrased in language that only depends on the topological group structures of the fundamental groups under consideration". Note that mono-anabelian results are stronger than bi-anabelian results.

We maintain the notation introduced in Definition \ref{def-3-1}. Then the mono-anabelian version of Definition \ref{def-3-1} is as follows: 

(a) We shall say that $\text{Inv}_{\mcF_{i}}$ can be {\it mono-anabelian reconstructed} from $\Pi_{\mcF_{i}}$ if there exists a group-theoretical algorithm whose input datum is $\Pi_{\mcF_{i}}$, and whose output datum is $\text{Inv}_{\mcF_{i}}(\mfS_{i})$. 

(b) We shall say that $\text{Add}_{\mcF_{i}}(\mfS_{i})$ can be {\it mono-anabelian reconstructed} from $\Pi_{\mcF_{i}}$ if there exists a group-theoretical algorithm whose input datum is $\Pi_{\mcF_{i}}$, and whose output datum is $\text{Add}_{\mcF_{i}}$. 

(c) Let $j_{1}$, $ j_{2} \in \{1, 2\}$ distinct from each other, and let $\theta: \Pi_{\mcF_{1}} \migi \Pi_{\mcF_{2}}$ be an open continuous homomorphism of profinite groups and $\mfS_{1}=\Pi_{\mcF_{1}}\times_{\theta, \Pi_{\mcF_{2}}}\mfS_{2}$. We shall say that a map (or a morphism) $\theta_{\rm add}: \text{Add}_{\mcF_{j_1}}(\mfS_{j_1}) \migi \text{Add}_{\mcF_{j_2}}(\mfS_{j_2})$ can be {\it mono-anabelian reconstructed} from $\theta: \Pi_{\mcF_{1}} \migi \Pi_{\mcF_{2}}$ if there exists a group-theoretical algorithm whose input datum is $\theta: \Pi_{\mcF_{1}} \migi \Pi_{\mcF_{2}}$, and whose output datum is $\theta_{\rm add}: \text{Add}_{\mcF_{j_1}}(\mfS_{j_1}) \migi \text{Add}_{\mcF_{j_2}}(\mfS_{j_2})$.  
\end{remarkA}

Let $i \in \{1, 2\}$, and let $X^{\bp}_{i}=(X_{i}, D_{X_{i}})$ be a pointed stable curve of type $(g_{X_{i}}, n_{X_{i}})$ over an algebraically closed field $k_{i}$ of characteristic $p_{i}>0$, $\Gamma_{X^{\bp}_{i}}$ the dual semi-graph of $X^{\bp}_{i}$, and $\Pi_{X_{i}^{\bp}}$ either the admissible fundamental group or the solvable admissible fundamental group of $X^{\bp}_{i}$. The following result was proved by Tamagawa for smooth pointed stable curves and by the author for arbitrary pointed stable curves.

\begin{theorem}\label{types}
We maintain the notation introduced above. Then the characteristic $p_{i}$ of $k_{i}$ and the type $(g_{X_{i}}, n_{X_{i}})$ can be reconstructed group-theoretically from $\Pi_{X_{i}^{\bp}}$. Moreover, $\Pi_{X^{\bp}_{i}}^{\text{\et}}$, $\Pi_{X_{i}^{\bp}}^{\rm top}$, ${\rm Ver}(\Pi_{X_{i}^{\bp}})$, ${\rm Edg}^{\rm op}(\Pi_{X_{i}^{\bp}})$, ${\rm Edg}^{\rm cl}(\Pi_{X_{i}^{\bp}})$, and $\Gamma_{X^{\bp}_{i}}$ can be reconstructed group-theoretically from $\Pi_{X_{i}^{\bp}}$.
\end{theorem}

\begin{proof}
See \cite[Theorem 1.2, Remark 1.2.1, Remark 1.2.2, and Proposition 6.1]{Y1}, \cite[Theorem 0.1]{T3}, and \cite[Theorem 1.3]{Y6}.
\end{proof}

\begin{remarkA}\label{types-1}
Note that \cite[Theorem 1.3]{Y6} gives a formula for $(g_{X_{i}}, n_{X_{i}})$. Then we obtain that the characteristic $p_{i}$ of $k_{i}$ and the type $(g_{X_{i}}, n_{X_{i}})$ can be {\it mono-anabelian reconstructed} from $\Pi_{X^{\bp}_{i}}$. In fact, we have that $\Pi_{X^{\bp}_{i}}^{\text{\et}}$, $\Pi_{X_{i}^{\bp}}^{\rm top}$, ${\rm Ver}(\Pi_{X_{i}^{\bp}})$, ${\rm Edg}^{\rm op}(\Pi_{X_{i}^{\bp}})$, ${\rm Edg}^{\rm cl}(\Pi_{X_{i}^{\bp}})$, and $\Gamma_{X^{\bp}_{i}}$ can be {\it mono-anabelian reconstructed} from $\Pi_{X_{i}^{\bp}}$ (see \cite[Theorem 0.2]{Y3}).
\end{remarkA}

\begin{remarkB}
We do not use the term ``mono-anabelian reconstruction" in the present paper. On the other hand, by applying Remark \ref{types-1}, all of the bi-anabelian results proved in the present paper can be generalized to the case of mono-anabelian reconstructions.  Moreover, mono-anabelian results will be used in \cite{Y8}, and play a fundamental role in \cite{Y9}.
\end{remarkB}

\begin{lemma}\label{surj}
We maintain the notation introduced above. Suppose that $p\defeq p_{1}=p_{2}$, that $(g_{X}, n_{X})\defeq (g_{X_{1}}, n_{X_{1}})=(g_{X_{2}}, n_{X_{2}})$. Let $\phi: \Pi_{X^{\bp}_{1}} \migi \Pi_{X_{2}^{\bp}}$ be an arbitrary  open continuous homomorphism. Then we have that $\phi$ is a surjection.
\end{lemma}

\begin{proof}
Let $\Pi_{\phi} \defeq \phi(\Pi_{X_{1}^{\bp}}) \subseteq \Pi_{X_{2}^{\bp}}$ be the image of $\phi$ which is an open subgroup of $\Pi_{X_{2}^{\bp}}$. Let $X_{\phi}^{\bp}=(X_{\phi}, D_{X_{\phi}})$ be the pointed stable curve of type $(g_{X_{\phi}}, n_{X_{\phi}})$ over $k_{2}$ induced by $\Pi_{\phi}$ and $$X_{\phi}^{\bp} \migi X^{\bp}_{2}$$ the admissible covering over $k_{2}$ induced by the natural inclusion $\Pi_{\phi} \migiinje\Pi_{X_{2}^{\bp}}$. The Riemann-Hurwitz formula implies that $g_{X_{\phi}}+n_{X_{\phi}} \geq g_{X}+n_{X}$. Moreover, by applying Theorem \ref{max and average} (a) and Lemma \ref{lem-0} (a), the natural surjection $\Pi_{X_{1}^{\bp}} \migisurj \Pi_{\phi}$ induced by $\phi$ imply that $g_{X}\geq g_{X_{\phi}}$ and $n_{X} \geq n_{X_{\phi}}$. Then we have $$(g_{X}, n_{X})=(g_{X_{\phi}}, n_{X_{\phi}}).$$ This means that the admissible covering $X^{\bp}_{\phi} \migi X^{\bp}_{2}$ is totally ramified over every marked point of $D_{X_{2}}$. Then the Riemann-Hurwitz formula implies that $[\Pi_{X_{1}^{\bp}}: \Pi_{\phi}]\neq 1$ if and only if $(g_{X}, n_{X})=(0, 2)$. Thus, we obtain that $\phi$ is a surjection.  
\end{proof}

In the remainder of this section, we suppose that $p\defeq p_{1}=p_{2}$, that $(g_{X}, n_{X})\defeq (g_{X_{1}}, n_{X_{1}})=(g_{X_{2}}, n_{X_{2}})$. Let $$\phi: \Pi_{X^{\bp}_{1}} \migi \Pi_{X_{2}^{\bp}}$$ be an arbitrary  open continuous homomorphism. By Lemma \ref{surj}, we see that $\phi$ is a {\it surjective} open continuous homomorphism. Let $\mfP$ be the set of prime numbers, $\Sigma \subseteq \mfP\setminus \{p\}$ a subset, $\Pi_{X_{i}^{\bp}}^{\Sigma}$  the maximal pro-$\Sigma$ quotient of $\Pi_{X_{i}^{\bp}}$, $pr_{i}^{\Sigma}: \Pi_{X_{i}^{\bp}} \migisurj \Pi_{X_{i}^{\bp}}^{\Sigma}$ the natural surjective homomorphism, and $$\phi^{\Sigma}: \Pi_{X^{\bp}_{1}}^{\Sigma} \isom \Pi_{X^{\bp}_{2}}^{\Sigma}$$ the isomorphism induced by $\phi$. In particular, if $\Sigma=\mfP\setminus\{p\}$, we use the notation $\Pi_{X_{i}^{\bp}}^{p'}$ and $\phi^{p'}: \Pi_{X^{\bp}_{1}}^{p'} \isom \Pi_{X^{\bp}_{2}}^{p'}$ to denote $\Pi_{X_{i}^{\bp}}^{\Sigma}$ and $\phi^{\Sigma}$, respectively.

\begin{lemma}\label{lem-1}
We maintain the notation introduced above. Then we have that $\Pi_{X_{i}^{\bp}}^{\rm cpt}$ can be reconstructed group-theoretically from $\Pi_{X_{i}^{\bp}}$, and that the (surjective) open continuous homomorphism $\phi: \Pi_{X^{\bp}_{1}} \migisurj \Pi_{X_{2}^{\bp}}$ induces a surjective homomorphism $$\phi^{\rm cpt}: \Pi_{X^{\bp}_{1}}^{\rm cpt} \migisurj \Pi_{X_{2}^{\bp}}^{\rm cpt}$$ group-theoretically. Moreover, the following commutative diagram of profinite groups
\[
\begin{CD}
\Pi_{X^{\bp}_{1}} @>\phi>> \Pi_{X^{\bp}_{2}}
\\
@VVV@VVV
\\
\Pi^{\rm cpt}_{X^{\bp}_{1}} @>\phi^{\rm cpt}>> \Pi^{\rm cpt}_{X^{\bp}_{2}}
\end{CD}
\]
can be reconstructed group-theoretically from $\phi$.  
\end{lemma}

\begin{proof}
By Theorem \ref{types}, we have that $(g_{X}, n_{X})$ can be reconstructed group-theoretically from $\Pi_{X_{i}^{\bp}}$. If $n_{X}=0$, the lemma is trivial. Then we may assume that $n_{X} >0$.

Let $H_{i} \subseteq \Pi_{X_{i}^{\bp}}$ be an open subgroup. Then the Riemann-Hurwitz formula implies that the admissible covering $$X^{\bp}_{H_{i}} \migi X^{\bp}_{i}$$ over $k_{i}$ induced by $H_{i} \subseteq \Pi_{X^{\bp}_{i}}$ is \'etale over $D_{X_{i}}$ if and only if $g_{X_{H_{i}}}=[\Pi_{X_{i}^{\bp}}: H_{i}](g_{X}-1)+1$. We put

$$\text{Et}^{\rm norm}_{D_{X_{i}}}(\Pi_{X^{\bp}_{i}}) \defeq \{H_{i} \subseteq \Pi_{X_{i}^{\bp}} \ \text{is an open normal subgroup}$$ $$\ | \ g_{X_{H_{i}}}=[\Pi_{X_{i}^{\bp}}: H_{i}](g_{X}-1)+1\}$$
$$ \subseteq \text{Et}_{D_{X_{i}}}(\Pi_{X^{\bp}_{i}}) \defeq \{H_{i} \subseteq \Pi_{X_{i}^{\bp}} \ \text{is an open subgroup} $$ $$\ | \ g_{X_{H_{i}}}=[\Pi_{X_{i}^{\bp}}: H_{i}](g_{X}-1)+1\}.$$ By Theorem \ref{types}, we have that $\text{Et}_{D_{X_{i}}}^{\rm norm}(\Pi_{X^{\bp}_{i}})$ and $\text{Et}_{D_{X_{i}}}(\Pi_{X^{\bp}_{i}})$ can be reconstructed group-theoretically from $\Pi_{X^{\bp}_{i}}$. Since $$\Pi_{X^{\bp}_{i}}^{\rm cpt} \defeq \Pi_{X^{\bp}_{i}}/\bigcap_{H_{i} \in \text{Et}^{\rm norm}_{D_{X_{i}}}(\Pi_{X_{i}^{\bp}})}H_{i}=\Pi_{X^{\bp}_{i}}/\bigcap_{H_{i} \in \text{Et}_{D_{X_{i}}}(\Pi_{X_{i}^{\bp}})}H_{i},$$ we obtain that $\Pi_{X_{i}^{\bp}}^{\rm cpt}$ can be reconstructed group-theoretically from $\Pi_{X_{i}^{\bp}}$.

Let $H_{2} \in \text{Et}^{\rm norm}_{D_{X_{2}}}(\Pi_{X^{\bp}_{2}})$, $H_{1} \defeq \phi^{-1}(H_{2})$, and $G \defeq \Pi_{X_{2}^{\bp}}/H_{2}=\Pi_{X_{1}^{\bp}}/H_{1}$. We will prove that $H_{1}  \in \text{Et}^{\rm norm}_{D_{X_{1}}}(\Pi_{X^{\bp}_{1}})$. Let $$f_{H_1}^{\bp}: X_{H_{1}}^{\bp} \migi X_{1}^{\bp}$$ be the Galois admissible covering over $k_{1}$ with Galois group $G$ corresponding to $H_{1}$, $x_{1} \in D_{X_{1}}$ a marked point of $X_{1}^{\bp}$, and $e_{f_{H_{1}}}(x_{1})$ the ramification index of a point of $f_{H_{1}}^{-1}(x_{1})$. Since $H_{2} \in \text{Et}^{\rm norm}_{D_{X_{2}}}(\Pi_{X_{2}^{\bp}})$, we have $g_{X_{H_{2}}}=\#G(g_{X}-1)+1$ and $n_{X_{H_{2}}}=\#Gn_{X}$. Then by applying the Riemann-Hurwitz formula, we obtain that $$g_{X_{H_{1}}}=\#G(g_{X}-1)+1+\frac{1}{2}\cdot\sum_{x_{1} \in D_{X_{1}}}\frac{\#G}{e_{f_{H_{1}}}(x_{1})}(e_{f_{H_{1}}}(x_{1})-1)$$$$=\#G(g_{X}-1)+1+\frac{1}{2}\cdot\sum_{x_{1} \in D_{X_{1}}}(\#G-\frac{\#G}{e_{f_{H_{1}}}(x_{1})})$$ and $$n_{X_{H_{1}}}=\sum_{x_{1}\in D_{X_{1}}}\frac{\#G}{e_{f_{H_{1}}}(x_{1})}.$$

By applying Theorem \ref{max and average} (a) and Lemma \ref{lem-0} (a), the surjective homomorphism $\phi|_{H_{1}}: H_{1} \migisurj H_{2}$ induces that $$\gamma^{\rm max}(H_{1})+2 =g_{X_{H_{1}}}+n_{X_{H_{1}}} \geq \gamma^{\rm max}(H_{2})+2=g_{X_{H_{2}}}+n_{X_{H_{2}}}.$$ Then we obtain that $$g_{X_{H_{1}}}+n_{X_{H_{1}}}=\#G(g_{X}-1)+1+\frac{1}{2}\cdot\sum_{x_{1} \in D_{X_{1}}}(\#G-\frac{\#G}{e_{f_{H_{1}}}(x_{1})})+\sum_{x_{1}\in D_{X_{1}}}\frac{\#G}{e_{f_{H_{1}}}(x_{1})}$$$$=\#G(g_{X}-1)+1+\frac{1}{2}\#Gn_{X}+\frac{1}{2}\cdot \sum_{x_{1} \in D_{X_{1}}}\frac{\#G}{e_{f_{H_{1}}}(x_{1})}$$$$\geq \#G(g_{X}-1)+1+\#Gn_{X}.$$ Thus, we have $$\sum_{x_{1} \in D_{X_{1}}}\frac{\#G}{e_{f_{H_{1}}}(x_{1})} \geq \#Gn_{X}.$$ Since $\#D_{X_{1}}=n_{X}$, we see immediately that $e_{f_{H_{1}}}(x_{1})=1$. This means that $f^{\bp}_{H_1}$ is \'etale, and that $$H_{1} \in \text{Et}_{D_{X_{1}}}^{\rm norm}(\Pi_{X_{1}^{\bp}}).$$ Thus we may define the following surjective homomorphism $$\phi^{\rm cpt}: \Pi_{X^{\bp}_{1}}^{\rm cpt} \defeq \Pi_{X^{\bp}_{1}}/\bigcap_{H_{1} \in \text{Et}^{\rm norm}_{D_{X_{1}}}(\Pi_{X_{1}^{\bp}})}H_{1}\migisurj \Pi_{X^{\bp}_{2}}^{\rm cpt} \defeq \Pi_{X^{\bp}_{2}}/\bigcap_{H_{2} \in \text{Et}^{\rm norm}_{D_{X_{2}}}(\Pi_{X_{2}^{\bp}})}H_{2}$$ which is induced by $\phi$ group-theoretically. Moreover, the commutative diagram 
\[
\begin{CD}
\Pi_{X^{\bp}_{1}} @>\phi>> \Pi_{X^{\bp}_{2}}
\\
@VVV@VVV
\\
\Pi^{\rm cpt}_{X^{\bp}_{1}} @>\phi^{\rm cpt}>> \Pi^{\rm cpt}_{X^{\bp}_{2}}
\end{CD}
\] follows immediately from the definition of $\phi^{\rm cpt}$. This completes the proof of the lemma.
\end{proof}

\begin{lemma}\label{lem-2}
Let $\ell$ be a prime number, and let $H_{2} \subseteq \Pi_{X_{2}^{\bp}}$ be an open normal subgroup and $H_{1} \defeq \phi^{-1}(H_{2}) \subseteq \Pi_{X^{\bp}_{1}}$. Suppose that $G \defeq \Pi_{X^{\bp}_{1}}/H_{1}=\Pi_{X^{\bp}_{2}}/H_{2}$ is a cyclic group which is isomorphic to $\mbZ/\ell\mbZ$. Then we have that $$(g_{X_{H_{1}}}, n_{X_{H_{1}}})=(g_{X_{H_{2}}}, n_{X_{H_{2}}}).$$
\end{lemma}

\begin{proof}
Let $$f^{\bp}_{H_{i}}: X^{\bp}_{H_{i}} \migi X^{\bp}_{i}$$ be the Galois admissible covering over $k_{i}$ with Galois group $G$ corresponding to $H_{i}$. Suppose that $\ell=p$. Then the definition of admissible coverings implies that $f^{\bp}_{H_{i}}$ is \'etale. Thus, we have that $(g_{X_{H_{1}}}, n_{X_{H_{1}}})=(g_{X_{H_{2}}}, n_{X_{H_{2}}}).$ Then we may suppose that $\ell \neq p$. 

By the Riemann-Hurwitz formula, we have $$g_{X_{H_i}}=\ell(g_{X}-1)+1+\frac{1}{2}\#e^{\rm op, ra}_{f_{H_{i}}}(\ell-1)$$ and $$n_{X_{H_{i}}}=\#e^{\rm op, ra}_{f_{H_{i}}}+\ell(n_{X}-\#e^{\rm op, ra}_{f_{H_{i}}}).$$ By applying Theorem \ref{max and average} (a) and Lemma \ref{lem-0} (a), the surjective homomorphism $\phi|_{H_{1}}: H_{1} \migisurj H_{2}$ implies that $$\gamma^{\rm max}(H_{1})+2=g_{X_{H_{1}}}+n_{X_{H_{1}}} \geq \gamma^{\rm max}(H_{2})+2=g_{X_{H_{2}}}+n_{X_{H_{2}}}.$$ Then we have $$\ell(g_{X}-1)+1+\frac{1}{2}\#e^{\rm op, ra}_{f_{H_{1}}}(\ell -1) +\#e^{\rm op, ra}_{f_{H_{1}}}+\ell(n_{X}-\#e^{\rm op, ra}_{f_{H_{1}}})$$$$=\ell(g_{X}-1)+1+\ell n_{X}+\frac{1}{2}(1-\ell)\#e^{\rm op, ra}_{f_{H_{1}}}$$ $$\geq \ell(g_{X}-1)+1+\frac{1}{2}\#e^{\rm op, ra}_{f_{H_{2}}}(\ell -1) +\#e^{\rm op, ra}_{f_{H_{2}}}+\ell(n_{X}-\#e^{\rm op, ra}_{f_{H_{2}}})$$$$=\ell(g_{X}-1)+1+\ell n_{X}+\frac{1}{2}(1-\ell)\#e^{\rm op, ra}_{f_{H_{2}}}.$$ Then we obtain that $$\#e^{\rm op, ra}_{f_{H_{1}}}\leq \#e^{\rm op, ra}_{f_{H_{2}}}.$$

Let $0 \leq m \leq n_{X}$ be a positive natural number. We put $$\mcA_{i, m} \defeq \{N_{i} \subseteq \Pi_{X^{\bp}_{i}}  \ \text{is an open normal subgroup} $$ $$\ | \ \Pi_{X^{\bp}_{i}}/N_{i} \cong \mbZ/\ell\mbZ \ \text{and} \ \#e^{\rm op, ra}_{f_{N_{i}}}=m\}$$ and $$\mcA_{i, \leq m} \defeq \bigcup_{0\leq j \leq m} \mcA_{i, j},$$ where  $f^{\bp}_{N_{i}}$ denotes the Galois admissible covering over $k_{i}$ corresponding to $N_{i}$. The isomorphism $\phi^{p'}$ induces a bijective map $$\phi^{*}_{\ell}: \mcA_{2, \leq n_{X}}\isom \mcA_{1, \leq n_{X}}.$$ To verify the lemma, it sufficient to prove that $\phi^{*}_{\ell}$ induces a bijection $$\phi^{*}_{\ell}|_{ \mcA_{2, m}}: \mcA_{2, m} \isom \mcA_{1, m}.$$ We note that since $(g_{X}, n_{X})=(g_{X_{1}}, n_{X_{1}})=(g_{X_{2}}, n_{X_{2}})$, the isomorphism $\phi^{p'}$ implies that $\#\mcA_{1, j}=\#\mcA_{2, j}$ for each $0\leq j \leq n_{X}$. Then by Lemma \ref{lem-1}, we have a bijection $$\phi^{*}_{\ell}|_{ \mcA_{2, 0}}: \mcA_{2, 0} \isom \mcA_{1, 0}.$$ We prove $\phi^{*}_{\ell}|_{ \mcA_{2, m}}: \mcA_{2, m} \isom \mcA_{1, m}$ by induction on $m$. Suppose that $m\geq 1$. The inequality $\#e^{\rm op, ra}_{f_{H_{1}}}\leq \#e^{\rm op, ra}_{f_{H_{2}}}$ concerning the cardinality of branch locus implies that we have a bijection $\phi_{\ell}^{*}|_{\mcA_{2, \leq m}}: \mcA_{2, \leq m} \isom \mcA_{1, \leq m}$. By induction, $\phi_{\ell}^{*}|_{\mcA_{2, \leq m-1}}: \mcA_{2, \leq m-1} \isom \mcA_{1, \leq m-1}$ is a bijection. Then we obtain that $$\phi_{\ell}^{*}|_{\mcA_{2,  m}}: \mcA_{2, m} \isom \mcA_{1, m}.$$ This completes the proof of the lemma.
\end{proof}

\begin{corollary}\label{coro-1}
Let $H_{2} \subseteq \Pi_{X_{2}^{\bp}}$ be an open normal subgroup and $H_{1} \defeq \phi^{-1}(H_{2}) \subseteq \Pi_{X^{\bp}_{1}}$. Suppose that $G \defeq \Pi_{X^{\bp}_{1}}/H_{1}=\Pi_{X^{\bp}_{2}}/H_{2}$ is a finite solvable group. Then we have that $$(g_{X_{H_{1}}}, n_{X_{H_{1}}})=(g_{X_{H_{2}}}, n_{X_{H_{2}}}).$$
\end{corollary}

\begin{proof}
The corollary follows immediately from Lemma \ref{lem-2}.
\end{proof}

\begin{lemma}\label{lem-3}
Let $H_{2} \subseteq \Pi_{X_{2}^{\bp}}$ be an open normal subgroup and $H_{1} \defeq \phi^{-1}(H_{2}) \subseteq  \Pi_{X_{1}^{\bp}}$. Suppose that $H_{2}$ contains the kernel of the natural homomorphism $ \Pi_{X_{2}^{\bp}} \migisurj  \Pi_{X_{2}^{\bp}}^{\rm cpt}$ (i.e., the admissible covering corresponding to $H_{2}$ is \'etale over $D_{X_{2}}$). Then we have that $$(g_{X_{H_{1}}}, n_{X_{H_{1}}})=(g_{X_{H_{2}}}, n_{X_{H_{2}}}).$$
\end{lemma}

\begin{proof}
By Lemma \ref{lem-1}, we have that $H_{1}$ contains the kernel of the natural homomorphism $\Pi_{X^{\bp}_{1}} \migisurj \Pi_{X_{1}^{\bp}}^{\rm cpt}$ (i.e., the admissible covering corresponding to $H_{1}$ is \'etale over $D_{X_{1}}$). Then the lemma follows immediately from the Riemann-Hurwitz formula.
\end{proof}

\begin{definition}\label{def-4}
Let $\Pi$ be an arbitrary profinite group and $m, n \in \mbN$ positive natural numbers. We define the closed normal subgroup $$D_{n}(\Pi)$$ of $\Pi$ to be the topological closure of $[\Pi, \Pi]\Pi^{n}$, where $[\Pi, \Pi]$ denotes the commutator subgroup of $\Pi$. Moreover, we define the closed normal subgroup $$D^{(m)}_{n}(\Pi)$$ of $\Pi$ inductively by $D^{(1)}_{n}(\Pi)\defeq D_{n}(\Pi)$ and $D_{n}^{(j+1)}(\Pi)\defeq D_{n}(D^{(j)}_{n}(\Pi))$, $j \in \{1, \dots, m-1\}$. Note that $\#(\Pi/D_{n}^{(m)}(\Pi))\leq \infty$ when $\Pi$ is topologically finitely generated.
\end{definition}

\begin{proposition}\label{prop-1}
Let $N_{2} \subseteq \Pi_{X_{2}^{\bp}}$ be an arbitrary open subgroup, $N_{1} \defeq \phi^{-1}(N_{2}) \subseteq \Pi_{X^{\bp}_{1}}$. Then there exist open normal subgroups $H_{2} \subseteq N_{2} \subseteq \Pi_{X_{2}^{\bp}}$ of $\Pi_{X_{2}^{\bp}}$ and $H_{1} \defeq \phi^{-1}(H_{2})\subseteq N_{1} \subseteq \Pi_{X_{1}^{\bp}}$ of $\Pi_{X_{1}^{\bp}}$ such that $$(g_{X_{H_{1}}}, n_{X_{H_{1}}})=(g_{X_{H_{2}}}, n_{X_{H_{2}}}).$$
\end{proposition}

\begin{proof}
Let $M_{i}$ be an open normal subgroup of $\Pi_{X^{\bp}_{i}}$ which is contained in $N_{i}$. We put $G \defeq \Pi_{X^{\bp}_{1}}/M_{1}=\Pi_{X^{\bp}_{2}}/M_{2}$, and write $m$ for $[G: G_{p}]$, where $G_{p}$ denotes a Sylow-$p$ subgroup of $G$. Then we have $(m, p)=1$. Let $$f_{M_i}^{\bp}: X_{M_{i}}^{\bp} \migi X_{i}^{\bp}$$ be the Galois admissible covering over $k_{i}$ with Galois group $G$ corresponding to $M_{i}$.

We put $Q_{2} \defeq D_{m}^{(4)}(\Pi_{X_{2}^{\bp}})$ and $Q_{1}\defeq \phi^{-1}(Q_{2})$. Note that since $m$ is prime to $p$ and $\phi^{p'}$ is an isomorphism, we have $Q_{1}=D_{m}^{(4)}(\Pi_{X^{\bp}_{1}})$. Let $H_{i} \defeq M_{i} \cap Q_{i} \subseteq \Pi_{X^{\bp}_{i}}$. We denote by $$f^{\bp}_{Q_{i}}: X^{\bp}_{Q_{i}} \migi X_{i}^{\bp}$$ and $$f^{\bp}_{H_{i}}: X^{\bp}_{H_{i}} \migi X_{i}^{\bp}$$ the Galois admissible covering over $k_{i}$ with Galois group $\Pi_{X_{i}^{\bp}}/Q_{i}$ corresponding to $Q_{i}$, and the Galois admissible covering over $k_{i}$ with Galois group $\Pi_{X_{i}^{\bp}}/H_{i}$ corresponding to $H_{i}$, respectively. Note that $f^{\bp}_{H_{i}}$ factors through $f_{M_{i}}^{\bp}$ and $f_{Q_{i}}^{\bp}$.

By applying \cite[Lemma 3.2 and Lemma 3.3]{Y7}, we have that the ramification index of every point of $D_{X_{Q_{i}}}$ is divided by $m$. Then Abhyankar's lemma implies that the Galois admissible covering $$g^{\bp}_{i}: X^{\bp}_{H_{i}} \migi X^{\bp}_{Q_{i}}$$ over $k_{i}$ induced by $H_{i} \subseteq Q_{i}$ is \'etale. Since $\Pi_{X_{i}^{\bp}}/Q_{i}$ is a finite solvable group, the proposition follows immediately from Corollary \ref{coro-1} and Lemma \ref{lem-3}. 
\end{proof}

\begin{lemma}\label{lem-4}
(a) Let $J_{i} \subseteq \Pi_{X_{i}}$ be a closed subgroup which is isomorphic to $\widehat\mbZ(1)^{p'}$. Then the following conditions are equivalent:

(i) There exists a unique closed subgroup $I_{i} \in {\rm Edg}^{\rm op}(\Pi_{X^{\bp}_{i}})$ such that $J_{i} \subseteq I_{i}$.

(ii) There exists an open subgroup $N_{i} \subseteq \Pi_{X^{\bp}_{i}}$ such that there exists a cofinal system $\mcC_{N_{i}}$ of $N_{i}$ which consists of open normal subgroups $H_{i} \subseteq N_{i} \subseteq \Pi_{X^{\bp}_{i}}$ of $\Pi_{X^{\bp}_{i}}$ (i.e., $N_{i}\cong \invlim_{H_{i}\in \mcC_{N_{i}}}N_{i}/H_{i}$), the image of the composition of the natural homomorphisms $$J_{i} \cap H_{i} \migiinje H_{i} \migi H_{i}^{\rm cpt, ab}$$ is trivial for every $H_{i}\in \mcC_{N_{i}}$.

(b) Let $\ell$ be a prime number distinct from $p$, $I_{i}, J_{i} \in {\rm Edg}^{\rm op}(\Pi_{X^{\bp}_{i}})$ arbitrary closed subgroups, and $\Pi_{X^{\bp}_{i}}^{\ell}$ the maximal pro-$\ell$ quotient of $\Pi_{X^{\bp}_{i}}$. Write $\overline I_{i}^{\ell}$ and $\overline J_{i}^{\ell}$ for $pr_{i}^{\ell}(I_{i})$ and $pr_{i}^{\ell}(J_{i})$, respectively. Suppose that $\overline I_{i}^{\ell}=\overline J_{i}^{\ell}$. Then we have $$I_{i}=J_{i}.$$
\end{lemma}

\begin{proof}
By applying similar arguments to the arguments given in the proof of \cite[Lemma 1.6]{HM}, we obtain (a). 

On the other hand, \cite[Proposition 1.2 (i)]{M} implies that $I_{i}\cap J_{i}$ is trivial. Then we see that, by replacing $\Pi_{X^{\bp}_{i}}$ by a certain open subgroup of $\Pi_{X^{\bp}_{i}}$, there exists an open normal subgroup $N_{i} \subseteq \Pi_{X^{\bp}_{i}}$ such that $\#(\Pi_{X^{\bp}_{i}}/N_{i})=\ell$, that $I_{i} \subseteq N_{i}$, and that $J_{i} \not\subseteq N_{i}$. This contradicts $\overline I_{i}^{\ell}=\overline J_{i}^{\ell}$. We complete the proof of (b). 
\end{proof}

Next, we prove the main result of this section.

\begin{theorem}\label{mainstep-1}
We maintain the notation introduced above. Then the (surjective) open continuous homomorphism $\phi: \Pi_{X_{1}^{\bp}} \migisurj \Pi_{X_{2}^{\bp}}$ induces  a surjective map $$\phi^{\rm edg, op}: {\rm Edg}^{\rm op}(\Pi_{X_{1}^{\bp}}) \migisurj {\rm Edg}^{\rm op}(\Pi_{X_{2}^{\bp}}),$$ group-theoretically.
Moreover, $\phi$ induces a bijection $$\phi^{\rm sg, op}: e^{\rm op}(\Gamma_{X^{\bp}_{1}}) \isom e^{\rm op}(\Gamma_{X^{\bp}_{2}})$$ of the sets of open edges of dual semi-graphs of $X_{1}^{\bp}$ and $X_{2}^{\bp}$ group-theoretically.
\end{theorem}

\begin{proof}
If $n_{X}=0$, the theorem is trivial. Then we may assume that $n_{X}>0$. Let $\mcC_{\Pi_{X^{\bp}_{2}}}$ be a cofinal system of $\Pi_{X^{\bp}_{2}}$ which consists of open normal subgroups of $\Pi_{X_{2}^{\bp}}$. We put $$\mcC_{\Pi_{X_{1}^{\bp}}}\defeq \{H_{1}\defeq \phi^{-1}(H_{2})\ | \ H_{2}\in \mcC_{\Pi_{X_{2}^{\bp}}}\}.$$ Note that $\mcC_{\Pi_{X_{1}^{\bp}}}$ is not a cofinal system of  $\Pi_{X_{1}^{\bp}}$ in general. Moreover, by applying Proposition \ref{prop-1}, we may assume that $$(g_{X_{H_{1}}}, n_{X_{H_{1}}})=(g_{X_{H_{2}}}, n_{X_{H_{2}}})$$ holds for every $H_{2} \in \mcC_{\Pi_{X_{2}^{\bp}}}$ and $H_{1} \defeq \phi^{-1}(H_{2})\in \mcC_{\Pi_{X^{\bp}_{1}}}$.

Let $I_{1} \in {\rm Edg}^{\rm op}(\Pi_{X_{1}^{\bp}})$ and $\phi(I_{1}) \subseteq \Pi_{X_{2}^{\bp}}$. We will prove that $\phi(I_{1}) \in {\rm Edg}^{\rm op}(\Pi_{X_{2}^{\bp}})$. Let $H_{2} \in \mcC_{\Pi_{X_{2}^{\bp}}}$. By replacing $\Pi_{X_{i}^{\bp}}$ and $\phi$ by $H_{i}$ and $\phi|_{H_{1}}$, respectively, Lemma \ref{lem-1} implies that we have the following commutative diagram: 
\[
\begin{CD}
I_{1} \cap H_{1}@>\phi|_{I_{1} \cap H_{1}}>>\phi(I_{1})  \cap H_{2}
\\
@VVV@VVV
\\
H_{1} @>\phi|_{H_{1}}>> H_{2}
\\
@VVV@VVV
\\
H_{1}^{\rm cpt, ab}@>\phi|_{H_{1}}^{\rm cpt, ab}>>H_{2}^{\rm cpt, ab}.
\end{CD}
\]
Since $I_{1} \in {\rm Edg}^{\rm op}(\Pi_{X_{1}^{\bp}})$, we have that $I_{1}\cap H_{1} \migiinje H_{1} \migi H_{1}^{\rm cpt, ab}$ is trivial. Then the commutative diagram above implies that the natural morphism $$\phi(I_{1}) \cap H_{2} \migiinje H_{2} \migi H_{2}^{\rm cpt, ab}$$ is trivial. Thus, by Lemma \ref{lem-4} (a), there exists $I_{2}  \in {\rm Edg}^{\rm op}(\Pi_{X_{2}^{\bp}})$ such that $\phi(I_{1})  \subseteq I_{2}$. 

Let us prove $\phi(I_{1}) =I_{2}$. Suppose that $\phi(I_{1})  \neq I_{2}$. We put $G \defeq I_{2}/\phi(I_{1}) $. Note that $G$ is a cyclic group, and that $(m, p)=1$, where $m\defeq \#G\geq 2$. 

Suppose that $g_{X}=0$. Then we have $n_{X}\geq 3$. Let $N_{2} \defeq D_{m}(\Pi_{X_{2}})$, $N_{1} \defeq \phi^{-1}(N_{2})=D_{m}(\Pi_{X_{1}})$, and $$f_{N_i}^{\bp}: X_{N_{i}}^{\bp} \migi X_{i}^{\bp}$$ the Galois admissible covering over $k_{i}$ corresponding to $N_{i}$. Since the ramification index of each point of $f_{N_i}^{-1}(D_{X_{i}})$ is equal to $m$, we have that $$I_{1} \not\subseteq N_{1}, \ I_{2} \not\subseteq N_{2}, \ \phi(I_{1})  \subseteq N_{2}.$$ On the other hand,  the isomorphism of maximal pro-prime-to-$p$ quotients $\phi^{p'}: \Pi_{X_{1}^{\bp}}^{p'} \isom \Pi_{X_{2}^{\bp}}^{p'}$ and $I_{1} \not\subseteq N_{1}$ imply that $\phi(I_{1}) \not\subseteq N_{2}$. This contradicts $\phi(I_{1})  \subseteq N_{2}$. Then we obtain $\phi(I_{1}) =I_{2}$.

Suppose that $g_{X}>0$. We put $$Q_{2} \defeq \text{ker}(\Pi_{X_{2}^{\bp}} \migisurj \Pi_{X_{2}^{\bp}}^{\rm cpt} \migisurj \Pi_{X_{2}^{\bp}}^{\rm cpt,ab}\otimes \mbZ/m\mbZ)$$ and $Q_{1} \defeq \phi^{-1}(Q_{2})$. Then Lemma \ref{lem-1} implies that $Q_{1}=\text{ker}(\Pi_{X_{1}^{\bp}} \migisurj \Pi_{X_{1}^{\bp}}^{\rm cpt} \migisurj \Pi_{X_{1}^{\bp}}^{\rm cpt, ab}\otimes \mbZ/m\mbZ)$. Note that the assumption $g_{X}>0$ implies that $\Pi_{X_{i}^{\bp}}^{\rm cpt} \migisurj \Pi_{X_{i}^{\bp}}^{\rm cpt,ab}\otimes \mbZ/m\mbZ$ is not trivial. Then  the nontrivial Galois admissible covering over $k_{i}$ corresponding to $Q_{i}$ is \'etale over $D_{X_{i}}$. Moreover, we have $I_{i} \subseteq Q_{i}$ and $$n_{X_{Q_{i}}}\geq 2.$$ Let $P_{2} \defeq D_{m}(Q_{2})$, $P_{1} \defeq \phi^{-1}(P_{2})=D_{m}(Q_{1})$, and $$g_{i}^{\bp}: X_{P_{i}}^{\bp} \migi X_{Q_{i}}^{\bp}$$ the Galois admissible covering over $k_{i}$ corresponding to $P_{i} \subseteq Q_{i}$. Since the ramification index of each point of $g_{i}^{-1}(D_{X_{Q_{i}}})$ is equal to $m$, we have that $$I_{1} \not\subseteq P_{1}, \ I_{2} \not\subseteq P_{2}, \ \phi(I_{1})  \subseteq P_{2}.$$ On the other hand,  the isomorphism of maximal pro-prime-to-$p$ quotients $\phi|_{P_{1}}^{p'}: P_{1}^{p'} \isom P_{2}^{p'}$ and $I_{1} \not\subseteq P_{1}$ imply that $\phi(I_{1}) \not\subseteq P_{2}$. This contradicts $\phi(I_{1})  \subseteq P_{2}$. Then we obtain $\phi(I_{1}) =I_{2}$. Thus, we may define the following map $$\phi^{\rm edg, op}: {\rm Edg}^{\rm op}(\Pi_{X_{1}^{\bp}}) \migi {\rm Edg}^{\rm op}(\Pi_{X_{2}^{\bp}}), \ I_{1} \mapsto I_{2} \defeq \phi(I_{1}).$$

Next, let us prove that $\phi^{\rm edg, op}$ is a surjection. Let $\ell$ be a prime number distinct from $p$ and $pr^{\ell}_{i}: \Pi_{X_{i}^{\bp}} \migisurj \Pi_{X_{i}^{\bp}}^{\ell}$. Let $J_{2} \in {\rm Edg}^{\rm op}(\Pi_{X_{2}^{\bp}})$ be an arbitrary subgroup, $\overline J^{\ell}_{2}\defeq pr^{\ell}_{2}(J_{2})$ the image of $J_{2}$, and $\mcC_{\Pi_{X_{i}^{\bp}}}^{\ell} \defeq \{\overline H_{i} \defeq pr_{2}^{\ell}(H_{i})\}_{H_{i} \in \mcC_{\Pi_{X_{i}^{\bp}}}}$, where $\mcC_{\Pi_{X_{i}^{\bp}}}$ is the set of normal subgroups of $\Pi_{X_{i}^{\bp}}$ defined above. Note that $\mcC^{\ell}_{\Pi_{X_{i}^{\bp}}}$ is a cofinal system of $\Pi_{X^{\bp}_{i}}^{\ell}$, and that $\overline H_{1}=(\phi^{\ell})^{-1}(\overline H_{2})$.

Let $\overline H_{2} \in \mcC^{\ell}_{\Pi_{X_{2}^{\bp}}}$, $\overline N_{2}\defeq \overline J^{\ell}_{2}\overline H_{2} \supseteq \overline H_{2}$, $\overline N_{1}\defeq (\phi^{\ell})^{-1}(\overline N_{2}) \supseteq \overline  H_{1}$, and $N_{i} \defeq (pr^{\ell}_{i})^{-1}(\overline N_{i})$. Note that $G\defeq \overline N_{1}/ \overline H_{1}=N_{1}/H_{1}=\overline N_{2}/\overline H_{2}=N_{2}/H_{2}$ is a cyclic $\ell$-group. Write $$g_{H_{i}, N_{i}}^{\bp}: X_{H_{i}}^{\bp} \migi X_{N_{i}}^{\bp}$$ for the Galois admissible covering over $k_{i}$ with Galois group $G$. Since $J_{2} \in  {\rm Edg}^{\rm op}(\Pi_{X_{2}^{\bp}})$, we obtain that $g^{\bp}_{H_{2}, N_{2}}$ is totally ramified at a marked point of $X_{H_{2}}^{\bp}$. We put $$\text{Edg}^{\rm op, \ell, ab}(N_{i}) \defeq \{ \text{the image of} \ I \ \text{of} $$$$\text{the natural homomorphism} \ N_{i} \migisurj  N_{i}^{\rm \ell, ab} \ | \ I \in {\rm Edg}^{\rm op}(N_{i})\}.$$ Note that $\#\text{Edg}^{\rm op, \ell, ab}(N_{i})=n_{X_{N_{i}}}$. Then the composition of the following natural homomorphisms $$\bigoplus_{I_{N_{2}} \in \text{Edg}^{\rm op, \ell, ab}(N_{2})} I_{N_{2}} \migi N_{2}^{\rm \ell, ab} \migisurj G$$ is a surjection. By applying Lemma \ref{lem-1}, we obtain that the isomorphism $\phi^{\ell}$ induces an isomorphism $$\text{Im}(\bigoplus_{I_{N_{1}} \in \text{Edg}^{\rm op, \ell, ab}(N_{1})} I_{N_{1}} \migi N_{1}^{\ell, \rm ab}) \isom \text{Im}(\bigoplus_{I_{N_{2}} \in \text{Edg}^{\rm op, \ell, ab}(N_{2})} I_{N_{2}}\migi N_{2}^{\ell, \rm ab}).$$ Then the composition of the following natural homomorphisms $$\bigoplus_{I_{N_{1}} \in \text{Edg}^{\rm op, \ell, ab}(N_{1})} I_{N_{1}} \migi N_{1}^{\rm \ell, ab} \migisurj G$$ is also a surjection. Since $G$ is a cyclic $\ell$-group, there exists $I'_{N_{1}} \in \text{Edg}^{\rm op, \ell, ab}(N_{1})$ such that the composition $I'_{N_{1}} \migiinje N_{1}^{\rm \ell, ab} \migisurj G$ is a surjection. This means that $g^{\bp}_{H_{1}, N_{1}}$ is also totally ramified at a marked point of $X_{H_{1}}^{\bp}$.

We put $$E_{\overline H_{1}}\defeq \{x_{1} \in D_{X_{H_{1}}} \ | \ g_{H_{1}, N_{1}}^{\bp} \ \text{is totally ramified at}\ x_{1}\}.$$ Then we have that $E_{\overline H_{1}}$ is a non-empty finite set. Thus, we obtain $$\invlim_{\overline H_{1} \in \mcC^{\ell}_{\Pi_{X_{1}^{\bp}}}}E_{\overline H_{1}}\neq \emptyset.$$ This means that there exists $J_{1} \in {\rm Edg}^{\rm op}(\Pi_{X_{1}^{\bp}})$ such that the image $pr^{\ell}_{2}(\phi(J_{1}))=\phi^{\ell}(pr^{\ell}_{1}(J_{1}))$ of $J_{1}$ of the composition of the natural homomorphisms 
\[
\begin{CD}
\Pi_{X_{1}^{\bp}}@>\phi>> \Pi_{X_{2}^{\bp}}
\\
@Vpr_{1}^{\ell}VV@Vpr_{2}^{\ell}VV
\\
\Pi_{X_{1}^{\bp}}^{\ell} @>\phi^{\ell}>> \Pi_{X_{2}^{\bp}}^{\ell}
\end{CD}
\]
 is equal to $\overline J_{2}^{\ell}$. Since $\phi(J_{1}) \in {\rm Edg}^{\rm op}(\Pi_{X_{2}^{\bp}})$, by applying Lemma \ref{lem-4} (b), we have $\phi(J_{1})=J_{2}$. Then $\phi^{\rm edg, op}$ is a surjection. Moreover, Theorem \ref{types} implies that ${\rm Edg}^{\rm op}(\Pi_{X_{i}^{\bp}})$ can be reconstructed group-theoretically from $\Pi_{X^{\bp}_{i}}$. This completes the proof of the first part of the theorem.

Let us prove the ``moreover" part of the theorem. We see immediately that $\phi^{\rm edg, op}: {\rm Edg}^{\rm op}(\Pi_{X_{1}^{\bp}}) \migisurj {\rm Edg}^{\rm op}(\Pi_{X_{2}^{\bp}})$ is compatible with the natural actions of $\Pi_{X_{1}^{\bp}}$ and $\Pi_{X_{2}^{\bp}}$, respectively. By using the surjectivity of $\phi^{\rm edg, op}$, we obtain immediately a surjection $$\phi^{\rm sg, op}: e^{\rm op}(\Gamma_{X^{\bp}_{1}}) \isom {\rm Edg}^{\rm op}(\Pi_{X_{1}^{\bp}})/\Pi_{X_{1}^{\bp}} \migisurj  {\rm Edg}^{\rm op}(\Pi_{X_{2}^{\bp}})/\Pi_{X_{2}^{\bp}} \isom e^{\rm op}(\Gamma_{X^{\bp}_{2}})$$ of the sets of open edges of dual semi-graphs of $X_{1}^{\bp}$ and $X_{2}^{\bp}$, where $(-)^{\rm sg}$ means ``semi-graph". Moreover, since $n_{X}=\#e^{\rm op}(\Gamma_{X^{\bp}_{1}}) =\#e^{\rm op}(\Gamma_{X^{\bp}_{2}})$, we have that $\phi^{\rm sg, op}$ is a bijection. On the other hand, Theorem \ref{types} implies that  $e^{\rm op}(\Gamma_{X_{i}^{\bp}})$ can be reconstructed group-theoretically from $\Pi_{X^{\bp}_{i}}$. This completes the proof of the theorem.
\end{proof}

\begin{corollary}
We maintain the notation introduced above. Let $H_{2} \subseteq \Pi_{X_{1}^{\bp}}$ be an arbitrary open subgroup and $H_{1} \defeq \phi^{-1}(H_{2})\subseteq \Pi_{X_{2}^{\bp}}$. Then we have that $$\gamma^{\rm max}(H_{1})=\gamma^{\rm max}(H_{2}).$$
\end{corollary}

\begin{proof}
By Theorem \ref{mainstep-1}, we obtain that $(g_{X_{H_{1}}}, n_{X_{H_{1}}})=(g_{X_{H_{2}}}, n_{X_{H_{2}}})$. Then Theorem \ref{max and average} (a) implies that  $\gamma^{\rm max}(H_{1})=\gamma^{\rm max}(H_{2}).$
\end{proof}

Let $\widehat X^{\bp}_{i}=(\widehat X_{i}, D_{\widehat X_{i}})$, $i\in \{1, 2\}$, be the universal admissible (resp. solvable admissible) covering associated to $\Pi_{X^{\bp}_{i}}$ if $\Pi_{X^{\bp}_{i}}$ is the admissible (resp. solvable admissible) fundamental group of $X^{\bp}_{i}$. Let $e_{i} \in e^{\rm op}(\Gamma_{X_{i}^{\bp}})$, $\widehat e_{i} \in e^{\rm op}(\Gamma_{\widehat X^{\bp}_{i}})$ over $e_{i}$, and $I_{\widehat e_{i}} \in \text{Edg}^{\rm op}(\Pi_{X_{i}^{\bp}})$ such that $\phi(I_{\widehat e_{1}}) =I_{\widehat e_{2}}$. Write $\overline \mbF_{p, i}$ for the algebraic closure of $\mbF_{p}$ in $k_{i}$. We put $$\mbF_{\widehat e_{i}}\defeq (I_{\widehat e_{i}}\otimes_{\mbZ} (\mbQ/\mbZ)_{i}^{p'}) \sqcup \{*_{\widehat e_{i}}\},$$ where $\{*_{\widehat e_{i}}\}$ is an one-point set, and $(\mbQ/\mbZ)_{i}^{p'}$ denotes the prime-to-$p$ part of $\mbQ/\mbZ$ which can be canonically identified with $$\bigcup_{(p, m)=1}\mu_{m}(\overline \mbF_{p, i}).$$ Moreover, $\mbF_{\widehat e_{i}}$ can be identified with $\overline \mbF_{p, i}$ as sets, hence, admits a structure of field, whose multiplicative group is $I_{\widehat e_{i}}\otimes_{\mbZ} (\mbQ/\mbZ)^{p'}_{i}$, and whose zero element is $*_{\widehat e_{i}}$. An important consequence of Theorem \ref{mainstep-1} is as follows.

\begin{theorem}\label{prop-3-12}
We maintain the notation introduced above. Then the field structure of $\mbF_{\widehat e_{i}}$ can be  reconstructed group-theoretically from $\Pi_{X^{\bp}_{i}}$. Moreover, $\phi$ induces a field isomorphism $$\phi^{\rm fd}_{\widehat e_{1}, \widehat e_{2}}:\mbF_{\widehat e_{1}}\isom \mbF_{\widehat e_{2}}$$ group-theoretically, where ``fd" means ``field".
\end{theorem}

\begin{proof}
We claim that we may assume that $n_{X}=3.$ If $g_{X}=0$, then $n_{X}\geq 3$. Suppose that $g_{X}\geq 1$. Theorem \ref{mainstep-1} implies that $\phi: \Pi_{X^{\bp}_{1}} \migisurj \Pi_{X^{\bp}_{2}}$ induces an open continuous surjection $\phi^{\rm cpt}: \Pi_{X^{\bp}_{1}}^{\rm cpt} \migisurj \Pi_{X^{\bp}_{2}}^{\rm cpt}.$ Let $H'_{2} \subseteq \Pi_{X^{\bp}_{2}}^{\rm cpt}$ be an open normal subgroup such that $\#(\Pi_{X^{\bp}_{2}}^{\rm cpt}/H'_{2})\geq 3$ and $H'_{1} \defeq (\phi^{\rm cpt})^{-1}(H_{2}')$. Write $H_{i}\subseteq \Pi_{X^{\bp}_{i}}^{\rm cpt}$, $i \in \{1, 2\}$, for the inverse image of $H_{i}'$ of the natural surjection $\Pi_{X^{\bp}_{2}}\migisurj \Pi_{X^{\bp}_{2}}^{\rm cpt}$, and $X_{H_{i}}^{\bp}$ for the pointed stable curve of type $(g_{X_{H_i}}, n_{X_{H_{i}}})$ over $k_{i}$ corresponding to $H_{i}$. Note that  $g_{X_{H_1}}=g_{X_{H_2}}\geq 1$ and $n_{X_{H_{1}}}=n_{X_{H_{2}}}\geq 3$. By replacing $X^{\bp}_{i}$ by $X^{\bp}_{H_{i}}$, we may assume that $g_{X}\geq 1$ and $n_{X} \geq 3$. 

By applying Theorem \ref{mainstep-1}, $\phi$ induces a bijection $$\phi^{\rm sg, op}: e^{\rm op}(\Gamma_{X^{\bp}_{1}}) \isom e^{\rm op}(\Gamma_{X^{\bp}_{2}}).$$ Let $E_{X_1}\defeq \{e_{1, 1}, e_{1, 2}, e_{1, 3}\} \subseteq e^{\rm op}(\Gamma_{X^{\bp}_{1}})$ and $E_{X_2} \defeq \phi^{\rm sg, op}(E_{X_{1}}) \subseteq  e^{\rm op}(\Gamma_{X^{\bp}_{2}})$. Write $D'_{X_{i}} \subseteq D_{X_{i}}$ for the set of marked points of $X^{\bp}_{i}$ corresponding to $E_{X_{i}}$. Then $(X_{i}, D'_{X_i})$, $i\in \{1, 2\}$, is a pointed stable curve of type $(g_{X}, 3)$ over $k_{i}$. Write $I_{i}$, $i \in \{1, 2\}$, for the closed subgroup of $\Pi_{X^{\bp}_{i}}$ generated by the subgroups $I_{\widehat e} \in \text{Edg}^{\rm op}(\Pi_{X^{\bp}_{i}})$ such that the image of $\widehat e$ in $e^{\rm op}(\Gamma_{X^{\bp}_{i}})$ is contained in $E_{X_{i}}$. Then we have a natural isomorphism $$\Pi_{(X_{i}, D'_{X_i})} \cong \Pi_{X^{\bp}_{i}}/I_{i}, \ i \in \{1, 2\}.$$ Moreover, Theorem \ref{mainstep-1} implies that $\phi$ induces a surjective open continuous  homomorphism $$\phi': \Pi_{(X_{1}, D'_{X_1})} \migisurj \Pi_{(X_{2}, D'_{X_2})}.$$ Thus, by replacing $X^{\bp}_{i}$, $\Pi_{X_{i}^{\bp}}$, and $\phi$ by $(X_{i}, D'_{X_i})$, $\Pi_{(X_{i}, D'_{X_i})}$, and $\phi'$, respectively, we may assume that $n_{X}=3.$ 

Then the theorem follows from \cite[Theorem 5.5]{Y6}.
\end{proof}

\begin{remarkA}
Theorem \ref{prop-3-12} generalizes \cite[Proposition 5.3]{T2} and \cite[Proposition 6.1]{Y2} to the case of arbitrary pointed stable curves. \cite[Proposition 5.3]{T2} and \cite[Proposition 6.1]{Y2} play key roles in the proofs of weak Isom-version of the Grothendieck conjecture of curves over algebraically closed fields of characteristic $p>0$ (cf. \cite[Theorem 0.2]{T2}) and weak Hom-version of the Grothendieck conjecture of curves over algebraically closed fields of characteristic $p>0$ (\cite[Theorem 1.2]{Y2}), respectively.
\end{remarkA}

\section{Combinatorial Grothendieck conjecture for surjections}\label{sec-4}

In this section, we will prove a version of combinatorial Grothendicek conjecture for surjective open continuous homomorphisms under certain assumption, which is an analogue of Theorem \ref{mainstep-1} for topological data and combinatorial data associated to pointed stable curves. {\it In the present section, we shall assume that all the fundamental groups of pointed stable curves are solvable admissible fundamental groups unless indicated otherwise.}

\subsection{Cohomology classes and sets of vertices}\label{sec-4-1}

We maintain the notation introduced in Section \ref{sec-1}. Let $X^{\bp}$ be a pointed stable curve of type $(g_{X}, n_{X})$ over an algebraically closed field $k$ of characteristic $p>0$, $\Gamma_{X^{\bp}}$ the dual semi-graph of $X^{\bp}$, and $\Pi_{X^{\bp}}$ the solvable admissible fundamental group of $X^{\bp}$.

Let $\ell$ be a prime number. We put $$v(\Gamma_{X^{\bp}})^{>0, \ell}\defeq \{v \in v(\Gamma_{X^{\bp}}) \ | \ \text{dim}_{\mbF_{\ell}}(\text{Hom}(\Pi^{\et}_{\widetilde X_{v}^{\bp}}, \mbF_{\ell}))>0 \},$$ $$M^{\text{\'et}}_{X^{\bp}} \defeq \text{Hom}(\Pi_{X^{\bp}}^{\text{\'et}}, \mbF_{\ell}),$$ $$M_{X^{\bp}}^{\rm top}\defeq \text{Hom}(\Pi_{X^{\bp}}^{\rm top}, \mbF_{\ell}).$$ On the other hand, we have the natural isomorphisms $\text{Hom}(\Pi^{\text{\'et}}_{\widetilde X_{v}^{\bp}}, \mbF_{\ell})\cong H_{\text{\'et}}^{1}(\widetilde X_{v}, \mbF_{\ell})$, $M^{\text{\'et}}_{X^{\bp}} \cong H^{1}_{\et}(X, \mbF_{\ell})$, and $M_{X^{\bp}}^{\rm top} \cong H^{1}(\Gamma_{X^{\bp}}, \mbF_{\ell})$. In the theory of anabelian geometry, we want to emphasize the objects under consideration are  arose from various fundamental groups. Then we do not use the standard notation $H_{\text{\'et}}^{1}(\widetilde X_{v}, \mbF_{\ell})$, $H^{1}_{\et}(X, \mbF_{\ell})$, and $H^{1}(\Gamma_{X^{\bp}}, \mbF_{\ell})$. Moreover, there is an injection $M_{X^{\bp}}^{\rm top}\migiinje M^{\text{\'et}}_{X^{\bp}}$ induced by the natural surjection $\Pi_{X^{\bp}} \migisurj \Pi_{X^{\bp}}^{\rm top}$. We put $$M_{X^{\bp}}^{\rm nt}\defeq \text{coker}(M_{X^{\bp}}^{\rm top}\migiinje M^{\text{\'et}}_{X^{\bp}}),$$ where $(-)^{\rm nt}$ means ``non-top". A non-zero element of $M^{\text{\'et}}_{X^{\bp}}$ corresponds to a Galois \'etale covering of the underlying curve $X$ of $X^{\bp}$ with Galois group $\mbZ/\ell\mbZ$ and an non-zero element of $M^{\rm top}_{X^{\bp}}$ corresponds to a Galois \'etale covering of $X^{\bp}$ with Galois group $\mbZ/\ell\mbZ$ such that the map of dual semi-graphs is a topological covering.

Let $V_{X, \ell}^{*} \subseteq M^{\text{\'et}}_{X^{\bp}}$ be the subset of elements such that the image of $M^{\et}_{X^{\bp}} \migisurj M^{\rm nt}_{X^{\bp}}$ is not $0$. Then an element of $V_{X, \ell}^{*} $ corresponds to a Galois \'etale covering of the underlying curve $X$ of $X^{\bp}$ with Galois group $\mbZ/\ell\mbZ$ such that the map of dual semi-graphs is not a topological covering. Let $\alpha \in V_{X, \ell}^{*}$ and $$f_{\alpha}^{\bp}: X^{\bp}_{\alpha} \migi X^{\bp}$$ the Galois \'etale covering corresponding to $\alpha$. Denote by $\Gamma_{X_{\alpha}^{\bp}}$ the dual semi-graph of $X_{\alpha}^{\bp}$. Then we obtain a map $$\iota: V_{X, \ell}^{*} \migi \mbZ_{>0}, \ \alpha \mapsto \#v(\Gamma_{X^{\bp}_{\alpha}}).$$ Furthermore, we put $$V_{X, \ell}^{\star} \defeq \{ \alpha\in V_{X, \ell}^{*}\ | \ \iota \ \text{attains its
maximum}\}$$$$=\{ \alpha\in V_{X, \ell}^{*}\ | \ \iota(\alpha)=\ell\#v(\Gamma_{X^{\bp}})-\ell+1\}$$$$=\{ \alpha\in V_{X, \ell}^{*}\ | \ \#v_{f_{\alpha}}^{\rm ra}=1\}.$$ For each $\alpha \in V^{\star}_{X, \ell}$, $\iota(\alpha)=\ell\#v(\Gamma_{X^{\bp}})-\ell+1$ implies that there exists a unique irreducible component $Z \subseteq X_{\alpha}$ whose decomposition group under the action of $\mbZ/\ell \mbZ$ is not trivial. Let $v_{\alpha} \in v(\Gamma_{X^{\bp}})$ such that $X_{v_{\alpha}}=f_{\alpha}(Z)$. Then we have  $v_{\alpha} \in v(\Gamma_{X^{\bp}})^{>0, \ell}$. This means that $V^{\star}_{X, \ell}=\emptyset$ if and only if $v(\Gamma_{X^{\bp}})^{>0, \ell}=\emptyset$.

On the other hand, let $H \subseteq \Pi_{X^{\bp}}$ be an open subgroup. Write $f^{\rm sg}_{H}: \Gamma_{X_{H}^{\bp}} \migi \Gamma_{X^{\bp}}$ for the map of dual semi-graphs induced by the admissible covering $f^{\bp}_{H}: X_{H}^{\bp} \migi X^{\bp}$ over $k$ corresponding to $H$. We define a map $$f^{\rm ver, \ell}_{H}: v(\Gamma_{X_{H}^{\bp}})^{>0, \ell} \migi v(\Gamma_{X^{\bp}})^{>0, \ell}$$ as follows: Let $v_{H} \in v(\Gamma_{X_{H}^{\bp}})^{>0, \ell}$ and $v \defeq f^{\rm sg}_{H}(v_{H}) \in v(\Gamma_{X_{H}^{\bp}})$. Then we have that $f^{\rm ver, \ell}_{H}(v_{H})=v$ if $\text{dim}_{\mbF_{\ell}}(\text{Hom}(\Pi^{\et}_{\widetilde X_{v}^{\bp}}, \mbF_{\ell}))\neq 0$; otherwise, $f^{\rm ver, \ell}_{H}(v_{H})=\emptyset$. Moreover, if $H \subseteq \Pi_{X^{\bp}}$ is an open normal subgroup, then $v(\Gamma_{X_{H}^{\bp}})^{>0, \ell}$ admits a natural action of $\Pi_{X^{\bp}}/H$. Then we have the following proposition.

\begin{proposition}\label{pro-2-1}
(a) We define a pre-equivalence relation $\sim$ on $V^{\star}_{X, \ell}$ as follows: 
\begin{quote}
Let $\alpha, \beta \in V^{\star}_{X, \ell}$. We have that $\alpha \sim \beta$ if, for each $\lambda, \mu \in \mbF_{\ell}^{\times}$ for which $\lambda\alpha+\mu\beta \in V_{X, \ell}^{*} $, $\lambda\alpha+\mu\beta \in V^{\star}_{X, \ell}$. 
\end{quote}
Then the pre-equivalence relation $\sim$ on $V^{\star}_{X, \ell}$ is an equivalence relation. 

(b) We denote by $V_{X, \ell}$ the quotient set of $V^{\star}_{X, \ell}$ by $\sim$. Then we have a natural bijection $$\kappa_{X, \ell}: V_{X, \ell}  \isom v(\Gamma_{X^{\bp}})^{>0, \ell}, \ [\alpha] \mapsto v_{\alpha},$$ where $[\alpha]$ denotes the equivalence class of $\alpha$.

(c) Let $\ell, \ell'$ be prime numbers distinct from each other. Suppose that $\ell \neq p$. Then we have a natural injection $$V_{X, \ell'} \migiinje V_{X, \ell},$$ which is a bijection if $\ell' \neq p$, and which fits into the following commutative diagram:
\[
\begin{CD}
V_{X, \ell'} @>\kappa_{X, \ell'}>> v(\Gamma_{X^{\bp}})^{>0, \ell'}
\\
@VVV@VVV
\\
V_{X, \ell} @>\kappa_{X, \ell}>> v(\Gamma_{X^{\bp}})^{>0, \ell},
\end{CD}
\]
where the vertical map of the right-hand side is the natural injection induced by the definitions of $v(\Gamma_{X^{\bp}})^{>0, \ell'}$ and $v(\Gamma_{X^{\bp}})^{>0, \ell}$.

(d) Let $H \subseteq \Pi_{X^{\bp}}$ be an open subgroup. Suppose that $([\Pi_{X}^{\bp}: H], \ell)=1$. Then the natural injection $H \migiinje \Pi_{X^{\bp}}$ induces a map $$\gamma^{\rm ver, \ell}_{H}: V_{X_{H}, \ell} \migi V_{X, \ell}$$ which fits into the following commutative diagram:
\[
\begin{CD}
V_{X_{H}, \ell} @>\kappa_{X_{H}, \ell}>> v(\Gamma_{X_{H}^{\bp}})^{>0, \ell}
\\
@V\gamma_{H}^{\rm ver, \ell}VV@Vf_{H}^{\rm ver, \ell}VV
\\
V_{X, \ell} @>\kappa_{X, \ell}>> v(\Gamma_{X^{\bp}})^{>0, \ell}.
\end{CD}
\]
Moreover, suppose that $H \subseteq \Pi_{X^{\bp}}$ is an open normal subgroup. Then $V_{X_{H}, \ell}$ admits an action of $\Pi_{X^{\bp}}/H$ such that $\kappa_{X_{H}, \ell}$ is compatible with $\Pi_{X^{\bp}}/H$-actions (i.e., $\kappa_{X_{H}, \ell}$ is $\Pi_{X^{\bp}}/H$-equivariant).
\end{proposition}

\begin{proof}
See \cite[Proposition 2.1 and Remark 2.1.1]{Y3} for (a), (b), and (c). Let us explain (d).  Let $[\alpha_{X}] \in V_{X, \ell}$. Then $\alpha_{X}$ induces an element $$\beta_{X_{H}}=\sum_{\beta\in L_{\alpha_{X}}}c_{\beta}\beta \in \text{Hom}(H, \mbF_{\ell}), \ c_{\beta} \in \mbF_{\ell}^{\times},$$ via the natural homomorphism $\text{Hom}(\Pi_{X^{\bp}}, \mbF_{\ell}) \migi \text{Hom}(H, \mbF_{\ell}),$ where $L_{\alpha_{X}}$ is a subset of $V^{\star}_{X_{H}, \ell}$ such that, if $\beta_{1}, \beta_{2} \in L_{\alpha_{X}}$ distinct from each other, then $[\beta_{1}] \neq [\beta_{2}]$. 

Let $[\alpha_{X_{H}}] \in V_{X_{H}, \ell}$. Then we define $$\gamma^{\rm ver, \ell}_{H}([\alpha_{X_{H}}])=[\alpha_{X}]$$ if there exists $[\alpha_{X}] \in V_{X, \ell}$ such that there exists $\beta \in L_{\alpha_{X}}$, and that $[\beta]=[\alpha_{X_{H}}]$ (i.e., $\beta \sim \alpha_{X_{H}}$). Otherwise, we put $\gamma^{\rm ver, \ell}_{H}([\alpha_{X_{H}}])=\emptyset$. It is easy to check that $\gamma^{\rm ver, \ell}_{H}$ is well-defined, and that the following diagram
\[
\begin{CD}
V_{X_{H}, \ell} @>\kappa_{X_{H}, \ell}>> v(\Gamma_{X_{H}^{\bp}})^{>0, \ell}
\\
@V\gamma_{H}^{\rm ver, \ell}VV@Vf_{H}^{\rm ver, \ell}VV
\\
V_{X, \ell} @>\kappa_{X, \ell}>> v(\Gamma_{X^{\bp}})^{>0, \ell}
\end{CD}
\]
is commutative. 

Moreover, suppose that $H$ is an open normal subgroup of $\Pi_{X^{\bp}}$. The natural exact sequence $$1\migi H \migi \Pi_{X^{\bp}} \migi \Pi_{X^{\bp}} /H \migi 1$$ induces an outer representation $$\Pi_{X^{\bp}}/H \migi \text{Out}(H) \defeq \frac{\text{Aut}(H)}{\text{Inn}(H)}.$$ Then we obtain an action of $\Pi_{X^{\bp}}/H$ on $$V^{\star}_{X_{H},\ell}\subseteq \text{Hom}(H^{\et}, \mbF_{\ell})$$ induced by the outer representation. Let $\sigma \in \Pi_{X^{\bp}}/H$ and $\alpha_{X_{H}}$, $\alpha'_{X_{H}} \in V^{\star}_{X_{H},\ell}$. Then we have that $\alpha_{X_{H}}\sim\alpha'_{X_{H}}$ if and only if $\sigma(\alpha_{X_{H}}) \sim \sigma(\alpha_{X_{H}}')$. Thus, we obtain an action of $\Pi_{X^{\bp}}/H$ on $V_{X_{H},\ell}$ induced by the natural injection $H \migiinje \Pi_{X^{\bp}}$. On the other hand, it is easy to check that  the commutative diagram above is compatible with the $\Pi_{X^{\bp}}/H$-actions. This completes the proof of the proposition.
\end{proof}

\begin{remarkA}\label{rem-pro-2-1}
By applying Theorem \ref{types},  we have that $\Pi_{X^{\bp}}^{\et}$ and $\Pi_{X^{\bp}}^{\rm top}$ can be reconstructed group-theoretically from $\Pi_{X^{\bp}}$. Then we obtain that $V_{X, \ell}$ (or $v(\Gamma_{X^{\bp}})^{>0, \ell}$) can be reconstructed group-theoretically from $\Pi_{X^{\bp}}$. Moreover, for every open subgroup $H \subseteq \Pi_{X^{\bp}}$, the map $$\gamma_{H}^{\rm ver, \ell}: V_{X_{H}, \ell} \migi V_{X, \ell}$$ constructed in Proposition \ref{pro-2-1} (d) can be reconstructed group-theoretically from the natural inclusion $H \migiinje \Pi_{X^{\bp}}$.  
\end{remarkA}

\subsection{Cohomology classes and sets of closed edges}\label{sec-4-2}

We maintain the notation introduced in Section \ref{sec-4-1}. Moreover, in this subsection, we suppose that the genus of the normalization of each irreducible component of $X$ is {\it positive} (i.e., $v(\Gamma_{X^{\bp}})=v(\Gamma_{X^{\bp}})^{>0, \ell}$ if $\ell \neq p$), and that $\Gamma^{\rm cpt}_{X^{\bp}}$ is {\it $2$-connected}.

We shall say that $$\mfT_{X^{\bp}}\defeq (\ell, d, f_{X}^{\bp}:Y^{\bp} \migi X^{\bp})$$ is an {\it edge-triple} associated to $X^{\bp}$ if the following conditions are satisfied: 

(i) $\ell$ and $d$ are prime numbers distinct from each other and from $p$.

(ii) $\ell \equiv 1 \ (\text{mod}\ d)$; this means that all $d$th roots of unity are contained in $\mbF_{\ell}$. Moreover, we write $\mu_{d} \subseteq \mbF_{\ell}^{\times}$ for the subgroup of $d$th roots of unity. 

(iii) $f_{X}^{\bp}: Y^{\bp} \migi X^{\bp}$ is a Galois admissible covering over $k$ whose Galois group is isomorphic to $\mu_{d}$ such that $f_{X}^{\bp}$ is \'etale, and that $\#v_{f_{X}}^{\rm sp} =0$. Note that since $v(\Gamma_{X^{\bp}})=v(\Gamma_{X^{\bp}})^{>0,d}$, $f_{X}^{\bp}$ exists.

On the other hand, we shall say that $$\mfT_{\Pi_{X^{\bp}}}\defeq (\ell, d, \alpha_{f_{X}})$$ is an {\it edge-triple} associated to $\Pi_{X^{\bp}}$ if the following conditions are satisfied: 

(i) $\alpha_{f_{X}} \in \text{Hom}(\Pi_{X^{\bp}}^{\rm \text{\'et}}, \mbF_{d})$. 

(ii) The composition of the following natural homomorphisms $$\Pi_{\widetilde X^{\bp}_{v}}^{\et} \migiinje \Pi_{X^{\bp}}^{\et} \overset{\alpha_{f_{X}}}\migi \mbF_{d}$$ is a surjection for every $v \in v(\Gamma_{X^{\bp}})$.

We see immediately that an edge-triple $\mfT_{X^{\bp}}$ associated to $X^{\bp}$ is equivalent to an edge-triple $\mfT_{\Pi_{X^{\bp}}}$ associated to $\Pi_{X^{\bp}}$, where the Galois admissible covering corresponding to the kernel of the composition of the natural homomorphisms $\Pi_{X^{\bp}} \migisurj \Pi_{X^{\bp}}^{\et} \overset{\alpha_{f_{X}}}\migi \mbF_{d}$ is $f_{X}^{\bp}$.

In the remainder of the present subsection, we fix an edge-triple $$\mfT_{\Pi_{X^{\bp}}}\defeq (\ell, d, \alpha_{f_{X}})$$ associated to $\Pi_{X^{\bp}}$. Write $\mfT_{X^{\bp}}\defeq (\ell, d, f_{X}^{\bp}:Y^{\bp}\migi X^{\bp})$ for the edge-triple associated to $X^{\bp}$ corresponding to $\mfT_{\Pi_{X^{\bp}}}$, $(g_{Y}, n_{Y})$ for the type of $Y^{\bp}$, $\Gamma_{Y^{\bp}}$ for the dual semi-graph of $Y^{\bp}$, $r_{Y}$ for the Betti number of $\Gamma_{Y^{\bp}}$, and $\Pi_{Y^{\bp}}$ for the kernel of the composition of the  homomorphisms $\Pi_{X^{\bp}} \migisurj \Pi_{X^{\bp}}^{\et} \overset{\alpha_{f_{X}}}\migi \mbF_{d}$. 

We put $$M_{Y^{\bp}} \defeq \text{Hom}(\Pi_{Y^{\bp}}, \mbF_{\ell}).$$ Note that there is a natural injection $M_{Y^{\bp}}^{\text{\'et}}\defeq \text{Hom}(\Pi_{Y^{\bp}}^{\et}, \mbF_{\ell})\migiinje M^{\rm}_{Y^{\bp}}$ induced by the natural surjection $\Pi_{Y^{\bp}} \migisurj \Pi_{Y^{\bp}}^{\text{\'et}}$. Then we obtain an exact sequence $$0\migi M_{Y^{\bp}}^{\text{\'et}}\migi M_{Y^{\bp}} \migi M_{Y^{\bp}}^{\rm ra}\defeq \text{coker}(M_{Y^{\bp}}^{\text{\'et}}\migiinje M^{\rm}_{Y^{\bp}})\migi 0$$ with a natural action of $\mu_{d}$, where ``ra" means ``ramification". For any element of $M_{Y^{\bp}}$, if the image of the element is not $0$ in $M_{Y^{\bp}}^{\rm ra}$, then the Galois admissible covering of $Y^{\bp}$ with Galois group $\mbZ/\ell\mbZ$ corresponding to the element is not \'etale.

Let $M^{\rm ra}_{Y^{\bp}, \mu_{d}} \subseteq M^{\rm ra}_{Y^{\bp}}$ be the subset of elements on which $\mu_{d}$ acts via the character $\mu_{d} \migiinje \mbF_{\ell}^{\times}$ and $E_{\mfT_{\Pi_{X^{\bp}}}}^{*} \subseteq M_{Y^{\bp}}$ the subset of elements that map to nonzero elements of $M^{\rm ra}_{Y^{\bp}, \mu_{n}}$. Let $\alpha \in E_{\mfT_{\Pi_{X^{\bp}}}}^{*}$. Write $$g^{\bp}_{\alpha}: Y^{\bp}_{\alpha} \migi Y^{\bp}$$ for the Galois admissible covering over $k$ corresponding to $\alpha$. Then we obtain a map $$\epsilon: E_{\mfT_{\Pi_{X^{\bp}}}}^{*} \migi \mbZ_{\geq 0}, \ \alpha \mapsto \#(e^{\rm op}(\Gamma_{Y^{\bp}_{\alpha}}) \cup e^{\rm cl}(\Gamma_{Y^{\bp}_{\alpha}})),$$ where $\Gamma_{Y^{\bp}_{\alpha}}$ denotes the dual semi-graph of $Y^{\bp}_{\alpha}$.  We put $$E^{\rm cl, \star}_{\mfT_{\Pi_{X^{\bp}}}}\defeq \{\alpha \in E_{\mfT_{\Pi_{X^{\bp}}}}^{*} \ | \  \#e_{g_{\alpha}}^{\rm op, ra}=0, \ \#e_{g_{\alpha}}^{\rm cl, ra}=d \}.$$ Note that $E^{\rm cl,\star}_{\mfT_{\Pi_{X^{\bp}}}}$ is not an empty set. For each $\alpha \in E^{\rm cl,\star}_{\mfT_{\Pi_{X^{\bp}}}}$, since the image of $\alpha$ is contained in $M^{\rm ra}_{Y^{\bp}, \mu_{d}}$, we obtain that the action of $\mu_{d}$ on the set $$\{y_{e}\}_{e\in e^{\rm cl, ra}_{g_{\alpha}}} \subseteq \text{Nod}(Y^{\bp})$$ is transitive, where $\text{Nod}(-)$ denotes the set of nodes of $(-)$, and $y_{e}$ denotes the node of $Y^{\bp}$ corresponding to $e$. Then there exists a unique node $x_{\alpha}$ of $X^{\bp}$ such that $f_{X}(y_{e})=x_{\alpha}$ for every $y_{e}\in \{y_{e}\}_{e\in e^{\rm cl, ra}_{g_{\alpha}}} $. We denote by $e_{\alpha} \in e^{\rm cl}(\Gamma_{X^{\bp}})$ the closed edge corresponding to $x_{\alpha}$. 

On the other hand, let $H \subseteq \Pi_{X^{\bp}}$ be an open subgroup. Write $f^{\rm sg}_{H}: \Gamma_{X_{H}^{\bp}} \migi \Gamma_{X^{\bp}}$ for the map of dual semi-graphs induced by the admissible covering $f^{\bp}_{H}: X_{H}^{\bp} \migi X^{\bp}$ over $k$ corresponding to $H$. We shall denote by $$f^{\rm cl}_{H}\defeq f^{\rm sg}_{H}|_{e^{\rm cl}(\Gamma_{X_{H}^{\bp}})}: e^{\rm cl}(\Gamma_{X_{H}^{\bp}}) \migi e^{\rm cl}(\Gamma_{X^{\bp}}).$$ Moreover, if $H \subseteq \Pi_{X^{\bp}}$ is an open normal subgroup, then $e^{\rm cl}(\Gamma_{X_{H}^{\bp}})$ admits a natural action of $\Pi_{X^{\bp}}/H$. Then we have the following result.

\begin{proposition}\label{pro-2-2}
(a) We define a pre-equivalence relation $\sim$ on $E^{\rm cl,\star}_{\mfT_{\Pi_{X^{\bp}}}}$ as follows: 
\begin{quote}
Let $\alpha, \beta \in E^{\rm cl,\star}_{\mfT_{\Pi_{X^{\bp}}}}$. We have that $\alpha \sim \beta$ if, for each $\lambda, \mu \in \mbF^{\times}_{\ell}$ for which $\lambda\alpha+\mu\beta \in E_{\mfT_{\Pi_{X^{\bp}}}}^{*}$, we have $\lambda\alpha+\mu\beta \in E^{\rm cl,\star}_{\mfT_{\Pi_{X^{\bp}}}}$.
\end{quote} 
Then the pre-equivalence relation $\sim$ on $E^{\rm cl, \star}_{\mfT_{\Pi_{X^{\bp}}}}$ is an equivalence relation. 

(b) We denote by $E_{\mfT_{\Pi_{X^{\bp}}}}^{\rm cl}$ the quotient set of $E^{\rm cl, \star}_{\mfT_{\Pi_{X^{\bp}}}}$ by $\sim$. Then we have a natural bijection $$\vartheta_{ \mfT_{\Pi_{X^{\bp}}}}: E_{\mfT_{\Pi_{X^{\bp}}}}^{\rm cl}  \isom e^{\rm cl}(\Gamma_{X^{\bp}}), \ [\alpha] \mapsto e_{\alpha},$$ where $[\alpha]$ denotes the equivalence class of $\alpha$.

(c) Let $\mfT_{\Pi_{X^{\bp}}}'$ be an arbitrary edge-triples associated to $\Pi_{X^{\bp}}$. Then we have a natural bijection $$E^{\rm cl}_{\mfT_{\Pi_{X^{\bp}}}'} \isom E^{\rm cl}_{\mfT_{\Pi_{X^{\bp}}}}$$ which fits into the following commutative diagram:
\[
\begin{CD}
E^{\rm cl}_{\mfT_{\Pi_{X^{\bp}}}'}@>\vartheta_{\mfT_{\Pi_{X^{\bp}}}'}>> e^{\rm cl}(\Gamma_{X^{\bp}})
\\
@VVV@|
\\
E^{\rm cl}_{\mfT_{\Pi_{X^{\bp}}}} @>\vartheta_{\mfT_{\Pi_{X^{\bp}}}}>> e^{\rm cl}(\Gamma_{X^{\bp}}).
\end{CD}
\]

(d) Let $H \subseteq \Pi_{X^{\bp}}$ be an open subgroup. Suppose that $([\Pi_{X^{\bp}}: H], \ell)=([\Pi_{X^{\bp}}: H], d)=1$. We have that $\mfT_{X^{\bp}}$ associated to $\Pi_{X^{\bp}}$ induces an edge-triple $$\mfT_{X_{H}^{\bp}}\defeq (\ell, d, f^{\bp}_{X_H}: Y^{\bp}_{X_{H}} \defeq Y^{\bp}\times_{X^{\bp}}X_{H}^{\bp} \migi X_{H}^{\bp})$$ associated to $X^{\bp}_{H}$, where $Y^{\bp}\times_{X^{\bp}}X_{H}^{\bp}$ denotes the fiber product in the category of pointed stable curves. Write $\mfT_{H}$ for the edge-triple associated to $H$ corresponding to $\mfT_{X^{\bp}_{H}}$. Then the natural injection $H \migiinje \Pi_{X^{\bp}}$ induces a surjective map $$\gamma^{\rm cl}_{\mfT_{\Pi_{X^{\bp}}}, H}: E^{\rm cl}_{\mfT_{H}} \migisurj E^{\rm cl}_{\mfT_{\Pi_{X^{\bp}}}}$$ which fits into the following commutative diagram:
\[
\begin{CD}
E^{\rm cl}_{\mfT_{H}}@>\vartheta_{\mfT_{H}}>> e^{\rm cl}(\Gamma_{X_{H}^{\bp}})
\\
@V\gamma_{\mfT_{\Pi_{X^{\bp}}}, H}^{\rm cl}VV@Vf_{H}^{\rm cl}VV
\\
E^{\rm cl}_{\mfT_{\Pi_{X^{\bp}}}} @>\vartheta_{\mfT_{\Pi_{X^{\bp}}}}>> e^{\rm cl}(\Gamma_{X^{\bp}}).
\end{CD}
\]
Moreover, suppose that $H \subseteq \Pi_{X^{\bp}}$ is an open normal subgroup. Then $E^{\rm cl}_{\mfT_{H}}$ admits an action of $\Pi_{X^{\bp}}/H$ such that $\vartheta_{\mfT_{H}}$ is compatible with $\Pi_{X^{\bp}}/H$-actions (i.e., $\vartheta_{\mfT_{H}}$ is $\Pi_{X^{\bp}}/H$-equivariant).
\end{proposition}

\begin{proof}
See \cite[Proposition 2.2 and Remark 2.2.1]{Y3} for (a), (b), and (c). Let us explain (d). Let $\alpha_{X} \in E^{\rm cl}_{\mfT_{\Pi_{X^{\bp}}}}$. Then $\alpha_{X}$ induces an element $$\beta_{X_{H}}=\sum_{\beta\in J_{\alpha_{X}}}c_{\beta}\beta, \ c_{\beta} \in \mbF_{\ell}^{\times}$$ via the natural homomorphism $\text{Hom}(\Pi_{Y^{\bp}_{X_{H}}}, \mbF_{\ell}) \migi \text{Hom}(\Pi_{Y^{\bp}}, \mbF_{\ell}),$ where $\Pi_{Y^{\bp}_{X_{H}}} \defeq \Pi_{Y^{\bp}} \cap H$, and $J_{\alpha_{X}}$ is a subset of $E^{\rm cl, \star}_{\mfT_{H}}$ such that, if $\beta_{1}, \beta_{2} \in J_{\alpha_{X}}$ distinct from each other, then $[\beta_{1}] \neq [\beta_{2}]$. Let $[\alpha_{X_{H}}] \in E^{\rm cl}_{\mfT_{H}}$. We define $$\gamma^{\rm cl}_{\mfT_{\Pi_{X^{\bp}}}, H}([\alpha_{X_{H}}])=[\alpha_{X}]$$ if there exists $\alpha_{X} \in E^{\rm cl}_{\mfT_{\Pi_{X^{\bp}}}}$ such that $[\beta]=[\alpha_{X_{H}}]$ for some $\beta \in J_{\alpha_{X}}$. It is easy to check that $\gamma^{\rm cl}_{\mfT_{\Pi_{X^{\bp}}}, H}$ is well-defined, and that the following diagram \[
\begin{CD}
E^{\rm cl}_{\mfT_{H}}@>\vartheta_{\mfT_{H}}>> e^{\rm cl}(\Gamma_{X_{H}^{\bp}})
\\
@V\gamma_{\mfT_{\Pi_{X^{\bp}}}, H}^{\rm cl}VV@Vf_{H}^{\rm cl}VV
\\
E^{\rm cl}_{\mfT_{\Pi_{X^{\bp}}}} @>\vartheta_{\mfT_{\Pi_{X^{\bp}}}}>> e^{\rm cl}(\Gamma_{X^{\bp}})
\end{CD}
\]
is commutative. 

Moreover, suppose that $H$ is an open normal subgroup of $\Pi_{X^{\bp}}$. Since $\Pi_{Y^{\bp}_{X_{H}}}$ is an open normal subgroup of $\Pi_{X^{\bp}}$, we have $$\Pi_{X^{\bp}}/\Pi_{Y^{\bp}_{X_{H}}} \cong \Pi_{X^{\bp}}/H \times \mbZ/d\mbZ.$$ Then the natural exact sequence $$1\migi \Pi_{Y^{\bp}_{X_{H}}} \migi \Pi_{X^{\bp}} \migi \Pi_{X^{\bp}} /\Pi_{Y^{\bp}_{X_{H}}} \migi 1$$ induces an outer representation $$\Pi_{X^{\bp}}/H \migiinje \Pi_{X^{\bp}}/\Pi_{Y^{\bp}_{X_{H}}} \migi \text{Out}(\Pi_{Y^{\bp}_{X_{H}}}).$$ Thus, we obtain an action of $\Pi_{X^{\bp}}/H$ on $$E^{\rm cl, \star}_{\mfT_{H}} \subseteq \text{Hom}(\Pi_{Y^{\bp}_{X_{H}}}, \mbF_{\ell})$$ induced by the outer representation.

Let $\sigma \in \Pi_{X^{\bp}}/H$ and $\alpha_{X_{H}}$, $\alpha_{X_{H}}' \in E^{\rm cl, \star}_{\mfT_{H}}$. We obverse that $\alpha_{X_{H}}\sim\alpha_{X_{H}}'$ if and only if $\sigma(\alpha_{X_{H}}) \sim \sigma(\alpha_{X_{H}}')$. Thus, we obtain an action of $\Pi_{X^{\bp}}/H$ on $E^{\rm cl}_{\mfT_{H}}$ induced by the natural injection $H \migiinje \Pi_{X^{\bp}}$. On the other hand, it is easy to check that  the commutative diagram above is compatible with the $\Pi_{X^{\bp}}/H$-actions. This completes the proof of the proposition.
\end{proof}

\begin{remarkA}\label{rem-pro-2-2}
By applying Theorem \ref{types}, we have that  $\Pi_{X^{\bp}}^{\et}$ can be reconstructed group-theoretically from $\Pi_{X^{\bp}}$. Then  $E^{\rm cl}_{\mfT_{\Pi_{X^{\bp}}}}$ (or $e^{\rm cl}(\Gamma_{X^{\bp}})$) can be reconstructed group-theoretically from $\Pi_{X^{\bp}}$. Moreover, for every open subgroup $H \subseteq \Pi_{X^{\bp}}$, the map $$\gamma_{\mfT_{\Pi_{X^{\bp}, H}}}^{\rm cl}: E^{\rm cl}_{\mfT_{H}}\migi E^{\rm cl}_{\mfT_{\Pi_{X^{\bp}}}}$$ constructed in Proposition \ref{pro-2-2} (d) can be reconstructed group-theoretically from the natural inclusion $H \migiinje \Pi_{X^{\bp}}$. 
\end{remarkA}

Next, let us calculate the cardinality $\#E^{\rm cl, \star}_{\mfT_{\Pi_{X^{\bp}}}}$ of $E^{\rm cl, \star}_{\mfT_{\Pi_{X^{\bp}}}}$. We put $$E^{\rm cl, \star}_{\mfT_{\Pi_{X^{\bp}}}, e}\defeq \{\alpha \in E^{\rm cl, \star}_{\mfT_{\Pi_{X^{\bp}}}}\ | \ e=e_{\alpha} \}, \ e\in e^{\rm cl}(\Gamma_{X^{\bp}}).$$ Note that $e=e_{\alpha}$, $\alpha \in E^{\rm cl, \star}_{\mfT_{\Pi_{X^{\bp}}}, e}$, means that the Galois admissible covering $g_{\alpha}^{\bp}: Y^{\bp}_{\alpha} \migi Y^{\bp}$ over $k$ induced by $\alpha$ is (totally) ramified over $f_{X}^{-1}(x_{e})$, where $x_{e}$ denotes the node of $X$ corresponding to $e$. 
Moreover, we have the following disjoint union $$E^{\rm cl, \star}_{\mfT_{\Pi_{X^{\bp}}}}=\bigsqcup_{e \in e^{\rm cl}(\Gamma_{X^{\bp}})} E^{\rm cl, \star}_{\mfT_{\Pi_{X^{\bp}}}, e}.$$ Let $m \in \mbZ_{\geq 0}$ and $e\in e^{\rm cl}(\Gamma_{X^{\bp}})$. We shall put $$E^{{\rm cl, \star}, m}_{\mfT_{\Pi_{X^{\bp}}}, e}\defeq \{\alpha \in E^{{\rm cl, \star}}_{\mfT_{\Pi_{X^{\bp}}}, e}\ | \ \#v_{g_{\alpha}}^{\rm sp}=m\}.$$

Let $e \in e^{\rm cl}(\Gamma_{X^{\bp}})$ be a closed edge. Write $Y_{e}$ for the normalization of the underlying curve $Y$ of $Y^{\bp}$ at $f_{X}^{-1}(x_{e})$ and $$\text{nor}_{e}: Y_{e} \migi Y$$  for the resulting normalization morphism. Since the genus of the normalization of each irreducible component of $X^{\bp}$ is positive, we obtain that the genus of the normalization of each irreducible component of $Y_{e}$ is also positive. Moreover, since $\Gamma_{X^{\bp}}$ is $2$-connected, $Y_{e}$ is connected. We have the following lemma.

\begin{lemma}\label{lem-2-1}
We maintain the notation introduced above. Let $e \in e^{\rm cl}(\Gamma_{X^{\bp}})$ be a closed edge. Then we have $$\#E^{\rm cl, \star}_{\mfT_{\Pi_{X^{\bp}}}, e}=\ell^{2g_{Y}-d-r_{Y}+1}-\ell^{2g_{Y}-d-r_{Y}}.$$ Moreover, we have $$\#E^{\rm cl, \star}_{\mfT_{\Pi_{X^{\bp}}}}=\#e^{\rm cl}(\Gamma_{X^{\bp}})(\ell^{2g_{Y}-d-r_{Y}+1}-\ell^{2g_{Y}-d-r_{Y}}).$$ 
\end{lemma}

\begin{proof}
Write $R_{e} \subseteq Y_{e}$ for the set of closed subset $(f_{X} \circ\text{nor}_{e})^{-1}(x_{e})$. Then $E^{\rm cl, \star}_{\mfT_{\Pi_{X^{\bp}}}, e}$ can be naturally regarded as a subset of $H^{1}_{\et}(Y_{e} \setminus R_{e}, \mbF_{\ell})$ via the natural open immersion $Y_{e}\setminus R_{e}  \migiinje Y_{e}$. Write $L_{e}$ for the $\mbF_{\ell}$-vector subspace spanned by $E^{\rm cl, \star}_{\mfT_{\Pi_{X^{\bp}}}, e}$ in $H^{1}_{\text{\rm \'et}}(Y_{e} \setminus R_{e}, \mbF_{\ell})$. Then we see that $$E^{\rm cl, \star}_{\mfT_{\Pi_{X^{\bp}}}, e}=L_{e}\setminus H^{1}_{\text{\rm \'et}}(Y_{e}, \mbF_{\ell}).$$ 

Write $H^{\rm ra}_{e}$ for the cokernel of the natural inclusion $H^{1}_{\text{\rm \'et}}(Y_{e}, \mbF_{\ell}) \migiinje L_{e}$. We obtain an exact sequence as follows:
$$0 \migi H^{1}_{\text{\rm \'et}}(Y_{e}, \mbF_{\ell}) \migi L_{e}\migi H^{\rm ra}_{e}\migi 0.$$ 
On the other hand, since the action of $\mu_{d}$ on $f^{-1}(x_{e})$ is translative, the structure of the maximal pro-$\ell$ quotient $\Pi_{Y^{\bp}}^{\ell}$ of $\Pi_{Y^{\bp}}$ implies that $$\text{dim}_{\mbF_{\ell}}(H^{\rm ra}_{e})=1.$$ Since $$\text{dim}_{\mbF_{\ell}}(H^{1}_{\text{\rm \'et}}(Y_{e}, \mbF_{\ell}))=2(g_{Y}-d)-(r_{Y}-d)=2g_{Y}-d-r_{Y},$$ we obtain that $$\#E^{\rm cl, \star}_{\mfT_{\Pi_{X^{\bp}}}, e}=\ell^{2g_{Y}-d-r_{Y}+1}-\ell^{2g_{Y}-d-r_{Y}}.$$ Thus, we have $$\#E^{\rm cl, \star}_{\mfT_{\Pi_{X^{\bp}}}}=\#e^{\rm cl}(\Gamma_{X^{\bp}})(\ell^{2g_{Y}-d-r_{Y}+1}-\ell^{2g_{Y}-d-r_{Y}}).$$ 
This completes the proof of the lemma.
\end{proof}

Next, we introduce some notation concerning open edges.  We put $$E^{\rm op, \star}_{\mfT_{\Pi_{X^{\bp}}}}\defeq \{\alpha \in E_{\mfT_{\Pi_{X^{\bp}}}}^{*} \ | \  \#e_{g_{\alpha}}^{\rm op, ra}=d, \ \#e_{g_{\alpha}}^{\rm cl, ra}=0 \}.$$ Note that $E^{\rm op,\star}_{\mfT_{\Pi_{X^{\bp}}}}$ is not an empty set if $n_{X} \neq 0$. For each $\alpha \in E^{\rm op,\star}_{\mfT_{\Pi_{X^{\bp}}}}$, since the image of $\alpha$ is contained in $M^{\rm ra}_{Y^{\bp}, \mu_{d}}$, we obtain that the action of $\mu_{d}$ on the set $$\{y_{e}\}_{e\in e^{\rm op, ra}_{g_{\alpha}}} \subseteq D_{Y}$$ is transitive, where  $y_{e}$ denotes the marked point of $Y^{\bp}$ corresponding to $e$. Then there exists a unique marked point $x_{\alpha} \in D_{X}$ of $X^{\bp}$ such that $f_{X}(y_{e})=x_{\alpha}$ for every $y_{e}\in \{y_{e}\}_{e\in e^{\rm op, ra}_{g_{\alpha}}} $. We denote by $e_{\alpha} \in e^{\rm op}(\Gamma_{X^{\bp}})$ the open edge corresponding to $x_{\alpha}$.  Moreover, we put $$E^{\rm op, \star}_{\mfT_{\Pi_{X^{\bp}}}, e}\defeq \{\alpha \in E^{\rm op, \star}_{\mfT_{\Pi_{X^{\bp}}}}\ | \ e=e_{\alpha} \}, \ e\in e^{\rm op}(\Gamma_{X^{\bp}}).$$ Note that $e=e_{\alpha}$, $\alpha \in E^{\rm op, \star}_{\mfT_{\Pi_{X^{\bp}}}, e}$, means that the Galois admissible covering $g_{\alpha}^{\bp}: Y^{\bp}_{\alpha} \migi Y^{\bp}$ over $k$ induced by $\alpha$ is (totally) ramified over $f_{X}^{-1}(x_{e})$, where $x_{e}$ denotes the marked point of $X^{\bp}$ corresponding to $e$. 
Moreover, we have the following disjoint union $$E^{\rm op, \star}_{\mfT_{\Pi_{X^{\bp}}}}=\bigsqcup_{e \in e^{\rm op}(\Gamma_{X^{\bp}})} E^{\rm op, \star}_{\mfT_{\Pi_{X^{\bp}}}, e}.$$ Let $m \in \mbZ_{\geq 0}$ and $e\in e^{\rm op}(\Gamma_{X^{\bp}})$. We shall put $$E^{{\rm op, \star}, m}_{\mfT_{\Pi_{X^{\bp}}}, e}\defeq \{\alpha \in E^{{\rm op, \star}}_{\mfT_{\Pi_{X^{\bp}}}, e}\ | \ \#v_{g_{\alpha}}^{\rm sp}=m\}.$$



Finally, we introduce the following conditions concerning pointed stable curves. Let $W^{\bp}$ be a pointed stable curve over $k$ of type $(g_{W}, n_{W})$, $\Gamma_{W^{\bp}}$ the dual semi-graph of $W^{\bp}$, and $\Pi_{W^{\bp}}$ the solvable admissible fundamental group of $W^{\bp}$.

\begin{conditiona}\label{conda}
We shall say that $W^{\bp}$ satisfies Condition A if the following conditions are satisfied:

(i) the genus of the normalization of each irreducible component of $W$ is positive; 

(ii) every irreducible component of $W$ is smooth over $k$;

(iii) $\Gamma_{W^{\bp}}^{\rm cpt}$ is $2$-connected; 

(iv) $\#(v(\Gamma_{W^{\bp}})^{b\leq 1})=0.$
\end{conditiona}

\begin{conditionb}\label{condb}
We shall say that $W^{\bp}$ satisfies Condition B if $\Gamma_{W^{\bp}_H}^{\rm cpt}$ is $2$-connected for every open subgroup $H \subseteq \Pi_{W^{\bp}}$.
\end{conditionb}

\begin{lemma}\label{rem-them-1-1-3}
Let $m$ be a positive natural number prime to $p$ and $H \defeq D^{(3)}_{m}(\Pi_{W^{\bp}}) \subseteq \Pi_{W^{\bp}}$. Then $W^{\bp}_{H}$ satisfies Condition A, and the Betti number of the dual semi-graph of $W^{\bp}_{H}$ is positive.
\end{lemma}

\begin{proof}
The lemma follows from the structure of $\Pi_{W^{\bp}}^{p'}$.
\end{proof}

\subsection{Reconstruction of sets of vertices, sets of closed edges, sets of genus, and sets of $p$-rank from surjections}\label{sec-4-3}

In this subsection, we prove that sets of vertices, sets of closed edges, and sets of genus can be reconstructed group-theoretically from a surjective open continuous homomorphism of solvable admissible fundamental groups. 

We fix some notation. Let $i\in \{1, 2\}$, $k_{i}$ an algebraically closed field of characteristic $p>0$, and $\ell$ a prime number distinct from $p$. Let $X^{\bp}_{i}$ be a pointed stable curve of type $(g_{X_{i}}, n_{X_{i}})$ over $k_{i}$, $\Pi_{X^{\bp}_{i}}$ the solvable admissible fundamental group of $X^{\bp}_{i}$, $\Gamma_{X^{\bp}_{i}}$ the dual semi-graph of $X^{\bp}_{i}$, and $r_{X_{i}}$  the Betti number of $\Gamma_{X^{\bp}_{i}}$. Moreover, let $v_{i} \in v(\Gamma_{X_{i}^{\bp}})$, $\widetilde X^{\bp}_{i, v_{i}}$ the smooth pointed stable curve of type $(g_{i, v_{i}}, n_{i, v_{i}})$ over $k_{i}$ associated to $v_{i}$, and  $\sigma_{i, v_{i}}$ the $p$-rank of  $\widetilde X_{i, v_{i}}^{\bp}$. We introduce the following condition:

\begin{conditionc}\label{condc}
We shall say that $X_{1}^{\bp}$ and $X_{2}^{\bp}$ satisfy Condition C if the following conditions are satisfied: 

(i) $(g_{X_{1}}, n_{X_{1}})=(g_{X_{2}}, n_{X_{2}})$; 

(ii) $\#v(\Gamma_{X^{\bp}_{1}})=\#v(\Gamma_{X^{\bp}_{2}})$; 

(iii) $\#e^{\rm cl}(\Gamma_{X^{\bp}_{1}})=\#e^{\rm cl}(\Gamma_{X^{\bp}_{2}})$. 
\end{conditionc}

In the remainder of the present subsection, we suppose that {\it $X_{1}^{\bp}$ and $X_{2}^{\bp}$  satisfy Condition A, Condition B, and Condition C}. Moreover, let $$\phi: \Pi_{X^{\bp}_{1}}\migisurj \Pi_{X^{\bp}_{2}}$$ be an arbitrary open continuous homomorphism of the solvable admissible fundamental groups of $X^{\bp}_{1}$ and $X^{\bp}_{2}$, and $$(g_{X}, n_{X})\defeq (g_{X_{1}}, n_{X_{1}})=(g_{X_{2}}, n_{X_{2}}).$$ Note that we have that $r_{X_{1}}=r_{X_{2}}$, and that by Lemma \ref{surj}, $\phi$ is a {\it surjective} open continuous homomorphism. First, we have the following lemma.


\begin{lemma}\label{rem-them-1-1-4}
We maintain the notation introduced above. Then we have $${\rm Avr}_{p}(\Pi_{X_{i}^{\bp}})=g_{X_{i}}-r_{X_{i}}.$$ 
\end{lemma}

\begin{proof}
The lemma follows immediately from Condition A and Theorem \ref{max and average} (b). 
\end{proof}

Let $G$ be a finite group such that $(\#G, p)=1$ and $$f_{i}^{\bp}: Y^{\bp}_{i} \migi X^{\bp}_{i}$$ a Galois admissible covering over $k_{i}$ with Galois group $G$. Let $j \in \{1, 2\}$ such that $i\neq j$. Then the isomorphism $\phi^{p'}: \Pi_{X_{1}^{\bp}}^{p'} \isom \Pi_{X^{\bp}_{2}}^{p'}$ induced by $\phi$ implies that $f_{i}^{\bp}$ induces a Galois admissible covering $$f_{j}^{\bp}: Y^{\bp}_{j} \migi X^{\bp}_{j}$$ over $k_{j}$ with Galois group $G$. We write $(g_{Y_{i}}, n_{Y_{i}})$ for the type of $Y_{i}^{\bp}$, $\Gamma_{Y^{\bp}_{i}}$ for the dual semi-graph of $Y^{\bp}_{i}$, and $r_{Y_{i}}$ for the Betti number of $\Gamma_{Y^{\bp}_{i}}$.

\begin{lemma}\label{lem-4-7}
We maintain the notation introduced above. Suppose that $G\cong \mbZ/\ell\mbZ$, that $f_{1}^{\bp}: Y^{\bp}_{1} \migi X^{\bp}_{1}$ is \'etale, and that $\#v_{f_{1}}^{\rm sp}=m$. Then we have $$\#e_{f_{2}}^{\rm cl, ra}+\frac{1}{2}\#e^{\rm op, ra}_{f_{2}}+\#v_{f_{2}}^{\rm sp}\leq m.$$
\end{lemma}

\begin{proof}
Since $f_{1}^{\bp}$ is an \'etale covering, the Riemann-Hurwitz formula implies that $$g_{Y_{1}}=\ell(g_{X}-1)+1$$ and $$g_{Y_{2}}=\ell(g_{X}-1)+\frac{1}{2}(\ell-1)\#e^{\rm op, ra}_{f_{2}}+1.$$ Then we obtain $$g_{Y_{1}}-g_{Y_{2}}=-\frac{1}{2}(\ell -1)\#e^{\rm op, ra}_{f_{2}}.$$ On the other hand, we have $$r_{Y_{1}}=\ell\#e^{\rm cl}(\Gamma_{X^{\bp}_{1}})-\#v(\Gamma_{X^{\bp}_{1}})+\#v_{f_{1}}^{\rm sp}-\ell \#v_{f_{1}}^{\rm sp}+1$$$$=\ell\#e^{\rm cl}(\Gamma_{X^{\bp}_{1}})-\#v(\Gamma_{X^{\bp}_{1}})-(\ell-1)m+1$$ and $$r_{Y_{2}}=\ell\#e_{f_{2}}^{\rm cl, \text{\'et}}+\#e^{\rm cl, ra}_{f_{2}}-\ell\#v^{\rm sp}_{f_{2}}-\#v^{\rm ra}_{f_{2}}+1.$$ Since $\#e(\Gamma_{X^{\bp}_{1}})=\#e(\Gamma_{X^{\bp}_{2}})$ and $\#v(\Gamma_{X^{\bp}_{1}})=\#v(\Gamma_{X^{\bp}_{2}})$, we obtain that $$r_{Y_{1}}-r_{Y_{2}}=(\ell-1)\#e_{f_{2}}^{\rm cl, ra}+(\ell -1)(\#v^{\rm sp}_{f_{2}}-m)$$ Moreover, by applying Lemma \ref{rem-them-1-1-4} and Lemma \ref{lem-0} (b), we have $$g_{Y_{1}}-g_{Y_{2}} \geq r_{Y_{1}}-r_{Y_{2}}.$$ Thus, we obtain $$\#e_{f_{2}}^{\rm cl, ra}+\frac{1}{2}\#e^{\rm op, ra}_{f_{2}}+\#v_{f_{2}}^{\rm sp}\leq m.$$ This completes the proof of the lemma.
\end{proof}

\begin{corollary}\label{coro-4-8}
We maintain the notation introduced above. Suppose that $G\cong \mbZ/\ell\mbZ$, that $f_{1}^{\bp}: Y^{\bp}_{1} \migi X^{\bp}_{1}$ is \'etale, and that $\#v_{f_{1}}^{\rm sp}=0$. Then we have that $f_{2}^{\bp}: Y^{\bp}_{2} \migi X^{\bp}_{2}$ is \'etale, and that $\#v_{f_{2}}^{\rm sp}=0$. 
\end{corollary}

\begin{proof}
The corollary follows immediately from Lemma \ref{lem-4-7}. 
\end{proof}

\begin{corollary}\label{coro-4-9}
We maintain the notation introduced above. Suppose that $G\cong \mbZ/\ell\mbZ$, that $f_{1}^{\bp}: Y^{\bp}_{1} \migi X^{\bp}_{1}$ is \'etale, and that $\#v_{f_{1}}^{\rm sp}=1$. Then we have that $f_{2}^{\bp}: Y^{\bp}_{2} \migi X^{\bp}_{2}$ is \'etale. 
\end{corollary}

\begin{proof}
In order to verify the corollary, it is sufficient to prove that $$\#e^{\rm cl, ra}_{f_{2}}=\#e^{\rm op, ra}_{f_{2}}=0.$$ By applying Lemma \ref{lem-4-7}, we have $$\#e_{f_{2}}^{\rm cl, ra}+\frac{1}{2}\#e^{\rm op, ra}_{f_{2}}+\#v_{f_{2}}^{\rm sp}\leq 1.$$

Suppose that $\#e^{\rm cl, ra}_{f_{2}} \neq 0$. Since $X_{2}^{\bp}$ satisfies Condition A, the inequality above and the structures of the maxmial prime-to-$p$ quotient of solvable admissible fundamental groups imply that either $\#e^{\rm cl, ra}_{f_{2}}=1$ and $\#e^{\rm op, ra}_{f_{2}}\geq 2$ or $\#e^{\rm cl, ra}_{f_{2}}\geq 2$ holds. Then we have $\#e_{f_{2}}^{\rm cl, ra}+\frac{1}{2}\#e^{\rm op, ra}_{f_{2}}+\#v_{f_{2}}^{\rm sp}>1$. Thus, we have $\#e^{\rm cl, ra}_{f_{2}}=0$.

Suppose that $\#e^{\rm op, ra}_{f_{2}} \neq 0$. Since $\#e^{\rm cl, ra}_{f_{2}}=0$, the inequality above implies that $\#e^{\rm op, ra}_{f_{2}}=2.$ Let $\ell' \neq p$ be a prime number distinct from $\ell$, and let $$g^{\bp}_{1}: Z_{1}^{\bp} \migi X^{\bp}_{1}$$ be a Galois \'etale covering of over $k_{1}$ with Galois group $\mbZ/\ell'\mbZ$ such that $\#v_{g_{1}}^{\rm sp}=0$. Then Corollary \ref{coro-4-8} implies that the Galois admissible covering $$g_{2}^{\bp}: Z_{2}^{\bp} \migi X_{2}^{\bp}$$ over $k_{2}$ with Galois group $\mbZ/\ell'\mbZ$ induced by $g_{2}^{\bp}$ is \'etale covering, and that $\#v_{g_{2}}^{\rm sp}=0$.  Write $\Gamma_{Z^{\bp}_{i}}$ for the dual semi-graphs of $Z^{\bp}_{i}$. We obtain $$\#v(\Gamma_{X^{\bp}_{1}})=\#v(\Gamma_{Z^{\bp}_{1}})=\#v(\Gamma_{Z^{\bp}_{2}})=\#v(\Gamma_{X^{\bp}_{2}}),$$ $$\ell'\#e^{\rm op}(\Gamma_{X^{\bp}_{1}})=\#e^{\rm op}(\Gamma_{Z^{\bp}_{1}})=\#e^{\rm op}(\Gamma_{Z^{\bp}_{2}})=\ell'\#e^{\rm op}(\Gamma_{X^{\bp}_{2}}),$$  $$\ell'\#e^{\rm cl}(\Gamma_{X^{\bp}_{1}})=\#e^{\rm cl}(\Gamma_{Z^{\bp}_{1}})=\#e^{\rm cl}(\Gamma_{Z^{\bp}_{2}})=\ell'\#e^{\rm cl}(\Gamma_{X^{\bp}_{2}}).$$ We have that $Z^{\bp}_{1}$ and $Z_{2}^{\bp}$ satisfy Condition A, Condition B, and Condition C.

We denote by $W_{i}^{\bp} \defeq Y^{\bp}_{i}\times_{X_{i}^{\bp}} Z_{i}^{\bp}$. Note that since $\ell'\neq \ell$, we have that $W_{i}^{\bp}$ is connected. Then $f_{i}^{\bp}$ induces a Galois admissible covering $$h^{\bp}_{i}: W^{\bp}_{i} \migi Z^{\bp}_{i}$$ over $k_{i}$ with Galois group $\mbZ/\ell\mbZ$. We have that $h_{1}^{\bp}$ is \'etale, that $\#v_{h_{1}}^{\rm sp}=1$, and that $\#e_{h_{2}}^{\rm op, ra}=2\ell'.$ Then Lemma \ref{lem-4-7} implies that $$1<\#e_{h_{2}}^{\rm cl, ra}+\frac{1}{2}\#e^{\rm op, ra}_{h_{2}}+\#v_{h_{2}}^{\rm sp}=\#e_{h_{2}}^{\rm cl, ra}+\ell'+\#v_{h_{2}}^{\rm sp}\leq 1.$$ Thus, we obtain $\#e^{\rm op, ra}_{f_{2}}=0$. This completes the proof of the corollary.
\end{proof}


We put $$M_{X_{i}^{\bp}} \defeq {\rm Hom}(\Pi_{X_{i}^{\bp}}, \mbF_{\ell}),$$ $$M^{\text{\rm \'et}}_{X_{i}^{\bp}} \defeq {\rm Hom}(\Pi_{X_{i}^{\bp}}^{\et}, \mbF_{\ell}),$$ $$M^{\rm top}_{X_{i}^{\bp}} \defeq {\rm Hom}(\Pi_{X_{i}^{\bp}}^{\rm top}, \mbF_{\ell}).$$ Note that we have the following injections (or weight-monodromy filtration) $$M^{\rm top}_{X_{i}^{\bp}}\migiinje M^{\text{\rm \'et}}_{X_{i}^{\bp}} \migiinje M_{X_{i}^{\bp}} \ (\text{or} \  M^{\rm top}_{X_{i}^{\bp}}\subseteq M^{\text{\rm \'et}}_{X_{i}^{\bp}} \subseteq M_{X_{i}^{\bp}})$$ induced by the natural surjections $\Pi_{X_{i}^{\bp}} \migisurj \Pi_{X_{i}^{\bp}}^{\et} \migisurj \Pi_{X_{i}^{\bp}}^{\rm top}$. Moreover,  we have an isomorphism $$\psi_{\rm \ell}: M_{X_{2}^{\bp}} \isom M_{X_{1}^{\bp}}$$ induced by the isomorphism $\phi^{\ell}: \Pi_{X^{\bp}_{1}}^{\ell} \isom \Pi_{X^{\bp}_{2}}^{\ell}$. Then we have the following propositions.

\begin{proposition}\label{prop-4-10}
We maintain the notation introduced above. Then the isomorphism $\psi_{\ell}: M_{X_{2}^{\bp}}\isom M_{X_{1}^{\bp}}$ induces an isomorphism $$\psi_{\ell}^{\text{\rm \'et}}: M^{\text{\rm \'et}}_{X_{2}^{\bp}} \isom M^{\text{\rm \'et}}_{X_{1}^{\bp}}$$ group-theoretically. Moreover, we have 
the following commutative diagram:
\[
\begin{CD}
M^{\text{\rm \'et}}_{X_{2}^{\bp}}
@>\psi_{\ell}^{\text{\rm \'et}}>>
M^{\text{\rm \'et}}_{X_{1}^{\bp}}
\\
@VVV
@VVV
\\
M_{X_{2}^{\bp}}
@>\psi_{\ell}>>
M_{X_{1}^{\bp}},
\end{CD}
\]
where the vertical arrows are injections.

\end{proposition}

\begin{proof}
To verify the proposition, it is sufficient to prove that $\psi_{\ell}^{-1}: M_{X_{1}^{\bp}}\isom M_{X_{2}^{\bp}}$ induces an isomorphism $\psi_{\ell}^{-1, \text{\rm \'et}}: M^{\text{\'et}}_{X_{1}^{\bp}} \isom M^{\text{\rm \'et}}_{X_{2}^{\bp}}$ which fits into the following commutative diagram:
\[
\begin{CD}
M_{X_{1}^{\bp}}^{\text{\'et}}
@>\psi_{\ell}^{-1, \text{\'et}}>>
M_{X_{2}^{\bp}}^{\text{\'et}}
\\
@VVV
@VVV
\\
M_{X_{1}^{\bp}}
@>\psi_{\ell}^{-1}>>
M_{X_{1}^{\bp}},
\end{CD}
\]
where the vertical arrows are injections. 

Let $\alpha_{1} \in M^{\text{\rm \'et}}_{X_{1}^{\bp}}$ be a non-trivial element and $f^{\bp}_{1, \alpha}:Y^{\bp}_{1, \alpha} \migi X^{\bp}_{1}$ the Galois \'etale covering over $k_{1}$ with Galois group $\mbZ/\ell\mbZ$ corresponding to $\alpha$. We put $$L_{X_{1}^{\bp}}\defeq \{\alpha_{1} \in M^{\text{\'et}}_{X_{1}^{\bp}} \ | \ \#v_{f_{1, \alpha_{1}}}^{\rm sp}=1 \}.$$ We see that $M^{\text{\rm \'et}}_{X_{1}^{\bp}}$ is spanned by $L_{X_{1}^{\bp}}$ as an $\mbF_{\ell}$-vector space. 

On the other hand, Corollary \ref{coro-4-9} implies that $f^{\bp}_{1, \alpha_{1}}$ induces a Galois \'etale covering of $X_{2}^{\bp}$ over $k_{2}$ with Galois group $\mbZ/\ell\mbZ$. This means that $\psi_{\ell}^{-1}$ induces an injection of $\mbF_{\ell}$-vector spaces $$\psi_{\ell}^{-1, \text{\rm \'et}}: M^{\text{\'et}}_{X_{1}^{\bp}}\migiinje M^{\text{\'et}}_{X_{2}^{\bp}}.$$ Moreover, since $\text{dim}_{\mbF_{\ell}}(M^{\text{\'et}}_{X_{1}^{\bp}})=2g_{X_{1}}-r_{X_{1}}=2g_{X_{2}}-r_{X_{2}}=\text{dim}_{\mbF_{\ell}}(M^{\text{\'et}}_{X_{2}^{\bp}})$, we obtain that $$\psi_{\ell}^{-1, \text{\rm \'et}}: M^{\text{\'et}}_{X_{1}^{\bp}}\isom M^{\text{\'et}}_{X_{2}^{\bp}}$$ is an isomorphism. This completes the proof of the proposition.
\end{proof}

\begin{proposition}\label{prop-4-11}
We maintain the notation introduced above. Then the isomorphism $\psi_{\ell}: M_{X_{2}^{\bp}}\isom M_{X_{1}^{\bp}}$ induces an isomorphism $$\psi_{\ell}^{\rm top}: M^{\rm top}_{X_{2}^{\bp}} \isom M^{\rm top}_{X_{1}^{\bp}}$$ group-theoretically. Moreover, we have the following commutative diagram:
\[
\begin{CD}
M^{\rm top}_{X_{2}^{\bp}}
@>\psi_{\ell}^{\rm top}>>
M^{\rm top}_{X_{1}^{\bp}}
\\
@VVV
@VVV
\\
M^{\text{\rm \'et}}_{X_{2}^{\bp}}
@>\psi_{\ell}^{\text{\rm \'et}}>>
M^{\text{\rm \'et}}_{X_{1}^{\bp}}
\\
@VVV
@VVV
\\
M_{X_{2}^{\bp}}
@>\psi_{\ell}>>
M_{X_{1}^{\bp}},
\end{CD}
\]
where the vertical arrows are injections.
\end{proposition}

\begin{proof}
First, by Proposition \ref{prop-4-10}, the isomorphism $\psi_{\ell}: M_{X_{2}^{\bp}}\isom M_{X_{1}^{\bp}}$ induces an isomorphism $\psi_{\ell}^{\et}: M^{\et}_{X_{2}^{\bp}} \isom M^{\et}_{X_{1}^{\bp}}$. Let $\alpha_{2} \in M^{\rm top}_{X_{2}^{\bp}} \subseteq M^{\et}_{X_{2}^{\bp}}$ be a non-trivial element and $$f^{\bp}_{2, \alpha_{2}}: Y^{\bp}_{2, \alpha_{2}}\migi X_{2}^{\bp}$$ the Galois \'etale covering over $k_{2}$ with Galois group $\mbZ/\ell\mbZ$ corresponding to $\alpha_{2}$. Then we obtain an element $\alpha_{1} \defeq \psi_{\ell}^{\text{\rm \'et}}(\alpha_{2}) \in M^{\text{\rm \'et}}_{X_{1}^{\bp}}$. Write $f^{\bp}_{1, \alpha_{1}}: Y^{\bp}_{1, \alpha_{1}}\migi X_{1}^{\bp}$ for the Galois \'etale covering over $k_{1}$ with Galois group $\mbZ/\ell\mbZ$ corresponding to $\alpha_{1}$. Note that the types of $Y^{\bp}_{1, \alpha_{1}}$ and $Y_{2, \alpha_{2}}^{\bp}$ are equal, and that $Y^{\bp}_{1, \alpha_{1}}$ and $Y_{2, \alpha_{2}}^{\bp}$ satisfy Condition A.

Lemma \ref{rem-them-1-1-4} and Lemma \ref{lem-0} (b)  imply that $$r_{Y_{1, \alpha_{1}}} \leq r_{Y_{2, \alpha_{2}}},$$ where $r_{Y_{1, \alpha_{1}}}$ and $r_{Y_{2, \alpha_{2}}}$ denote the Betti numbers of the dual semi-graphs of $Y^{\bp}_{1, \alpha_{1}}$ and $Y^{\bp}_{2, \alpha_{2}}$, respectively. Since $\#v^{\rm sp}_{f_{2, \alpha_{2}}}=\#v(\Gamma_{X^{\bp}_{2}})=\#v(\Gamma_{X^{\bp}_{1}})$, the inequality implies $\#v^{\rm sp}_{f_{1, \alpha_{1}}}=\#v(\Gamma_{X^{\bp}_{1}}).$ Thus, we have $$\alpha_{1} \in M^{\text{\rm top}}_{X_{1}^{\bp}}.$$ Then $\alpha_{1}$ induces an injection $$\psi_{\ell}^{\rm top}: M^{\text{\rm top}}_{X_{2}^{\bp}} \migiinje M^{\text{\rm top}}_{X_{1}^{\bp}}.$$ Moreover, since $\text{dim}_{\mbF_{\ell}}(M^{\text{\rm top}}_{X_{2}^{\bp}})=r_{X_{2}}=r_{X_{1}}=\text{dim}_{\mbF_{\ell}}(M^{\text{\rm top}}_{X_{1}^{\bp}})$, we have that $\psi_{\ell}^{\rm top}$ is an isomorphism. This completes the proof of the proposition.
\end{proof}

\begin{remarkA}
Proposition \ref{prop-4-10} and Proposition \ref{prop-4-11} mean that the weight-monodromy filtrations can be reconstructed group-theoretically from $\phi$.
\end{remarkA}

\begin{lemma}\label{lem-4-12}
We maintain the notation introduced above. Suppose that $G\cong \mbZ/\ell\mbZ$, that $f_{2}^{\bp}$ is \'etale, and that $\#v_{f_{2}}^{\rm ra}=1$. Then we have that $f_{1}^{\bp}$ is \'etale, and that $\#v_{f_{1}}^{\rm ra}=1$.
\end{lemma}

\begin{proof}
By Proposition \ref{prop-4-10}, we obtain that $f_{1}^{\bp}$ is \'etale. This implies that $g_{Y_{1}}=g_{Y_{2}}$, and that $\#e^{\rm cl}(\Gamma_{Y^{\bp}_{1}})=\ell\#e^{\rm cl}(\Gamma_{X^{\bp}_{1}})=\ell\#e^{\rm cl}(\Gamma_{X^{\bp}_{2}})=\#e^{\rm cl}(\Gamma_{Y^{\bp}_{2}})$. On the other hand, Lemma \ref{rem-them-1-1-4} and Lemma \ref{lem-0} (b) imply that $$r_{Y_{1}} \leq r_{Y_{2}}.$$ Thus, we obtain $$\ell\#e^{\rm cl}(\Gamma_{X^{\bp}_{1}})-\ell(\#v(\Gamma_{X^{\bp}_{1}})-\#v^{\rm ra}_{f_{1}})-\#v^{\rm ra}_{f_{1}}+1\leq \ell\#e^{\rm cl}(\Gamma_{X^{\bp}_{2}})-\ell(\#v(\Gamma_{X^{\bp}_{2}})-1)-1+1 $$ This implies that $\#v_{f_{1}}^{\rm ra}\leq1$.

Suppose that $\#v_{f_{1}}^{\rm ra}=0$. Let $\alpha_{f_{1}}\in M_{X^{\bp}_{1}}$ be an element corresponding to $f^{\bp}_{1}$. Then $\alpha_{f_{1}} \in M_{X^{\bp}_{1}}^{\text{top}}$. Note that $\alpha_{f_{2}}\defeq 
(\psi_{\ell}^{\et})^{-1}(\alpha_{f_{1}})\in M_{X^{\bp}_{2}}^{\et}$ is the element corresponding to $f_{2}^{\bp}$. Then Proposition \ref{prop-4-11} implies that $\alpha_{f_{2}}$ is contained in $M_{X^{\bp}_{2}}^{\text{top}}$. This means that $\#v_{f_{2}}^{\rm ra}=0$. This contradicts the assumption $\#v_{f_{2}}^{\rm ra}=1$. Thus, we have $\#v_{f_{1}}^{\rm ra}=1$. We complete the proof of the lemma.
\end{proof}

We reconstruct the sets of vertices and the sets of genus of irreducible components group-theoretically from $\phi$ as follows.

\begin{theorem}\label{them-4-13}
We maintain the notation introduced above. Then the (surjective) open continuous homomorphism $\phi: \Pi_{X_{1}^{\bp}} \migisurj \Pi_{X_{2}^{\bp}}$ induces a bijection of the set of vertices $$\phi^{\rm sg, ver}: v(\Gamma_{X^{\bp}_{1}}) \isom v(\Gamma_{X^{\bp}_{2}})$$ group-theoretically. Moreover, let $v_{1}\in v(\Gamma_{X_{1}^{\bp}})$ and $v_{2}\defeq \phi^{\rm sg, vex}(v_{1})$. Then we have $$g_{1, v_{1}}=g_{2, v_{2}}.$$
\end{theorem}

\begin{proof}
We maintain the notation introduced in Section \ref{sec-4-1}. By applying Theorem \ref{types}, Proposition \ref{prop-4-10}, and Proposition \ref{prop-4-11}, we obtain that the following homomorphisms of the natural exact sequences can be induced by $\phi$ group-theoretically:
\[
\begin{CD}
0@>>>M_{X_{2}^{\bp}}^{\rm top}@>>>M^{\et}_{X_{2}^{\bp}}@>>>M_{X_{2}^{\bp}}^{\rm nt}@>>>0
\\
@.@V\psi^{\rm top}_{\ell}VV@V\psi^{\et}_{\ell}VV@VVV@.
\\
0@>>>M_{X_{1}^{\bp}}^{\rm top}@>>>M^{\et}_{X_{1}^{\bp}}@>>>M_{X_{1}^{\bp}}^{\rm nt}@>>>0.
\end{CD}
\]
Then we obtain $$\psi_{\ell}^{\text{\rm \'et}}(V^{*}_{X_{2}, \ell})=V^{*}_{X_{1}, \ell}.$$ Moreover, Lemma \ref{lem-4-12} implies that $$\psi_{\ell}^{\text{\rm \'et}}(V^{\star}_{X_{2}, \ell})=V^{\star}_{X_{1}, \ell}.$$ Let $\alpha_{2}$, $\alpha'_{2} \in V^{\star}_{X_{2}, \ell}$ distinct from each other such that $\alpha_{2}\sim \alpha'_{2}$. By applying Lemma \ref{lem-4-12} again, for any $a$, $b \in \mbF^{\times}_{\ell}$, we see that $a\alpha_{2}+b\alpha'_{2} \in V^{\star}_{X_{2}, \ell}$ if and only $\psi_{\ell}^{\et}(a\alpha_{2}+b\alpha'_{2} )=a\psi_{\ell}^{\text{\'et}}(\alpha_{2})+b\psi_{\ell}^{\text{\'et}}(\alpha'_{2}) \in V^{\star}_{X_{1}, \ell}$. Thus, we obtain a bijection  $$V_{X_{2}, \ell} \isom V_{X_{1}, \ell}.$$ Then the first part of the theorem follows from Proposition \ref{pro-2-1}. 

Next, let us prove the ``moreover" part of the theorem. Let $v_{i} \in v(\Gamma_{X_{i}^{\bp}})$. We put $$L_{X^{\bp}_{i}}^{v_{i}} \defeq \{\alpha_{i} \in M_{X^{\bp}_{i}}^{\text{\'et}} \ | \ v_{f_{i, \alpha_{i}}}^{\rm ra}=\{v_{i}\}\},$$ where $f_{i, \alpha_{i}}^{\bp}$ denotes the Galois admissible covering of $X_{i}^{\bp}$ over $k_{i}$ corresponding to $\alpha_{i}$. Moreover, we denote by $$[L_{X^{\bp}_{i}}^{v_{i}}]$$ the image of $L_{X^{\bp}_{i}}^{v_{i}}$ in $M_{X^{\bp}_{i}}^{\rm nt}$. Then we have $$\#[L_{X^{\bp}_{i}}^{v_{i}}]=\ell^{g_{i, v_{i}}}-1.$$

Suppose that $v_{2}=\phi^{\rm sg, ver}(v_{1})$. Proposition \ref{prop-4-11} and Lemma \ref{lem-4-12} imply that $\psi^{\text{\'et}}_{\ell}$ induces an injection $$[L_{X^{\bp}_{2}}^{v_{2}}] \migiinje [L_{X^{\bp}_{1}}^{v_{1}}].$$ Thus, we have $$\ell^{g_{2, v_{2}}}-1=\#[L_{X^{\bp}_{2}}^{v_{2}}] \leq \#[L_{X^{\bp}_{1}}^{v_{1}}]=\ell^{g_{1, v_{1}}}-1.$$ This means that $$g_{2, v_{2}} \leq g_{1, v_{1}}.$$  On the other hand, since $$\sum_{v_{1} \in v(\Gamma_{X_{1}^{\bp}})}g_{1, v_{1}}=g_{X}-r_{X_{1}}=g_{X}-r_{X_{2}}=\sum_{v_{2} \in v(\Gamma_{X_{2}^{\bp}})}g_{2, v_{2}},$$ we obtain $$g_{1, v_{1}}=g_{2, v_{2}}.$$ This completes the proof of the theorem.
\end{proof}

Next, let us reconstruct the sets of closed edges from $\phi$. In the remainder of the present subsection, we fix an edge-triple $$\mfT_{\Pi_{X_{1}^{\bp}}} \defeq (\ell, d, \alpha_{f_{X_{1}}}: \Pi_{X^{\bp}_{1}}^{\et} \migisurj \mbF_{d})$$ associated to $\Pi_{X_{1}^{\bp}}$. Then Corollary \ref{coro-4-8} implies that $\phi$ and the edge-triple $\mfT_{\Pi_{X_{1}^{\bp}}}$ induces an edge-triple  $$\mfT_{\Pi_{X_{2}^{\bp}}} \defeq (\ell, d, \alpha_{f_{X_{2}}}: \Pi_{X^{\bp}_{2}}^{\et} \migisurj \mbF_{d})$$ associated to $\Pi_{X_{2}^{\bp}}$ group-theoretically. Write $\Pi_{Y_{i}^{\bp}}$ for the kernel of $\alpha_{f_{X_{i}}}$. The surjection $\phi: \Pi_{X^{\bp}_{1}} \migisurj \Pi_{X^{\bp}_{2}}$ induces a sujection $$\phi_{Y}: \Pi_{Y^{\bp}_{1}} \migisurj \Pi_{Y^{\bp}_{2}}.$$  Moreover, the constructions of $Y^{\bp}_{1}$ and $Y_{2}^{\bp}$ imply that $Y^{\bp}_{1}$ and $Y^{\bp}_{2}$ satisfy Condition A, Condition B, and Condition C.

We put $$M_{Y_{i}^{\bp}} \defeq \text{Hom}(\Pi_{Y_{i}^{\bp}}, \mbF_{\ell}),$$ $$M_{Y_{i}^{\bp}}^{\text{\'et}}\defeq \text{Hom}(\Pi^{\et}_{Y_{i}^{\bp}}, \mbF_{\ell}),$$ $$M_{Y_{i}^{\bp}}^{\rm ra}\defeq M_{Y_{i}^{\bp}}/M_{Y_{i}^{\bp}}^{\et}.$$ Then, by Theorem \ref{types} and Proposition \ref{prop-4-10}, the following commutative diagram can be induced by $\phi_{Y}$ group-theoretically:
\[
\begin{CD}
0@>>>M_{Y_{2}^{\bp}}^{\text{\'et}}@>>>M_{Y_{2}^{\bp}}@>>>M_{Y_{2}^{\bp}}^{\rm ra}@>>>0
\\
@.
@V\psi_{Y, \ell}^{\et}VV
@V\psi_{Y, \ell}VV
@VVV
@.
\\
0@>>>M_{Y_{1}^{\bp}}^{\text{\'et}}@>>>M_{Y_{1}^{\bp}}@>>>M_{Y_{1}^{\bp}}^{\rm ra}@>>>0,
\end{CD}
\]
where all the vertical arrows are isomorphisms. Let $E^{*}_{\mfT_{\Pi_{X_{i}^{\bp}}}}$ be the subset of $M_{Y_{i}^{\bp}}$ defined in Section \ref{sec-4-2}. Since the actions of $\mu_{d}$ on the exact sequences are compatible with the isomorphisms appearing in the commutative diagram above, we have $$\psi_{Y,\ell}(E^{*}_{\mfT_{\Pi_{X_{2}^{\bp}}}})=E^{*}_{\mfT_{\Pi_{X_{1}^{\bp}}}}.$$ Let $m \in \mbZ_{\geq 0}$ and $e_{i} \in e^{\rm cl}(\Gamma_{X^{\bp}_{i}})$. Recall that $E^{{\rm cl, \star}, m}_{\mfT_{\Pi_{X_{i}^{\bp}}}, e_{i}}$ is the subset of $E^{{\rm cl, \star}}_{\mfT_{\Pi_{X_{i}^{\bp}}}, e_{i}}$ whose element $\alpha_{i}$ satisfies $\#v_{g_{i, \alpha_{i}}}^{\rm sp}=m$.  Then we have the following lemma.

\begin{lemma}\label{lem-4-14}
We maintain the notation introduced above. Then we have $$\psi^{-1}_{Y,\ell}(\bigsqcup_{e_{1} \in e^{\rm op}(\Gamma_{X^{\bp}_{1}})}E^{{\rm cl, \star}, 0}_{\mfT_{\Pi_{X_{1}^{\bp}}}, e_{1}}) \subseteq \bigsqcup_{e_{2} \in e^{\rm op}(\Gamma_{X^{\bp}_{2}})}E^{{\rm cl, \star}, 0}_{\mfT_{\Pi_{X_{2}^{\bp}}}, e_{2}}.$$ Moreover, we have $$\psi^{-1}_{Y,\ell}(E^{{\rm cl, \star}}_{\mfT_{\Pi_{X_{1}^{\bp}}}})=E^{{\rm cl, \star}}_{\mfT_{\Pi_{X_{2}^{\bp}}}}.$$
\end{lemma}

\begin{proof}
Let $e_{1} \in e^{\rm cl}(\Gamma_{X^{\bp}_{1}})$ and $\alpha_{1} \in E^{{\rm cl, \star}, 0}_{\mfT_{\Pi_{X_{1}^{\bp}}}, e_{1}}$. Then the Galois admissible covering $$g_{1, \alpha_{1}}^{\bp}: Y^{\bp}_{1, \alpha} \migi Y^{\bp}_{1}$$ over $k_{1}$ with Galois group $\mbZ/\ell\mbZ$ corresponding to $\alpha_{1}$ induces a Galois admissible covering $$g^{\bp}_{2, \alpha_{2}}: Y^{\bp}_{2, \alpha_{2}} \migi Y^{\bp}_{2}$$ over $k_{2}$ with Galois group $\mbZ/\ell\mbZ$. Write $\alpha_{2} \in M_{Y_{2}^{\bp}}$ for the element corresponding to $g^{\bp}_{2, \alpha_{2}}$. We have $$\alpha_{2} \in E^{*}_{\mfT_{\Pi_{Y^{\bp}_{2}}}}.$$ Write $g_{Y_{i, \alpha_{i}}}$ for the genus of $Y^{\bp}_{i, \alpha_{i}}$, and $r_{Y_{i, \alpha_{i}}}$ for the Betti number of the dual semi-graph $\Gamma_{Y^{\bp}_{i, \alpha_{i}}}$. Then the Riemann-Hurwitz formula and Theorem \ref{mainstep-1} imply that $$g_{Y_{1, \alpha_{1}}}-g_{Y_{2, \alpha_{2}}}=-\frac{1}{2}(\#e^{\rm op, ra}_{g_{2, \alpha_{2}}})(\ell-1)=0.$$ On the other hand, we have $$r_{Y_{1, \alpha_{1}}}=\ell(\#e^{\rm cl}(\Gamma_{Y^{\bp}_{1}})-d)+d-\#v(\Gamma_{Y^{\bp}_{1}})+1$$ and $$r_{Y_{2, \alpha_{2}}}=\ell\#e^{\rm cl, \text{\'et}}_{g_{2, \alpha_{2}}}+\#e^{\rm cl, ra}_{g_{2, \alpha_{2}}}-\ell\#v^{\rm cl, sp}_{g_{2, \alpha_{2}}}-\#v^{\rm cl, ra}_{g_{2, \alpha_{2}}}+1.$$ Then Lemma \ref{rem-them-1-1-4} and Lemma \ref{lem-0} (b) imply that $$0=g_{Y_{1, \alpha_{1}}}-g_{Y_{2, \alpha_{2}}}\geq r_{Y_{1, \alpha_{1}}}-r_{Y_{2, \alpha_{1}}}.$$ Thus, we have $$\#e^{\rm cl, ra}_{g_{2, \alpha_{2}}}+\#v^{\rm sp}_{g_{2, \alpha_{2}}}+\frac{1}{2}\#e^{\rm op, ra}_{g_{2, \alpha_{2}}} =\#e^{\rm cl, ra}_{g_{2, \alpha_{2}}}+\#v^{\rm sp}_{g_{2, \alpha_{2}}}\leq d.$$

If $\#e^{\rm cl, ra}_{g_{2, \alpha_{2}}}=0$, then $g_{2, \alpha_{2}}$ is \'etale. By replacing $X^{\bp}_{1}$ and $X^{\bp}_{2}$ by $Y^{\bp}_{1}$ and $Y^{\bp}_{2}$, respectively, Proposition \ref{prop-4-10} implies that $g_{1, \alpha_{1}}$ is also \'etale. This contradicts the definition of $\alpha_{1}$. Thus,  we obtain  $\#e^{\rm cl, ra}_{g_{2, \alpha_{2}}}\neq 0$. 


If $\#e^{\rm cl, ra}_{g_{2, \alpha_{2}}} \neq 0$, then we have $\#e^{\rm cl, ra}_{g_{2, \alpha_{2}}}=d$ and $\#v^{\rm sp}_{g_{2, \alpha_{2}}}=\#e^{\rm op, ra}_{g_{2, \alpha_{2}}}=0$. This means that $$\alpha_{2} \in \bigsqcup_{e_{2} \in e^{\rm cl}(\Gamma_{Y^{\bp}_{2}})}E^{{\rm cl, \star}, 0}_{\mfT_{\Pi_{Y_{2}^{\bp}}}, e_{2}}.$$ Thus, we have $$\psi^{-1}_{Y,\ell}(\bigsqcup_{e_{1} \in e^{\rm cl}(\Gamma_{Y^{\bp}_{1}})}E^{{\rm cl, \star}, 0}_{\mfT_{\Pi_{Y_{1}^{\bp}}}, e_{1}}) \subseteq \bigsqcup_{e_{2} \in e^{\rm cl}(\Gamma_{Y^{\bp}_{2}})}E^{{\rm cl, \star}, 0}_{\mfT_{\Pi_{Y_{2}^{\bp}}}, e_{2}}.$$ 

Moreover, let $\beta_{i} \in E^{{\rm cl, \star}}_{\mfT_{\Pi_{Y_{i}^{\bp}}}}$. Then $\beta_{i}$ is a linear combination of the elements of $$\bigsqcup_{e_{i} \in e^{\rm cl}(\Gamma_{Y^{\bp}_{i}})}E^{{\rm cl, \star}, 0}_{\mfT_{\Pi_{Y_{i}^{\bp}}}, e_{i}}.$$ Then we have $$\psi^{-1}_{Y,\ell}(E^{{\rm cl, \star}}_{\mfT_{\Pi_{X_{1}^{\bp}}}})\subseteq E^{{\rm cl, \star}}_{\mfT_{\Pi_{X_{2}^{\bp}}}}.$$ On the other hand, since $g_{Y_{1}}=g_{Y_{2}}$ and $r_{Y_{1}}=r_{Y_{2}}$, Lemma \ref{lem-2-1} implies that $\#\psi^{-1}_{Y,\ell}(E^{{\rm cl, \star}}_{\mfT_{\Pi_{X_{1}^{\bp}}}})=\#E^{{\rm cl, \star}}_{\mfT_{\Pi_{X_{2}^{\bp}}}}$. Thus, we obtain $$\psi^{-1}_{Y,\ell}(E^{{\rm cl, \star}}_{\mfT_{\Pi_{X_{1}^{\bp}}}})= E^{{\rm cl, \star}}_{\mfT_{\Pi_{X_{2}^{\bp}}}}.$$ This completes the proof of the lemma.
\end{proof}

We reconstruct the sets of closed edges group-theoretically from $\phi$  as follows.

\begin{theorem}\label{them-4-15}
We maintain the notation introduced above. Then the (surjective) open continuous homomorphism $\phi: \Pi_{X_{1}^{\bp}} \migisurj \Pi_{X_{2}^{\bp}}$ induces a bijection of the set of closed edges $$\phi^{\rm sg, cl}: e^{\rm cl}(\Gamma_{X^{\bp}_{1}}) \isom e^{\rm cl}(\Gamma_{X^{\bp}_{2}})$$ group-theoretically. 
\end{theorem}

\begin{proof}
Let $\alpha_{2}$, $\alpha'_{2} \in E^{\rm cl,\star}_{\mfT_{\Pi_{X_{2}^{\bp}}}}$ and $\alpha_{2}\defeq \psi_{Y, \ell}(\alpha_{1}')$, $\alpha'_{1} \defeq \psi_{Y, \ell}(\alpha_{2}') \in E^{\rm cl,\star}_{\mfT_{\Pi_{X_{1}^{\bp}}}}$. We see immediately that $\alpha_{1} \sim \alpha_{1}'$ if and only if $\alpha_{2} \sim \alpha_{2}'$. Then the theorem follows from Lemma \ref{lem-4-14} and Proposition \ref{pro-2-2}.
\end{proof}

Next, let us reconstruct the sets of $p$-rank from $\phi$. Note that the surjection $\phi$ induces a surjection of the maximal pro-$p$ quotients $$\phi^{p}: \Pi_{X^{\bp}_{1}}^{p} \migisurj \Pi_{X^{\bp}_{2}}^{p}$$ of solvable admissible fundamental groups. Then every Galois (\'etale) admissible covering $h_{2}^{\bp}: Z_{2}^{\bp} \migi X_{2}^{\bp}$ over $k_{2}$ with Galois group $\mbZ/p\mbZ$ induces a Galois (\'etale) admissible covering $h_{1}^{\bp}: Z_{1}^{\bp} \migi X_{1}^{\bp}$ over $k_{1}$ with Galois group $\mbZ/p\mbZ$.  Moreover, $\phi^{p}$ induces an injection $$\psi_{p}: N_{X_{2}^{\bp}} \defeq \text{Hom}(\Pi_{X_{2}^{\bp}}, \mbF_{p}) \migiinje N_{X_{1}^{\bp}} \defeq \text{Hom}(\Pi_{X_{1}^{\bp}}, \mbF_{p}).$$  We have the following lemmas.

\begin{lemma}\label{lem-4-16}
We maintain the notation introduced above. Suppose that $\#v_{h_{2}}^{\rm ra}=0$. Then we have that $h^{\bp}_{1}$ is an \'etale covering, and that $\#v_{h_{1}}^{\rm ra}=0$. In particular, we obtain that $$\psi_{p}^{\rm top}: N^{\rm top}_{X_{2}^{\bp}} \defeq {\rm Hom}(\Pi_{X_{2}^{\bp}}^{\rm top}, \mbF_{p}) \isom N_{X_{1}^{\bp}}^{\rm top} \defeq {\rm Hom}(\Pi_{X_{1}^{\bp}}^{\rm top}, \mbF_{p})$$ is an isomorphism.
\end{lemma}

\begin{proof}
Since $h_{i}^{\bp}$ is \'etale, the Riemann-Hurwitz formula implies that $$g_{Z_{1}}=g_{Z_{2}}.$$ Thus, similar arguments to the arguments given in the proofs of Proposition \ref{prop-4-11} imply that $$\#v_{h_{1}}^{\rm ra}=0.$$ This completes the proof of the lemma.
\end{proof}

\begin{lemma}\label{lem-4-17}
We maintain the notation introduced above. Suppose that $\#v_{h_{2}}^{\rm ra}=1$. Then we obtain that $h^{\bp}_{1}$ is \'etale, and that $\#v_{h_{1}}^{\rm ra}=1$.
\end{lemma}

\begin{proof}
Similar arguments to the arguments given in the proofs of Lemma \ref{lem-4-12} imply that $\#v_{h_{1}}^{\rm ra}\leq 1.$ If $\#v_{h_{1}}^{\rm ra}=0$, then the ``in particular" part of Lemma \ref{lem-4-16} implies that $\#v_{h_{2}}^{\rm ra}=0$. This contradicts our assumption. Then we obtain that $\#v_{h_{1}}^{\rm ra}= 1.$
\end{proof}

We reconstruct the sets of $p$-rank of smooth pointed stable curves associated to vertices from $\phi$ as follows.

\begin{theorem}\label{them-4-18}
We maintain the notation introduced above. Then the (surjective) open continuous homomorphism $\phi: \Pi_{X_{1}^{\bp}} \migisurj \Pi_{X_{2}^{\bp}}$ induces an injection of the set of vertices $$\psi_{p}^{{\rm sg, ver}}: v(\Gamma_{X^{\bp}_{2}})^{>0, p} \migiinje v(\Gamma_{X^{\bp}_{1}})^{>0, p}$$ group-theoretically. Moreover, let $v_{2}\in v(\Gamma_{X_{2}^{\bp}})^{>0, p}$ and $v_{1}\defeq \psi_{p}^{\rm sg, vex}(v_{2})$. Then we have $$\sigma_{2, v_{2}}\leq \sigma_{1, v_{1}}.$$ 
\end{theorem}

\begin{proof}
Lemma \ref{lem-4-17} implies that $$\psi_{p}(V^{\star}_{X_{2}, p}) \subseteq V^{\star}_{X_{1}, p}.$$ Let $\alpha_{2}, \alpha'_{2} \in V^{\star}_{X_{2}, p}$ be elements distinct from each other such that $\alpha_{2} \sim \alpha'_{2}$. It is easy to see that $a\alpha_{2}+b\alpha'_{2} \in V^{\star}_{X_{2}, p}$ if and only if $a\psi_{p}(\alpha_{2})+b\psi_{p}(\alpha'_{2}) \in V^{\star}_{X_{1}, p}$ for each $a, b \in \mbF^{\times}_{p}$. Thus, by Proposition \ref{pro-2-1}, we obtain an injection of the set of vertices $$\psi^{{\rm sg, ver}}_{p}: v(\Gamma_{X^{\bp}_{2}})^{>0, p} \migiinje v(\Gamma_{X^{\bp}_{1}})^{>0, p}.$$ 

Let $v_{i} \in v(\Gamma_{X_{i}^{\bp}})$. We put  $$L_{X^{\bp}_{i}}^{v_{i}, p}\defeq \{\alpha_{i} \in N_{X_{i}^{\bp}} \ | \ v^{\rm ra}_{h_{i, \alpha_{i}}}=\{v_{i}\}\},$$ where $h_{i, \alpha_{i}}^{\bp}$ denotes the Galois (\'etale) admissible covering corresponding to $\alpha_{i}$. Moreover, we denote by $$[L_{X^{\bp}_{i}}^{v_{i}, p}]$$ the image of $L_{X^{\bp}_{i}}^{v_{i}, p}$ in $N_{X_{i}^{\bp}}/N_{X_{i}^{\bp}}^{\rm top}$, where $N^{\rm top}_{X_{i}^{\bp}} \defeq \text{Hom}(\Pi_{X^{\bp}_{i}}^{\rm top}, \mbF_{p})$. Then we have $$\#[L_{X^{\bp}_{i}}^{v_{i}, p}]=p^{\sigma_{i, v_{i}}}-1.$$

Suppose that $v_{1} \defeq \psi^{\rm sg, ver}_{p}(v_{2})$. Lemma \ref{lem-4-16} implies that $\psi_{p}$ induces an injection $$[L_{X^{\bp}_{2}}^{v_{2}, p}] \migiinje [L_{X^{\bp}_{1}}^{v_{1}, p}].$$ Thus, we have $$p^{\sigma_{2, v_{2}}}-1=\#[L_{X^{\bp}_{2}}^{v_{2}, p}] \leq \#[L_{X^{\bp}_{1}}^{v_{1}, p}]=p^{\sigma_{1, v_{1}}}-1.$$ This means that $$\sigma_{2, v_{2}} \leq \sigma_{1, v_{1}}$$ for each $v_{2} \in v(\Gamma_{X_{2}^{\bp}})^{>0, p}$. This completes the proof of the theorem.
\end{proof} 

In the remainder of the present subsection, we prove a proposition which will be used in Section \ref{sec-4-5}.

\begin{proposition}\label{prop-op-cl}
We maintain the notation introduced above. Then the following statements hold:

(a) Let $S_{1}^{\rm cl} \subseteq e^{\rm cl}(\Gamma_{X_{1}^{\bp}})$ be a subset of closed edges, $\alpha_{e_{1}} \in E^{\rm cl, \star, 0}_{\mfT_{\Pi_{X^{\bp}_{1}}}, e_{1}}$ for every $e_{1} \in S^{\rm cl}_{1}$,  $$\alpha_{1} \defeq \sum_{e_{1} \in S^{\rm cl}_{1}}\alpha_{e_{1}} \in E^{*}_{\mfT_{\Pi_{X^{\bp}_{1}}}},$$ and $g_{1, \alpha_{1}}^{\bp}: Y_{1, \alpha_{1}}^{\bp} \migi Y_{1}^{\bp}$ the Galois admissible covering over $k_{1}$ with Galois group $\mbZ/\ell\mbZ$ corresponding to $\alpha_{1}$. Let $\phi^{\rm sg, cl}: e^{\rm cl}(\Gamma_{X_{1}^{\bp}}) \isom  e^{\rm cl}(\Gamma_{X_{2}^{\bp}})$ be the bijection of the sets of closed edges obtained in Theorem \ref{them-4-15}, $\alpha_{\phi^{\rm sg, cl}(e_{1})} \in E^{\rm cl, \star, 0}_{\mfT_{\Pi_{X^{\bp}_{2}}}, \phi^{\rm sg, cl}(e_{1})}$ the element induced by $\phi$ for every $e_{1} \in S^{\rm cl}_{1}$, $$\alpha_{2} \defeq \sum_{e_{1} \in S^{\rm cl}_{1}}\alpha_{\phi^{\rm sg, cl}(e_{1})} \in E^{*}_{\mfT_{\Pi_{X^{\bp}_{2}}}},$$ and $g_{2, \alpha_{2}}^{\bp}: Y_{2, \alpha_{2}}^{\bp} \migi Y_{2}^{\bp}$ the Galois admissible covering over $k_{2}$ with Galois group $\mbZ/\ell\mbZ$ corresponding to $\alpha_{2}$. Suppose that $\#v^{\rm sp}_{g_{1, \alpha_{1}}}=0$. Then we have that $$\#e^{\rm op, ra}_{g_{2, \alpha_{2}}}=\#v^{\rm sp}_{g_{2, \alpha_{2}}}=0.$$

(b) Let $E^{\rm op, \star, 0}_{\mfT_{\Pi_{X^{\bp}_{i}}}, e_{i}}$, $e_{i} \in e^{\rm op}(\Gamma_{X_{i}^{\bp}})$, be the set of cohomology classes defined in Section \ref{sec-4-2}, and let $S_{1}^{\rm op} \subseteq e^{\rm op}(\Gamma_{X_{1}^{\bp}})$ be a subset of open edges, $\alpha_{e_{1}} \in E^{\rm op, \star, 0}_{\mfT_{\Pi_{X^{\bp}_{1}}}, e_{1}}$ for every $e_{1} \in S^{\rm op}_{1}$,  $$\alpha_{1} \defeq \sum_{e_{1} \in S^{\rm op}_{1}}\alpha_{e_{1}} \in E^{*}_{\mfT_{\Pi_{X^{\bp}_{1}}}},$$ and $g_{1, \alpha_{1}}^{\bp}: Y_{1, \alpha_{1}}^{\bp} \migi Y_{1}^{\bp}$ the Galois admissible covering over $k_{1}$ with Galois group $\mbZ/\ell\mbZ$ corresponding to $\alpha_{1}$. Let $\phi^{\rm sg, op}: e^{\rm op}(\Gamma_{X_{1}^{\bp}}) \isom  e^{\rm op}(\Gamma_{X_{2}^{\bp}})$ be the bijection of the sets of open edges obtained in Theorem \ref{mainstep-1}, $\alpha_{\phi^{\rm sg, op}(e_{1})}\in E^{\rm cl, \star, 0}_{\mfT_{\Pi_{X^{\bp}_{2}}}, \phi^{\rm sg, op}(e_{1})}$ the element induced by $\phi$ for every $e_{1} \in S^{\rm op}_{1}$, $$\alpha_{2} \defeq \sum_{e_{1} \in S^{\rm op}_{1}}\alpha_{\phi^{\rm sg, op}(e_{1})} \in E^{*}_{\mfT_{\Pi_{X^{\bp}_{2}}}},$$ and $g_{2, \alpha_{2}}^{\bp}: Y_{2, \alpha_{2}}^{\bp} \migi Y_{2}^{\bp}$ the Galois admissible covering over $k_{2}$ with Galois group $\mbZ/\ell\mbZ$ corresponding to $\alpha_{2}$. Suppose that $\#v^{\rm sp}_{g_{1, \alpha_{1}}}=0$. Then we have that $$\#e^{\rm cl, ra}_{g_{2, \alpha_{2}}}=\#v^{\rm sp}_{g_{2, \alpha_{2}}}=0.$$
\end{proposition}

\begin{proof}
(a)  Since $\#e^{\rm op, ra}_{g_{1, \alpha_{1}}}=0$, Theorem \ref{mainstep-1} imply that $\#e^{\rm op, ra}_{g_{2, \alpha_{2}}}=0$. On the other hand, we have $$r_{Y_{1, \alpha_{1}}}=\ell(\#e^{\rm cl}(\Gamma_{Y^{\bp}_{1}})-d\#S_{1}^{\rm cl})+d\#S_{1}^{\rm cl}-\#v(\Gamma_{Y^{\bp}_{1}})+1$$ and $$r_{Y_{2, \alpha_{2}}}=\ell\#e^{\rm cl, \text{\'et}}_{g_{2, \alpha_{2}}}+\#e^{\rm cl, ra}_{g_{2, \alpha_{2}}}-\ell\#v^{\rm cl, sp}_{g_{2, \alpha_{2}}}-\#v^{\rm cl, ra}_{g_{2, \alpha_{2}}}+1.$$ Then Lemma \ref{rem-them-1-1-4} and Lemma \ref{lem-0} (b) imply that $$0=g_{Y_{1, \alpha_{1}}}-g_{Y_{2, \alpha_{2}}}\geq r_{Y_{1, \alpha_{1}}}-r_{Y_{2, \alpha_{1}}}.$$ Thus, we have $$\#e^{\rm cl, ra}_{g_{2, \alpha_{2}}}+\#v^{\rm sp}_{g_{2, \alpha_{2}}}+\frac{1}{2}\#e^{\rm op, ra}_{g_{2, \alpha_{2}}} =\#e^{\rm cl, ra}_{g_{2, \alpha_{2}}}+\#v^{\rm sp}_{g_{2, \alpha_{2}}}\leq d\#S_{1}^{\rm cl}.$$

On the othe hand, Lemma \ref{lem-4-14} implies that $\#e^{\rm cl, ra}_{g_{2, \alpha_{2}}}=d\#S_{1}^{\rm cl}$. Then we obtain $\#v^{\rm sp}_{g_{2, \alpha_{2}}}=0$. This completes the proof of (a).

(b) The Riemann-Hurwitz formula and Theorem \ref{mainstep-1} imply that $$g_{Y_{1, \alpha_{1}}}-g_{Y_{2, \alpha_{2}}}=\frac{1}{2}(d\#S^{\rm op}_{1}-\#e^{\rm op, ra}_{g_{2, \alpha_{2}}})(\ell-1)=0.$$ On the other hand, we have $$r_{Y_{1, \alpha_{1}}}=\ell\#e^{\rm cl}(\Gamma_{Y^{\bp}_{1}})-\#v(\Gamma_{Y^{\bp}_{1}})+1$$ and $$r_{Y_{2, \alpha_{2}}}=\ell\#e^{\rm cl, \text{\'et}}_{g_{2, \alpha_{2}}}+\#e^{\rm cl, ra}_{g_{2, \alpha_{2}}}-\ell\#v^{\rm sp}_{g_{2, \alpha_{2}}}-\#v^{\rm ra}_{g_{2, \alpha_{2}}}+1.$$ Then Lemma \ref{rem-them-1-1-4} and Lemma \ref{lem-0} (b) imply that  $$g_{Y_{1, \alpha_{1}}}-g_{Y_{2, \alpha_{2}}}\geq r_{Y_{1, \alpha_{1}}}-r_{Y_{2, \alpha_{2}}}.$$ Thus, we have $$\#e^{\rm cl, ra}_{g_{2, \alpha_{2}}}+\#v^{\rm sp}_{g_{2, \alpha_{2}}}+\frac{1}{2}\#e^{\rm op, ra}_{g_{2, \alpha_{2}}} - \frac{d\#S_{1}^{\rm op}}{2}\leq 0.$$ This means that $$\#e^{\rm cl, ra}_{g_{2, \alpha_{2}}}=\#v^{\rm sp}_{g_{2, \alpha_{2}}}=0.$$ We complete the proof of (b).
\end{proof}

\subsection{Reconstruction of commutative diagrams of sets of vertices, sets of open edges, and sets of closed edges from surjections}\label{sec-4-4}

We maintain the notation introduced in Section \ref{sec-4-3}. In the present subsection, we suppose that $X_{1}^{\bp}$ and $X_{2}^{\bp}$ {\it satisfy Condition A, Condition B, and Condition C}. Moreover, let $$\phi: \Pi_{X^{\bp}_{1}}\migi \Pi_{X^{\bp}_{2}}$$ be an arbitrary open continuous homomorphism of the solvable admissible fundamental groups of $X^{\bp}_{1}$ and $X^{\bp}_{2}$, and $$(g_{X}, n_{X})\defeq (g_{X_{1}}, n_{X_{1}})=(g_{X_{2}}, n_{X_{2}}).$$ Note that we have $r_{X_{1}}=r_{X_{2}}$, and that by Lemma \ref{surj}, $\phi$ is a surjection.

We fix some notation. Let $H_{2}$ be an open normal subgroup of $\Pi_{X_{2}}^{\bp}$, $H_{1}\defeq \phi^{-1}(H_{2})$ the open normal subgroup of $\Pi_{X_{1}^{\bp}}$, $G\defeq \Pi_{X_{1}^{\bp}}/H_{1}=\Pi_{X_{2}^{\bp}}/H_{2}$, and $\phi_{H_{1}}$ the surjection $\phi|_{H_{1}}: H_{1} \migisurj H_{2}$. Let $i \in \{1, 2\}$. We write $$f^{\bp}_{H_{i}}: X^{\bp}_{H_{i}} \migi X_{i}^{\bp}$$ for the Galois admissible covering over $k_{i}$ with Galois group $G$, $(g_{X_{H_i}}, n_{X_{H_{i}}})$ for the type of $X^{\bp}_{H_{i}}$, and $\Gamma_{X^{\bp}_{H_{i}}}$ for the dual semi-graph of $X^{\bp}_{H_{i}}$. Furthermore, {\it we suppose that $X_{H_1}^{\bp}$ and $X_{H_2}^{\bp}$ satisfy Condition A, Condition B, and Condition C.}

Let $\ell$ and $d$ be prime numbers distinct from $p$ such that $\ell \neq d$ and  $(\#G, \ell)=(\#G, d)=1$, and let $$\mfT_{\Pi_{X^{\bp}_{2}}} \defeq (\ell, d, \alpha_{f_{X_{2}}}: \Pi_{X_{2}^{\bp}}^{\et} \migisurj \mbF_{d})$$ be an edge-triple associated to $\Pi_{X^{\bp}_{2}}$ and $\mfT_{X^{\bp}_{2}} \defeq (\ell, d, f^{\bp}_{X_2}: Y_{2}^{\bp} \migi  X^{\bp}_{2})$ the edge-triple associated to $X_{2}^{\bp}$ corresponding to $\mfT_{\Pi_{X^{\bp}_{2}}}$. By Corollary \ref{coro-4-8}, we obtain an edge-triple $$\mfT_{\Pi_{X^{\bp}_{1}}} \defeq (\ell, d, \alpha_{f_{X_{1}}}: \Pi_{X_{1}^{\bp}}^{\et} \migisurj \mbF_{d})$$ induced group-theoretically from $\phi$ and $\mfT_{\Pi_{X^{\bp}_{2}}}$. We write $\mfT_{X^{\bp}_{1}} \defeq (\ell, d, f^{\bp}_{X_1}: Y_{1}^{\bp} \migi  X^{\bp}_{1})$ for the edge-triple associated to $X_{1}^{\bp}$ corresponding to $\mfT_{\Pi_{X^{\bp}_{1}}}$. On the other hand, we put $$Q_{i} \defeq \text{ker}(\Pi_{X_{i}^{\bp}} \migisurj \Pi_{X_{i}^{\bp}}^{\et}\overset{\alpha_{f_{X_{i}}}}\migisurj \mbF_{d}).$$ We have that $H_{i} \migisurj H_{i}/(H_{i}\cap Q_{i})\cong \mbF_{d}$ factors through $H_{i}^{\et}$. Write $\alpha_{f_{X_{H_{i}}}}: H_{i}^{\et} \migisurj \mbF_{d}$ for this homomorphism. We see that $$\mfT_{H_{i}} \defeq (\ell, d, \alpha_{f_{X_{H_{i}}}})$$ is an edge-triple associated to $H_{i}$ which is induced group-theoretically from $H_{i} \subseteq \Pi_{X^{\bp}_{i}}$ and $\mfT_{\Pi_{X^{\bp}_{i}}}$. Note that  $\mfT_{H_{1}}$  coincides with  the edge-triple associated to $H_{1}$ induced group-theoretically from $\phi_{H_{1}}$ and $\mfT_{H_{2}}$. Moreover, we denote by  $$\mfT_{X_{H_{i}}^{\bp}} \defeq (\ell, d, f_{X_{H_{i}}}^{\bp}: Y_{X_{H_{i}}}^{\bp}\defeq Y^{\bp}_{i} \times_{X^{\bp}_{i}}X^{\bp}_{H_{i}} \migi X^{\bp}_{H_{i}})$$ the edge-triple associated to $X^{\bp}_{H_{i}}$ corresponding to $\mfT_{H_{i}}$. 


By applying Proposition \ref{pro-2-1}, Remark \ref{rem-pro-2-1}, Proposition \ref{pro-2-2}, and Remark \ref{rem-pro-2-2}, we have that the natural inclusion $H_{i} \migiinje \Pi_{X_{i}^{\bp}}$ induces the  maps
$$\gamma^{\rm ver, \ell}_{H_{i}}: V_{X_{H_{i}}, \ell} \migi V_{X_{i}, \ell},$$ 
$$\gamma^{\rm cl}_{\mfT_{\Pi_{X^{\bp}_{i}}}, H_{i}}: E^{\rm cl}_{\mfT_{H_{i}}} \migi E^{\rm cl}_{\mfT_{\Pi_{X_{i}}}}$$ group-theoretically. We put $$\gamma^{\rm ver}_{H_{i}}: v(\Gamma_{X_{H_{i}}^{\bp}}) \overset{\kappa^{-1}_{X_{H_{i}}, \ell}}\isom V_{X_{H_{i}}, \ell} \overset{\gamma^{\rm ver, \ell}_{H_{i}}}\migi V_{X_{i}, \ell}  \overset{\kappa_{X_{i}, \ell}}\isom v(\Gamma_{X^{\bp}_{i}}),$$$$\gamma^{\rm cl}_{H_{i}}: e^{\rm cl}(\Gamma_{X_{H_{i}}^{\bp}})\overset{\vartheta^{-1}_{\mfT_{H_{i}}}}\isom E^{\rm cl}_{\mfT_{H_{i}}} \overset{\gamma^{\rm cl}_{\mfT_{\Pi_{X^{\bp}_{i}}}, H_{i}}}\migi E^{\rm cl}_{\mfT_{\Pi_{X^{\bp}_{i}}}} \overset{\vartheta_{\mfT_{\Pi_{X^{\bp}_{i}}}}} \isom e^{\rm cl}(\Gamma_{X^{\bp}_{i}}).$$  Then the maps $\gamma^{\rm ver}_{H_{i}}$ and $\gamma^{\rm cl}_{H_{i}}$ can be reconstructed group-theoretically from the inclusion $H_{i} \migiinje \Pi_{X^{\bp}_{i}}$.

On the other hand, Theorem \ref{types} implies that the sets $\text{Edg}^{\rm op}(\Pi_{X^{\bp}_{i}})$ and $\text{Edg}^{\rm op}(H_{i})$ can be reconstructed group-theoretically from $\Pi_{X^{\bp}_{i}}$ and $H_{i}$, respectively. Note that we have a natural map $$\text{Edg}^{\rm op}(H_{i}) \migi \text{Edg}^{\rm op}(\Pi_{X^{\bp}_{i}})$$ induced by the natural inclusions of the stabilizer subgroups. Moreover, this map compatible with the actions of $H_{i}$ and $\Pi_{X^{\bp}_{i}}$. Then we obtain a map  $$\gamma^{\rm op}_{H_{i}}: e^{\rm op}(\Gamma_{X_{H_i}^{\bp}}) \isom  \text{Edg}^{\rm op}(H_{i})/H_{i} \migi \text{Edg}^{\rm op}(\Pi_{X^{\bp}_{i}})/\Pi_{X_{i}^{\bp}} \isom e^{\rm op}(\Gamma_{X_{i}^{\bp}})$$ which can be induced by the inclusion $H_{i} \migiinje \Pi_{X_{i}^{\bp}}$ group-theoretically.

By Theorem \ref{mainstep-1}, Theorem \ref{them-4-13}, and Theorem \ref{them-4-15}, the following maps 
$$\phi_{H_{1}}^{\rm sg, ver}: v(\Gamma_{X_{H_{1}}^{\bp}}) \isom v(\Gamma_{X_{H_{2}}^{\bp}}),$$
$$\phi_{H_{1}}^{\rm sg, op}: e^{\rm op}(\Gamma_{X_{H_{1}}^{\bp}}) \isom e^{\rm op}(\Gamma_{X_{H_{2}}^{\bp}}),$$
$$\phi_{H_{1}}^{\rm sg, cl}: e^{\rm cl}(\Gamma_{X_{H_{1}}^{\bp}}) \isom e^{\rm cl}(\Gamma_{X_{H_{2}}^{\bp}}),$$
$$\phi^{\rm sg, ver}: v(\Gamma_{X_{1}^{\bp}}) \isom v(\Gamma_{X_{2}^{\bp}}),$$
$$\phi^{\rm sg, op}: e^{\rm op}(\Gamma_{X_{1}^{\bp}}) \isom e^{\rm op}(\Gamma_{X_{2}^{\bp}}),$$
$$\phi^{\rm sg, cl}: e^{\rm cl}(\Gamma_{X_{1}^{\bp}}) \isom e^{\rm cl}(\Gamma_{X_{2}^{\bp}})$$  can be reconstructed group-theoretically from $\phi: \Pi_{X_{1}^{\bp}} \migisurj \Pi_{X_{2}^{\bp}}$ and $\phi_{H_{1}}: H_{1}\migisurj H_{2}$, respectively.

\begin{proposition}\label{prop-4-19}
We maintain the notation introduced above. Then the following diagrams
\[
\begin{CD}
v(\Gamma_{X^{\bp}_{H_{1}}}) @>\phi^{\rm sg, ver}_{H_{1}}>> v(\Gamma_{X^{\bp}_{H_{2}}}) 
\\
@V\gamma_{H_{1}}^{\rm ver}VV
@V\gamma_{H_{2}}^{\rm ver}VV
\\
v(\Gamma_{X^{\bp}_{1}}) @>\phi^{\rm sg,  ver}>> v(\Gamma_{X^{\bp}_{2}}),
\end{CD}
\]
\[
\begin{CD}
e^{\rm op}(\Gamma_{X^{\bp}_{H_{1}}}) @>\phi^{\rm sg, op}_{H_{1}}>> e^{\rm op}(\Gamma_{X^{\bp}_{H_{2}}}) 
\\
@V\gamma_{H_{1}}^{\rm op}VV
@V\gamma_{H_{2}}^{\rm op}VV
\\
e^{\rm op}(\Gamma_{X^{\bp}_{1}}) @>\phi^{\rm sg, op}>> e^{\rm op}(\Gamma_{X^{\bp}_{2}}), 
\end{CD}
\]

\[
\begin{CD}
e^{\rm cl}(\Gamma_{X^{\bp}_{H_{1}}}) @>\phi^{\rm sg, cl}_{H_{1}}>> e^{\rm cl}(\Gamma_{X^{\bp}_{H_{2}}}) 
\\
@V\gamma_{H_{1}}^{\rm cl}VV
@V\gamma_{H_{2}}^{\rm cl}VV
\\
e^{\rm cl}(\Gamma_{X^{\bp}_{1}}) @>\phi^{\rm sg, cl}>> e^{\rm cl}(\Gamma_{X^{\bp}_{2}})
\end{CD}
\]
are commutative. Moreover, all the commutative diagrams above are compatible with the natural actions of $G$. 
\end{proposition}

\begin{proof}
The commutativity of the second diagram follows immediately from Theorem \ref{mainstep-1}. We treat the third diagram. To verify the commutativity of the third diagram, we only need to prove the commutativity of the following diagram \[
\begin{CD}
e^{\rm cl}(\Gamma_{X^{\bp}_{H_{2}}}) @>(\phi^{\rm sg, cl}_{H_{1}})^{-1}>> e^{\rm cl}(\Gamma_{X^{\bp}_{H_{1}}}) 
\\
@V\gamma_{H_{2}}^{\rm cl}VV
@V\gamma_{H_{1}}^{\rm cl}VV
\\
e^{\rm cl}(\Gamma_{X^{\bp}_{2}}) @>(\phi^{\rm sg, cl})^{-1}>> e^{\rm cl}(\Gamma_{X^{\bp}_{1}}).
\end{CD}
\]

Let $e_{H_{2}} \in e^{\rm cl}(\Gamma_{X^{\bp}_{H_{2}}})$, $e_{H_{1}}\defeq (\phi^{\rm sg, cl}_{H_{1}})^{-1}(e_{H_{2}}) \in e^{\rm cl}(\Gamma_{X^{\bp}_{H_{1}}})$, $ e_{2}\defeq \gamma_{H_{2}}^{\rm cl}(e_{H_{2}}) \in e^{\rm cl}(\Gamma_{X^{\bp}_{2}})$,  $ e_{1}\defeq (\gamma_{H_{1}}^{\rm cl}\circ (\phi^{\rm sg, cl}_{H_{1}})^{-1})(e_{H_{2}}) \in e^{\rm cl}(\Gamma_{X^{\bp}_{1}})$,  $e'_{1}\defeq (\phi^{\rm sg, cl})^{-1}(e_{2}) \in e^{\rm cl}(\Gamma_{X^{\bp}_{1}}).$ We will prove that $e_{1}=e'_{1}$. 

Write $S_{H_{1}}$ and $S_{H_{2}}$ for the sets $(\gamma_{H_{1}}^{\rm cl})^{-1}(e_{1}')$ and $(\gamma_{H_{2}}^{\rm cl})^{-1}(e_{2})$, respectively. Note that $e_{H_{2}} \in S_{H_{2}}$. To verify $e_{1}=e'_{1}$, it is sufficient to prove that $e_{H_{1}} \in S_{H_{1}}$.

Let $\alpha_{2} \in E^{\rm cl, \star}_{\mfT_{\Pi_{X_{2}^{\bp}}}, e_{2}}$. Then the proof of Lemma \ref{lem-4-14} implies that $\alpha_{2}$ induces an element $$\alpha_{1} \in E^{\rm cl, \star}_{\mfT_{\Pi_{X_{1}^{\bp}}}, e'_{1}}.$$ Write $Y_{\alpha_{i}}^{\bp}$ for the pointed stable curve over $k_{i}$ corresponding to $\alpha_{i}$. We consider the Galois admissible covering $$Y^{\bp}_{\alpha_{2}} \times_{X^{\bp}_{2}} X^{\bp}_{H_{2}} \migi Y_{X_{H_2}}^{\bp}$$ over $k_{2}$ with Galois group $\mbZ/\ell\mbZ$, and denote by $\beta_{2}$ the element of $E^{*}_{\mfT_{H_{2}}}$ corresponding to this Galois admissible covering. Then we have $$\beta_{2}=\sum_{c_{2}\in S_{H_{2}}} t_{c_{2}}\beta_{c_{2}},$$ where $t_{c_{2}} \in (\mbZ/\ell\mbZ)^{\times}$ and $\beta_{c_{2}} \in E^{\rm cl, \star}_{\mfT_{H_{2}}, c_{2}}$. Note that we have $t_{e_{H_2}}\neq 0$. On the other hand, the proof of Lemma \ref{lem-4-14} implies that $\beta_{c_{2}}$ induces an element $ \beta_{(\phi_{H_{1}}^{\rm cl})^{-1}(c_{2})} \in E^{\rm cl, \star}_{\mfT_{H_{1}}, (\phi_{H_{1}}^{\rm cl})^{-1}(c_{2})}$. Then $\beta_{2}$ induces an element $$\beta_{1}\defeq \sum_{c_{2}\in S_{X_{H_{2}}}\setminus \{e_{H_{2}}\}} t_{c_{2}}\beta_{(\phi_{H_{1}}^{\rm cl})^{-1}(c_{2})} +t_{e_{H_{2}}}\beta_{e_{H_{1}}}\in E_{\mfT_{H_{1}}}^{*}.$$  Note that since $\beta_{1}$ corresponds to the Galois admissible covering $$Y^{\bp}_{\alpha_{1}} \times_{X^{\bp}_{1}} X^{\bp}_{H_{1}} \migi Y_{X_{H_1}}^{\bp}$$ over $k_{1}$ with Galois group $\mbZ/\ell\mbZ$, the composition of the Galois admissible coverings $Y^{\bp}_{\alpha_{1}} \times_{X^{\bp}_{1}} X^{\bp}_{H_{1}} \migi Y_{X_{H_1}}^{\bp} \overset{f_{X_{H_{1}}}^{\bp}}\migi X^{\bp}_{H_{1}}$ is ramified over $S_{H_{1}}$. This means that $e_{H_{1}}$ is contained in $S_{H_{1}}$. 

Similar arguments to the arguments given in the proof above imply the first diagram is commutative. It is easy to check the ``moreover" part of the lemma. This completes the proof of the proposition.
\end{proof}

\subsection{Combinatorial Grothendieck conjecture for surjections}\label{sec-4-5}

We maintain the notation introduced in Section \ref{sec-4-3}. In the present subsection, we suppose that {\it $X_{1}^{\bp}$ and $X_{2}^{\bp}$ satisfy Condition A, Condition B, and Condition C} unless indicated otherwise. Then we have $r_{X_{1}}=r_{X_{2}}$. We put $(g_{X}, n_{X})\defeq (g_{X_{1}}, n_{X_{1}})=(g_{X_{2}}, n_{X_{2}}).$ Let $$\phi: \Pi_{X^{\bp}_{1}}\migi \Pi_{X^{\bp}_{2}}$$ be an arbitrary open continuous homomorphism of the solvable admissible fundamental groups of $X^{\bp}_{1}$ and $X^{\bp}_{2}$. By Lemma \ref{surj}, we have that $\phi$ is a surjection.

We fix some notation. Let $H_{2}$ be an open normal subgroup of $\Pi_{X_{2}}^{\bp}$, $H_{1}\defeq \phi^{-1}(H_{2})$ the open normal subgroup of $\Pi_{X_{1}^{\bp}}$, $G \defeq \Pi_{X_{1}^{\bp}}/H_{1}=\Pi_{X_{2}^{\bp}}/H_{2}$, and $\phi_{H_{1}}\defeq \phi|_{H_{1}}: H_{1} \migisurj H_{2}$ the surjection induced by $\phi$. Let $i \in \{1, 2\}$. We write $$f^{\bp}_{H_i}: X^{\bp}_{H_{i}} \migi X_{i}^{\bp}$$ for the Galois admissible covering over $k_{i}$ with Galois group $G$, $(g_{X_{H_i}}, n_{X_{H_{i}}})$ for the type of $X^{\bp}_{H_{i}}$, $\Gamma_{X^{\bp}_{H_{i}}}$ for the dual semi-graph of $X^{\bp}_{H_{i}}$, and $r_{X_{H_{i}}}$ for the Betti number of $\Gamma_{X^{\bp}_{H_{i}}}$. 


\begin{lemma}\label{lem-4-20}
We maintain the notation introduced above.  Then $X^{\bp}_{H_{i}}$ satisfies Condition A,  Condition B, Condition C (i).
\end{lemma}

\begin{proof}
The first condition, the second condition, and the fourth condition of Condition A follow immediately from the definition of admissible coverings. Since $X_{i}^{\bp}$ satisfies Condition B and the third condition of Condition A, $X^{\bp}_{H_{i}}$ also satisfies Condition B and the third condition of Condition A. Moreover, Condition C (i) follows immediately from Theorem \ref{mainstep-1}.
This completes the proof of the lemma.
\end{proof}

\begin{lemma}\label{cond-0}
We maintain the notation introduced above. Suppose that there exists an open normal subgroup $H_{2}' \subseteq H_{2}$ such that $X^{\bp}_{H'_{1}}$ and $X^{\bp}_{H'_{2}}$ satisfy Condition A, Condition B, and Condition C, where $H_{1}' \defeq \phi^{-1}(H'_{1}) \subseteq H_{1}$. Then $X^{\bp}_{H_{1}}$ and $X^{\bp}_{H_{2}}$ satisfy Condition A, Condition B, and Condition C.
\end{lemma}

\begin{proof}
By Lemma \ref{lem-4-20}, to verify the lemma, we only need to prove that $X^{\bp}_{H_{1}}$ and $X^{\bp}_{H_{2}}$ satisfy Condition C (ii) and Condition C (iii).

Let $G' \defeq \Pi_{X^{\bp}_{1}}/H_{1}'= \Pi_{X^{\bp}_{2}}/H_{2}'$ and $G''\defeq H_{1}/H_{1}'=H_{2}/H_{2}' \subseteq G'$. By applying Proposition \ref{prop-4-19}, the following commutative diagrams  
\[
\begin{CD}
v(\Gamma_{X^{\bp}_{H'_{1}}}) @>\phi^{\rm sg, ver}_{H'_{1}}>> v(\Gamma_{X^{\bp}_{H'_{2}}}) 
\\
@V\gamma_{H'_{1}}^{\rm ver}VV
@V\gamma_{H'_{2}}^{\rm ver}VV
\\
v(\Gamma_{X^{\bp}_{1}}) @>\phi^{\rm sg,  ver}>> v(\Gamma_{X^{\bp}_{2}}),
\end{CD}
\]
\[
\begin{CD}
e^{\rm op}(\Gamma_{X^{\bp}_{H'_{1}}}) @>\phi^{\rm sg, op}_{H'_{1}}>> e^{\rm op}(\Gamma_{X^{\bp}_{H'_{2}}}) 
\\
@V\gamma_{H'_{1}}^{\rm op}VV
@V\gamma_{H'_{2}}^{\rm op}VV
\\
e^{\rm op}(\Gamma_{X^{\bp}_{1}}) @>\phi^{\rm sg, op}>> e^{\rm op}(\Gamma_{X^{\bp}_{2}}), 
\end{CD}
\]

\[
\begin{CD}
e^{\rm cl}(\Gamma_{X^{\bp}_{H'_{1}}}) @>\phi^{\rm sg, cl}_{H'_{1}}>> e^{\rm cl}(\Gamma_{X^{\bp}_{H'_{2}}}) 
\\
@V\gamma_{H'_{1}}^{\rm cl}VV
@V\gamma_{H'_{2}}^{\rm cl}VV
\\
e^{\rm cl}(\Gamma_{X^{\bp}_{1}}) @>\phi^{\rm sg, cl}>> e^{\rm cl}(\Gamma_{X^{\bp}_{2}})
\end{CD}
\] can be reconstructed group-theoretically from $H'_{i}\migiinje \Pi_{X^{\bp}_{i}}$, $\phi$, and $\phi_{H'_{1}} \defeq \phi|_{H'_{1}}$. Moreover, the commutative diagrams are compatible with the actions of $G''$ and $G'$. Then we obtain that $$\#v(\Gamma_{X^{\bp}_{H_{1}}})=\#(v(\Gamma_{X^{\bp}_{H'_{1}}})/G'')=\#(v(\Gamma_{X^{\bp}_{H'_{2}}})/G'')=\#v(\Gamma_{X^{\bp}_{H_{2}}}),$$
$$\#e^{\rm op}(\Gamma_{X^{\bp}_{H_{1}}})=\#(e^{\rm op}(\Gamma_{X^{\bp}_{H'_{1}}})/G'')=\#(e^{\rm op}(\Gamma_{X^{\bp}_{H'_{2}}})/G'')=\#e^{\rm op}(\Gamma_{X^{\bp}_{H_{2}}}),$$
$$\#e^{\rm cl}(\Gamma_{X^{\bp}_{H_{1}}})=\#(e^{\rm cl}(\Gamma_{X^{\bp}_{H'_{1}}})/G'')=\#(e^{\rm cl}(\Gamma_{X^{\bp}_{H'_{2}}})/G'')=\#e^{\rm cl}(\Gamma_{X^{\bp}_{H_{2}}}).$$ This means that $X^{\bp}_{H_{1}}$ and $X^{\bp}_{H_{2}}$ satisfy Condition C.
\end{proof}

\begin{lemma}\label{cond-1}
We maintain the notation introduced above. Suppose that $(\#G, p)=1$, and that $f_{H_{2}}$ is \'etale. Then $X^{\bp}_{H_{1}}$ and $X^{\bp}_{H_{2}}$ satisfy Condition A, Condition B, and Condition C.
\end{lemma}

\begin{proof}
By Lemma \ref{lem-4-20}, to verify the lemma, we only need to prove that $X^{\bp}_{H_{1}}$ and $X^{\bp}_{H_{2}}$ satisfy Condition C (ii) and Condition C (iii).  Moreover, since $G$ is a finite solvable group, to verify the lemma, it is sufficient to prove the lemme when $G\cong \mbZ/\ell\mbZ$, where $\ell$ is a prime number distinct from $p$. Thus, Proposition \ref{prop-4-10} implies that $f_{H_{1}}$ is also \'etale.

We denote by $H_{2}' \subseteq H_{2}$ the inverse image $D_{\ell}(\Pi_{X_{2}^{\bp}}^{\et})$ of the natural surjection $\Pi_{X_{2}^{\bp}} \migisurj \Pi_{X_{2}^{\bp}}^{\et}$. Then $H_{2}'$  is an open normal subgroup of $\Pi_{X_{2}^{\bp}}$. Let $H_{1}' \defeq \phi^{-1}(H_{2}')  \subseteq H_{1}$. We see that $H_{1}'$ is equal to the inverse image $D_{\ell}(\Pi_{X_{1}^{\bp}}^{\et})$ of the natural surjection $\Pi_{X_{1}^{\bp}} \migisurj \Pi_{X_{1}^{\bp}}^{\et}$. Since $X^{\bp}_{1}$ and $X_{2}^{\bp}$ satisfy Condition C, Theorem \ref{them-4-13} and the structures of the maximal prime-to-$p$ quotients of solvable admissible fundamental groups imply that $X_{H_{1}'}^{\bp}$ and $X_{H_{2}'}^{\bp}$ also satisfy Condition C. Then the lemma follows from Lemma \ref{cond-0}.
\end{proof}

\begin{lemma}\label{lem-4-21}
We maintain the notation introduced above. Suppose that $(\#G, p)=1$. Then $X^{\bp}_{H_{1}}$ and $X^{\bp}_{H_{2}}$ satisfy Condition A, Condition B, and Condition C.
\end{lemma}

\begin{proof}
By Lemma \ref{lem-4-20}, to verify the lemma, we only need to prove that $X^{\bp}_{H_{1}}$ and $X^{\bp}_{H_{2}}$ satisfy Condition C (ii) and Condition C (iii).

Since $G$ is a finite solvable group, to verify the lemma, it is sufficient to prove the lemme when $G\cong \mbZ/\ell\mbZ$, where $\ell$ is a prime number distinct from $p$.

Let $\mfT_{\Pi_{X_{2}^{\bp}}}=(\ell, d, \alpha_{f_{X_{2}}}: \Pi_{X_{2}^{\bp}}^{\et} \migisurj \mbF_{d})$ be an edge-triple associated to $\Pi_{X_{2}^{\bp}}$, $\mfT_{\Pi_{X_{1}^{\bp}}}=(\ell, d, \alpha_{f_{X_{1}}}: \Pi_{X_{1}^{\bp}}^{\et} \migisurj \mbF_{d})$ the edge-triple associated to $\Pi_{X_{1}^{\bp}}$ induced by $\phi$, and $\mfT_{X_{i}^{\bp}}=(\ell, d, f_{X_{i}}^{\bp}: Y_{i}^{\bp} \migi X_{i}^{\bp})$ the edge-triple associated to $X_{i}^{\bp}$ corresponding to $\mfT_{\Pi_{X_{i}^{\bp}}}$. 

First, we suppose that $f_{H_{2}}$ is \'etale over $D_{X_{2}}$. Then Theorem \ref{mainstep-1} implies that $f_{H_{1}}$ is also \'etale over $D_{X_{1}}$. Let $\alpha_{e_{1}} \in E^{\rm cl, \star, 0}_{\mfT_{\Pi_{X^{\bp}_{1}}}, e_{1}}$, $e_{1} \in e^{\rm cl}(\Gamma_{X_{1}^{\bp}})$, $$\alpha_{1} \defeq \sum_{e_{1} \in e^{\rm cl}(\Gamma_{X_{1}^{\bp}})}\alpha_{e_{1}} \in E^{*}_{\mfT_{\Pi_{X^{\bp}_{1}}}},$$ and $g_{1, \alpha_{1}}^{\bp}: Y_{1, \alpha_{1}}^{\bp} \migi Y_{1}^{\bp}$ the Galois admissible covering over $k_{1}$ corresponding to $\alpha_{1}$. Note that we have that $\#e_{g_{1, \alpha_{1}}}^{\rm op, ra}=\#v_{g_{1, \alpha_{1}}}^{\rm sp}=0$. Let $\phi^{\rm sg, cl}: e^{\rm cl}(\Gamma_{X_{1}^{\bp}}) \isom  e^{\rm cl}(\Gamma_{X_{2}^{\bp}})$ be the bijection of the sets of closed edges obtained in Theorem \ref{them-4-15}, $\alpha_{\phi^{\rm sg, cl}(e_{1})} \in E^{\rm cl, \star, 0}_{\mfT_{\Pi_{X^{\bp}_{2}}}, \phi^{\rm sg, cl}(e_{1})}$ the element induced by $\phi$ for every $e_{1} \in e^{\rm cl}(\Gamma_{X_{1}^{\bp}})$, $$\alpha_{2} \defeq \sum_{e_{1} \in e^{\rm cl}(\Gamma_{X_{1}^{\bp}})}\alpha_{\phi^{\rm sg, cl}(e_{1})} \in E^{*}_{\mfT_{\Pi_{X^{\bp}_{2}}}},$$ and $g_{2, \alpha_{2}}^{\bp}: Y_{2, \alpha_{2}}^{\bp} \migi Y_{2}^{\bp}$ the Galois admissible covering over $k_{2}$ corresponding to $\alpha_{2}$.  Then Proposition \ref{prop-op-cl} (a) implies that $$\#e_{g_{2, \alpha_{2}}}^{\rm op, ra}=\#v_{g_{2, \alpha_{2}}}^{\rm sp}=0.$$ We obtain that $g_{i, \alpha_{i}}$ is totally ramified over every node of $Y_{i}$, and that $Y_{1, \alpha_{1}}^{\bp}$ and $Y_{2, \alpha_{2}}^{\bp}$ satisfy Condition A, Condition B, and Condition C. Write $N_{i} \subseteq \Pi_{X_{i}^{\bp}}$ for the open normal subgroup corresponding to $Y^{\bp}_{i, \alpha_{i}}$.

Let $H'_{i} \defeq H_{i}\cap N_{i}$ and $X_{H'_{i}}^{\bp}$ the pointed stable curve over $k_{i}$ corresponding to $H'_{i}$. Note that $X_{H'_{i}}^{\bp}$ is isomorphic to a connected component of $$X_{H_{i}}^{\bp} \times_{X_{i}^{\bp}} Y^{\bp}_{g_{i, \alpha_{i}}}.$$ We denote by $h_{i}^{\bp}: X_{H'_{i}}^{\bp} \migi Y^{\bp}_{g_{i, \alpha_{i}}}$ the Galois admissible covering over $k_{i}$ corresponding to the injection $H'_{i} \migiinje N_{i}$. By applying Abhyankar's lemma, $f_{H_{i}}$ is \'etale over $D_{X_{i}}$ implies that  $h_{i}$ is \'etale. Then the lemma follows from Lemma \ref{cond-0} and Lemma \ref{cond-1}. This completes the proof of the lemme when $f_{H_{2}}$ is \'etale over $D_{X_{2}}$.

Next, let us prove the lemme in the general case. We take $\beta_{e_{1}} \in E^{\rm op, \star, 0}_{\mfT_{\Pi_{X^{\bp}_{1}}}, e_{1}}$ for every $e_{1} \in e^{\rm op}(\Gamma_{X_{1}^{\bp}})$ such that $\#v^{\rm sp}_{g_{1, \beta_{1}}}=0$, where $$\beta_{1} \defeq \sum_{e_{1} \in e^{\rm op}(\Gamma_{X_{1}^{\bp}})}\beta_{e_{1}} \in E^{*}_{\mfT_{\Pi_{X^{\bp}_{1}}}},$$ and $g_{1, \beta_{1}}^{\bp}: Y_{1, \beta_{1}}^{\bp} \migi Y_{1}^{\bp}$ is the Galois admissible covering over $k_{1}$ corresponding to $\beta_{1}$. Note that we have that $\#e_{g_{1, \beta_{1}}}^{\rm cl, ra}=\#v_{g_{1, \beta_{1}}}^{\rm sp}=0$. Let $\phi^{\rm sg, op}: e^{\rm op}(\Gamma_{X_{1}^{\bp}}) \isom  e^{\rm op}(\Gamma_{X_{2}^{\bp}})$ be the bijection of the sets of open edges obtained in Theorem \ref{mainstep-1}, $\beta_{\phi^{\rm sg, op}(e_{1})} \in E^{\rm op, \star, 0}_{\mfT_{\Pi_{X^{\bp}_{2}}}, \phi^{\rm sg, op}(e_{1})}$ the element induced by $\phi$ for every $e_{1} \in e^{\rm op}(\Gamma_{X_{1}^{\bp}})$, $$\beta_{2} \defeq \sum_{e_{1} \in e^{\rm op}(\Gamma_{X_{1}^{\bp}})}\beta_{\phi^{\rm sg, op}(e_{1})} \in E^{*}_{\mfT_{\Pi_{X^{\bp}_{2}}}},$$ and $g_{2, \beta_{2}}^{\bp}: Y_{2, \beta_{2}}^{\bp} \migi Y_{2}^{\bp}$ the Galois admissible covering over $k_{2}$ corresponding to $\beta_{2}$.  Then Proposition \ref{prop-op-cl} (b) implies that $$\#e_{g_{2, \beta_{2}}}^{\rm cl, ra}=\#v_{g_{2, \beta_{2}}}^{\rm sp}=0.$$ We obtain that $g_{i, \beta_{i}}$ is totally ramified over every marked point of $Y_{i}$, and that $Y_{1, \beta_{1}}^{\bp}$ and $Y_{2, \beta_{2}}^{\bp}$ satisfy Condition A, Condition B, and Condition C. Write $Q_{i} \subseteq \Pi_{X_{i}^{\bp}}$ for the open normal subgroup corresponding to $Y^{\bp}_{i, \beta_{i}}$.

Let $H''_{i} \defeq H_{i}\cap Q_{i}$ and $X_{H''_{i}}^{\bp}$ the pointed stable curve over $k_{i}$ corresponding to $H''_{i}$. Note that $X_{H''_{i}}^{\bp}$ is isomorphic to a connected component of $$X_{H_{i}}^{\bp} \times_{X_{i}^{\bp}} Y^{\bp}_{g_{i, \beta_{i}}}.$$ We denote by $h_{i}^{*, \bp}: X_{H''_{i}}^{\bp} \migi Y^{\bp}_{g_{i, \beta_{i}}}$ the Galois admissible covering over $k_{i}$ corresponding to the injection $H''_{i} \migiinje Q_{i}$. By applying Abhyankar's lemma,  $h_{i}^{*}$ is \'etale over $D_{Y_{g_{i, \beta_{i}}}}$. By applying the lemma in the case where $g_{i, \beta_{i}}$ is \'etale over $D_{Y_{i}}$, we obtain that $X_{H''_{1}}^{\bp}$ and $X_{H''_{2}}^{\bp}$ satisfy Condition A, Condition B, and Condition C. Then the lemma follows from Lemma \ref{cond-0}. We complete the proof of the lemma. 
\end{proof}

\begin{lemma}\label{lem-4-22}
We maintain the notation introduced above. Suppose that $G$ is a $p$-group. Then  $X^{\bp}_{H_{1}}$ and $X^{\bp}_{H_{2}}$ satisfy Condition A, Condition B, and Condition C.
\end{lemma}

\begin{proof}
By Lemma \ref{lem-4-20}, to verify the lemma, we only need to prove that $X^{\bp}_{H_{1}}$ and $X^{\bp}_{H_{2}}$ satisfy Condition C (ii) and Condition C (iii).

To verify the lemma, without loss the generality, it is sufficient to treat the case where $G\cong \mbZ/p\mbZ$. Since $f_{H_{i}}^{\bp}$ is \'etale, $X^{\bp}_{H_{1}}$ and $X^{\bp}_{H_{2}}$ satisfy Condition C (iii). 

Let $V_{i} \subseteq v(\Gamma_{X_{i}^{\bp}})^{>0, p}$ be the subset of vertices such that the natural homomorphism $$\Pi_{\widetilde X^{\bp}_{i, v_{i}}} \migiinje \Pi_{X_{i}^{\bp}} \migisurj  G \defeq \Pi_{X_{i}^{\bp}}/H_{i}$$ is non-trivial (since $G \cong \mbZ/p\mbZ$, the homomorphism is a surjection). Then we obtain $\#v(\Gamma_{X_{H_{i}}^{\bp}})=p(\#v(\Gamma_{X_{i}^{\bp}})-\#V_{i})+\#V_{i}$ and  $\#e^{\rm cl}(\Gamma_{X_{H_{i}}^{\bp}})=p\#e^{\rm cl}(\Gamma_{X_{i}^{\bp}})$.

Let $$\psi_{p}^{{\rm sg, ver}}: v(\Gamma_{X^{\bp}_{2}})^{>0, p} \migiinje v(\Gamma_{X^{\bp}_{1}})^{>0, p}$$ be the injection induced by $\phi$, which is obtained in Theorem \ref{them-4-18}. We put $$V'_{1} \defeq \{\psi_{p}^{{\rm sg, ver}}(v_{2})\}_{v_{2} \in V_{2}}\subseteq v(\Gamma_{X_{1}^{\bp}})^{>0, p}.$$ By applying Lemma \ref{lem-4-17}, we see that $$V_{1}=V'_{1}.$$ Thus, we have $\#v(\Gamma_{X_{H_{1}}^{\bp}})=\#v(\Gamma_{X_{H_{2}}^{\bp}})$ and $\#e^{\rm cl}(\Gamma_{X_{H_{1}}^{\bp}})=\#e^{\rm cl}(\Gamma_{X_{H_{2}}^{\bp}})$. This completes the proof of the lemma.
\end{proof}

\begin{proposition}\label{prop-4-23}
We maintain the notation introduced above. Then  $X^{\bp}_{H_{1}}$ and $X^{\bp}_{H_{2}}$ satisfy Condition A, Condition B, and Condition C.
\end{proposition}

\begin{proof}
Since $G$ is a solvable group, the proposition follows from Lemma \ref{lem-4-21} and Lemma \ref{lem-4-22}.
\end{proof}

Next, we prove the main result of the present section which we call combinatorial Grothendieck conjecture for surjections.

\begin{theorem}\label{comgc}
We maintain the notation introduced above. Then the surjective open continuous homomorphism $\phi: \Pi_{X^{\bp}_{1}} \migisurj \Pi_{X_{2}^{\bp}}$ induces  surjective maps $$\phi^{\rm ver}: {\rm Ver}(\Pi_{X_{1}^{\bp}}) \migisurj {\rm Ver}(\Pi_{X_{2}^{\bp}}),$$ $$\phi^{\rm edg, op}: {\rm Edg}^{\rm op}(\Pi_{X_{1}^{\bp}}) \migisurj {\rm Edg}^{\rm op}(\Pi_{X_{2}^{\bp}}),$$ $$\phi^{\rm edg, cl}: {\rm Edg}^{\rm cl}(\Pi_{X_{1}^{\bp}}) \migisurj {\rm Edg}^{\rm cl}(\Pi_{X_{2}^{\bp}})$$ group-theoretically. Moreover, $\phi$ induces an isomorphism $$\phi^{\rm sg}: \Gamma_{X_{1}^{\bp}} \isom \Gamma_{X^{\bp}_{2}}$$ of the dual semi-graphs of $X_{1}^{\bp}$ and $X_{2}^{\bp}$ group-theoretically.
\end{theorem}

\begin{proof}
By applying Theorem \ref{mainstep-1}, the homomorphism $\phi: \Pi_{X^{\bp}_{1}} \migisurj \Pi_{X_{2}^{\bp}}$ induces a surjective map $\phi^{\rm edg, op}: {\rm Edg}^{\rm op}(\Pi_{X_{1}^{\bp}}) \migisurj {\rm Edg}^{\rm op}(\Pi_{X_{2}^{\bp}})$ group-theoretically. We only need to treat the cases of $\phi^{\rm ver}$ and $\phi^{\rm edg, cl}$, respectively.

Let $\mcC_{\Pi_{X^{\bp}_{2}}}$ be a cofinal system of $\Pi_{X^{\bp}_{2}}$ which consists of open normal subgroups of $\Pi_{X_{2}^{\bp}}$. We put $$\mcC_{\Pi_{X_{1}^{\bp}}}\defeq \{H_{1}\defeq \phi^{-1}(H_{2})\ | \ H_{2}\in \mcC_{\Pi_{X_{2}^{\bp}}}\}.$$ Note that $\mcC_{\Pi_{X_{1}^{\bp}}}$ is not a cofinal system of  $\Pi_{X_{1}^{\bp}}$ in general. Moreover, by applying Proposition \ref{prop-4-23}, we have that $X_{H_{1}}^{\bp}$ and $X_{H_{2}}^{\bp}$ satisfy Condition A, Condition B, and Condition C for every $H_{2} \in \mcC_{\Pi_{X_{2}^{\bp}}}$ and $H_{1} \defeq \phi^{-1}(H_{2})\in \mcC_{\Pi_{X^{\bp}_{1}}}$.  

We treat the case of $\phi^{\rm ver}$. Let $\widehat X_{i}^{\bp}$ be the universal solvable admissible covering of $X^{\bp}_{i}$ associated to $\Pi_{X_{i}^{\bp}}$ and $\Gamma_{\widehat X^{\bp}_{i}}$ the dual semi-graph of $\widehat X^{\bp}_{i}$. Let $\widehat w_{1} \in v(\Gamma_{\widehat X^{\bp}_{1}})$ and $\Pi_{\widehat w_{1}}$ the stabilizer subgroup of $\widehat w_{1}$. Write $w_{H_{1}} \in v(\Gamma_{X_{H_{1}}^{\bp}})$, $H_{1} \in \mcC_{\Pi_{X^{\bp}_{1}}}$, for the image of $\widehat w_{1}$. Proposition \ref{prop-4-19} implies that $\phi$ induces a cofinal system of vertices $$\mcC_{\widehat w_{2}}\defeq \{w_{H_{2}} \defeq \phi_{H_{1}}^{\rm ver}(w_{H_{1}})\}_{H_{2} \in \mcC_{\Pi_{X^{\bp}_{2}}}},$$ which admits a natural action of $\Pi_{X_{2}^{\bp}}$.  Then we obtain an element $\widehat w_{2} \in v(\Gamma_{\widehat X^{\bp}_{2}})$. Moreover, the stabilizer of $\mcC_{\widehat w_{2}}$ is $\Pi_{\widehat w_{2}}$. We see immediately that $\phi$ induces a surjective open continuous homomorphism $$\phi|_{\Pi_{\widehat w_{1}}}: \Pi_{\widehat w_{1}} \migisurj \Pi_{\widehat w_{2}}$$ group-theoretically. Then we define $$\phi^{\rm ver}: \text{Ver}(\Pi_{X_{1}^{\bp}}) \migi \text{Ver}(\Pi_{X_{2}^{\bp}}), \ \Pi_{\widehat w_{1}} \mapsto \Pi_{\widehat w_{2}}.$$

Next, we prove that $\phi^{\rm ver}$ is a surjective. Let $\widehat v_{2} \in v(\Gamma_{\widehat X_{2}^{\bp}})$ and $\Pi_{\widehat v_{2}}$ the stabilizer subgroup of $\widehat v_{2}$. Write $v_{H_{2}} \in v(\Gamma_{X_{H_{2}}^{\bp}})$, $H_{2} \in \mcC_{\Pi_{X^{\bp}_{2}}}$, for the image of $\widehat v_{2}$. Then we obtain a cofinal system of vertices $$\mcC_{\widehat v_{2}}\defeq \{v_{H_{2}}\}_{H_{2} \in \mcC_{\Pi_{X^{\bp}_{2}}}}$$ associated to $\widehat v_{2}$. The cofinal system $\mcC_{\widehat v_{2}}$ admits a natural action of $\Pi_{X_{2}^{\bp}}$. We see immediately that the stabilizer of $\mcC_{\widehat v_{2}}$ is equal to $\Pi_{\widehat v_{2}}$. Proposition \ref{prop-4-19} implies that $\phi$ and $\mcC_{\widehat v_{2}}$ induce a set of $$\mcC'\defeq \{v_{H_{1}}\defeq (\phi^{\rm sg, vex})^{-1}(v_{H_{2}})\}_{H_{1} \in \mcC_{\Pi_{X^{\bp}_{2}}}}$$ group-theoretically. By extending $\mcC'$ to a cofinal system of vertices. Then we obtain an element $\widehat v_{1} \in v(\Gamma_{\widehat X_{1}^{\bp}})$ such that the image of $\widehat v_{1}$ in $v(\Gamma_{X_{H_{1}}})$ is $v_{H_{1}}$. Thus, $\phi$ induces a surjective $$\phi|_{\Pi_{\widehat v_{1}}}: \Pi_{\widehat v_{1}} \migisurj \Pi_{\widehat v_{2}}.$$ This means that $\phi^{\rm ver}$ is a surjection.

By applying similar arguments to the arguments given in the proof above, we obtain that $\phi$ induces a surjective map $\phi^{\rm edg, cl}: {\rm Edg}^{\rm cl}(\Pi_{X_{1}^{\bp}}) \migisurj {\rm Edg}^{\rm cl}(\Pi_{X_{2}^{\bp}})$ group-theoretically. This completes the proof of the first part of the theorem.

The surjections $\phi^{\rm ver}$, $\phi^{\rm edg, op}$, and $\phi^{\rm edg, cl}$ imply the following surjections $$\phi^{\rm sg, ver}: v(\Gamma_{X_{1}^{\bp}}) \isom {\rm Ver}(\Pi_{X_{1}^{\bp}})/\Pi_{X_{1}^{\bp}}\migisurj {\rm Ver}(\Pi_{X_{2}^{\bp}})/\Pi_{X_{2}^{\bp}}\isom v(\Gamma_{X_{2}^{\bp}}),$$ $$\phi^{\rm sg, op}: e^{\rm op}(\Gamma_{X_{1}^{\bp}}) \isom {\rm Edg}^{\rm op}(\Pi_{X_{1}^{\bp}})/\Pi_{X_{1}^{\bp}} \migisurj {\rm Edg}^{\rm op}(\Pi_{X_{2}^{\bp}})/\Pi_{X_{2}^{\bp}} \isom e^{\rm op}(\Gamma_{X_{2}^{\bp}}),$$ $$\phi^{\rm sg, cl}: e^{\rm cl}(\Gamma_{X_{1}^{\bp}}) \isom {\rm Edg}^{\rm cl}(\Pi_{X_{1}^{\bp}})/\Pi_{X_{1}^{\bp}}\migisurj {\rm Edg}^{\rm cl}(\Pi_{X_{2}^{\bp}})/\Pi_{X_{2}^{\bp}} \isom e^{\rm cl}(\Gamma_{X_{2}^{\bp}}).$$ Since $X_{1}^{\bp}$ and $X_{2}^{\bp}$ satisfy Condition C, we have that $\phi^{\rm sg, ver}$, $\phi^{\rm sg, op}$, and $\phi^{\rm sg, cl}$ are bijections. Furthermore, by applying \cite[Lemma 1.5, Lemma 1.7, and Lemma 1.9]{HM}, $\phi$ induces an isomorphism of dual semi-graphs $$\phi^{\rm sg}: \Gamma_{X_{1}^{\bp}} \isom \Gamma_{X_{2}^{\bp}}$$ group-theoretically. This completes the proof of the theorem.
\end{proof}

\begin{remarkA}\label{rem-comgc}
We maintain the notation introduced above. We see immediately that Theorem \ref{comgc} does not hold without Condition C (e.g. $X_{1}^{\bp}$ is a generic curve of $\overline \mcM_{g, n}$, and $X_{2}^{\bp}$ is a singular curve).

On the other hand, the author believes that Theorem \ref{comgc} also holds without Condition B (e.g. $n_{X_{i}}=0$). The main difficult is that we do not have a precise formula for limits of $p$-averages of arbitrary pointed stable curves. Moreover, if the question of \cite[Remark 4.10.2]{Y5} is true, without too much difficulty, similar arguments to the arguments given in the proofs of this section imply that Theorem \ref{comgc} holds without Condition B.

Furthermore, the author also believes that Theorem \ref{comgc} holds without Condition A and Condition B. For example, Theorem \ref{mainstep-2} below shows that Theorem \ref{comgc} holds without Condition A and Condition B if $g_{X}=0$. Moreover, Theorem \ref{mainstep-2} will play a key role in the proof of the main theorem of the present paper.
\end{remarkA}

\begin{corollary}\label{coro-4-25}
We maintain the notation introduced above. Let $Q_{2} \subseteq \Pi_{X_{2}^{\bp}}$ be an arbitrary open subgroup and $Q_{1} \defeq \phi^{-1}(Q_{2}) \subseteq \Pi_{X_{1}^{\bp}}$. Then we have $${\rm Avr}_{p}(Q_{1})={\rm Avr}_{p}(Q_{2}).$$
\end{corollary}

\begin{proof}
The corollary follows immediately from Theorem \ref{comgc}.
\end{proof}


\begin{lemma}\label{lem-4-26}
Let $E^{\bp}$ be a pointed stable curve of type $(0, n)$ over an algebraically closed field $k$ of characteristic $p>0$, $\Pi_{E^{\bp}}$ the solvable admissible fundamental group of $E^{\bp}$, and $\ell$ a prime number such that $\ell \neq p$, and that $\ell >>n$. We put $${\rm Edg}^{\rm op, \ell, ab}(\Pi_{E^{\bp}})\defeq \{ pr^{\ell, \rm ab}(I_{\widehat e}) \ | \ I_{\widehat e} \in {\rm Edg}^{\rm op}(\Pi_{E^{\bp}})\}=\{I_{e}\}_{e \in e^{\rm op}(\Gamma_{E^{\bp}})},$$ where $pr^{\rm \ell, ab}$ denotes the natural surjective homomorphism $\Pi_{E^{\bp}} \migisurj \Pi_{E^{\bp}}^{\rm \ell, ab}$, and $I_{e} \defeq pr^{\ell, \rm ab}(I_{\widehat e})$. Let $a_{e}\in I_{e}$, $e\in e^{\rm op}(\Gamma_{E^{\bp}})$, be a generator of $I_{e}$ such that $$\prod_{e\in e^{\rm op}(\Gamma_{E^{\bp}})}a_{e}=1,$$ and let $\alpha: \Pi_{E^{\bp}}^{\rm \ell, ab} \migisurj \mbF_{\ell}$ be a surjection and $r_{e} \defeq \alpha(a_{e})$. Write $$g^{\bp}: X^{\bp} \migi E^{\bp}$$ for the Galois admissible covering over $k$ with Galois group $\mbZ/\ell\mbZ$ corresponding to $\alpha$. Suppose that $r_{e} \neq 0$ for every $e\in e^{\rm op}(\Gamma_{E^{\bp}})$, and that $$\sum_{e\in e^{\rm op}(\Gamma_{E^{\bp}})}r_{e}=\ell$$ if we identify $\mbF_{\ell}$ to $\{0, 1, \dots, \ell-1\} \subseteq \mbZ$. Then $g^{\bp}$ is totally ramified over every node and every marked point of $E^{\bp}$. In particular, we have that the map of dual semi-graphs $\Gamma_{X^{\bp}} \migi \Gamma_{E^{\bp}}$ of $X^{\bp}$ and $E^{\bp}$ induced by $g^{\bp}$ is a bijection, and that $X^{\bp}$ satisfies Condition A. 
\end{lemma}

\begin{proof}
We prove the lemma by induction on $\#v(\Gamma_{E^{\bp}})$. Suppose that $\#v(\Gamma_{E^{\bp}})=1$. Then the lemma is trivial. 

Suppose that $\#v(\Gamma_{E^{\bp}})\geq 2$. Let $v_{0} \in v(\Gamma_{E^{\bp}})$ and $\widetilde E_{v_{0}}^{\bp}$ the smooth pointed stable curve associated to $v_{0}$. Note that the underlying curve $\widetilde E_{v_{0}}$ coincides with the irreducible component of $E$ corresponding to $v_{0}$.  On the other hand, we define a pointed stable curve over $k$ to be $$E_{0}^{\bp} =(E_{0}\defeq \overline {E\setminus \widetilde E_{v_{0}}}, D_{E_{0}} \defeq (D_{E}\cap E_{0}) \cup (E_{0} \cap \widetilde E_{v_{0}})),$$ where $\overline {E\setminus \widetilde E_{v_{0}}}$ denotes the topological closure of $E\setminus \widetilde E_{v_{0}}$ in $E$. Then $g^{\bp}$ induces Galois admissible coverings $$g^{\bp}_{v_{0}}: \widetilde X_{v_{0}}^{\bp} \migi \widetilde E_{v_{0}}^{\bp},$$  $$g_{0}^{\bp}: X^{\bp}_{0} \migi E_{0}^{\bp}$$ over $k$ with Galois group $\mbZ/\ell\mbZ$. To verify the lemma, we only need to prove that $g_{v_{0}}^{\bp}$ and $g_{0}^{\bp}$ are totally ramified over every node and every marked point of $\widetilde E_{v_{0}}^{\bp}$ and $E_{0}^{\bp}$, respectively.

Let $\Pi_{\widetilde E_{v_{0}}^{\bp}}$ and $\Pi_{E_{0}^{\bp}}$ be the solvable admissible fundamental groups of $\widetilde E_{v_{0}}^{\bp}$ and $E^{\bp}_{0}$, respectively. Since $\Gamma_{E^{\bp}}^{\rm cpt}$ is $2$-connected, \cite[Corollary 3.5]{Y5} implies that the natural homomorphism $$\theta_{v_{0}}: \Pi_{\widetilde E_{v_{0}}^{\bp}}^{\rm \ell, ab} \migi \Pi^{\rm \ell ,ab}_{E^{\bp}}$$ is an injection. Let $$\theta_{0}: \Pi_{E_{0}^{\bp}}^{\rm \ell, ab} \migi \Pi^{\rm \ell ,ab}_{E^{\bp}}$$ be the homomorphism induced by the natural (outer) injective homomorphism $\Pi_{E_{0}^{\bp}} \migiinje \Pi_{E^{\bp}}$ (in fact, $\theta_{0}$ is also an injection).

Let $\{x\} = E_{0} \cap \widetilde E_{v_{0}}$, $e_{v_{0}} \in e^{\rm op}(\Gamma_{\widetilde E_{v_{0}}^{\bp}})$ the open edge corresponding to $x$, $e_{0} \in e^{\rm op}(\Gamma_{E_{0}^{\bp}})$ the open edge corresponding to $x$, $\widehat {\widetilde E_{v_{0}}^{\bp}}$ the universal admissible covering of $\widetilde E_{v_{0}}^{\bp}$, $\widehat E_{0}^{\bp}$ the universal admissible covering of $E_{0}^{\bp}$, $\widehat e_{v_0} \in v(\Gamma_{\widehat {\widetilde E_{v_{0}}^{\bp}}})$ an element over $e_{v_0}$, and $\widehat e_{0} \in v(\Gamma_{\widehat E_{0}^{\bp}})$ an element over $e_{0}$. We denote by $I_{e_{v_0}}$ the image of $I_{\widehat e_{v_0}}$ of $\Pi_{\widetilde E_{v_{0}}^{\bp}}\migisurj \Pi_{\widetilde E_{v_{0}}^{\bp}}^{\rm \ell, ab}$, and by $I_{e_{0}}$ the image of $I_{\widehat e_{0}}$ of $\Pi_{E_{0}^{\bp}}\migisurj \Pi_{E_{0}^{\bp}}^{\rm \ell, ab}$. We put $$a_{e_{v_0}}=(\prod_{e \in e^{\rm op}(\Gamma_{\widetilde E_{v_{0}}^{\bp}}) \setminus \{e_{v_0}\}}a_{e})^{-1},$$ $$a_{e_{0}}=(\prod_{e \in e^{\rm op}(\Gamma_{ E_{0}^{\bp}}) \setminus \{e_{0}\}}a_{e})^{-1}.$$ Then $a_{e_{v_0}}$ and $a_{e_{0}}$ are generators of $I_{e_{v_0}}$ and $I_{e_{0}}$, respectively. Moreover, we put $$\widetilde \alpha_{v_{0}}: \Pi^{\rm \ell, ab}_{\widetilde E_{v_{0}}^{\bp}}\overset{\theta_{v_{0}}}\migi \Pi_{E^{\bp}}^{\rm \ell, ab} \overset{\alpha} \migisurj \mbF_{\ell}.$$ and $$\alpha_{0}: \Pi^{\rm \ell, ab}_{ E_{0}^{\bp}}\overset{\theta_{0}}\migi \Pi_{E^{\bp}}^{\rm \ell, ab} \overset{\alpha} \migisurj \mbF_{\ell}.$$ Then the structures of maximal pro-prime-to-$p$ quotients of solvable admissible fundamental groups implies that $$\widetilde \alpha_{v_{0}}(a_{e_{v_0}})=\ell-\sum_{e \in e^{\rm op}(\Gamma_{\widetilde E_{v_{0}}^{\bp}}) \setminus \{e_{v_0}\}}r_{e}=\sum_{e \in e^{\rm op}(\Gamma_{ E_{0}^{\bp}}) \setminus \{e_{0}\}}r_{e}$$ and $$\alpha_{0}(a_{e_{0}})=\sum_{e \in e^{\rm op}(\Gamma_{\widetilde E_{v_{0}}^{\bp}}) \setminus \{e_{v_0}\}}r_{e}.$$ Thus, by induction, we have that $g^{\bp}_{v_{0}}$ and $g^{\bp}_{0}$ are totally ramified over every node and every marked point of $\widetilde E_{v_{0}}^{\bp}$ and $E_{0}^{\bp}$, respectively. We complete the proof of the lemma.
\end{proof}

\begin{lemma}\label{lem-4-27}
Let $E^{\bp}$ be a pointed stable curve of type $(0, n)$ over an algebraically closed field $k$ of characteristic $p>0$. Then $E^{\bp}$ satisfies Condition B.
\end{lemma}

\begin{proof}
Let $f^{\bp}: W^{\bp} \migi E^{\bp}$ be an arbitrary admissible covering over $k$, $\Gamma_{W^{\bp}}$ the dual semi-graph of $W^{\bp}$, and $f^{\rm sg}: \Gamma_{W^{\bp}} \migi \Gamma_{E^{\bp}}$ the map of dual semi-graphs of $W^{\bp}$ and $X^{\bp}$ induced by $f^{\bp}$. To verify the lemma, we only need to prove that $\Gamma_{W^{\bp}}^{\rm cpt}$ is $2$-connected.

Suppose that $f^{\bp}$ is trivial. Then the lemma follows from that $\Gamma_{E^{\bp}}^{\rm cpt}$ is $2$-connected.

Suppose that $f^{\bp}$ is non-trivial. Let $w \in v(\Gamma_{W^{\bp}})$ and $v \in v(\Gamma_{E^{\bp}})$. We denote by $\pi_{0}(w)$ and $\pi_{0}(v)$ the sets of connected components of $\Gamma_{W^{\bp}} \setminus \{w\}$ and $\Gamma_{E^{\bp}} \setminus \{v\}$, respectively.  Suppose that $v=f^{\rm sg}(w)$. Let $C_{w} \in \pi_{0}(w)$. We see immediately that $f^{\rm sg}(C_{w})$ is a connected component of $\Gamma_{E^{\bp}} \setminus \{v\}$. Write $C_{v}$ for $f^{\rm sg}(C_{w})$. Since $C_{v} \cap e^{\rm op}(\Gamma_{E^{\bp}}) \neq \emptyset$, we obtain that $C_{w} \cap e^{\rm op}(\Gamma_{W^{\bp}}) \neq \emptyset$. Thus, $\Gamma_{W^{\bp}}^{\rm cpt}$ is $2$-connected. This completes the proof of the lemma.
\end{proof}

Moreover, Theorem \ref{comgc} implies the following important result.

\begin{theorem}\label{mainstep-2}
Let $i\in \{1, 2\}$, and let $E_{i}^{\bp}$ be a pointed stable curve of type $(0, n)$ over $k_{i}$ of characteristic $p>0$, $\Pi_{E_{i}^{\bp}}$ the solvable admissible fundamental group of $E_{i}^{\bp}$, and $$\phi_{E}: \Pi_{E_{1}^{\bp}} \migi \Pi_{E_{2}^{\bp}}$$ an arbitrary open continuous homomorphism. Suppose that $E_{1}^{\bp}$ and $E_{2}^{\bp}$ satisfy Condition C. Then  $\phi_{E}: \Pi_{E^{\bp}_{1}} \migi \Pi_{E_{2}^{\bp}}$ induces  surjective maps $$\phi_{E}^{\rm ver}: {\rm Ver}(\Pi_{E_{1}^{\bp}}) \migisurj {\rm Ver}(\Pi_{E_{2}^{\bp}}),$$ $$\phi_{E}^{\rm edg, op}: {\rm Edg}^{\rm op}(\Pi_{E_{1}^{\bp}}) \migisurj {\rm Edg}^{\rm op}(\Pi_{E_{2}^{\bp}}),$$ $$\phi_{E}^{\rm edg, cl}: {\rm Edg}^{\rm cl}(\Pi_{E_{1}^{\bp}}) \migisurj {\rm Edg}^{\rm cl}(\Pi_{E_{2}^{\bp}})$$ group-theoretically. Moreover, $\phi_{E}$ induces an isomorphism $$\phi_{E}^{\rm sg}: \Gamma_{E_{1}^{\bp}} \isom \Gamma_{E^{\bp}_{2}}$$ of the dual semi-graphs of $E_{1}^{\bp}$ and $E_{2}^{\bp}$ group-theoretically.
\end{theorem}

\begin{proof}
Lemma \ref{surj} implies that $\phi_{E}$ is a surjective. By applying Theorem \ref{mainstep-1}, the homomorphism $\phi_{E}: \Pi_{E^{\bp}_{1}} \migisurj \Pi_{E_{2}^{\bp}}$ induces a surjective map $\phi^{\rm edg, op}: {\rm Edg}^{\rm op}(\Pi_{E_{1}^{\bp}}) \migisurj {\rm Edg}^{\rm op}(\Pi_{E_{2}^{\bp}})$ group-theoretically. We only need to treat the cases of $\phi^{\rm ver}_{E}$ and $\phi_{E}^{\rm edg, cl}$, respectively.

Let $\ell$ be a prime number such that $\ell \neq p$, and that $\ell>>n$. Let $$\alpha_{2}: \Pi^{\rm \ell, ab}_{E_{2}^{\bp}} \migisurj \mbF_{\ell}$$ satisfying the assumptions of Lemma \ref{lem-4-26}. Then Theorem \ref{mainstep-1} implies that $\phi_{E}$ and $\alpha_{2}$ induces a surjection $$\alpha_{1}: \Pi^{\rm \ell, ab}_{E_{1}^{\bp}} \migisurj \mbF_{\ell},$$ which satisfies the assumptions of Lemma \ref{lem-4-26}. Write $g_{i}^{\bp}: X^{\bp}_{i} \migi E_{i}^{\bp}$ for the Galois admissible covering over $k_{i}$ with Galois group $\mbZ/\ell\mbZ$. Then Lemma \ref{lem-4-26} and Lemma \ref{lem-4-27} imply that $X_{1}^{\bp}$ and $X^{\bp}_{2}$ satisfy Condition A, Condition B, and Condition C.

Write $\Pi_{X_{i}^{\bp}} \subseteq \Pi_{E_{i}^{\bp}}$ for the open normal subgroup corresponding to $g_{i}^{\bp}$.  Let $\Pi_{\widehat v_{X_{i}}} \in \text{Ver}(\Pi_{X_{i}^{\bp}})$, $I_{\widehat e_{X_{i}}} \in \text{Edg}^{\rm cl}(\Pi_{X_{i}^{\bp}})$, $\Pi_{\widehat v_{i}}\in \text{Ver}(\Pi_{E_{i}^{\bp}})$ the unique element which contains $\Pi_{\widehat v_{X_{i}}}$, and $I_{\widehat e_{i}}\in \text{Edg}^{\rm cl}(\Pi_{E_{i}^{\bp}})$ the unique element which contains $I_{\widehat e_{X_{i}}}$. Since $\Pi_{\widehat v_{i}}$ and $I_{\widehat e_{i}}$ the normalizers of $\Pi_{\widehat v_{X_{i}}}$ and $I_{\widehat e_{X_{i}}}$
in $\Pi_{E_{i}^{\bp}}$, respectively, the theorem follows immediately from Theorem \ref{comgc}. This completes the proof of the theorem.
\end{proof}

\section{The Homeomorphism Conjecture for closed points when $g=0$}\label{sec-6}

We maintain the notation introduced in Section \ref{sec-5}. In this section, we will prove that $\pi_{g, n}^{\rm adm}([q])$ is a closed point of $\overline \Pi_{g, n}$ for every $[q] \in \overline \mfM_{g, n}^{\rm cl}$ if $g=0$. In particular, the Homeomorphism Conjecture holds when $(g, n)=(0, 4)$.  {\it In the present section, we shall assume that all the fundamental groups of pointed stable curves are solvable admissible fundamental groups unless indicated otherwise.}

We fix some notation. Let $i \in \{1, 2\}$, and let  $X^{\bp}_{i}$ be a pointed stable curve of type $(0, n)$ over an algebraically closed field $k_{i}$ of characteristic $p>0$, $\Gamma_{X_{i}^{\bp}}$ the dual semi-graph of $X^{\bp}_{i}$, and $r_{X_{i}}$ the Betti number of $\Gamma_{X^{\bp}}$. Note that $\Gamma_{X_{i}^{\bp}}$ is a tree, and that $X_{i, v_{i}}$ is isomorphic to $\mbP_{k_{i}}^{1}$ for every $v_{i} \in v(\Gamma_{X_{i}^{\bp}})$. In particular, $X_{i, v_{i}}$ is smooth over $k_{i}$. For simplicity, we shall use the notation $X_{i, v_{i}}^{\bp}$ to denote the smooth pointed stable curve $\widetilde X_{i, v_{i}}^{\bp}$ of type $(0, n_{i, v_{i}})$ over $k_{i}$ associated to $v_{i} \in v(\Gamma_{X^{\bp}_{i}})$. On the other hand, let $\Pi_{X_{i}^{\bp}}$ be the solvable admissible fundamental group of $X^{\bp}_{i}$ and $$\phi: \Pi_{X_{1}^{\bp}} \migi \Pi_{X_{2}^{\bp}}$$ an arbitrary open continuous homomorphism. By Lemma \ref{surj}, we see that $\phi$ is a {\it surjective} open continuous homomorphism. Then $\phi$ induces an isomorphism $$\phi^{p}: \Pi_{X^{\bp}_{1}}^{p'} \isom \Pi_{X^{\bp}_{2}}^{p'}$$ of the maximal prime-to-$p$ quotients of solvable admissible fundamental groups. Let $\widehat X^{\bp}_{i}$ be a universal solvable admissible covering of $X_{i}^{\bp}$ corresponding to $\Pi_{X_{i}^{\bp}}$, $\Gamma_{\widehat X^{\bp}_{i}}$ the dual semi-graph of $\widehat X^{\bp}_{i}$, and $e_{i} \in e^{\rm op}(\Gamma_{X_{i}^{\bp}})$. We put $$\text{Edg}_{e_{i}}^{\rm op}(\Pi_{X_{i}^{\bp}}) \defeq \{I_{\widehat e_{i}} \in \text{Edg}^{\rm op}(\Pi_{X_{i}^{\bp}})\ | \ \widehat e_{i} \in e^{\rm op}(\Gamma_{\widehat X^{\bp}_{i}}) \ \text{is an open edge over} \ e_{i}\}.$$ Moreover, in the remainder of the present section, we shall suppose that {\it $k_{1}$ is an algebraic closure of $\mbF_{p}$.}

Denote by $${\rm Hom}^{\rm open}_{\rm pro\text{-}gps}(-, -), \ {\rm Isom}_{\rm pro\text{-}gps}(-, -)$$ the set of open continuous homomorphisms of profinit groups and the set of continuous isomorphisms of profinite groups, respectively. First, we have the following theorem which was proved by the author (cf. \cite[Theorem 1.2 and Remark 7.3.1]{Y2}). 

\begin{theorem}\label{mainthemsm}
We maintain the notation introduced above. Suppose that $X_{1}^{\bp}$ and $X_{2}^{\bp}$ are smooth over $k_{1}$ and $k_{2}$, respectively. Then we have that $${\rm Hom}^{\rm open}_{\rm pro\text{-}gps}(\Pi_{X_{1}^{\bp}}, \Pi_{X_{2}^{\bp}}) \neq \emptyset$$ if and only if $X_{1}^{\bp}$ is Frobenius equivalent to $X_{2}^{\bp}$. In particular, if this is the case, we have that $X_{2}^{\bp}$ can be defined over the algebraic closure of $\mbF_{p}$ in $k_{2}$,  and that $${\rm Hom}^{\rm open}_{\rm pro\text{-}gps}(\Pi_{X_{1}^{\bp}}, \Pi_{X_{2}^{\bp}}) = {\rm Isom}_{\rm pro\text{-}gps}(\Pi_{X_{1}^{\bp}}, \Pi_{X_{2}^{\bp}}).$$ 
\end{theorem}

\begin{remarkA}\label{rem-mainthemsm}
Let $[q] \in \mfM_{0, n}^{\rm cl}$ be an arbitrary point. Theorem \ref{mainthemsm} and Proposition \ref{prop-5-5} (a) imply that $$V(\pi_{0, n}^{\rm sol}([q])) \cap \Pi_{0, n}^{\rm sol}=\pi_{0, n}^{\rm sol}([q]).$$ Then we have that $[\pi_{1}^{\rm sol}(q)]$ is a closed point of $\Pi_{0, n}^{\rm sol}$. In particular, $$\pi_{0, 4}^{\rm t}: \mfM_{0, 4} \migisurj \Pi_{0, 4}, \ \pi_{0, 4}^{\rm t, sol}: \mfM_{0, 4} \migisurj \Pi_{0, 4}^{\rm sol}$$ are homeomorphisms. Note that Theorem \ref{mainthemsm} cannot tell us whether or not $[\pi_{1}^{\rm sol}(q)]$ is closed in $\overline \Pi_{0, n}^{\rm sol}$. In fact, this is highly non-trivial, see Proposition \ref{prop-6-5} below.
\end{remarkA}

\begin{lemma}\label{lem-6-2}
We maintain the notation introduced above. Suppose that $X_{1}^{\bp}$ is a singular curve. Then $X_{2}^{\bp}$ is also a singular curve.
\end{lemma}

\begin{proof}
Lemma \ref{rem-them-1-1-3} implies that there exists a Galois admissible covering $$f_{1}^{\bp}: Y_{1}^{\bp} \migi X_{1}^{\bp}$$ over $k_{1}$ with Galois group $G$ such that $(\#G, p)=1$, that the Betti number of the dual semi-graph of $Y_{1}^{\bp}$ is positive, and that $Y_{1}^{\bp}$ satisfies Condition A. Then $\phi^{p'}$ induces a Galois admissible covering $$f_{2}^{\bp}: Y_{2}^{\bp} \migi X_{2}^{\bp}$$ over $k_{2}$ with Galois group $G$. Write $g_{Y_{i}}$ for the genus of $Y^{\bp}_{i}$, and $r_{Y_{i}}$ for the Betti number of the dual semi-graph of $Y_{i}^{\bp}$.

By applying Theorem \ref{mainstep-1}, we obtain that $g_{Y_{1}}=g_{Y_{2}}$. Moreover, Theorem \ref{max and average} and Lemma \ref{lem-0} (b) imply that $$0<r_{Y_{1}}\leq r_{Y_{2}}.$$ This means that $X_{2}^{\bp}$ is a singular curve. We complete the proof of the lemma.
\end{proof}

Let $\overline \mbF_{p}$ be an algebraic closure of the finite field $\mbF_{p}$, and let $X^{\bp}$ be a {\it smooth} pointed stable curve of type $(0, n)$ over $\overline \mbF_{p}$. We fix two marked points $x_{ \infty}$, $x_{0} \in D_{X}$ distinct from each other. Moreover, we choose any field $k'\cong \overline \mbF_{p}$, and choose any isomorphism $\varphi: X \isom \mbP_{k'}^{1}$ as schemes such that $\varphi(x_{\infty})=\infty$ and $\varphi(x_{0})=0$. Then the set of $\overline \mbF_{p}$-rational points $X(\overline \mbF_{p}) \setminus \{x_{\infty}\} \isom \mbA^{1}_{k'}(k')$ is equipped with a structure of $\mbF_{p}$-module via the bijection $\varphi$. Note that since any $k'$-isomorphism of $\mbP_{k'}^{1}$ fixing $\infty$ and $0$ is a scalar multiplication, the $\mbF_{p}$-module structure of $X(\overline \mbF_{p}) \setminus \{x_{\infty}\}$ does not depend on the choices of $k'$ and $\varphi$ but depends only on the choices of $x_{ \infty}$ and $x_{0}$. We shall say that {\it $X(\overline \mbF_{p}) \setminus \{x_{\infty}\}$ is equipped with a structure of $\mbF_{p}$-module with respect to $x_{\infty}$ and $x_{0}$}. Then we have the following lemma.

\begin{lemma}\label{lem-6-3}
We maintain the notation introduced above. Suppose that $X_{i}^{\bp}$ is smooth over $k_{i}$. Let $e_{1, 0}$, $e_{1, \infty} \in e^{\rm op}(\Gamma_{X_{1}^{\bp}})$ be open edges distinct from each other. Theorem \ref{mainstep-1} implies that $\phi$ induces a bijection $\phi^{\rm sg, op}: e^{\rm op}(\Gamma_{X_{1}^{\bp}}) \isom e^{\rm op}(\Gamma_{X_{2}^{\bp}})$ group-theoretically. We put $e_{2, 0} \defeq \phi^{\rm sg, op}(e_{1, 0})$ and $e_{2, \infty} \defeq \phi^{\rm sg, op}(e_{1, \infty})$. Let $$\sum_{e_{1}\in e^{\rm op}(\Gamma_{X^{\bp}_{1}})\setminus \{e_{1, \infty}, e_{1, 0}\}}b_{e_{1}}x_{e_{1}}=x_{e_{1, 0}}$$ be a linear condition with respect to $e_{1, \infty}$ and $e_{1, 0}$ on $X^{\bp}_{1}$, where $b_{e_{1}} \in \mbF_{p}$ for every $e_{1}\in  e^{\rm op}(\Gamma_{X^{\bp}_{1}})$. Then the linear condition  $$\sum_{e_{1}\in e^{\rm op}(\Gamma_{X^{\bp}_{1}})\setminus \{e_{1, \infty}, e_{1, 0}\}}b_{e_{1}}x_{\phi^{\rm sg, op}(e_{1})}=x_{\phi^{\rm sg, op}(e_{1, 0})}=x_{e_{2, 0}}$$ with respect to $x_{e_{2, \infty}}$ and $x_{e_{2, 0}}$ on $X^{\bp}_{2}$ also holds.
\end{lemma}

\begin{proof}
See \cite[Lemma 7.1]{Y2}.
\end{proof}

\begin{lemma}\label{lem-6-4}
Let $X^{\bp}$ be a pointed stable curve of type $(0,n)$ over an algebraically closed field $k$ of characteristic $p>0$ and $\ell \geq 3$ a prime number distinct from $p$. Then there exists a Galois admissible covering $f^{\bp}: Y^{\bp} \migi X^{\bp}$ over $k$ with Galois group $\mbZ/\ell\mbZ$ such that the genus of $Y^{\bp}$ is $0$, and that there exists an irreducible component $Y_{v}$ of $Y$ satisfying $\#(Y_{v} \cap D_{Y})\geq 3$.
\end{lemma}

\begin{proof}
Suppose that $X^{\bp}$ is smooth over $k$. Then the lemma is trivial. We may suppose that $X^{\bp}$ is singular. Since the type of $X^{\bp}$ is $(0, n)$, there exists irreducible components $X_{v_{1}}$, $X_{v_{2}}$ of $X$ distinct from each other such that $\#(X_{v_{1}} \cap D_{X})\geq 2$ and  $\#(X_{v_{2}} \cap D_{X})\geq 2$. 

Let $x_{1} \in  X_{v_{1}} \cap D_{X}$,  $x_{2} \in  X_{v_{2}} \cap D_{X}$, and $$f^{\bp}: Y^{\bp} \migi X^{\bp}$$ a Galois admissible covering over $k$ with Galois group $\mbZ/\ell\mbZ$ such that $f$ is totally ramified over $x_{1}$ and $x_{2}$, and that $f$ is \'etale over $D_{X} \setminus \{x_{1}, x_{2}\}$. We see immediately that the irreducible components $Y_{v_{1}}\defeq f^{-1}(X_{v_{1}})$ and $Y_{v_{2}} \defeq f^{-1}(X_{v_{2}})$ of $Y$ satisfy the conditions $\#(Y_{v_{1}} \cap D_{Y})\geq 3$ and $\#(Y_{v_{2}} \cap D_{Y})\geq 3$, respectively. Moreover, the Riemann-Hurwitz formula implies that the genus of $Y^{\bp}$ is $0$. This completes the proof of the lemma.
\end{proof}

Next, we generalize Theorem \ref{mainthemsm} to the case where we only assume that $X_{1}^{\bp}$ is smooth over $k_{1}$. 

\begin{proposition}\label{prop-6-5}
We maintain the notation introduced above. Suppose that $X^{\bp}_{1}$ is smooth over $k_{1}$. Then $X_{1}^{\bp}$ is Frobenius equivalent to $X^{\bp}_{2}$. In particular, we have that $X_{2}^{\bp}$ is smooth over $k_{2}$, and that $X_{2}^{\bp}$ can be defined over the algebraic closure of $\mbF_{p}$ in $k_{2}$.
\end{proposition}

\begin{proof}
If $X_{2}^{\bp}$ is smooth over $k_{2}$, the proposition follows immediately from Theorem \ref{mainthemsm}. Then we may assume that $X_{2}^{\bp}$ is singular (i.e., $\#v(\Gamma_{X_{2}^{\bp}})\geq 2$).

Let $\ell\geq 3$ be a prime number distinct from $p$. Lemma \ref{lem-6-4} implies that there exists an open normal subgroup $H_{2} \subseteq \Pi_{X_{2}^{\bp}}$ such that $\Pi_{X_{2}^{\bp}}/H_{2} \cong \mbZ/\ell\mbZ$, that the Galois admissible covering $f_{H_2}^{\bp}: X^{\bp}_{H_2} \migi X_{2}^{\bp}$ corresponding to $H_{2}$ is totally ramified over two marked points of $X_{2}^{\bp}$, and that there exists $w_{H_2} \in v(\Gamma_{X_{H_{2}}^{\bp}})$ such that $\#(X_{H_{2}, w_{H_2}} \cap D_{X_{H_{2}}})\geq 3$. Write $H_{1} \defeq \phi^{-1}(H_{2}) \subseteq \Pi_{X_{1}^{\bp}}$ for the open subgroup and $f_{H_{1}}^{\bp}: X_{H_{1}}^{\bp} \migi X_{1}^{\bp}$ for the Galois admissible covering over $k_{1}$ corresponding to $H_{1}$. Theorem \ref{mainstep-1} implies that $f_{H_{1}}^{\bp}$ is totally ramified over two marked points of $X_{1}^{\bp}$, and that $n_{X_{H_{1}}}=n_{X_{H_{2}}}$. Since $f_{H_{i}}^{\bp}$ is totally ramified over two marked points, we have that $$g_{X_{H_{1}}}=g_{X_{H_{2}}}=0.$$

If we can prove the proposition holds for $X_{H_{1}}^{\bp}$, $X_{H_{2}}^{\bp}$, and $\phi|_{H_{1}}: H_{1} \migisurj H_{2}$, then we obtain that $X_{2}^{\bp}$ is also smooth over $k_{2}$. Then the proposition follows immediately from Theorem \ref{mainthemsm}. Thus, by replacing $X_{1}^{\bp}$, $X_{2}^{\bp}$, and $\phi$ by $X_{H_{1}}^{\bp}$, $X_{H_{2}}^{\bp}$, and $\phi|_{H_{1}}$, respectively, we may assume that there exists $w_{2}\in v(\Gamma_{X_{2}^{\bp}})$ such that $\#(X_{2, w_{2}} \cap D_{X_{2}})\geq 3$.

Let $e_{2, \muge}$, $e_{2, 0}$, $e_{2, 1} \in e^{\rm op}(\Gamma_{X_{2}^{\bp}}) \cap e^{\rm op}(\Gamma_{X_{2, w_{2}}^{\bp}})$ distinct from each other. Theorem \ref{mainstep-1} implies that $\phi$ induces a bijection $$\phi^{\rm sg, op}: e^{\rm op}(\Gamma_{X_{1}^{\bp}}) \isom e^{\rm op}(\Gamma_{X_{2}^{\bp}})$$ group-theoretically. We put $$e_{1,\muge}\defeq (\phi^{\rm sg, op})^{-1}(e_{2, \muge}), \ e_{1,0}\defeq (\phi^{\rm sg, op})^{-1}(e_{2, 0}), \ e_{1,1}\defeq (\phi^{\rm sg, op})^{-1}(e_{2, 1}).$$ Without loss of generality, we may assume that $$\ x_{e_{i, \muge}} \defeq \infty,\ x_{e_{i, 0}}\defeq 0, \ x_{e_{i, 1}}\defeq 1,$$ and that $$X_{1}=\mbP^{1}_{k_{1}}, \ X_{w_{2}}=\mbP^{1}_{k_{2}}.$$ 

Let $\pi_{0}(\Gamma_{X_{2}^{\bp}}\setminus \{w_{2}\})$ denote the set of connected components of $\Gamma_{X_{2}^{\bp}}\setminus \{w_{2}\}$ in $\Gamma_{X_{2}^{\bp}}$. Let $C_{2} \in \pi_{0}(\Gamma_{X_{2}^{\bp}}\setminus \{w_{2}\})$. Since $X_{2}^{\bp}$ is a pointed stable curve of type $(0, n)$ over $k_{2}$, we have that $\#(C_{2} \cap e^{\rm op}(\Gamma_{X_{2}^{\bp}}))\geq 2$. Let $e_{2, C_{2}, 1}$, $e_{2, C_{2}, 2} \in C_{2} \cap e^{\rm op}(\Gamma_{X_{2}^{\bp}})$ be open edges distinct from each other. We put $$e_{1, 2}\defeq (\phi^{\rm sg, op})^{-1}(e_{2, C_{2}, 1}) \in e^{\rm op}(\Gamma_{X_{1}^{\bp}}),$$ $$e_{1, 3}\defeq (\phi^{\rm sg, op})^{-1}(e_{2, C_{2}, 2}) \in e^{\rm op}(\Gamma_{X_{1}^{\bp}}).$$ We denote by $X_{2, C_{2}}$ the semi-stable subcurve of $X_{2}$ whose irreducible components are the irreducible components corresponding to the vertices of $\Gamma_{X_{2}^{\bp}}$ contained in $C_{2}$. Moreover, we write $e_{2, 2}$ for the unique closed edge of $\Gamma_{X_{2}^{\bp}}$ connecting $w_{2}$ and $C_{2}$. Then the node $x_{e_{2, 2}}$ corresponding to $e_{2, 2}$ is the unique closed point of $X_{2}$ contained in $X_{w_{2}} \cap X_{2, C_{2}}$. 

We put $$Z_{1}^{\bp}=(Z_{1}\defeq X_{1}, D_{Z_{1}}\defeq \{x_{e_{1, \muge}}, x_{e_{1, 0}}, x_{e_{1, 1}}, x_{e_{1, 2}}, x_{e_{1, 3}}\}),$$
$$Y_{1, 1}^{\bp}=(Y_{1,1}\defeq X_{1}, D_{Y_{1,1}}\defeq \{x_{e_{1, \muge}}, x_{e_{1, 0}}, x_{e_{1, 1}}, x_{e_{1, 2}}\}),$$ $$Y_{1,2}^{\bp}=(Y_{1,2}\defeq X_{1}, D_{Y_{1}}\defeq \{x_{e_{1, \muge}}, x_{e_{1, 0}}, x_{e_{1, 1}}, x_{e_{1, 3}}\}),$$ $$Y_{2}^{\bp}=(Y_{2}\defeq X_{w_{2}}, D_{Y_{2}} \defeq \{x_{e_{2, \muge}}, x_{e_{2, 0}}, x_{e_{2, 1}}, x_{e_{2, 2}}\}).$$ 
Moreover, we denote by $Z_{2}^{\bp}$ the pointed stable curve of type $(0, 5)$ over $k_{2}$ associated to the pointed semi-stable curve $$(X_{2}, \{x_{e_{2, \muge}}, x_{e_{2, 0}}, x_{e_{2, 1}}, x_{e_{2, C_{2}, 1}}, x_{e_{2, C_{2}, 2}}\})$$ over $k_{2}$ (i.e., the pointed stable curve obtained by contracting the $(-1)$-curves and the $(-2)$-curves of $(X_{2}, \{x_{e_{2, \muge}}, x_{e_{2, 0}}, x_{e_{2, 1}}, x_{e_{2, C_{2}, 1}}, x_{e_{2, C_{2}, 2}}\})$. We see  that $Z_{2}$ has two irreducible components $Z_{w_{2}}$ and $Z_{ C_{2}}$ such that $Z_{w_{2}}$ is equal to $X_{w_{2}}$, that $x_{e_{2, 2}}=Z_{w_{2}}\cap Z_{C_{2}}$, that $\{x_{e_{2, \muge}}, x_{e_{2, 0}}, x_{e_{2, 1}}\} \subseteq Z_{w_{2}}$, and that $\{x_{e_{2, C_{2}, 1}}, x_{e_{2, C_{2}, 2}}\} \subseteq Z_{C_{2}}$.

Next, we will see that the solvable admissible fundamental groups and the natural homomorphisms of the the solvable admissible fundamental groups of pointed stable curves constructing above can be reconstructed group-theoretically from $\phi$.
Let $$I_{1} \subseteq \Pi_{X_{1}^{\bp}}, I_{2} \subseteq \Pi_{X_{2}^{\bp}}$$ be the closed subgroups generated by the inertia subgroups of $$\bigcup_{e_{1} \in e^{\rm op}(\Gamma_{X_{1}^{\bp}})\setminus \{e_{1, \muge}, e_{1, 0},e_{1, 1},e_{1, 2}, e_{1, 3}\}}\text{Edg}_{e_{1}}^{\rm op}(\Pi_{X_{1}^{\bp}}),$$ $$\bigcup_{e_{2} \in e^{\rm op}(\Gamma_{X_{2}^{\bp}})\setminus \{e_{2, \muge}, e_{2, 0},e_{2, 1},e_{2,C_{2},1}, e_{2, C_{2}, 2}\}}\text{Edg}_{e_{2}}^{\rm op}(\Pi_{X_{2}^{\bp}}),$$ respectively, $$I_{1, 1} \subseteq \Pi_{X_{1}^{\bp}}, I_{1, 2} \subseteq \Pi_{X_{1}^{\bp}}$$ the closed subgroups generated by the inertia subgroups of $$\bigcup_{e_{1} \in e^{\rm op}(\Gamma_{X_{1}^{\bp}})\setminus \{e_{1, \muge}, e_{1, 0},e_{1, 1},e_{1, 2}\}}\text{Edg}_{e_{1}}^{\rm op}(\Pi_{X_{1}^{\bp}}),$$ $$\bigcup_{e_{1} \in e^{\rm op}(\Gamma_{X_{1}^{\bp}})\setminus \{e_{1, \muge}, e_{1, 0},e_{1, 1},e_{1, 3}\}}\text{Edg}_{e_{1}}^{\rm op}(\Pi_{X_{1}^{\bp}}),$$ respectively, and $$I_{2, 1} \subseteq \Pi_{X_{2}^{\bp}}, \ I_{2, 2} \subseteq \Pi_{X_{2}^{\bp}}$$ the closed subgroups generated by the inertia subgroups of $$\bigcup_{e_{2} \in e^{\rm op}(\Gamma_{X_{2}^{\bp}})\setminus \{e_{2, \muge}, e_{2, 0},e_{2, 1},e_{2, C_{2}, 1}\}}\text{Edg}_{e_{2}}^{\rm op}(\Pi_{X_{2}^{\bp}}),$$ $$\bigcup_{e_{2} \in e^{\rm op}(\Gamma_{X_{2}^{\bp}})\setminus \{e_{2, \muge}, e_{2, 0},e_{2, 1},e_{2, C_{2}, 2}\}}\text{Edg}_{e_{2}}^{\rm op}(\Pi_{X_{2}^{\bp}}),$$ respectively.

 Then Theorem \ref{mainstep-1} implies that $\phi(I_{1})=I_{2}$, $\phi(I_{1, 1})=I_{2, 1}$, and $\phi(I_{1, 2})=I_{2, 2}$. Moreover, we see that $\Pi_{X_{1}^{\bp}}/I_{1}$ and $\Pi_{X_{2}^{\bp}}/I_{2}$ are (outer) isomorphic to the solvable admissible fundamental groups of $Z^{\bp}_{1}$ and $Z^{\bp}_{2}$, respectively, that $\Pi_{X_{1}^{\bp}}/I_{1,1}$ and $\Pi_{X_{1}^{\bp}}/I_{1,2}$ are (outer) isomorphic to the solvable admissible fundamental groups of $Y^{\bp}_{1, 1}$ and $Y^{\bp}_{1, 2}$, respectively, and that $\Pi_{X_{2}^{\bp}}/I_{2,1}$ and $\Pi_{X_{2}^{\bp}}/I_{2,2}$ are (outer) isomorphic to the solvable admissible fundamental group of $Y_{2}^{\bp}$. Note that $I_{1, 1} \supseteq I_{1} \subseteq I_{1, 2}$ and $I_{2, 1} \supseteq I_{2} \subseteq I_{2, 2}$.

On the other hand, $\phi$ induces the following surjective open continuous homomorphisms 
$$\overline\phi: \Pi_{Z_{1}^{\bp}}\defeq \Pi_{X_{1}^{\bp}}/I_{1} \migisurj \Pi_{Z_{2}^{\bp}}\defeq \Pi_{X_{2}^{\bp}}/I_{2},$$
$$\overline\phi_{1, 1}: \Pi_{Y_{1, 1}^{\bp}}\defeq \Pi_{X_{1}^{\bp}}/I_{1, 1} \migisurj \Pi_{Y_{2}^{\bp}}\defeq \Pi_{X_{2}^{\bp}}/I_{2,1},$$$$ \overline\phi_{1, 2}: \Pi_{Y_{1, 2}^{\bp}}\defeq \Pi_{X_{1}^{\bp}}/I_{1, 2} \migisurj \Pi_{Y_{2}^{\bp}}\defeq \Pi_{X_{2}^{\bp}}/I_{2,2}$$ which fit into the following commutative diagram:
\[
\begin{CD}
\Pi_{Y_{1, 1}^{\bp}}@>\overline \phi_{1, 1}>>\Pi_{Y_{2}^{\bp}}
\\
@A\psi_{1,1}AA@A\psi_{2,1}AA
\\
\Pi_{Z_{1}^{\bp}}@>\overline \phi>>\Pi_{Z_{2}^{\bp}}
\\
@V\psi_{1,2}VV@V\psi_{2,2}VV
\\
\Pi_{Y_{1, 2}^{\bp}}@>\overline \phi_{1, 2}>>\Pi_{Y_{2}^{\bp}},
\end{CD}
\]
where $\psi_{1, 1}$, $\psi_{1,2}$, $\psi_{2,1}$, and $\psi_{2,2}$ denote the natural quotient homomorphisms.

Note that 
$\psi_{2, 1}\circ\overline \phi\neq \psi_{2, 2}\circ\overline \phi$, and that the homomorphisms of maximal prime-to-$p$ quotients of solvable admissible fundamental groups $\overline \phi^{p'}_{1,1}$, $\overline \phi^{p'}$, and $\overline \phi_{1,2}^{p'}$ induced by $\overline \phi_{1,1}$, $\overline \phi$, and $\overline \phi_{1,2}$, respectively, are isomorphisms. Moreover, we see that $\psi_{2,1}(I_{\widehat e_{2, C_{2}, 1}}) \in \text{Edg}^{\rm op}_{e_{2, 2}}(\Pi_{Y_{2}^{\bp}})$ and $\psi_{2,2}(I_{\widehat e_{2, C_{2}, 2}}) \in \text{Edg}^{\rm op}_{e_{2, 2}}(\Pi_{Y_{2}^{\bp}})$ for every $I_{\widehat e_{2, C_{2}, 1}} \in \text{Edg}^{\rm op}_{e_{2, C_{2}, 1}}(\Pi_{Z_{2}^{\bp}})$ and every $I_{\widehat e_{2, C_{2}, 2}} \in \text{Edg}^{\rm op}_{e_{2, C_{2}, 2}}(\Pi_{Z_{2}^{\bp}})$. 

Let $\widehat e_{i, 0} \in e^{\rm op}(\Gamma_{\widehat X_{i}^{\bp}})$ be an open edge over $e_{i, 0}$. By applying Theorem \ref{prop-3-12}, $$\mbF_{\widehat e_{i, 0}}\defeq (I_{\widehat e_{i, 0}}\otimes_{\mbZ} (\mbQ/\mbZ)_{i}^{p'}) \sqcup \{*_{\widehat e_{i, 0}}\}$$ admits a structure of field which can be reconstructed group-theoretically from $\Pi_{X_{i}^{\bp}}$. Since we assume that $k_{1}$ is an algebraic closure of $\mbF_{p}$, we may suppose that $k_{1}=\mbF_{\widehat e_{1, 0}}$. Moreover, we have that $\phi$ induces a field isomorphism $$\phi^{\rm fd}_{\widehat e_{1, 0}, \widehat e_{2, 0}}: \mbF_{\widehat e_{1, 0}} \isom \mbF_{\widehat e_{2, 0}}.$$ group-theoretically. By [T2, Lemma 3.4], there exists a natural number $m$ prime to $p$ such that $\mbF_{p}(\zeta_{1, m})$ contains $m$th roots of $x_{e_{1, 2}}$, $x_{e_{1, 3}}$, where $\zeta_{1, m}$ denotes a fixed primitive $m$th root of unity in $\mbF_{\widehat e_{1, 0}}$. Let $s\defeq [\mbF_{p}(\zeta_{1, m}): \mbF_{p}]$. For each $e_{1, u} \in \{e_{1, 2},  e_{1, 3}\}$, we fix an $m$th root $x_{e_{1, u}}^{\frac{1}{m}}$ in $\mbF_{\widehat e_{1, 0}}$. Then we have $$x_{e_{1, u}}^{\frac{1}{m}}=\sum_{t=0}^{s-1}b_{1, u, t}\zeta_{1, m}^{t}, \ u \in \{2, 3\},$$ where $b_{1, u, t} \in \mbF_{p}$ for each $u\in \{2, 3\}$ and each $t\in \{0, \dots, s-1\}$. Note that since $x_{e_{1, 2}}\neq x_{e_{1, 3}}$, there exists $t' \in \{0, \dots, s-1\}$ such that $b_{1, 2, t'}\neq b_{1, 3, t'}$.

Let $Z_{1} \setminus \{x_{e_{1, \muge}}\} =\spec \mbF_{\widehat e_{1, 0}}[x_{1}]$, $$f^{\bp}_{Q_{1}}: Z^{\bp}_{Q_{1}} \migi Z^{\bp}_{1}$$ the Galois admissible covering over $\mbF_{\widehat e_{1, 0}}$ with Galois group $\mbZ/m\mbZ$ determined by the equation $y_{1}^{m}=x_{1}$, and $Q_{1} \subseteq \Pi_{Z_{1}^{\bp}}$ the open normal subgroup induced by $f^{\bp}_{Q_{1}}$. Then $f_{Q_{1}}$ is totally ramified over $\{x_{e_{1, 0}}=0, x_{e_{1, \muge}}=\infty\}$ and is \'etale over $D_{Z_{1}} \setminus \{x_{e_{1, 0}}, x_{e_{1, \muge}}\}$. Note that $Z_{Q_{1}}= \mbP_{\mbF_{\widehat e_{1, 0}}}^{1}$, and that the marked points of $D_{Z_{Q_{1}}}$ over $\{x_{e_{1, 0}}, x_{e_{1, \muge}}\}$ are $\{x_{e_{Q_{1}, 0}}\defeq 0, x_{e_{Q_{1}, \muge}}\defeq \infty\}$. We put $$x_{e_{Q_{1}, u}}\defeq x_{e_{1, u}}^{\frac{1}{m}} \in D_{Z_{Q_{1}}}, \ u \in \{2, 3\},$$ and $$x_{e^{t}_{Q_{1}, 1}}\defeq \zeta^{t}_{1, m} \in D_{Z_{Q_{1}}}, \ t \in \{0, \dots, s-1\}.$$ Thus, we obtain a linear condition
$$x_{e_{Q_{1}, u}}=\sum_{t=0}^{s-1}b_{1, u, t}x_{e^{t}_{Q_{1}, 1}}, \ u \in \{2, 3\}$$ with respect to $x_{e_{Q_{1}, 0}}$ and $x_{e_{Q_{1}, \muge}}$ on $Z^{\bp}_{Q_{1}}$. 

Since $(m, p)=1$, there exists a unique open normal subgroup $Q_{2} \subseteq \Pi_{Z_{2}^{\bp}}$ such that $\overline \phi^{-1}(Q_{2})=Q_{1}$. On the other hand, we put $$Q_{1, 1} \defeq \psi_{1,1}(Q_{1}) \subseteq \Pi_{Y_{1,1}^{\bp}},$$ $$Q_{1, 2} \defeq \psi_{1,2}(Q_{1}) \subseteq \Pi_{Y_{1,2}^{\bp}},$$ $$Q_{2, 1} \defeq \psi_{2,1}(Q_{2}) \subseteq \Pi_{Y_{2}^{\bp}},$$ $$Q_{2, 2} \defeq \psi_{2,2}(Q_{2}) \subseteq \Pi_{Y_{2}^{\bp}}.$$ Note that the constructions of $Q_{1}$ and $Q_{2}$ imply that $P_{2}\defeq Q_{2,1}=Q_{2,2}$. The commutative diagram of profinite groups above induces the following commutative diagram of profinite groups:
\[
\begin{CD}
Q_{1,1}@>\overline \phi_{Q_{1, 1}}>>P_{2}
\\
@AAA@AAA
\\
Q_{1}@>\overline \phi_{Q_{1}}>>Q_{2}
\\
@VVV@VVV
\\
Q_{1, 2}@>\overline \phi_{Q_{1, 2}}>>P_{2}.
\end{CD}
\]

Let $j \in \{1, 2\}$. Write $Y^{\bp}_{Q_{1, j}}$ for the pointed stable curve over $k_{1}$ corresponding to $Q_{1, j}$. Then we see that  $e^{\rm op}(\Gamma_{Y_{Q_{1, j}}^{\bp}})$ can be regarded as a subset of $e^{\rm op}(\Gamma_{Z_{Q_{1}}^{\bp}})$. By applying Theorem \ref{mainstep-1} for $\overline \phi_{Q_{1, 1}}$ and $\overline \phi_{Q_{1, 2}}$, respectively, the commutative diagram of profinite groups above implies that we may put $$e_{P_{2}, \muge} \defeq\overline\phi^{\rm sg, op}_{Q_{1, 1}}(e_{Q_{1}, \muge})=\overline\phi^{\rm sg, op}_{Q_{1, 2}}(e_{Q_{1}, \muge}), \ e_{P_{2},0} \defeq\overline\phi^{\rm sg, op}_{Q_{1, 1}}(e_{Q_{1}, 0})=\overline\phi^{\rm sg, op}_{Q_{1, 2}}(e_{Q_{1}, 0}),$$ $$ e_{P_{2}, 1}^{t} \defeq \overline\phi^{\rm sg, op}_{Q_{1, 1}}(e_{Q_{1}, 1}^{t})=\overline\phi^{\rm sg, op}_{Q_{1, 2}}(e_{Q_{1}, 1}^{t}), \ t\in \{0, \dots, s-1\},$$ $$e_{P_{2}, 2} \defeq \overline \phi_{Q_{1, 1}}^{\rm sg, op}(e_{Q_{1}, 2})= \overline \phi_{Q_{1, 2}}^{\rm sg, op}(e_{Q_{1}, 3}).$$

We denote by $\zeta_{2, m}\defeq \phi^{\rm fd}_{\widehat e_{1, 0}, \widehat e_{2, 0}}(\zeta_{1, m})$. Then we have  $$x_{e_{Q_{2}, 1}^{t}}=\zeta_{2, m}^{t}, \ t\in \{0, \dots, s-1\}.$$ Let $Y_{P_{2}}^{\bp}$ be the pointed stable curve over $k_{2}$ corresponding to $P_{2} \subseteq \Pi_{Y_{2}^{\bp}}$. Moreover, by applying Lemma \ref{lem-6-3} for $\overline \phi_{Q_{1, 1}}$, we obtain that $$x_{e_{P_{2}, 2}}=\sum_{t=0}^{s-1}b_{1, 2, t}x_{e^{t}_{Q_{2}, 1}}$$ with respect to $x_{e_{P_{2}, 0}}$ and $x_{e_{P_{2}, \muge}}$ on $Y_{P_{2}}^{\bp}$. On the other hand,  by applying Lemma \ref{lem-6-3} for $\overline \phi_{Q_{1, 2}}$, we obtain that $$x_{e_{P_{2}, 2}}=\sum_{t=0}^{s-1}b_{1, 3, t}x_{e^{t}_{Q_{2}, 1}}$$ with respect to $x_{e_{P_{2}, 0}}$ and $x_{e_{P_{2}, \muge}}$ on $Y_{P_{2}}^{\bp}$. This means that $$\sum_{t=0}^{s-1}b_{1, 2, t}\zeta_{2, m}^{t}=\sum_{t=0}^{s-1}b_{1, 3, t}\zeta_{2, m}^{t},$$ which is impossible as $b_{1, 2, t'}\neq b_{1, 3, t'}$ for some $t' \in \{0, \dots, s-1\}$. Then we obtain that $X_{2}^{\bp}$ is smooth over $k_{2}$. Thus, the proposition follows immediately from Theorem \ref{mainthemsm}. This completes the proof of the proposition.
\end{proof}

Now, we can prove the first form of our main theorem of the present paper. 

\begin{theorem}\label{mainthem-form-1}
Let $X_{i}^{\bp}$, $i\in \{1, 2\}$, be an arbitrary pointed stable curve of type $(0, n)$ over an algebraically closed field $k_{i}$ of characteristic $p>0$ and $\Pi_{X_{i}^{\bp}}$ either the admissible fundamental group of $X_{i}^{\bp}$ or the solvable admissible fundamental group of $X_{i}^{\bp}$. Suppose that $k_{1}$ is an algebraic closure of $\mbF_{p}$. Then we have that $${\rm Hom}^{\rm open}_{\rm pro\text{-}gps}(\Pi_{X_{1}^{\bp}}, \Pi_{X_{2}^{\bp}}) \neq \emptyset$$ if and only if $X_{1}^{\bp}$ is Frobenius equivalent to $X_{2}^{\bp}$. In particular, if this is the case, we have that $X_{2}^{\bp}$ can be defined over the algebraic closure of $\mbF_{p}$ in $k_{2}$,  and that $${\rm Hom}^{\rm open}_{\rm pro\text{-}gps}(\Pi_{X_{1}^{\bp}}, \Pi_{X_{2}^{\bp}}) = {\rm Isom}_{\rm pro\text{-}gps}(\Pi_{X_{1}^{\bp}}, \Pi_{X_{2}^{\bp}}).$$ 
\end{theorem}

\begin{proof}
To verify the theorem, it is sufficient to prove the theorem when $\Pi_{X_{i}^{\bp}}$ is the solvable admissible fundamental group of $X_{i}^{\bp}$. The ``if" part of the theorem follows from \cite[Proposition 3.7]{Y7}. Let us prove the ``only if" part of the theorem. Suppose that ${\rm Hom}^{\rm open}_{\rm pro\text{-}gps}(\Pi_{X_{1}^{\bp}}, \Pi_{X_{2}^{\bp}}) \neq \emptyset$, and let $\phi \in {\rm Hom}^{\rm open}_{\rm pro\text{-}gps}(\Pi_{X_{1}^{\bp}}, \Pi_{X_{2}^{\bp}})$ be an arbitrary element. Then Lemma \ref{surj} implies that $\phi$ is a surjection.

Suppose that $X_{1}^{\bp}$ is smooth over $k_{1}$. Then the theorem follows from Proposition \ref{prop-6-5}. Thus, we may assume that $X_{1}^{\bp}$ is a singular pointed stable curve.

Note that since  $X_{1}^{\bp}$ is singular, we have $n=\#e^{\rm op}(\Gamma_{X_{1}^{\bp}})\geq 4$. We prove the theorem by induction on $\#e^{\rm op}(\Gamma_{X_{1}^{\bp}})$. Suppose that $\#e^{\rm op}(\Gamma_{X_{1}^{\bp}})=4$. Since $X_{1}^{\bp}$ is a singular pointed stable curve of type $(0, 4)$, we obtain that $\#v(\Gamma_{X_{1}^{\bp}})=2$ and $\#e^{\rm cl}(\Gamma_{X_{1}^{\bp}})=1$. On the other hand, by applying Lemma \ref{lem-6-2}, we obtain that $X_{2}^{\bp}$ is also a singular pointed stable curve of type $(0, 4)$. Thus, we have that $\#e^{\rm op}(\Gamma_{X_{2}^{\bp}})=4$, $\#v(\Gamma_{X_{2}^{\bp}})=2$, and $\#e^{\rm cl}(\Gamma_{X_{2}^{\bp}})=1$. Then $X_{1}^{\bp}$ and $X_{2}^{\bp}$ satisfy Condition C defined in Section \ref{sec-4}. Thus, by Theorem \ref{mainstep-2} and Theorem \ref{mainthemsm}, we obtain that $X_{1}^{\bp}$ is Frobenius equivalent to  $X_{2}^{\bp}$. 

Suppose that $\#e^{\rm op}(\Gamma_{X_{1}^{\bp}}) \geq 5$.  Theorem \ref{mainstep-1} implies that $\phi$ induces a bijection $$\phi^{\rm sg, op}: e^{\rm op}(\Gamma_{X_{1}^{\bp}}) \isom e^{\rm op}(\Gamma_{X_{2}^{\bp}})$$ group-theoretically. Let $e_{1, n} \in e^{\rm op}(\Gamma_{X_{1}^{\bp}})$ and $e_{2, n} \defeq \phi^{\rm sg, op}(e_{1, n})$. We denote by $$Z_{i}^{\bp}$$ the pointed stable curve of type $(0, n-1)$ over $k_{i}$ associated to the pointed semi-stable curve $(X_{i}, D_{X_{i}} \setminus \{x_{e_{i, n}}\})$ whose underlying curve $Z_{i}$ can be regarded naturally as a subcurve of $X_{i}$. Write $I_{i, n} \subseteq \Pi_{X_{i}^{\bp}}$ for the closed subgroup generated by the subgroups $\text{Edg}_{e_{i, n}}^{\rm op}(\Pi_{X_{i}^{\bp}})$. Then we see that $$\Pi_{Z_{i}^{\bp}} \defeq \Pi_{X_{i}^{\bp}}/I_{i, n}$$ is (outer) isomorphic to the solvable admissible fundamental group of $Z_{i}^{\bp}$. Moreover, Theorem \ref{mainstep-1} implies that $\phi(I_{1, n})=I_{2, n}$. Then $\phi$ induces a surjective open continuous homomorphism $$\overline \phi: \Pi_{Z_{1}^{\bp}}\migisurj \Pi_{Z_{2}^{\bp}}.$$ By induction, we obtain that $Z_{1}^{\bp}$ is Frobenius equivalent to $Z_{2}^{\bp}$. Then $\phi$ induces a bijection of dual semi-graphs $$\overline \phi^{\rm sg}: \Gamma_{Z_{1}^{\bp}} \isom \Gamma_{Z_{2}^{\bp}}.$$ In particular, we put $$\overline \phi^{\rm sg, ver}=\overline \phi^{\rm sg}|_{v(\Gamma_{Z_{1}^{\bp}})}: v(\Gamma_{Z_{1}^{\bp}}) \isom v(\Gamma_{Z_{2}^{\bp}}),$$$$\overline \phi^{\rm sg, op}=\overline \phi^{\rm sg}|_{e^{\rm op}(\Gamma_{Z_{1}^{\bp}})}: e^{\rm op}(\Gamma_{Z_{1}^{\bp}}) \isom e^{\rm op}(\Gamma_{Z_{2}^{\bp}}).$$ Note that $\Gamma_{Z_{i}^{\bp}}$
can be regarded naturally as a sub-semi-graph of $\Gamma_{X_{i}^{\bp}}$. Moreover, one of the following cases  may occur: (i) $\#v(\Gamma_{X^{\bp}_{1}})=\#v(\Gamma_{Z^{\bp}_{1}})=\#v(\Gamma_{X^{\bp}_{2}})=\#v(\Gamma_{Z^{\bp}_{2}})$; (ii) $\#v(\Gamma_{X^{\bp}_{1}})-1=\#v(\Gamma_{Z^{\bp}_{1}})=\#v(\Gamma_{X^{\bp}_{2}})-1=\#v(\Gamma_{Z^{\bp}_{2}})$;  (iii) $\#v(\Gamma_{X^{\bp}_{1}})=\#v(\Gamma_{Z^{\bp}_{1}})=\#v(\Gamma_{X^{\bp}_{2}})-1=\#v(\Gamma_{Z^{\bp}_{2}})$; (iv) $\#v(\Gamma_{X^{\bp}_{1}})-1=\#v(\Gamma_{Z^{\bp}_{1}})=\#v(\Gamma_{X^{\bp}_{2}})=\#v(\Gamma_{Z^{\bp}_{2}})$.

Suppose that either (i) or (ii) holds. Then $X_{1}^{\bp}$ and $X_{2}^{\bp}$ satisfy Condition C defined in Section \ref{sec-4}. Thus, by Theorem \ref{mainstep-2} and Theorem \ref{mainthemsm}, we obtain that $X_{1}^{\bp}$ is Frobenius equivalent to  $X_{2}^{\bp}$. 

Suppose that (iii) holds. Let $v_{2} \in v(\Gamma_{X^{\bp}_{2}})$ such that $x_{e_{2, n}} \in X_{v_{2}}$. Since $\#v(\Gamma_{X^{\bp}_{2}})=\#v(\Gamma_{Z^{\bp}_{2}})+1$, we have that $\#(X_{v_{2}} \cap D_{X_{2}})=2$. Note that $\{v_{2}\}=v(\Gamma_{X_{2}^{\bp}}) \setminus v(\Gamma_{Z_{1}^{\bp}})$.

Let $x_{e_{2, n-1}} \in X_{v_{2}} \cap D_{X_{2}}$ be the marked point distinct from $x_{e_{2, n}}$ and $e_{2, n-1}\in e^{\rm op}(\Gamma_{X_{2}^{\bp}})$ the open edge corresponding to the marked point $x_{e_{2, n-1}}$. On the other hand, let $w_{1} \in v(\Gamma_{X^{\bp}_{1}})$ such that $x_{e_{1, n}} \in X_{w_{1}}$. We put $$w_{2}\defeq \overline \phi^{\rm sg, ver}(w_{1}) \in v(\Gamma_{Z_{2}^{\bp}}) \subseteq v(\Gamma_{X_{2}^{\bp}}),$$ $$e_{1, n-1}\defeq (\phi^{\rm sg, op})^{-1}(e_{2, n-1})  \in e^{\rm op}(\Gamma_{Z_{1}^{\bp}}) \subseteq e^{\rm op}(\Gamma_{X_{1}^{\bp}}).$$
Since $Z_{1}^{\bp}$ is a pointed stable curve of type $(0, n-1)$, we have that $$\#(X_{w_{1}} \cap D_{Z_{1}})+\#(X_{w_{1}} \cap Z_{1}^{\rm sing})\geq 3.$$ Then we see that there exist marked points $x_{e_{1, n-2}}$, $x_{e_{1, n-3}} \in D_{Z_{1}} \setminus \{x_{e_{1, n-1}}\}$ distinct from each other such that the following conditions are satisfied: 

(1) if $\#(X_{w_{1}} \cap D_{Z_{1}})\geq 3$, then $x_{e_{1, n-2}}$, $x_{e_{1, n-3}} \in X_{w_{1}}$; 

(2) if $\#(X_{w_{1}} \cap D_{Z_{1}})=2$ and $x_{e_{1, n-1}} \not\in X_{w_{1}}$, then $x_{e_{1, n-2}}$, $x_{e_{1, n-3}} \in X_{w_{1}}$;

(3) if $\#(X_{w_{1}} \cap D_{Z_{1}})=1$ and $x_{e_{1, n-1}} \not\in X_{w_{1}}$, then we have that $x_{e_{1, n-3}} \in X_{w_{1}}$, and that the connected components of $Z_{1} \setminus X_{w_{1}}$ (note that since $\#(X_{w_{1}} \cap D_{Z_{1}})=1$, the cardinality of the set of connected components of $Z_{1} \setminus X_{w_{1}}$ is $\geq 2$) containing $x_{e_{2, n-1}}$ and $x_{e_{n-2}}$ respectively are distinct from each other; 

(4) if $\#(X_{w_{1}} \cap D_{Z_{1}})=2$ and $x_{e_{1, n-1}} \in X_{w_{1}}$, then we have that $x_{e_{1, n-3}} \in X_{w_{1}}$, and that $x_{e_{1, n-2}}$ is contained in a connected component of $Z_{1} \setminus X_{w_{1}}$; 

(5) if $\#(X_{w_{1}} \cap D_{Z_{1}})=1$ and $x_{e_{1, n-1}} \in X_{w_{1}}$, then we have that the connected components of $Z_{1} \setminus X_{w_{1}}$ (note that since $\#(X_{w_{1}} \cap D_{Z_{1}})=1$, the cardinality of the set of connected components of $Z_{1} \setminus X_{w_{1}}$ is $\geq 2$) containing $x_{e_{1, n-2}}$ and $x_{e_{1, n-3}}$ respectively are distinct from each other; 

(6) if $\#(X_{w_{1}} \cap D_{Z_{1}})=0$, then we have that the connected components of $Z_{1} \setminus X_{w_{1}}$ (note that since $\#(X_{w_{1}} \cap D_{Z_{1}})=0$, the cardinality of the set of connected components of $Z_{1} \setminus X_{w_{1}}$ is $\geq 3$) containing $x_{e_{1, n-1}}$, $x_{e_{1, n-2}}$, and $x_{e_{1, n-3}}$ respectively are distinct from each other. 

Write $e_{1, n-2}$ and $e_{1, n-3} \in e^{\rm op}(\Gamma_{Z_{1}^{\bp}})$ for the open edges corresponding to the marked points $x_{e_{1, n-2}}$ and $x_{e_{1, n-3}}$, respectively. We put $$e_{2, n-2} \defeq  \overline \phi^{\rm sg, op}(e_{1, n-2}), \ e_{2, n-3} \defeq  \overline \phi^{\rm sg, op}(e_{1, n-3}).$$ Let $Y_{i}^{\bp}$ be the pointed stable curve of type $(0, 4)$ over $k_{i}$ associated to the pointed semi-stable curve $$(X_{i}, \{x_{e_{i, n}}, x_{e_{i, n-1}}, x_{e_{i, n-2}}, x_{e_{i, n-3}}\}).$$ By the construction of the set of marked points $\{x_{e_{i, n}}, x_{e_{i, n-1}}, x_{e_{i, n-2}}, x_{e_{i, n-3}}\}$, we see that $Y_{1}^{\bp}$ is smooth over $k_{1}$ whose underlying curve is $X_{w_{1}}$, and that $Y_{2}^{\bp}$ is singular whose irreducible components are $X_{w_{2}}$ and $X_{v_{2}}$. 

Next, we will see that the solvable admissible fundamental groups and the natural homomorphisms of the the solvable admissible fundamental groups of pointed stable curves constructing above can be reconstructed group-theoretically from $\phi$. Let $I_{i}\subseteq \Pi_{X^{\bp}_{i}}$ be the closed subgroups generated by the subgroups $$\bigcup_{e_{i} \in e^{\rm op}(\Gamma_{X_{i}^{\bp}}) \setminus \{e_{i, n}, e_{i, n-1}, e_{i, n-2}, e_{i, n-3}\}}\text{Edg}^{\rm op}_{e_{i}}(\Pi_{X_{i}^{\bp}}).$$ We see that $$\Pi_{Y_{i}^{\bp}} \defeq \Pi_{X^{\bp}_{i}}/I_{i}$$ is (outer) isomorphic to the solvable admissible fundamental group of $Y_{i}^{\bp}$. Moreover, Theorem \ref{mainstep-1} implies that $\phi(I_{1})=I_{2}$. Then we obtain a surjective open continuous homomorphism $$\overline {\overline \phi}: \Pi_{Y_{1}^{\bp}} \migisurj \Pi_{Y_{2}^{\bp}}.$$ This contradicts Proposition \ref{prop-6-5}, since Proposition \ref{prop-6-5} implies that $Y_{2}^{\bp}$ is smooth over $k_{2}$. Then (iii) does not occur.

Suppose that (iv) holds. Similar arguments to the arguments given in the proof of (iii) imply that (iv) does not occur. More precisely, we have the following.

Let $v_{1} \in v(\Gamma_{X^{\bp}_{1}})$ such that $x_{e_{1, n}} \in X_{v_{1}}$. Since $\#v(\Gamma_{X^{\bp}_{1}})=\#v(\Gamma_{Z^{\bp}_{1}})+1$, we have that $\#(X_{v_{1}} \cap D_{X_{1}})=2$. Note that $\{v_{1}\}=v(\Gamma_{X_{1}^{\bp}}) \setminus v(\Gamma_{Z_{1}^{\bp}})$.

Let $x_{e_{1, n-1}} \in X_{v_{1}} \cap D_{X_{1}}$ be the marked point distinct from $x_{e_{1, n}}$ and $e_{1, n-1}\in e^{\rm op}(\Gamma_{X_{1}^{\bp}})$ the open edge corresponding to the marked point $x_{e_{1, n-1}}$. On the other hand, let $w_{2} \in v(\Gamma_{X^{\bp}_{2}})$ such that $x_{e_{2, n}} \in X_{w_{2}}$. We put $$w_{1}\defeq (\overline \phi^{\rm sg, ver})^{-1}(w_{2}) \in v(\Gamma_{Z_{1}^{\bp}}) \subseteq v(\Gamma_{X_{1}^{\bp}}),$$ $$e_{2, n-1}\defeq \phi^{\rm sg, op}(e_{1, n-1}) \in e^{\rm op}(\Gamma_{Z_{2}^{\bp}})\subseteq e^{\rm op}(\Gamma_{X_{2}^{\bp}}).$$
Since $Z_{2}^{\bp}$ is a pointed stable curve of type $(0, n-1)$, we have that $$\#(X_{w_{2}} \cap D_{Z_{2}})+\#(X_{w_{2}}\cap Z_{2}^{\rm sing})\geq 3.$$ Then we see that there exist marked points $x_{e_{2, n-2}}$, $x_{e_{2, n-3}} \in D_{Z_{2}} \setminus \{x_{e_{2, n-1}}\}$ distinct from each other such that the following conditions are satisfied: 

(1) if $\#(X_{w_{2}} \cap D_{Z_{2}})\geq 3$, then $x_{e_{2, n-2}}$, $x_{e_{2,n-3}} \in X_{w_{2}}$; 

(2) if $\#(X_{w_{2}} \cap D_{Z_{2}})=2$ and $x_{e_{2, n-1}} \not\in X_{w_{2}}$, then $x_{e_{2, n-2}}$, $x_{e_{2,n-3}} \in X_{w_{2}}$; 

(3) if $\#(X_{w_{2}} \cap D_{Z_{2}})=1$ and $x_{e_{2, n-1}} \not\in X_{w_{2}}$, then we have that $x_{e_{2, n-3}} \in X_{w_{2}}$, and that the connected components of $Z_{2} \setminus X_{w_{2}}$ (note that since $\#(X_{w_{2}} \cap D_{Z_{2}})=1$, the cardinality of the set of connected components of $Z_{2} \setminus X_{w_{2}}$ is $\geq 2$) containing $x_{e_{2, n-1}}$ and $x_{e_{n-2}}$ respectively are distinct from each other; 

(4) if $\#(X_{w_{2}} \cap D_{Z_{2}})=2$ and $x_{e_{2, n-1}} \in X_{w_{2}}$, then we have that $x_{e_{2, n-3}} \in X_{w_{2}}$, and that $x_{e_{2, n-2}}$ is contained in a connected component of $Z_{2} \setminus X_{w_{2}}$; 

(5) if $\#(X_{w_{2}} \cap D_{Z_{2}})=1$ and $x_{e_{2, n-1}} \in X_{w_{2}}$, then we have that the connected components of $Z_{2} \setminus X_{w_{2}}$ (note that  since $\#(X_{w_{2}} \cap D_{Z_{2}})=1$, the cardinality of the set of connected components of $Z_{2} \setminus X_{w_{2}}$ is $\geq 2$) containing $x_{e_{2, n-2}}$ and $x_{e_{n-3}}$ respectively are distinct from each other; 

(6) if $\#(X_{w_{2}} \cap D_{Z_{2}})=0$, then we have that the connected components of $Z_{2} \setminus X_{w_{2}}$ (note that since $\#(X_{w_{2}} \cap D_{Z_{2}})=0$, the cardinality of the set of connected components of $Z_{2} \setminus X_{w_{2}}$ is $\geq 3$) containing $x_{e_{2, n-1}}$, $x_{e_{n-2}}$, and $x_{e_{2, n-3}}$ respectively are distinct from each other. 

Write $e_{2, n-2}$ and $e_{2, n-3} \in e^{\rm op}(\Gamma_{Z_{2}^{\bp}})$ for the open edges corresponding to the marked points $x_{e_{2, n-2}}$ and $x_{e_{2, n-3}}$, respectively. We put $$e_{1, n-2} \defeq  (\overline \phi^{\rm sg, op})^{-1}(e_{2, n-2}), \ e_{1, n-3} \defeq  (\overline \phi^{\rm sg, op})^{-1}(e_{2, n-3}).$$ Let $Y_{i}^{\bp}$ be the pointed stable curve of type $(0, 4)$ over $k_{i}$ associated to the pointed semi-stable curve $$(X_{i}, \{x_{e_{i, n}}, x_{e_{i, n-1}}, x_{e_{i, n-2}}, x_{e_{i, n-3}}\}).$$ By the construction of the set of marked points $\{x_{e_{i, n}}, x_{e_{i, n-1}}, x_{e_{i, n-2}}, x_{e_{i, n-3}}\}$, we see that $Y_{1}^{\bp}$ is singular whose irreducible component are $X_{w_{1}}$ and $X_{v_{1}}$, and that $Y_{2}^{\bp}$ is smooth over $k_{2}$ whose underlying curve is $X_{w_{2}}$. 

Let $I_{i}\subseteq \Pi_{X^{\bp}_{i}}$ be the closed subgroups generated by the subgroups $$\bigcup_{e_{i} \in e^{\rm op}(\Gamma_{X_{i}^{\bp}}) \setminus \{e_{i, n}, e_{i, n-1}, e_{i, n-2}, e_{i, n-3}\}}\text{Edg}^{\rm op}_{e_{i}}(\Pi_{X_{i}^{\bp}}).$$ We see that $$\Pi_{Y_{i}^{\bp}} \defeq \Pi_{X^{\bp}_{i}}/I_{i}$$ is (outer) isomorphic to the solvable admissible fundamental group of $Y_{i}^{\bp}$. Moreover, Theorem \ref{mainstep-1} implies that $\phi(I_{1})=I_{2}$. Then we obtain a surjective open continuous homomorphism $$\overline {\overline \phi}: \Pi_{Y_{1}^{\bp}} \migisurj \Pi_{Y_{2}^{\bp}}.$$ This contradicts Lemma \ref{lem-6-2}, since Lemma \ref{lem-6-2} implies that $Y_{2}^{\bp}$ is singular. Then (iv) does not occur. This completes the proof of the theorem.
\end{proof}

Theorem \ref{mainthem-form-1} implies the following result concerning the Homeomorphism Conjecture which is the main theorem of the present paper.

\begin{theorem}\label{mainthem-form-2}
We maintain the notation introduced in Section \ref{sec-5}. Let $[q] \in \overline \mfM^{\rm cl}_{0, n}$ be an arbitrary closed point. Then $\pi_{0, n}^{\rm adm}([q])$ and $\pi_{0, n}^{\rm sol}([q])$ are closed points of $\overline \Pi_{0, n}$ and $\overline \Pi_{0, n}^{\rm sol}$, respectively. In particular, the Homeomorphism Conjecture and the Solvable Homeomorphism Conjecture hold when $(g, n)=(0, 4)$.  
\end{theorem}

\begin{proof}
To verify the theorem, we only need to treat the case of solvable admissible fundamental groups.

Let $V(\pi_{0, n}^{\rm sol}([q]))$ be the topological closure of $\pi_{0, n}^{\rm sol}([q])$ in $\overline \Pi_{0, n}^{\rm sol}$ and $[\pi_{1}^{\rm sol}(q')] \in V(\pi_{0, n}^{\rm sol}([q]))$ an arbitrary point. Then by Proposition \ref{prop-5-5} (a), we obtain that there exists a surjective open continuous homomorphism $$\phi: \pi_{1}^{\rm sol}(q) \migisurj \pi_{1}^{\rm sol}(q').$$ Theorem \ref{mainthem-form-1} implies that $q \sim_{fe} q'$. Thus, we obtain that $[\pi_{1}^{\rm sol}(q)]=[\pi_{1}^{\rm sol}(q')]$. This means that $V(\pi_{0, n}^{\rm sol}([q]))=[\pi_{1}^{\rm sol}(q)]$ is a closed point of $\overline \Pi_{0, n}^{\rm sol}$. Moreover, the ``in particular" part of the theorem follows  from Theorem \ref{them-5-4} (b). This completes the proof of the theorem.
\end{proof}

\begin{remarkA}
In \cite{Y8}, the author proved a similar result of Theorem \ref{mainthem-form-2} when $(g, n)=(1,1)$ and $p>2$.
\end{remarkA}

\section{Continuity of $\pi_{g, n}^{\rm adm}$}\label{sec-7}

In this section, we prove Theorem \ref{continuous}.

\subsection{Moduli spaces of curves with level structures}\label{sec-7-1}

We fix some notation. Let $p$ be a prime number, $k$ an algebraic closure of the finite field $\mbF_{p}$, and $B$, $B'$ schemes over $k$. Let $X_{B}^{\bp} = (X_{B}, D_{X_{B}})$ be a pointed stable curve of type $(g, n)$ over $B$ and $B' \migi B$ a $k$-morphism. We shall write $X_{B'}^{\bp}$ for $X^{\bp}_{B} \times _{B} B'$ the pointed stable curve over $B'$. Let $f_{B}^{\bp}: Y^{\bp}_{B} \migi X_{B}^{\bp}$ be a finite flat morphism of pointed stable curves over $B$, $b \in B$, and $\overline b \migi B$ a geometric point over $b$. We shall say $f_{B}^{\bp}$ an admissible covering (resp. a Galois admissible covering) over $B$ if $f_{\overline b}^{\bp}\defeq f^{\bp}_{B}\times_{B} \overline b: Y_{\overline b}^{\bp} \migi X_{\overline b}^{\bp}$ is an admissible covering (resp. a Galois admissible covering) over $\overline b$ for all $b \in B$.

Let $\overline \mcM_{g, n, \mbZ}$ be the moduli stack over $\spec \mbZ$  parameterizing pointed stable curves of type $(g,n)$ and $\mcM_{g, n, \mbZ}$ the open substack of $\overline \mcM_{g, n, \mbZ}$ parameterizing smooth pointed stable curves. Let $\overline \mcM_{g, n}\defeq\overline \mcM_{g, n, \mbZ}\times_{\mbZ} k$, $\mcM_{g, n}\defeq\mcM_{g, n, \mbZ}\times_{\mbZ} k$, $\overline M_{g, n}$ the coarse moduli space of $\overline \mcM_{g, n}$, and $M_{g, n}$ the coarse moduli space of $\mcM_{g, n}$. We denote by $\overline \omega_{g, n}: \overline \mcM_{g,n} \migi \overline M_{g, n}$ and $\omega_{g, n}: \mcM_{g, n} \migi M_{g, n}$ the natural morphisms, respectively. 

When $g=0$, then $\mcM_{0, n}$ is a scheme over $k$. Thus, we have $\mcM_{0, n}=M_{0, n}$. Note that $M_{0, n}$ is a quasi-projective variety over $k$. In general, the coarse moduli space $M_{g, n}$ is not a fine moduli space. In order to build a family of curves over schemes, we use the level structures. Let $m \geq 3$ be an integer number prime to $p$. 

Suppose that $g=1$. We denote by $M^{(m)}_{1, 1}$ the moduli stack over $k$ classifying smooth pointed stable curves of type $(1, 1)$ with level $m$-structure (i.e., the moduli stack of elliptic curves in characteristic $p$ with level $m$-structure). Then we have a natural morphism $\omega_{1,1}^{(m)}: M_{1, 1}^{(m)} \migi \mcM_{1, 1} \migi M_{1,1}.$ Moreover, we put $$M_{1, n}^{(m)}\defeq M_{1, 1}^{(m)}\times_{M_{1,1}}M_{1, n},$$ where $M_{1, n} \migi M_{1, 1}$ is the natural morphism induced by the forgetting morphism $\mcM_{1, n} \migi \mcM_{1,1}$ determined by forgetting the last $n-1$ marked points. Then we obtain a morphism $$\omega_{1, n}^{(m)}: M_{1, n}^{(m)} \migi M_{1, n}$$ determined by the second projection. Note that $M^{(m)}_{1, n}$ is a quasi-projective variety over $k$, and that $M_{1, n}^{(m)}(B)$ is the set of $B$-isomorphism classes of smooth pointed stable curves of type $(1, n)$ over $B$ such that, by forgetting the last $n-1$ marked points, the smooth pointed stable curves of type $(1, 1)$ are elliptic curves over $B$ with level $m$-structure.

Suppose that $g \geq 2$. Let $M^{(m)}_{g, 0}$ be the moduli stack over $k$ classifying smooth pointed stable curves of type $(g, 0)$ with level $m$-structure. Then we obtain a natural morphism $\omega^{(m)}_{g, 0}: M_{g, 0}^{(m)} \migi \mcM_{g, 0} \migi M_{g, 0}.$ Moreover, we put $$M_{g, n}^{(m)}\defeq M_{g, 0}^{(m)}\times_{M_{g, 0}} M_{g, n},$$ where $M_{g, n} \migi M_{g, 0}$ is the natural morphism induced by the forgetting morphism $\mcM_{g, n} \migi \mcM_{g, 0}$ determined by forgetting marked points. Then we obtain a  morphism $$\omega_{g, n}^{(m)}: M_{g, n}^{(m)} \migi M_{g, n}$$ determined by the second projection.  Note that $M^{(m)}_{g, n}$ is a quasi-projective variety over $k$, and that $M_{g, n}^{(m)}(B)$ is the set of $B$-isomorphism classes of smooth pointed stable curves of type $(g, n)$ over $B$ whose underlying curves are smooth projective curves of genus $g$ over $B$ with level $m$-structure.

For simplicity, we shall write $H_{g, n}$ for $M_{g, n}^{(m)}$ when $g \geq 1$, $H_{0, n}$ for $M_{0, n}$ when $g=0$, $\omega_{g, n}^{(m)}$ for the natural morphism $H_{g, n}\defeq M_{g, n}^{(m)} \migi M_{g, n}$ (note that $\omega_{0, n}^{(m)}=\text{id}_{M_{0, n}}: M_{0, n} \migi M_{0,n}$ when $g=0$). Moreover, we denote by  $$X^{\bp}_{H_{g, n}}=(X_{H_{g, n}}, D_{X_{H_{g, n}}})$$ the universal smooth pointed stable curve of type $(g, n)$ over $H_{g, n}$ with a level $m$-structure $\tau_{H_{g, n}}\defeq\tau_{H_{g, 0}}\times_{H_{g, 0}} H_{g, n}$ induced by the universal level $m$-structure $$\tau_{H_{g, 0}}: \text{Pic}_{X_{H_{g, 0}}/H_{g, 0}}^{0}[m] \isom (\mbZ/m\mbZ)^{2g}_{H_{g, 0}}$$ when $g\geq 2$, by the universal level $m$-structure $$\tau_{H_{1, 1}}: \text{Pic}_{X_{H_{1, 1}}/H_{1, 1}}^{0}[m] \isom (\mbZ/m\mbZ)^{2}_{H_{1, 1}}$$ when $g=1$, and by the trivial level $m$-structure when $g=0$.

\subsection{The sets of finite quotients of admissible fundamental groups}\label{sec-7-2}

We maintain the notation introduced in Section \ref{sec-5} and Section \ref{sec-7-1}. Let $q \in \overline M_{g, n}$ be an arbitrary point, $\pi^{\rm adm}_{1}(q)$ the admissible fundamental group of the pointed stable curve $X_{q}^{\bp}$ over an algebraic closure of the residue field $k(q)$ of $q$, $\Gamma_{q}$ the dual semi-graph of $X_{q}^{\bp}$, and $\pi_{A}^{\rm adm}(q)$ the set of finite quotients of $\pi_{1}^{\rm adm}(q)$. Since $\pi_{1}^{\rm adm}(q)$ is topologically finitely generated, the isomorphism class of $\pi_{1}^{\rm adm}(q)$ is determined completely by $\pi_{A}^{\rm adm}(q)$ (cf. \cite[Proposition 16.10.6]{FJ}). First, we have the following lemmas.

\begin{lemma}\label{lem-7-1}
Let $q_{1}, q_{2} \in \overline M_{g, n}$ be arbitrary points such that $q_{2} \in \overline {\{q_{1}\}}$, where $\overline {\{(-)\}}$ denotes the topological closure of $(-)$ in $\overline M_{g, n}$. Then we have that $$\pi_{A}^{\rm adm}(q_{2}) \subseteq \pi_{A}^{\rm adm}(q_{1}).$$
\end{lemma}

\begin{proof}
The lemma follows immediately from the specialization theorem of admissible fundamental groups of pointed stable curves (e.g. \cite[Th\'eor\`eme 2.2]{V}).
\end{proof}

\begin{lemma}\label{lem-7-2}
Let $S$ be a smooth variety over $k$, $\eta_{S}$ the generic point of $S$, and $X_{S}^{\bp}$ a smooth pointed stable curve over $S$. Let $Y^{\bp}_{\eta_{S}}$ be a smooth pointed stable curve over $\eta_{S}$ and $$f^{\bp}_{\eta_{S}}: Y_{\eta_{S}}^{\bp}\migi X_{\eta_{S}}^{\bp}$$ a Galois admissible covering over $\eta_{S}$. Then there exist an open subset $U \subseteq S$ and a Galois admissible covering $$f_{U}^{\bp}: Y^{\bp}_{U} \migi X^{\bp}_{U}$$ of smooth pointed stable curves over $U$ such that $f^{\bp}_{U}\times_{U}\eta_{S}=f^{\bp}_{\eta_{S}}$.
\end{lemma}

\begin{proof}
Write $Y_{S}$ for the normalization of $X_{S}$ in the function field of $Y_{\eta_{S}}$ and $D_{Y_{S}}$ for the set of the topological closures of points of $D_{Y_{\eta_{S}}}$ in $Y_{S}$. \cite[Proposition 5]{Har} implies that, by replacing $S$ by an open subset of $S$, we may assume that the fiber $Y_{s}\defeq Y_{S} \times_{S}s$ is geometrically irreducible over every {\it closed} point $s \in S$. 

The normalization $f_{S}: Y_{S} \migi X_{S}$ induces a morphism $$g_{S}\defeq f_{S}|_{Y_{S} \setminus D_{Y_{S}}}: Y_{S} \setminus D_{Y_{S}} \migi X_{S} \setminus D_{X_{S}}$$ over $S$. Since the restriction of $g_{S}$ on the generic fiber $\eta_{S}$ is a Galois \'etale covering over $\eta_{S}$, there exists an open subset $U \subseteq S$ such that $g_{U}\defeq g_{S}|_{U}$ is a Galois \'etale covering over $U$. Thus, by replacing $S$ by  $U$, we may assume that $g_{S}$ is a Galois \'etale covering. Since the fiber $Y_{s}\defeq Y_{S} \times_{S}s$ is generically smooth over each $s \in S$,  $Y_{s}$ is geometrically irreducible over each point $s \in S$.

The normalization $f_{S}: Y_{S} \migi X_{S}$ induces a morphism $$g_{S}\defeq f_{S}|_{Y_{S} \setminus D_{Y_{S}}}: Y_{S} \setminus D_{Y_{S}} \migi X_{S} \setminus D_{X_{S}}$$ over $S$. Since the restriction of $g_{S}$ on the generic fiber $\eta_{S}$ is \'etale, there exists an open subset $U \subseteq S$ such that $$g_{u}: Y_{S} \setminus D_{Y_{S}} \times_{S} u \migi X_{S} \setminus D_{X_{S}} \times_{S} u$$ is \'etale for each $u \in U$. Thus, by replacing $S$ by the open subset $U$, we may assume that $g_{S}$ is \'etale. Since the fiber $Y_{s}\defeq Y_{S} \times_{S}s$ is generically smooth over each point $s \in S$, $Y_{s}$ is geometrically irreducible over each point $s \in S$.

Let $X^{\log}_{S}$ be the log scheme over $S$ whose underlying scheme is $X_{S}$, and whose log structure is determined by the marked points $D_{X_{S}}$. Since $S$ is smooth over $k$, we see that $X^{\log}_{S}$ is log regular. Note that $f_{S}$ is tamely ramified over the generic point of every marked point of $D_{X_{S}}$. Then the log purity (cf. \cite[Theorem B]{M-1}) implies that $g_{S}$ extends uniquely to a Galois log \'etale morphism $f_{S}^{\log}: Y^{\log}_{S} \migi X^{\log}_{S}$ over $S$. Let $Y^{\bp}_{S}\defeq(Y_{S}, D_{S})$. Then $Y^{\bp}_{S}$ is a smooth pointed stable curve over $S$. Thus, $f^{\log}_{S}$ induces a morphism $f^{\bp}_{S}: Y^{\bp}_{S} \migi X^{\bp}_{S}$ such that the restriction of $f^{\bp}_{S}$ on $\eta_{S}$ is equal to $f^{\bp}_{\eta_{S}}$, and that $f_{s}: Y^{\bp}_{s} \migi X^{\bp}_{s}$ induced by $f^{\bp}_{S}$ is a Galois admissible covering over every point $s\in S$.
\end{proof}

\begin{lemma}\label{lem-7-3}
Let $q \in M_{g, n}$ be an arbitrary point, $V_{q}^{\rm sm}$ the topological closure of $q$ in $M_{g, n}$, and $C \subseteq V_{q}^{\rm sm, cl}$ a subset of closed points of $V_{q}^{\rm sm}$. Suppose that $C$ is dense in $V_{q}^{\rm sm}$. Then we have that $$\pi_{A}^{\rm adm}(q)=\bigcup_{c\in C}\pi_{A}^{\rm adm}(c).$$
\end{lemma}

\begin{proof}
If $q$ is a closed point, then the lemma is trivial. We may assume that $q$ is not a closed point. Lemma \ref{lem-7-1} implies that, to verify the lemma, it is sufficient to prove that, for any $G \in \pi_{A}^{\rm adm}(q)$, there exists a closed point $c \in C$ such that $G \in \pi_{A}^{\rm adm}(c)$.

Let $q^{(m)} \in (\omega^{(m)}_{g, n})^{-1}(q)$ be a point of $H_{g, n}$, $V_{q^{(m)}}$ the topological closure of $q^{(m)}$ in $H_{g, n}$, and $k(q^{(m)})$ the residue field of $q^{(m)}$ which is the function field of $V_{q^{(m)}}$. Write $M'$ for the normalization of $V_{q^{(m)}}$ in $k(q^{(m)})$. Then there exists an open subset of $M \subseteq M'$ such that $M$ is smooth over $k$. Moreover, the composition of the natural morphisms $M \migiinje M' \migi V_{q^{(m)}} \migiinje H_{g, n}$ determines a smooth pointed stable curve $X^{\bp}_{M}\defeq X^{\bp}_{H_{g, n}}\times_{H_{g, n}} M$ over $M$.

Let $k_{q}$ be an algebraic closure of $k(q^{(m)})$. By the construction, $k_{q}$ is also an algebraic closure of the residue field $k(q)$ of $q$. Let $Y^{\bp}_{k_{q}} \migi X^{\bp}_{k_{q}}$ be a Galois admissible covering over $k_{q}$ with Galois group $G$. By replacing $k(q^{(m)})$ by a finite extension $l$ of $k(q^{(m)})$, the Galois admissible covering can be descended to a Galois admissible covering $Y^{\bp}_{l} \migi X^{\bp}_{l}$ over $l$ with Galois group $G$. Write $N$ for the normalization of $M$ in $l$, $X_{N}^{\bp}$ for $X^{\bp}_{M}\times_{M} N$, and $Y^{\bp}_{N}$ for the normalization of $X^{\bp}_{N}$ in the function field of $Y^{\bp}_{l}$. Then we obtain a covering $$Y^{\bp}_{N} \migi X^{\bp}_{N}$$ over $N$ such that the base change via the natural morphism $\spec l \migi N$ is the Galois admissible covering $Y^{\bp}_{l} \migi X^{\bp}_{l}$ over $l$ with Galois group $G$. Since $N$ is generically smooth over $k$, by replacing $N$ by an open subset of $N$, we may assume that $N$ is smooth over $k$. Thus, Lemma \ref{lem-7-2} implies that there exists an open subset $U \subseteq N$ such that the morphism $$Y^{\bp}_{U} \defeq Y_{N}^{\bp} \times_{N} U \migi X_{U}^{\bp} \defeq X_{N}^{\bp} \times_{N} U$$ is a Galois admissible covering over $U$ with Galois group $G$.

We denote by $$W \subseteq V^{\rm sm}_{q}$$ the image of $U$ of the composition of the natural morphisms $U \migiinje N  \migi M \migiinje M' \migi V_{q^{(m)}} \migiinje H_{g, n} \migi M_{g, n}$, which is a dense constructible subset of $V^{\rm sm}_{q}$. Then $W$ contains an open subset $W'$ of $V_{q}^{\rm sm}$. Since $C$ is dense in $V_{q}^{\rm sm}$, we obtain that $W \cap C \neq \emptyset$. This means that, there exists a closed point $c \in C$ such that $G \in \pi_{A}^{\rm adm}(c)$. We complete the proof of the lemma.
\end{proof}

We maintain the notation introduced in the proof of Lemma \ref{lem-7-3}. We shall denote by $$U_{q} \subseteq U$$ the inverse image of $W'$ of the composition of the natural morphisms $U \migiinje N \migi M \migiinje M' \migi V_{q^{(m)}} \migiinje H_{g, n} \migi M_{g, n}$, which is an open subset of $U$. Then the proof of Lemma \ref{lem-7-3} implies the following corollary.

\begin{corollary}\label{coro-7-4}
We maintain the notation introduced in the proof of Lemma \ref{lem-7-3}. Let $q \in M_{g, n}$ be an arbitrary point, $k_{q}$ an algebraic closure of the residue field $k(q)$, $V_{q}^{\rm sm}$ the topological closure of $q$ in $M_{g, n}$, and $f^{\bp}_{k_{q}}: Y^{\bp}_{k_{q}} \migi X^{\bp}_{k_{q}}$ a Galois admissible covering over $k_{q}$ with Galois group $G$. Then there exist a smooth $k$-variety $U_{q}$ and a quasi-finite morphism $U_{q} \migi H_{g, n}$ (not necessary a surjection) such that the following conditions are satisfied:

(i) the image of $U_{q}$ of the composition of the natural  morphisms $U_{q} \migi H_{g, n}\overset{\omega_{g, n}^{(m)}}{\migi} M_{g, n}$ is an open subset of $V_{q}^{\rm sm}$;

(ii) the morphism $U_{q} \migi H_{g, n}$ induces a smooth pointed stable curve $$X^{\bp}_{U_{q}}\defeq X^{\bp}_{H_{g, n}} \times_{H_{g, n}} U_{q}$$ over $U_{q}$ with a level $m$-structure $\tau_{U_{q}}\defeq\tau_{H_{g, n}} \times_{H_{g, n}} U_{q}$;

(iii) there exists a Galois admissible covering $f^{\bp}_{U_{q}}: Y^{\bp}_{U_{q}} \migi X^{\bp}_{U_{q}}$ of smooth pointed stable curves over $U_{q}$ with Galois group $G$ such that $f_{U_{q}}^{\bp}\times _{U_{q}}\spec k_{q}=f^{\bp}_{k_{q}}$.
\end{corollary}

In the remainder of this subsection, we will generalize Lemma \ref{lem-7-3} to the case where $q \in \overline M_{g, n}$.

\begin{lemma}\label{lem-7-5}
Let $S$ be a $k$-variety and $s_{1}, s_{2} \in S$ two points of $S$ such that $s_{1} \neq s_{2}$ and $s_{2} \in \overline {\{s_{1}\}}$. Then there exist a complete discrete valuation ring $R$ and a morphism $\spec R \migi S$ such that the image of the morphism is $\{s_{1}, s_{2}\}$.
\end{lemma}

\begin{proof}
The lemma follows immediately from elementary algebraic geometry.
\end{proof}




\begin{lemma}\label{lem-7-6}
Let $R$ be a complete discrete valuation ring, $K_{R}$ the quotient field of $R$, and $k_{R}$ the residue field of $R$ such that $k_{R}$ is an algebraically closed field containing $k$. Let $$f^{\bp}_{K_R}: Y^{\bp}_{K_R} \migi X^{\bp}_{K_R}$$ be a Galois admissible covering over $K_R$ with Galois group $G$. Write $\Gamma_{X_{K_{R}}^{\bp}}$ for the dual semi-graph of $X^{\bp}_{K_R}$ and $\Gamma_{Y^{\bp}_{K_{R}}}$ for the dual semi-graph of $Y^{\bp}_{K_{R}}$. Suppose that $\widetilde X^{\bp}_{K_{R}, v_{X}}$, $v_{X} \in v(\Gamma_{X_{K_{R}}^{\bp}})$, and $\widetilde Y^{\bp}_{K_{R}, v_{Y}}$, $v_{Y} \in v(\Gamma_{Y_{K_{R}}^{\bp}})$, have good reduction over $R$, where $\widetilde X^{\bp}_{K_{R}, v_{X}}$ and $\widetilde Y^{\bp}_{K_{R}, v_{Y}}$ denote the smooth pointed stable curves of types $(g_{v_{X}}, n_{v_{X}})$ and $(g_{v_{Y}}, n_{v_{Y}})$ associated to $v_{X}$ and $v_{Y}$, respectively. Then there exists a Galois admissible covering $$f^{\bp}_{R}: Y^{\bp}_{R} \migi X^{\bp}_{R}$$ over $R$ with Galois group $G$ such that $f_{K_{R}}^{\bp}= f_{R}^{\bp}\times_{R} K_{R}$.
\end{lemma}

\begin{proof}
The smooth pointed stable curve $\widetilde Y^{\bp}_{K_{R}, v_{Y}}$ over $K_{R}$ determines a morphism $$c_{v_{Y}}: \spec K_{R} \migi \mcM_{g_{v_{Y}}, n_{v_{Y}}}.$$ Write $c_{Y_{K_{R}}}: \spec K_{R} \migi \overline \mcM_{g_{Y}, n_{Y}}$ for the morphism determined by $Y^{\bp}_{K_{R}}$ over $K_{R}$, where $(g_{Y}, n_{Y})$ denotes the type of $Y^{\bp}_{K_{R}}$. Then the pointed stable curve $Y^{\bp}_{K_{R}}$ determines a clutching morphism $$\kappa_{Y_{K_{R}}}: \bigtimes_{v_{Y} \in v(\Gamma_{Y^{\bp}_{K_{R}}})} \mcM_{g_{v_{Y}}, n_{v_{Y}}} \migi \overline \mcM_{g_{Y}, n_{Y}}$$ such that  $$\kappa_{Y_{K_{R}}}\circ  (\bigtimes_{v_{Y} \in v(\Gamma_{Y^{\bp}_{K_{R}}})}c_{v_{Y}})=c_{Y_{K_{R}}}.$$ We denote by $Y_{R, v_{Y}}^{\bp}$ the pointed stable model of $Y_{K_{R}, v_{Y}}^{\bp}$ over $R$ which is a smooth pointed stable curve of type $(g_{v_{Y}}, n_{v_{Y}})$ over $R$. By applying the clutching morphism $\kappa_{Y_{K_{R}}}$, we may glue the pointed stable curves $\{Y_{R, v_{Y}}^{\bp}\}_{v_{Y} \in v(\Gamma_{Y^{\bp}_{K_{R}}})}$ in a way that is compatible with the
gluing of $\{Y_{K_{R}, v_{Y}}^{\bp}\}_{v_{Y} \in v(\Gamma_{Y^{\bp}_{K_{R}}})}$ that gives rise to $Y^{\bp}_{K_{R}}$. Then we obtain a pointed stable curve $Y^{\bp}_{R}$ over $R$.

Since $Y_{K_{R}}^{\bp}$ admits an action of $G$, we obtain an action of $G$ on the pointed stable model $Y^{\bp}_{R}$. Let $Z^{\bp}_{R}\defeq Y^{\bp}_{R}/G$, $f^{\bp}_{R}: Y^{\bp}_{R} \migi Z^{\bp}_{R}$ the quotient morphism, $Z^{\bp}_{K_{R}}$ the generic fiber over $K_{R}$, and $Z^{\bp}_{k_{R}}$ the special fiber over $k_{R}$. \cite[Appendice, Corollaire]{R0} (or \cite[Proposition 10.3.48]{L}) implies that $Z^{\bp}_{R}$ is a pointed semi-stable curve over $R$. Moreover, since $f_{K_{R}}^{\bp}$ is a Galois admissible covering over $K_{R}$ with Galois group $G$, $Z^{\bp}_{K_{R}}$ is isomorphic to $X^{\bp}_{K_{R}}$ over $K_{R}$. 

On the other hand, write $f^{\rm sg}_{K_{R}}: \Gamma_{Y^{\bp}_{K_{R}}} \migi \Gamma_{X^{\bp}_{K_{R}}}$ for the map of dual semi-graphs induced by $f^{\bp}_{K_{R}}$. Note that, for every $v_{X} \in v(\Gamma_{X^{\bp}_{K_{R}}})$ and every $v_{Y} \in (f^{\rm sg}_{K_{R}})^{-1}(v_{X})$, $f_{K_{R}}^{\bp}$ can be extended to a Galois admissible covering $f^{\bp}_{R, v_{Y}, v_{X}}: Y_{R,v_{Y}}^{\bp} \migi X^{\bp}_{R, v_{X}}$ of smooth pointed stable curves over $R$ with Galois group $G$. Then we obtain $$Y^{\bp}_{R, v_{Y}}/G \cong X^{\bp}_{R, v_{X}}$$ over $R$. This implies that $Z^{\bp}_{k_{R}}$ is a pointed stable curve over $k_{R}$. Then we have $X^{\bp}_{R} \cong Z^{\bp}_{R}$ over $R$. We complete the proof of the lemma.
\end{proof}



\begin{proposition}\label{pro-7-7}
Let $q \in \overline M_{g, n}$ be an arbitrary point, $V_{q}$ the topological closure of $q$ in $\overline M_{g, n}$, and $G \in \pi_{A}^{\rm adm}(q)$ a finite group. Then there exists a closed point $c \in V_{q}^{\rm cl}$ such that $\Gamma_{q}$ is isomorphic to $\Gamma_{c}$ as dual semi-graphs, and that $G \in \pi_{A}^{\rm adm}(c)$.
\end{proposition}

\begin{proof}
If $q$ is a closed point, then the proposition is trivial. To verify the proposition, we may assume that $q$ is not a closed point. If $q \in M_{g, n}$, then the proposition follows from Lemma \ref{lem-7-3}. Then we may assume that $q \in \overline M_{g, n} \setminus M_{g, n}$.

Let $k_{q}$ be an algebraic closure of the residue field $k(q)$ of $q$. The natural morphism $\spec k_q \migi \spec k(q) \migi \overline M_{g, n}$ determines a pointed stable curve $X^{\bp}_{k_q}$ over $k_q$. Let $\widetilde X^{\bp}_{k_{q}, v_{X}}$, $v_{X} \in v(\Gamma_{q})$, be the smooth pointed stable curve of type $(g_{v_{X}}, n_{v_{X}})$ associated to $v_{X}$.

Let $Y^{\bp}_{k_{q}}$ be a pointed stable curve of type $(g_{Y}, n_{Y})$ over $k_q$, $f_{k_{q}}^{\bp}: Y^{\bp}_{k_{q}} \migi X^{\bp}_{k_{q}}$ a Galois admissible covering over $k_{q}$ with Galois group $G$, $\Gamma_{Y^{\bp}_{k_q}}$ the dual semi-graph of $Y^{\bp}_{k_{q}}$, and $f^{\rm sg}_{k_{q}}: \Gamma_{Y^{\bp}_{k_q}} \migi \Gamma_{q}$ the map of dual semi-graphs induced by $f_{k_q}^{\bp}$. For each $v_{X} \in v(\Gamma_{q})$, write $\widetilde Y^{\bp}_{k_q, v_{Y}}$, $v_{Y} \in (f^{\rm sg}_{k_{q}})^{-1}(v_{X})$, for the smooth pointed stable curve of type $(g_{v_{Y}}, n_{v_{Y}})$ associated to $v_{Y}$. Then $f_{k_{q}}^{\bp}$ induces a Galois multi-admissible covering $$f^{\bp}_{k_{q}, v_{X}}: \bigsqcup_{v_{Y} \in (f^{\rm sg}_{k_{q}})^{-1}(v_{X})} \widetilde Y_{k_{q}, v_{Y}}^{\bp} \migi \widetilde X^{\bp}_{k_{q}, v_{X}}$$ over $k_{q}$ with Galois group $G$. Note that $\bigsqcup_{v_{Y} \in (f^{\rm sg}_{k_{q}})^{-1}(v_{X})}\widetilde Y_{k_{q}, v_{Y}}^{\bp}$ admits an action of $G$ induced by the action of $G$ on $Y^{\bp}_{k_{q}}$. This action induces an action of $G$ on the set $(f^{\rm sg}_{k_{q}})^{-1}(v_{X})$. For each $v_{Y} \in (f^{\rm sg}_{k_{q}})^{-1}(v_{X})$, write $G_{v_{Y}}$ for the inertia subgroup of $v_{Y}$. Then we obtain a Galois admissible covering $$f_{k_{q}, v_{Y}, v_{X}}^{\bp}: \widetilde Y^{\bp}_{k_{q}, v_{Y}} \migi \widetilde X^{\bp}_{k_{q}, v_{X}}, \ v_{Y} \in (f^{\rm sg}_{k_{q}})^{-1}(v_{X}),$$ over $k_{q}$ with Galois group $G_{v_{Y}}$. 

The pointed stable curves $X^{\bp}_{k_q}$, $\{\widetilde X^{\bp}_{k_{q}, v_{X}}\}_{v_{X} \in v(\Gamma_{q})}$, $Y^{\bp}_{k_{q}}$, and $\{\widetilde Y^{\bp}_{k_{q}, v_{Y}}\}_{v_{Y} \in v(\Gamma_{Y^{\bp}_{k_q}})}$ over $k_q$ determine morphisms $c_{X_{k_q}}: \spec k_q \migi \overline \mcM_{g, n}$, $\{c_{v_{X}}: \spec k_{q} \migi \mcM_{g_{v_{X}}, n_{v_{X}}}\}_{v_{X}\in v(\Gamma_{q})}$, $c_{Y_{k_q}}: \spec k_q \migi \overline \mcM_{g_{Y}, n_{Y}}$, and $\{c_{v_{Y}}: \spec k_q \migi \mcM_{g_{v_{Y}}, n_{v_{Y}}}\}_{v_{Y} \in v(\Gamma_{Y^{\bp}_{k_q}})}$, respectively. Then the pointed stable curves $X^{\bp}_{k_{q}}$ and $Y^{\bp}_{k_{q}}$ over $k_{q}$ induce clutching morphisms
$$\kappa_{X_{k_q}}: \bigtimes_{v_{X}\in v(\Gamma_{q})}\mcM_{g_{v_{X}}, n_{v_{X}}} \migi \overline \mcM_{g, n},$$
$$\kappa_{Y_{k_q}}: \bigtimes_{v_{Y}\in v(\Gamma_{Y^{\bp}_{k_q}})}\mcM_{g_{v_{Y}}, n_{v_{Y}}} \migi \overline \mcM_{g_{Y}, n_{Y}},$$ respectively, such that $$\kappa_{X_{k_q}}\circ (\bigtimes_{v_{X} \in v(\Gamma_{q})} c_{v_{X}})=c_{X_{k_q}},$$ $$\kappa_{Y_{k_q}}\circ (\bigtimes_{v_{Y} \in v(\Gamma_{Y^{\bp}_{k_q}})} c_{v_{Y}})=c_{Y_{k_q}}.$$

On the other hand, the smooth pointed stable curve $\widetilde X_{k_{q}, v_{X}}^{\bp}$, $v_{X} \in v(\Gamma_{q})$, over $k_{q}$ determines a morphism $\spec k_{q} \migi M_{g_{v_{X}}, n_{v_{X}}}$, and we denote by $q_{v_{X}} \in M_{g_{v_{X}}, n_{v_{X}}}$ the image of the morphism. Write $V_{q_{v_{X}}}^{\rm sm}$ for the topological closure of $q_{v_{X}}$ in $M_{g_{v_{X}}, n_{v_{X}}}$. Let $k_{q_{v_{X}}}$ be an algebraic closure of the residue field $k(q_{v_{X}})$ of $q_{v_{X}}$. Since the admissible coverings over algebraically closed fields do not depend on the choices of base fields, for each $v_{Y} \in (f^{\rm sg}_{k_{q}})^{-1}(v_{X})$, $f^{\bp}_{k_{q}, v_{Y}, v_{X}}$ induces a $G_{v_{Y}}$-Galois admissible covering $$f^{\bp}_{k_{q_{v_{X}}}, v_{Y}, v_{X}}: \widetilde Y^{\bp}_{k_{q_{v_{X}}}, v_{Y}} \migi \widetilde X^{\bp}_{k_{q_{v_{X}}}, v_{X}}$$ over $k_{q_{v_{X}}}$. Then Corollary \ref{coro-7-4} implies that there exist a smooth $k$-variety $U_{q_{v_{X}}}$ and a quasi-finite morphism $U_{q_{v_{X}}} \migi H_{g_{v_{X}}, n_{v_{X}}}$ (not necessary a surjection) such that the following conditions are satisfied:

(i) the image of $U_{q_{v_{X}}}$ of the composition of the morphisms $U_{q_{v_{X}}} \migi H_{g_{v_{X}}, n_{v_{X}}}\overset{\omega_{g_{v_{X}}, n_{v_{X}}}^{(m)}}{\migi} M_{g_{v_{X}}, n_{v_{X}}}$ is an open subset of $V_{q_{v_{X}}}^{\rm sm}$;

(ii) the morphism $U_{q_{v_{X}}} \migi H_{g_{v_{X}}, n_{v_{X}}}$ induces a smooth pointed stable curve $$X^{\bp}_{U_{q_{v_{X}}}, v_{X}}\defeq X^{\bp}_{H_{g_{v_{X}}, n_{v_{X}}}} \times_{H_{g_{v_{X}}, n_{v_{X}}}} U_{q_{v_{X}}}$$ over $U_{q_{v_{X}}}$ with a level $m$-structure $\tau_{U_{q_{v_{X}}}}\defeq\tau_{H_{g_{v_{X}}, n_{v_{X}}}} \times_{H_{g_{v_{X}}, n_{v_{X}}}} U_{q_{v_{X}}}$;

(iii) for each $v_{Y} \in (f^{\rm sg}_{k_{q}})^{-1}(v_{X})$, there exists a Galois admissible covering $$f^{\bp}_{U_{q_{v_{X}}}, v_{Y}, v_{X}}: Y^{\bp}_{U_{q_{v_{X}}}, v_{Y}} \migi X^{\bp}_{U_{q_{v_{X}}}, v_{X}}$$ of smooth pointed stable curves over $U_{q_{v_{X}}}$ with Galois group $G_{v_{Y}}$ such that $f_{U_{q_{v_{X}}}, v_{Y}, v_{X}}^{\bp}\times_{U_{q_{v_{X}}}}\spec k_{q_{v_{X}}}$ is equal to $f^{\bp}_{k_{q_{v_{X}}}, v_{X}, v_{Y}}$.

Then the clutching morphism induces a morphism $$\kappa: \bigtimes_{v_{X} \in v(\Gamma_{q})}U_{q_{v_{X}}} \migi \bigtimes_{v_{X} \in v(\Gamma_{q})} H_{g_{v_{X}}, n_{v_{X}}} \migi \bigtimes_{v_{X}\in v(\Gamma_{q})}\mcM_{g_{v_{X}}, n_{v_{X}}} \overset{\kappa_{X_{k_q}}}{\migi} \overline \mcM_{g, n} \overset{\overline \omega_{g, n}}{\migi} \overline M_{g, n}$$ over $k$. Since the image of $\kappa$ is a dense constructible subset of $V_{q}$, the image of $\kappa$ contains an open subset $$W \subseteq V_{q}.$$

Let $c$ be a closed point of $W$. Note that $\Gamma_{c}$ is isomorphic to $\Gamma_{q}$ as dual semi-graphs. Then Lemma \ref{lem-7-5} implies that there exist a complete discrete valuation ring $R$ with algebraically closed residue field and a morphism $\spec R \migi W$ such that the image of the morphism is $\{q, c\}$. By replacing $R$ by a finite extension of $R$, there is a pointed stable curve $X^{\bp}_{R}$ over $R$. Write $K_{R}$ for the quotient field of $R$, $\overline K_{R}$ for an algebraic closure of $K_{R}$, and $k_{R}$ for the residue field of $R$. Moreover, we may assume that $\overline K_{R}$ contains $k_{q}$. For each $v_{X} \in v(\Gamma_{q})$, the smooth pointed stable curve $$\widetilde X^{\bp}_{\overline K_{R}, v_{X}}\defeq \widetilde X^{\bp}_{k_{q}, v_{X}} \times_{k_{q}} \overline K_{R}$$ of type $(g_{v_{X}}, n_{v_{X}})$ over $\overline K_{R}$ determines a morphism $$\spec \overline K_{R} \migi \mcM_{g_{v_{X}}, n_{v_{X}}} \migi M_{g_{v_{X}}, n_{v_{X}}}.$$

 
Let $\spec \overline K_{R} \migi H_{g_{v_{X}}, n_{v_{X}}}$ be a morphism obtained by restricting the natural morphism $$\bigsqcup \spec \overline K_{R}=\spec \overline K_{R} \times_{M_{g_{v_{X}}, n_{v_{X}}}} H_{g_{v_{X}}, n_{v_{X}}} \migi H_{g_{v_{X}}, n_{v_{X}}}$$ on a connected component of $\spec \overline K_{R} \times_{M_{g_{v_{X}}, n_{v_{X}}}} H_{g_{v_{X}}, n_{v_{X}}}$ such that the image of $\spec \overline K_{R} \migi H_{g_{v_{X}}, n_{v_{X}}}$ is contained in the image of $U_{q_{v_{X}}}  \migi H_{g_{v_{X}}, n_{v_{X}}}$. The morphism $\spec \overline K_{R} \migi H_{g_{v_{X}}, n_{v_{X}}}$ above induces a level $m$-structure $\tau_{\overline K_{R}}\defeq\tau_{H_{g_{v_{X}}, n_{v_{X}}}}\times_{H_{g_{v_{X}}, n_{v_{X}}}} \spec \overline K_{R}.$ 

By replacing $R$ by a finite extension of $R$, $\widetilde X^{\bp}_{\overline K_{R}, v_{X}}$ descends to a smooth pointed stable curve $$\widetilde X^{\bp}_{K_{R}, v_{X}}$$ over $K_{R}$, and the level $m$-structure $\tau_{\overline K_{R}}$ descends to a level $m$-structure $\tau_{K_{R}}$ on the smooth pointed stable curve $\widetilde X_{K_{R}, v_{X}}^{\bp}$ over $K_{R}$. Let $$X^{\bp}_{R, v_{X}}$$ be the pointed stable model of $\widetilde X^{\bp}_{K_{R}, v_{X}}$ over $R$. Note that, by the construction, $X^{\bp}_{R, v_{X}}$ is smooth over $R$. Then the level $m$-structure $\tau_{K_{R}}$ extends to a level $m$-structure $\tau_{R}$. Thus, for each $v_{X} \in v(\Gamma_{q})$, the smooth pointed stable curve $X^{\bp}_{R, v_{X}}$ over $R$ with the level $m$-structure $\tau_{R}$ determines a morphism $\spec R \migi H_{g_{v_{X}}, n_{v_{X}}}$ such that the image of the composition of the morphisms $$\spec R \migi \bigtimes_{v_{X} \in v(\Gamma_{q})} H_{g_{v_{X}}, n_{v_{X}}} \migi \bigtimes_{v_{X} \in v(\Gamma_{q})} \mcM_{g_{v_{X}}, n_{v_{X}}} \overset{\kappa_{X_{k_q}}}{\migi} \overline \mcM_{g, n} \overset{\overline \omega_{g, n}}\migi \overline M_{g, n}$$ is $\{q, c\}$. Moreover, by replacing $R$ by a finite extension of $R$, we may assume that the morphism $\spec R \migi H_{g_{v_{X}}, n_{v_{X}}}$ obtained above factors through the morphism $U_{q_{v_{X}}} \migi H_{g_{v_{X}}, n_{v_{X}}}.$ Thus, for each $v_{X} \in v(\Gamma_{q})$ and each $v_{Y} \in (f^{\rm sg}_{k_{q}})^{-1}(v_{X})$, we obtain a Galois covering $$f^{\bp}_{R, v_{Y}, v_{X}}\defeq f^{\bp}_{U_{q_{v_{X}}}, v_{Y}, v_{X}} \times_{U_{q_{v_{X}}}} \spec R: Y^{\bp}_{R, v_{Y}}\defeq Y^{\bp}_{U_{q_{v_{X}}}, v_{Y}} \times_{U_{q_{v_{X}}}} \spec R$$ $$\migi X^{\bp}_{R, v_{X}}\defeq X^{\bp}_{U_{q_{v_{X}}}, v_{X}} \times_{U_{q_{v_{X}}}} \spec R$$ of smooth pointed stable curves over $R$ with Galois group $G_{v_{Y}}$. Moreover, the clutching morphism $\kappa_{Y_{k_q}}$ implies that we may glue $\{Y_{R, v_{Y}}^{\bp}\}_{v_{Y} \in v(\Gamma_{Y_{k_{q}}^{\bp}})}$ in a way that is compatible with the
gluing of $\{Y_{k_{q}, v_{Y}}^{\bp}\}_{v_{Y} \in v(\Gamma_{Y^{\bp}_{k_{q}}})}$ that gives rise to $Y^{\bp}_{k_{q}}$. Then we obtain a pointed stable curve $Y^{\bp}_{R}$ over $R$ such that the following conditions are satisfied: (i) $Y^{\bp}_{R} \times_{K_{R}} \overline K_{R} \cong Y^{\bp}_{k_q}\times_{k_{q}} \overline K_{R}$ over $\overline K_{R}$; (ii) there exists a Galois admissible covering $f_{K_R}^{\bp}:Y^{\bp}_{K_{R}} \migi X^{\bp}_{K_{R}}$ over $K_{R}$ with Galois group $G$ such that $f_{K_{R}}^{\bp}\times_{K_{R}} \overline K_{R}=f_{k_{q}}^{\bp}\times_{k_{q}} \overline K_{R}$.

Then by applying Lemma \ref{lem-7-6},  there exists a Galois admissible covering $f^{\bp}_{R}: Y_{R}^{\bp} \migi X^{\bp}_{R}$ over $R$ with Galois group $G$ such that the restriction of $f^{\bp}_{R}$ on the special fibers is a connected Galois admissible covering over $k_{R}$ with Galois group $G$. This means that $G \in \pi_{A}^{\rm adm}(c)$. We  complete the proof of the proposition.
\end{proof}

\subsection{Continuity of $\pi_{g, n}^{\rm adm}$}\label{sec-7-2}

In this subsection, we prove the continuity of $\pi_{g, n}^{\rm adm}$.

\begin{lemma}\label{lem-7-8}
Let $v$ be a closed point of $H_{g, n}$, $\widehat \mcO_{H_{g, n},v}$ the completion of the local ring $\mcO_{H_{g, n},v}$, $\widehat V\defeq \spec \widehat \mcO_{H_{g, n},v}$ with the natural morphism $\widehat V \migi H_{g, n}$, and $X_{\widehat V}^{\bp} \defeq X^{\bp}_{H_{g, n}} \times_{H_{g, n}}\widehat V$ the smooth pointed stable curve over $\widehat V$ with a level $m$-structure $\tau_{\widehat V}\defeq\tau_{H_{g, n}}\times_{H_{g, n}} \widehat V$. Let $Y^{\bp}_{\widehat V}$ be a smooth pointed stable curve over $\widehat V$ and $$f_{\widehat V}^{\bp}: Y^{\bp}_{\widehat V} \migi X_{\widehat V}^{\bp}$$ a Galois admissible covering over $\widehat V$ with Galois group $G$. Then there exist a subring $A \subseteq \widehat \mcO_{H_{g, n}, v}$, a morphism $\alpha_{E}: E\defeq \spec A \migi H_{g, n}$, and a morphism $f^{\bp}_{E}: Y^{\bp}_{E} \migi X^{\bp}_{E}\defeq X^{\bp}_{H_{g, n}}\times_{H_{g, n}} E$ such that the following conditions are satisfied:

(i) $X_{E}^{\bp} \times_{E} \widehat V$ is isomorphic to $X^{\bp}_{\widehat V}$ over $\widehat V$, and the pulling-back of $f_{E}^{\bp}\times_{E} \widehat V$ via the natural morphism $\widehat V \migi E$ is isomorphic to $f^{\bp}_{\widehat V}$ over $\widehat V$;

(ii) $f_{E}^{\bp}$ is a Galois admissible covering with Galois group $G$ over $E$.
\end{lemma}

\begin{proof}
Similar arguments to the arguments given in the proof of \cite[Lemma 4.1]{Ste} imply the lemma holds. For readers convenience, we attach the proof.

By applying  \cite[Proposition 4.3 (1)]{V0}, there exists a subring $A' \subseteq \widehat \mcO_{H_{g, n}, v}$ which is of finite type over $k$ such that $Y_{E'}^{\bp}$ is smooth over $E'\defeq \spec A'$, that the Galois admissible covering $f^{\bp}_{\widehat V}$ can be descended to a finite morphism (a Galois covering) $f_{E'}^{\bp}: Y_{E'}^{\bp} \migi X^{\bp}_{E'}$ over $E'\defeq\spec A'$ with a level $m$-structure $\tau_{E'}$ on $X^{\bp}_{E'}$, and that $f_{E'}^{\bp}|_{e'}$ is  a Galois multi-admissible covering with Galois group $G$ over every $e' \in E'$. Moreover, by the construction, the pulling-back $f_{E'}^{\bp} \times_{E'} \widehat V$ via $\widehat V \migi E'$ is isomorphic to $f^{\bp}_{\widehat V}$ over $\widehat V$. Moreover, the smooth pointed stable curve $X^{\bp}_{E'}$ over $E'$ with the level $m$-structure $\tau_{E'}$ determines a morphism $\alpha_{E'}: E' \migi H_{g, n}$.

We denote by $v_{E'} \in E'$ the image of $v \in \widehat V$ via the natural morphism $\widehat V \migi E'$ which is a closed point of $E'$. \cite[Proposition 5]{Har} implies that, there exists an affine open subset $v_{E'} \in E \subseteq E'$ such that the fiber $Y^{\bp}_{e}\defeq Y^{\bp}_{E} \times_{E} e$ is geometrically irreducible over each {\it closed} point $e \in E$. We shall put $ A \defeq \mcO_{E}(E) \subseteq \widehat \mcO_{H_{g, n}, v}$. Moreover, since the underlying curve of $Y^{\bp}_{e}\defeq Y^{\bp}_{E} \times_{E} e$ is smooth over each $e$, we have that $Y^{\bp}_{e}$ is geometrically irreducible over each point $e \in E$. Thus, for each point $e \in E$, the restriction of $f_{E}^{\bp}\defeq f^{\bp}_{E'}\times_{E'}E$ on $e$ is a Galois admissible covering over $e$ with Galois group $G$. We put $$\alpha_{E}\defeq\alpha_{E'}|_{E}: E \migi H_{g, n}.$$ Then we obtain the desired curve. This completes the proof of the lemma.
\end{proof}

\begin{definition}
Let $q \in \overline M_{g, n}$ be an arbitrary point. For each $G \in \pi_{A}^{\rm adm}(q)$, we define $$U_{G}\defeq\{q' \in \overline M_{g, n} \ | \ G \in \pi_{A}^{\rm adm}(q') \}.$$ Moreover, we put $U^{\rm sm}_{G}\defeq U_{G} \cap M_{g, n}.$ 
\end{definition}

First, let us prove that $U_{G}^{\rm sm}$ is an open subset of $M_{g, n}$.

\begin{proposition}\label{prop-7-10}
Let $q$ be an arbitrary point of $M_{g, n}$ and $G \in \pi_{A}^{\rm adm}(q)$. Then $U_{G}^{\rm sm}$ is an open subset of $M_{g, n}$. 
\end{proposition}

\begin{proof}
To verify the proposition, Lemma \ref{lem-7-3} (or Proposition \ref{pro-7-7}) implies that it is sufficient to prove that, for every {\it closed point} $c \in U_{G}^{\rm sm}$, there exists an open subset $c\in U_{c} \subseteq M_{g, n}$ which is contained in $U_{G}^{\rm sm}$.  

Let $v \in H_{g, n}$ be a closed point such that $\omega_{g, n}^{(m)}(v)=c$. We maintain the notation introduced in Lemma \ref{lem-7-8}. Then we obtain an affine $k$-variety $E$ and a morphism $\alpha_{E}: E \migi H_{g, n}$ over $k$ such that $(\omega^{(m)}_{g, n}\circ\alpha_{E})(v_{E'})=c$. Moreover, since the image $\widehat V$ of the composition of the morphisms $\widehat V \migi E \overset{\alpha_{E}}{\migi} H_{g, n} \overset{\omega_{g, n}^{(m)}}{\migi} M_{g, n}$ is dense in $M_{g, n}$, the image of the composition of the morphisms $E \overset{\alpha_{E}}{\migi} H_{g, n} \overset{\omega_{g, n}^{(m)}}{\migi} M_{g, n}$ is also a dense constructible subset of $M_{g, n}$.

Write $W$ for the image of $E$ in $M_{g, n}$. Since $W$ is a constructible subset, we have that $$W=\bigcup_{i=1}^{r} W_{i}$$ is a finite disjoint union of local closed subsets $\{W_{i}\}_{i=1, \dots, r}$, of $M_{g, n}$. Without loss of generality, we may assume that $c \in W_{1}$. Since $W_{1}$ contains the image of $\widehat V$, we obtain that $W_{1}$ is an open subset of $M_{g, n}$. This completes the proof of the proposition.
\end{proof}

\begin{remarkA}
In \cite[Section 4]{Ste}, Stevenson proved that $U_{G}^{\rm sm}$ contains an open subset of $M_{g, n}$ when $n=0$.
\end{remarkA}

In the remainder of this subsection, we generalize Proposition \ref{prop-7-10} to the case of  arbitrary points of $\overline M_{g, n}$.

\begin{lemma}\label{lem-7-11}
Let $R$ be a complete discrete valuation ring, $K_{R}$ the quotient field of $R$, and $k_{R}$ the residue field of $R$ such that $k_{R}$ is an algebraically closed field containing $k$. Let $X^{\bp}_{R}$ be a pointed stable curve of type $(g, n)$ over $R$ and $$f_{k_{R}}^{\bp}: Y^{\bp}_{k_{R}} \migi X^{\bp}_{k_{R}}$$ a Galois admissible covering over $k_{R}$ with Galois group $G$. Then, by replacing $R$ by a finite extension of $R$, there exist a pointed stable curve $Y^{\bp}_{R}$ over $R$ and a Galois admissible covering $$f^{\bp}_{R}: Y^{\bp}_{R} \migi X^{\bp}_{R}$$ over $R$ with Galois group such that $f^{\bp}_{R}\times_{R}k_{R}=f^{\bp}_{k_{R}}$. 
\end{lemma}

\begin{proof}
Let $X_{\mcM'}^{\bp}$ be the versal formal deformation of the special fiber $X^{\bp}_{k_{R}}$ of $X^{\bp}_{R}$ over $\mcM'\defeq \spec \mcO_{k_{R}}[[t_{1}, \dots, t_{3g-3+n}]],$ where $\mcO_{k_{R}}$ is a regular local ring with maximal ideal $p\mcO_{k_{R}}$ and residue field $k_{R}$ (cf. \cite[p79]{DM}). The pointed stable curve $X^{\bp}_{R}$ over $R$ determines a morphism $\spec R \migi \mcM'$ such that $X^{\bp}_{\mcM'}\times_{\mcM'}\spec R$ is isomorphic to $X^{\bp}_{R}$ over $R$. Moreover, since $R \cong k_{R}[[t]]$, the morphism $\spec R \migi \mcM'$ factors through a morphism $$\spec R \migi \mcM\defeq \spec k_{R}[[t_{1}, \dots, t_{3g-3+n}]].$$ The natural morphism $\mcM \migi \mcM'$ induces a pointed stable curve $X^{\bp}_{\mcM} \defeq X_{\mcM'}^{\bp} \times_{\mcM'} \mcM$ over $\mcM$. 

Let $\overline \mcM_{g, n}^{\log}$ be the log stack obtained by equipping $\overline \mcM_{g, n}$ with the natural log structure associated to the divisor with normal crossings $\overline \mcM_{g, n} \setminus \mcM_{g, n}$. Then we obtain a log scheme $\mcM^{\log}$ whose underlying scheme is $\mcM$, and whose log structure is the pulling-back log structure induced by the natural morphism $\mcM \migi \mcM' \migi \overline \mcM_{g, n}.$ Moreover, we obtain a stable log curve $$X^{\log}_{\mcM}\defeq\overline \mcM^{\log}_{g, n+1}\times_{\overline \mcM_{g, n}^{\log}} \mcM^{\log}$$ over $\mcM^{\log}$ whose underlying curve is $X_{\mcM}$. Note that $X^{\log}_{\mcM}$ is log regular.

By replacing $\mcM^{\log}$ by a finite log \'etale covering $\mcN^{\log}$, and replacing $R$ by a finite extension of $R$, we obtain a morphism $\spec R \migi \mcN$ induced by the morphism $\spec R \migi \mcM$. We obtain a log scheme $s_{k_{R}}^{\log}$ whose underlying scheme is $s_{k_{R}}\defeq \spec k_{R}$, and whose log structure is the pulling-back log structure induced by the composition of the morphisms $s_{k_{R}} \migi \spec R \migi \mcN$. Moreover, the Galois admissible covering $f^{\bp}_{k_{R}}$ determines a log \'etale covering $f^{\log}_{k_{R}}: Y^{\log}_{k_{R}} \migi X^{\log}_{k_{R}}$ over $s_{k_{R}}^{\log}$ such that the underlying morphism of $f^{\log}_{k_{R}}$ is $f_{k_{R}}$.  By applying \cite[Corollary 1]{Hos}, there exists a Galois log \'etale covering $$f^{\log}_{\mcN}: Y^{\log}_{\mcN} \migi X^{\log}_{\mcN}\defeq X^{\log}_{\mcM} \times_{\mcM^{\log}} \mcN^{\log}$$ with Galois group $G$ over $\mcN^{\log}$ such that $$f^{\log}_{\mcN} \times_{\mcN^{\log}}s_{k_{R}}^{\log}: Y^{\log}_{\mcN} \times_{\mcN^{\log}} s_{k_{R}}^{\log} \migi X^{\log}_{\mcN} \times_{\mcN^{\log}} s_{k_{R}}^{\log}$$ is isomorphic to $f^{\log}_{k_{R}}$ over $s_{k_{R}}^{\log}$. Furthermore, by replacing $\mcN^{\log}$ by a finite log \'etale covering of $\mcN^{\log}$, we may assume that the underlying morphism of $f^{\log}_{\mcN}$ is a morphism of pointed stable curves over $\mcN$. 

Let $s^{\log}_{R}$ be the log scheme whose underlying scheme is $\spec R$, and whose log structure is the pulling-back log structure induced by the morphism $\spec R \migi \mcN$. Then we obtain a log \'etale covering $$f^{\log}_{\mcN} \times_{\mcN^{\log}}s_{R}^{\log}: Y^{\log}_{\mcN} \times_{\mcN^{\log}} s_{R}^{\log} \migi X^{\log}_{\mcN} \times_{\mcN^{\log}} s_{R}^{\log}$$ over $s_{R}^{\log}$. We denote by $$f^{\bp}_{R}: Y^{\bp}_{R} \migi X^{\bp}_{R}$$ the morphism induced by the underlying morphism of $f^{\log}_{\mcN^{\log}} \times_{\mcN^{\log}}s_{R}^{\log}$ over $R$. Since the special fiber $Y^{\bp}_{R}$ is geometrically connected, the Zariski main theorem implies that $Y^{\bp}_{R}\times_{R} R'$ is connected for every finite extension $R'$ of $R$. Thus, the generic fiber of $Y^{\bp}_{R}$ is geometrically connected. 


Let us prove that $f^{\bp}_{R}$ is a Galois admissible covering over $R$ with Galois group $G$. Note that we have a log scheme $s_{K_{R}}^{\log}$ whose underlying scheme is $s_{K_{R}}\defeq\spec K_{R}$, and whose log structure is the pulling-back log structure induced by the composition of the natural morphisms $s_{K_{R}}\migi \spec R \migi \mcN$. Then we see that $$f^{\log}_{\mcN} \times_{\mcN^{\log}}s_{K_{R}}^{\log}: Y^{\log}_{\mcN} \times_{\mcN^{\log}} s_{K_{R}}^{\log} \migi X^{\log}_{\mcN} \times_{\mcN^{\log}} s_{K_{R}}^{\log}$$ is geometrically connected Galois log \'etale covering over $s_{K_{R}}^{\log}$. This means that the underlying morphism of $f^{\log}_{\mcN} \times_{\mcN^{\log}}s_{K_{R}}^{\log}$ induces a Galois admissible covering over $K_{R}$ with Galois group $G$. This completes the proof of the lemma.
\end{proof}

Let $c \in \overline M_{g, n}$ be a {\it closed point} and $k_{c}= k$ the residue field of $c$. Then the closed point $c$ determines a pointed stable curve $$X_{k_c}^{\bp}=(X_{k_{c}}, D_{X_{k_{c}}})$$ over $k$. For each $v_{X} \in v(\Gamma_{c})$, let $\widetilde X^{\bp}_{k_{c}, v_{X}}$ be the smooth pointed stable curve of type $(g_{v_X}, n_{v_X})$ over $k_{c}$ associated to $v_{X}$. Then we obtain a morphism  $$c_{v_{X}}: \spec k_{c} \migi \mcM_{g_{v_{X}}, n_{v_{X}}}, \ v_{X} \in v(\Gamma_{c}),$$ determined by $\widetilde X^{\bp}_{k_{c}, v_{X}}$ over $k_{c}$. Write $c_{X_{k_{c}}}: \spec k_{c} \migi \overline \mcM_{g, n}$ for the morphism induced by $X^{\bp}_{k_{c}}$ over $k_{c}$. Moreover, $X^{\bp}_{k_c}$ over $k_{c}$ determines a clutching morphism $$\kappa_{X_{k_c}}: \bigtimes_{v_{X}\in v(\Gamma_{c})} \mcM_{g_{v_{X}}, n_{v_{X}}} \migi \overline \mcM_{g, n}$$ satisfying $\kappa_{X_{k_{c}}} \circ (\bigtimes_{v_{X} \in v(\Gamma_{c})} c_{v_{X}})=c_{X_{k_{c}}}$. We denote by $$M_{c} \defeq \text{Im}(\bigtimes_{v\in v(\Gamma_{c})} \mcM_{g_{v}, n_{v}} \overset{\kappa_{X_{k_c}}} \migi \overline \mcM_{g, n} \overset{\overline \omega_{g, n}}\migi \overline M_{g, n})$$ the image of the composition of the natural morphisms. Note that $\Gamma_{q}\cong \Gamma_{c}$ for each $q \in M_{c}$.

\begin{lemma}\label{lem-7-12}
We maintain the notation introduced above. Let $G \in \pi_{A}^{\rm adm}(c)$ be a finite group. Then $$U_{G} \cap M_{c}$$ contains an open subset of $M_{c}$ in which $c$ is contained.
\end{lemma}

\begin{proof}

Let $Y^{\bp}_{k_{c}}=(Y_{k_{c}}, D_{Y_{k_{c}}})$ be a pointed stable curve of type $(g_{Y}, n_{Y})$ over $k_{c}$ and $f^{\bp}_{k_c}: Y^{\bp}_{k_{c}} \migi X^{\bp}_{k_{c}}$ a Galois admissible covering over $k_{c}$ with Galois group $G$. Write $\Gamma_{Y^{\bp}_{k_c}}$ for the dual semi-graph of $Y^{\bp}_{k_{c}}$, and $f^{\rm sg}_{k_{c}}: \Gamma_{Y^{\bp}_{k_c}} \migi \Gamma_{c}$ for the map of dual semi-graphs induced by $f_{k_c}^{\bp}$. For every $v_{X} \in v(\Gamma_{c})$ and every $v_{Y} \in (f^{\rm sg}_{k_{c}})^{-1}(v_{X})$, let $\widetilde Y^{\bp}_{k_c, v_{Y}}$ be the smooth pointed stable curve of type $(g_{v_{Y}}, n_{v_{Y}})$ over $k_{c}$ associated to $v_{Y}$. Then $f_{k_{c}}^{\bp}$ induces a Galois multi-admissible covering $$f^{\bp}_{k_{c}, v_{X}}: \bigsqcup_{v_{Y} \in (f^{\rm sg}_{k_{c}})^{-1}(v_{X})} \widetilde Y_{k_{c}, v_{Y}}^{\bp} \migi \widetilde X^{\bp}_{k_{c}, v_{X}}$$ over $k_{c}$. Note that $\bigsqcup_{v_{Y} \in (f^{\rm sg}_{k_{c}})^{-1}(v_{X})} \widetilde Y_{k_{c}, v_{Y}}^{\bp} $ admits an action of $G$ induced by the action of $G$ on $Y^{\bp}_{k_{c}}$. This action induces an action of $G$ on the set $(f^{\rm sg}_{k_{c}})^{-1}(v_{X})$. For each $v_{Y} \in  (f^{\rm sg}_{k_{c}})^{-1}(v_{X})$, write $G_{v_{Y}} \subseteq G$ for the inertia subgroup of $v_{Y}$. Then we obtain a Galois admissible covering $$f_{k_{c}, v_{Y}, v_{X}}^{\bp}: \widetilde Y^{\bp}_{k_{c}, v_{Y}} \migi \widetilde X^{\bp}_{k_{c}, v_{X}}$$ over $k$ with Galois group $G_{v_{Y}}$. Write $c_{Y_{k_{c}}}: \spec k_{c} \migi \overline \mcM_{g_{Y}, n_{Y}}$ for the morphism determined by $Y^{\bp}_{k_{c}}$ over $k_{c}$, and $c_{v_{Y}}: \spec k_{c} \migi \overline \mcM_{g_{v_{Y}}, n_{v_{Y}}}$, $v_{Y} \in v(\Gamma_{Y_{k_{c}}^{\bp}})$, for the morphism determined by $\widetilde Y^{\bp}_{k_{c}, v_{Y}}$ over $k_{c}$. Then the pointed stable curve $Y^{\bp}_{k_{c}}$ over $k_{c}$ determines a clutching morphism 
$$\kappa_{Y_{k_c}}: \bigtimes_{v_{Y}\in v(\Gamma_{Y^{\bp}_{k_c}})}\mcM_{g_{v_{Y}}, n_{v_{Y}}} \migi \overline \mcM_{g_{Y}, n_{Y}}$$ satisfying $\kappa_{Y_{k_{c}}}\circ  (\bigtimes_{v_{Y} \in v(\Gamma_{Y^{\bp}_{k_{c}}})}c_{v_{Y}})=c_{Y_{k_{c}}}$. 

On the other hand, for each $v_{X} \in v(\Gamma_{c})$, we denote by $$q_{v_{X}} \in M_{g_{v_{X}}, n_{v_{X}}}$$ the image of $\omega_{g_{v_{X}}, n_{v_{X}}}\circ c_{v_{X}}$. Then the last paragraph of the proof of Proposition \ref{prop-7-10} implies that, for each $v_{X} \in v(\Gamma_{c})$, there exist an affine $k$-variety $E_{q_{v_{X}}}$ and a morphism $\alpha_{E_{q_{v_{X}}}}: E_{q_{v_{X}}} \migi H_{g_{v_{X}}, n_{v_{X}}}$ satisfying the following conditions:

(i) the image of $\alpha_{E_{q_{v_{X}}}}$ contains an open subset $U_{q_{v_{X}}}$ of $H_{g_{v_{X}}, n_{v_{X}}}$ such that the image $\omega_{g_{v_{X}}, n_{v_{X}}}^{(m)}(U_{q_{v_{X}}}) \subseteq M_{g_{v_{X}}, n_{v_{X}}}$ contains $q_{v_{X}}$;

(ii) we have a smooth pointed stable curve $X^{\bp}_{E_{q_{v_{X}}}} \defeq X^{\bp}_{H_{g_{v_{X}}, n_{v_{X}}}}\times_{H_{g_{v_{X}}, n_{v_{X}}}}E_{q_{v_{X}}}$ over $E_{q_{v_{X}}}$ with a level $m$-structure $\tau_{E_{q_{v_{X}}}}\defeq\tau_{H_{g_v, n_{v}}} \times_{H_{g_{v_{X}}, n_{v_{X}}}} E_{q_{v_{X}}}$;

(iii) for each $v_{Y}\in (f^{\rm sg}_{k_{c}})^{-1}(v_{X})$, there exists a Galois admissible covering $$f^{\bp}_{E_{q_{v_{X}}}, v_{Y}, v_{X}}: Y^{\bp}_{E_{q_{v_{X}}}, v_{Y}} \migi X^{\bp}_{E_{q_{v_{X}}}, v_{X}}$$ over $E_{q_{v_{X}}}$ with Galois group $G_{v_{Y}}$ such that the pulling-back of $f^{\bp}_{E_{q_{v_{X}}}, v_{Y}, v_{X}}$ to every point of $(\omega_{g_{v_{X}}, n_{v_{X}}}^{(m)}\circ \alpha_{E_{q_{v_{X}}}})^{-1}(q_{v_{X}})$ is isomorphic to the Galois admissible covering $f_{k_{c}, v_{Y}, v_{X}}^{\bp}$ over $k_{c}$ with Galois group $G_{v_{X}}$.

Then the image of the composition of the natural morphisms $$\bigtimes_{v_{X} \in v(\Gamma_{c})} U_{q_{v_{X}}} \migiinje  \bigtimes_{v_{X} \in v(\Gamma_{c})} H_{g_{v_{X}}, n_{v_{X}}} \migi \bigtimes_{v_{X} \in v(\Gamma_{c})} \mcM_{g_{v_{X}}, n_{v_{X}}} \overset{\kappa_{X_{k_c}}}{\migi} \overline \mcM_{g, n} \overset{\overline \omega_{g, n}}\migi \overline M_{g, n}$$ contains an open subset $W_{c}\subseteq M_{c}$ of $M_{c}$ in which $c$ is contained. To verify the lemma, it is sufficient to prove that $G\in \pi_{A}^{\rm adm}(c')$ for each {\it closed point} $c' \in W_{c}$.

Since $W_{c}$ is a $k$-variety, there exists a $k$-curve  $C' \subseteq W_{c}$ (not necessary irreducible) which contains $c$ and $c'$. Suppose that $C'$ is irreducible. Write $C$ for the normalization of $C'$, $c_{1}$ for a closed point of $C$ over $c$, and $c_{2}$ for a closed point of $C$ over $c'$. Let $R_{i}$, $i\in \{1, 2\}$, be a finite extension of $\widehat \mcO_{C, c_{i}}$ (then $R_{i}$ is a complete discrete valuation ring),  $K_{R_{i}}$ the quotient field of $R_{i}$, $\overline K_{R_{i}}$ an algebraic closure of $K_{R_{i}}$, and $k_{R_{i}}=k$ the residue field of $R_{i}$. 

By replacing $R_{1}$ by a finite extension of $R_{1}$, there is a smooth pointed stable curve $X^{\bp}_{R_{1}}$ over $R_{1}$ whose special fiber $X^{\bp}_{k_{R_{1}}}$ over the residue field $k_{R_{1}}=k$ of $R_{1}$ is isomorphic to $X_{k_c}^{\bp}$ over $k$. Lemma \ref{lem-7-11} implies that the Galois admissible covering $f^{\bp}_{k_c}$ over $k_{c}=k$ with Galois group $G$ can be lifted to a Galois admissible covering $$f^{\bp}_{R_{1}}: Y_{R_{1}}^{\bp} \migi X^{\bp}_{R_{1}}$$ over $R_{1}$ with Galois group $G$.

Let $\overline K$ be an algebraically closed field which contains $\overline K_{R_{1}}$ and $\overline K_{R_{2}}$. Since the admissible fundamental groups do not depend on the choices of base fields, the Galois admissible covering $f^{\bp}_{R_{1}}\times_{R_{1}}\overline K_{R_{1}}$ over $\overline K_{R_{1}}$ with Galois group $G$ induces a Galois admissible covering $f^{\bp}_{\overline K_{R_{2}}}: Y^{\bp}_{\overline K_{R_{2}}} \migi X^{\bp}_{\overline K_{R_{2}}}$ over $\overline K_{R_{2}}$ with Galois group $G$ such that $f^{\bp}_{\overline K_{R_{1}}} \times_{\overline K_{R_{1}}} \overline K$ is isomorphic to $f^{\bp}_{\overline K_{R_{2}}} \times_{\overline K_{R_{2}}} \overline K$ over $\overline K$. By replacing $R_{2}$ by a finite extension of $R_{2}$, there is a pointed stable curve $X^{\bp}_{R_{2}}$ over $R_{2}$ whose special fiber $X^{\bp}_{k_{R_{2}}}$ over the residue field $k_{R_{2}}=k$ of $R_{2}$ is isomorphic to $X_{k_{c'}}^{\bp}$ over $k$, and $f_{\overline K_{R_{2}}}^{\bp}$ can be descended to a Galois admissible covering $$f^{\bp}_{K_{R_{2}}}: Y^{\bp}_{K_{R_{2}}} \migi X^{\bp}_{K_{R_{2}}}$$ over $K_{R_{2}}$ with Galois group $G$. Moreover, for each $v_{X} \in v(\Gamma_{c})$ and each $v_{Y} \in (f_{k_{c}}^{\rm sg})^{-1}(v_{X})$, $f^{\bp}_{K_{R_{2}}}$ induces a Galois admissible covering $f^{\bp}_{K_{R_{2}}, v_{Y}, v_{X}}: Y^{\bp}_{K_{R_{2}}, v_{Y}} \migi X^{\bp}_{K_{R_{2}}, v_{X}}$ over $K_{R_{2}}$ with Galois group $G_{v_{Y}}$. By choosing a suitable level $m$-structure $\tau_{R_{2}, v_{X}}$ on $X^{\bp}_{R_{2}, v_{X}}$, we obtain a morphism $$l_{R_{2}, v_{X}}: \spec R_{2} \migi H_{g_{v_{X}}, n_{v_{X}}}$$ determined by $X^{\bp}_{R_{2}, v_{X}}$ over $R_{2}$ and $\tau_{R_{2}, v_{X}}$ such that image of $l_{R_{2}, v_{X}}$ is contained in $U_{q_{v_{X}}}$.

By replacing $R_{2}$ by a finite extension of $R_{2}$, we may assume that $Y^{\bp}_{K_{R_{2}, v_{Y}}}$, $v_{Y} \in v(\Gamma_{Y^{\bp}_{k_{c}}})$, has pointed stable reduction over $R_{2}$. Next, let us prove that $Y^{\bp}_{K_{R_{2}, v_{Y}}}$ has good reduction over $R_{2}$. 

If the image of $l_{R_{2}, v_{X}}$ is a constant morphism, then $Y^{\bp}_{K_{R_{2}, v_{Y}}}$ has good reduction over $R_{2}$. We may assume that $l_{R_{2}, v_{X}}$ is not a constant morphism. Let $\eta'_{v_{X}}$ be a generic point of $E_{q_{v_{X}}} \times_{H_{g_{v_{X}}, n_{v_{X}}}} \spec K_{R_{2}}$. Since $c' \in W_{c}$, there exists a closed point $s_{2, v_{X}} \in E_{q_{v_{X}}} \times_{H_{g_{v_{X}}, n_{v_{X}}}} k_{R_{2}} \migiinje E_{q_{v_{X}}}$ contained in $V_{\eta'_{v_{X}}}\defeq\overline {\{\eta'_{v_{X}}\}}$ such that $\alpha_{E_{q_{v_{X}}}}(s_{2, v_{X}})$ is equal to the image of $$\spec k_{R_{2}} \migiinje \spec R_{2} \overset{l_{R_{2}, v_{X}}}\migi H_{g_{v_{X}}, n_{v_{X}}},$$ where $\overline {\{\eta'_{v_{X}}\}}$ denotes the topological closure of $\eta_{v'_{X}}$ in $E_{q_{v_{X}}}\times_{H_{g_{v_{X}}, n_{v_{X}}}} R_{2}$. Since $R_{2} \cong k[[t]]$, the scheme-theoretic image of $l_{R_{2}, v_{X}}$ is a local ring of dimension one, $V_{\eta_{v}'}$ is an one dimension $k$-scheme. Write $A_{2, v_{X}}$ for the normalization of $\widehat \mcO_{V_{\eta'_{v_{X}}}, s_{2, v_{X}}}$. Note that $A_{2, v_{X}}$ is a complete discrete valuation ring, and the natural morphism $\spec A_{2, v_{X}} \migi \spec R_{2}$ is finite. We may assume that $A_{2, v_{X}} \subseteq \overline K_{R_{2}}$. Thus, we obtain a Galois admissible covering $$f^{\bp}_{A_{2, v_{X}}, v_{Y}, v_{X}}\defeq f^{\bp}_{E_{q_{v_{X}}}, v_{Y}}\times_{E_{q_{v_{X}}}} \spec A_{2, v_{X}}: Y^{\bp}_{A_{2, v_{X}}, v_{Y}} \migi X^{\bp}_{A_{2, v_{X}}, v_{X}}$$ of smooth pointed stable curves over $A_{2, v_{X}}$ with Galois group $G_{v_{Y}}$ such that $f^{\bp}_{A_{2, v_{X}}, v_{Y}, v_{X}}\times_{A_{2, v_{X}}} \overline K_{R_{2}}$ is isomorphic to $f^{\bp}_{\overline K_{R_{2}}, v_{Y}, v_{X}} $ over $\overline K_{R_{2}}$. This implies that $Y^{\bp}_{K_{R_{2}, v_{Y}}}$ has good reduction.

The clutching morphism $\kappa_{Y_{k_c}^{\bp}}$ implies that we may glue $\{Y_{R, v_{Y}}^{\bp}\}_{v_{Y} \in v(\Gamma_{Y_{k_c}^{\bp}})}$  in a way that is compatible with the
gluing of $\{Y_{K_{R_{2}}, v_{Y}}^{\bp}\}_{v_{Y} \in v(\Gamma_{Y^{\bp}_{K_{R_{2}}}})}$ that gives rise to $Y^{\bp}_{K_{R_{2}}}$.  Then we obtain a pointed stable curve $Y^{\bp}_{R_{2}}$ over $R_{2}$ such that

(i) $Y^{\bp}_{R_{2}} \times_{R_{2}}  K_{R_{2}} \cong Y^{\bp}_{K_{R_{2}}}$ over $K_{R_{2}}$;

(ii) there exists a Galois admissible covering $f_{K_{R_{2}}}^{\bp}:Y^{\bp}_{K_{R_{2}}} \migi X^{\bp}_{K_{R_{2}}}$ of pointed stable curves over $K_{R_{2}}$ which is a Galois admissible covering over $K_{R_{2}}$ with Galois group $G$ such that $f_{K_{R_{2}}}^{\bp}\times_{K_{R_{2}}} \overline K$ is isomorphic to $f^{\bp}_{K_{R_{1}}}\times_{K_{R_{1}}}\overline K$.

Then Lemma \ref{lem-7-6} implies that there exists a Galois admissible covering $f^{\bp}_{R_{2}}: Y_{R_{2}}^{\bp} \migi X^{\bp}_{R_{2}}$ over $R_{2}$ with Galois group $G$ such that the restriction of $f^{\bp}_{R_{2}}$ on the special fibers is a Galois admissible covering over $k_{R_{2}}$ with Galois group $G$. Since $X_{k_{R_{2}}}^{\bp}$ is isomorphic to $X_{k_{c'}}^{\bp}$ over $k=k_{c'}$, we have that $G \in \pi_{A}^{\rm adm}(c')$. This completes the proof of the lemma if $C' \subseteq W_{c}$ is irreducible. 

Suppose that $C'$ is not irreducible. Then we see that there is a set of closed points $\{c=c_{1}, c_{2}, \dots, c_{r}=c'\}$ of $C'$ such that $c_{i}$, $c_{i+1}$, $i\in \{1, \dots, r-1\}$, are contained in an irreducible component of $C'$. Then the lemma follows from the lemma when $C'$ is irreducible. This completes the proof of the lemma.
\end{proof}

\begin{corollary}\label{coro-7-13}
We maintain the notation introduced in Lemma \ref{lem-7-12}. Let $G \in \pi_{A}^{\rm adm}(c)$ be a finite group. Then $$U_{G} \cap M_{c}$$ is an open subset of $M_{c}$.
\end{corollary}

\begin{proof}
The corollary follows immediately from Proposition \ref{pro-7-7} and Lemma \ref{lem-7-12}.
\end{proof}

Next, we prove the main result of the present section. 

\begin{theorem}\label{continuousmap}
Let $q$ be an arbitrary point of $\overline M_{g, n}$ and $G \in \pi_{A}^{\rm adm}(q)$. Then $U_{G}$ is an open subset of $\overline M_{g, n}$. In particular, the maps $$\pi_{g, n}^{\rm adm}: \overline \mfM_{g, n} \migi \overline \Pi_{g, n},$$ $$\pi_{g, n}^{\rm sol}: \overline \mfM_{g, n} \migi \overline \Pi^{\rm sol}_{g, n}$$ defined in Section \ref{sec-5-2} are continuous.
\end{theorem}

\begin{proof}
For each $j \in \mbZ_{\geq 0}$, we put $M_{j}\defeq\{q' \in \overline M_{g, n} \ | \ \#e^{\rm cl}(\Gamma_{q'})=j\} \subseteq \overline M_{g, n},$ and denote by $\text{Gen}(M_{j})$ the set of generic points of $M_{j}$. Note that $M_{0}=M_{g, n}$, and that $M_{j} = \emptyset$ if $j >>0$. Since $M_{j'} \cap M_{j''} =\emptyset$ if $j' \neq j''$, we have that $$\overline M_{g, n}=\bigsqcup_{j \in \mbZ_{\geq 0}}M_{j}.$$ Moreover, for each $\eta_{j} \in \text{Gen}(M_{j})$, we put $M_{\eta_{j}}\defeq V_{\eta_{j}}\cap M_{j},$ where $V_{\eta_{j}}$ denotes the topological closure of $\eta_{j}$ in $\overline M_{g, n}$. Since $M_{\eta'_{j}} \cap M_{\eta''_{j}}=\emptyset$ if $\eta'_{j}\neq \eta''_{j}$ for each $j \in \mbZ_{\geq 0}$, we obtain a disjoint union $$M_{j}= \bigsqcup_{\eta_{j} \in \text{Gen}(M_{j})} M_{\eta_{j}}.$$ Then we obtain $$\overline M_{g, n}=\bigsqcup_{j \in \mbZ_{\geq 0}}\bigsqcup_{\eta_{j} \in \text{Gen}(M_{j})} M_{\eta_{j}}.$$ Thus, we have  $$U_{G}=\bigsqcup_{j \in \mbZ_{\geq 0}}\bigsqcup_{\eta_{j} \in \text{Gen}(M_{j})} M_{\eta_{j}} \cap U_{G}.$$ Corollary \ref{coro-7-13} implies that $M_{\eta_{j}} \cap U_{G}$ is an open subset of $M_{\eta_{j}}$. This means that $M_{\eta_{j}} \cap U_{G}$ is a constructible subset of $M_{\eta_{j}}$, and $M_{\eta_{j}} \cap U_{G}$ is stable under generization in $M_{\eta_{j}} \cap U_{G}$. Since $M_{\eta_{j}}$ is a constructible subset of $\overline M_{g, n}$, we obtain that $U_{G}$ is a constructible subset of $\overline M_{g, n}$.

Let $j' \geq j''$. If $M_{\eta_{j'}}$ is contained in the topological closure of $M_{\eta_{j''}}$ in $\overline M_{g, n}$ and $M_{\eta_{j'}} \cap U_{G} \neq \emptyset$, then Lemma \ref{lem-7-1} implies that $M_{\eta_{j''}} \cap U_{G} \neq \emptyset$. Since $M_{\eta_{j}} \cap U_{G}$ is stable under generization in $M_{\eta_{j}} \cap U_{G}$, $j\in \mbZ_{\geq 0}$, we obtain that $U_{G}$ is stable under generization in $\overline M_{g, n}$. Thus, $U_{G}$ is an open subset of $\overline M_{g, n}$. This completes the proof of the theorem.
\end{proof}

\newpage
\markright{ }

\bigskip

\bigskip

{\it Yu Yang}

\bigskip

{\it Address: Research Institute for Mathematical Sciences, Kyoto University, Kyoto 606-8502, Japan}

\bigskip

{\it E-mail: yuyang@kurims.kyoto-u.ac.jp}

\end{document}